\documentclass[12pt]{caltech_thesis}
\usepackage[hyphens]{url}
\usepackage{lipsum,amsthm,mathtools,amssymb,amsmath}
\usepackage{graphicx}

\usepackage{todonotes}

\usepackage[utf8]{inputenc}
\usepackage[T1]{fontenc}
\usepackage{newtxtext,newtxmath}

\usepackage[
    backend=biber,natbib,
    style=numeric
]{biblatex}

\newtheorem{theorem}{Theorem}[section]
\newtheorem{lemma}[theorem]{Lemma}

\newtheorem{corollary}[theorem]{Corollary}
\newtheorem{proposition}[theorem]{Proposition}
\newtheorem{definition}[theorem]{Definition}

\newtheorem{remark}[theorem]{Remark}
\numberwithin{equation}{section}
\setsecnumdepth{subsubsection}
\begin{document}

% Do remember to remove the square bracket!
\title{Singularity Formation: Synergy in Theoretical, Numerical and Machine Learning Approaches.}
\author{Yixuan Wang}

\degreeaward{Doctor of Philosophy}                 % Degree to be awarded
\university{California Institute of Technology}    % Institution name
\address{Pasadena, California}                     % Institution address
\unilogo{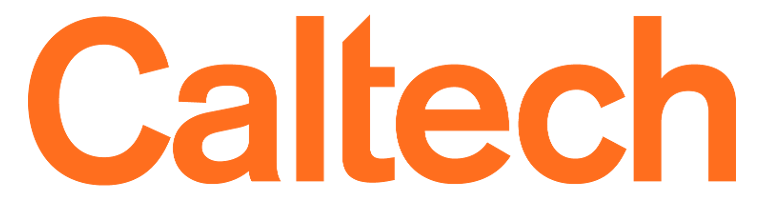}                                 % Institution logo
\copyyear{2026}  % Year (of graduation) on diploma
\defenddate{April 16th}          % Date of defense

\orcid{0000-0001-7305-5422}

%% IMPORTANT: Select ONE of the rights statement below.
\rightsstatement{All rights reserved}
% \rightsstatement{All rights reserved except where otherwise noted}
% \rightsstatement{Some rights reserved. This thesis is distributed under a [name license, e.g., ``Creative Commons Attribution-NonCommercial-ShareAlike License'']}

%%  If you'd like to remove the Caltech logo from your title page, simply remove the "[logo]" text from the maketitle command
\maketitle[logo]
%\maketitle

\begin{acknowledgements} 	
I am finishing up my thesis with extreme emotions, humility, and gratitude. What a journey! I spent some of my prime years towards my PhD, and it has been extremely rewarding: tons of ups and downs, but what is important is that it has been fun all along.

First and foremost, I express my sincere gratitude to the thesis and candidacy committee, along with other faculty at Caltech. Prof. Tom Hou, my supervisor, with great admiration, I grew and thrived under his training. Tom is a giant scholar, but even before that, he is a wonderful person: patient with the whole group, caring and devoting tons of his time, working extremely hard by himself, always there for us. It's through Tom that I learned how to devote myself to research and be a good person. I extend my thanks to the extended Tom's group, including his kind wife Dr. Chang, all the PhD students I have interacted with, and Mr. Choi. Prof. Anima Anandkumar, I have collaborated with her at Caltech most extensively besides Tom. Her vision and passion ignite me and guide me through my early stages with machine learning. Prof. Houman Owhadi, my chair and the very first person I knew at Caltech, has kindly provided guidance and interactions with his research group. Prof. Franca Hoffman, inspires me greatly through her outreach and caring for the community, where I have built the SIAM student chapter at Caltech myself. Prof. Kewen Wu, who will join online, my undergrad cohort, every conversation with him has been stimulating and insightful as a wonderful junior researcher and motivates me to stay eager and humble for learning. Prof. Andrew Stuart, who I had the pleasure to collaborate with, will unfortunately be away for my defense, has helped throughout my PhD.
   The help of CMS staff makes my stay in the department even nicer. I am forever in debt to Diana and Sydney, who have left. Diana really made Steele House her proper house and cared about each individual; everything about her is emotional. Then I would also like to thank Joy, Minzhi, and Jolene for their help.

I have to mention my amazing collaborators, many of whom are also close and dear friends. Prof. Ziming Liu, of course, what KAN I say but so many words: it is just amazing being the best friend all along, and thanks for being a role model. Prof. Zongyi Li, your humility and kindness always physically inform me. Prof. Van Tien Nguyen, thanks for being so encouraging and guiding me through the blowup world. Changhe Yang, the uniquely devoted collaborator, your dedication pays off greatly, and it was extremely rewarding on board the grand project you carried. Prof. Yifan Chen, we have known each other for so long and quite deeply; our relationship has always been multiscale, but thanks for watching out for me. Prof. Jiajie Chen, you are simply the strongest PDE person, period; thanks for being patient as a big brother. Tao Zhou, it was so much fun meeting you at a conference, becoming good friends, and working on good things together, like what academia should be. Prof. Jonathan Siegel, your insights and kindness help make our paper happen.
Dr. Valentin Duruisseaux, you are a dedicated and extremely talented collaborator, but you are also such a good friend and tennis buddy. I would also like to thank all the kind people writing me academic reference letters, although I decided to pursue math with AI at scale in the industry instead: Prof. Tom Hou, Prof. Anima Anandkumar, Prof. Andrew Stuart, Prof. Javier Serrano-Gomez, Prof. Hatem Zaag, Prof. Tristan Buckmaster, Prof. Arthur Jacot, Prof. Alexandru Ionescu, and Prof. Philippe Rigollet. The majority of math professors are just so nice to junior folks like me.

Caltech is great only because of its people. I would first like to thank my best friends in the department, without whom I would say, without exaggeration, that I would not have pulled through my PhD. Pau, my very first friend, we just have to pat each other on the shoulder to know that we are there. Matthieu, we have so much in common: the movie nights, concert nights, national parks, squash sessions, deep talks, and academic discussions; we are everything. Peicong, despite being a younger one in the group, you helped me so much emotionally. We experienced a lot of things together and had so much fun, and I really appreciated it.
Han, thanks for playing tennis, hanging out at concerts and others, but most importantly, for bringing in deep thoughts about life, about the journey. It is just a shame that despite many math discussions, I did not end up having a single paper with my best friends at Caltech. Then, I would say that my life beyond research also provided me with extreme joy and emotional support. I would like to thank the team, busy beaver, and my tennis buddies: Xijin, Haowen, Aaron, Daniel, Blake, Analiese, Sam, Steven, Sean, Max, Frid,  Edo, and Robert. I was often not the best player, to put it mildly, but playing and being on the court is the most important thing. Concert is a big part of life that makes me energetic, and I appreciate the company of my good friends: Tianshu, Berthy, Alex, and Jingyuan. Movies extend my and everybody's life, and my foodie friends, including Mansour and Steven, take it to another level. Cardio dance, especially Mariam, empowers me and gives me a little community, including the kind lady at the postal office who kindly let me mail my tax just before closing. It is just so easy to get scared and lost without a community. I cherish in particular my Chinese heritage, with my international horizon, just like the favorite dish I cook, Avocado Mapo Tofu, sometimes with a Parmesan twist. I would like to thank the Chinese community here, including Xinyi, Yulu, Ding, Hongkai, Ray, Bohan, Qiren, Chris, Xiaozhou, and Jim. 

I am extremely thankful to my friends all over the world: Nina, Luna, Roger, Anita, Vincent, Shawn, and Yihan. But especially 
friends from my college, Peking, the greatest place for an undergrad due to its comprehensive vision and greatest people, including the ones in the math department: Weiyuan, Zexuan, Ran, Jiasen, Jiayi, Jialei,  Shangqin,  Mengxi, and Yibo. I wanted to be emotional and write a paragraph to each one of you folks, but I will do (and have done) more on a postcard. You all have shaped me and witnessed me. Of course, I was grateful for the training and teachers there as well, including my undergrad advisor Prof. Ruo Li.

Finally, I thank my dearest family. I am extremely proud, but may not have expressed it due to the East Asian shyness, to have the best parents in the world. My dad Tao Wang is a silent protector for me, and my mom Fang Liu, oh my mom, is just so caring and kind. They are always there for me, as if I am the center of their universe. They are never too busy, and nothing about me is too difficult for them. I regretted that in my adolescence, we had some fights, but overall, we have a very lovely and healthy relationship. I want to look up to them so that in the future, I could be just as marginally good a parent.
 I am dedicating this thesis to my grandmom, Xia Lan. As I grew up with the four of us in a family, she is the closest person in my childhood. Ms. Lan, she grew up in troubled times in China, so she is not lucky to experience the world as I am fortunate to now. Her vision, open-mindedness, and insights are nonetheless profound and shaped me.

Moving forward, as I have always been, I will embark on and enjoy a love-hate, hopefully love on the upper side, relationship with Math, at DeepMind. It is a creative time, so do creative things. I am looking forward to the new paradigm of AI+math (PDEs): embracing the scale of things and hopefully witnessing grokking and phase-transition therein. I express my sincere gratitude to the senior people who helped me take this giant leap of faith, Dr. Yongji Wang, Prof. Chingyao Lai, Prof. Javier Serrano-Gomez,  Prof. Tristan Buckmaster, and my team leader Dr. Ray Jiang.
\end{acknowledgements}
\begin{abstract}
  This thesis develops numerical and theoretical approaches for understanding and analyzing singularity formation in Partial Differential Equations (PDEs). The singularity formation in the Navier-Stokes Equation (NSE) is famously challenging as one of the seven Clay Prize problems. Unlike simpler equations such as the Nonlinear Heat (NLH) or Keller-Segel (KS) equations, where formal asymptotics near blowup are better understood, the intrinsic complexity of NSE makes quantitative analytical treatment difficult, if not impossible, without numerical guidance. 

  Building on numerical insights, Chapter \ref{chap:2} introduces a robust analytical framework to simplify and systematize pen-and-paper proofs for simpler singular PDEs.  We present a novel approach based on enforcing vanishing modulation conditions for perturbations around approximate blowup profiles, complemented by singularly weighted energy estimates. Blowups are proven with a clear notion of stability, with rates automatically inferred,  without the need to know the asymptotics a priori or explicit spectral information of the linearized operator. We demonstrate the efficacy of our method on PDEs with complicated asymptotics, such as NLH and the Complex Ginzburg-Landau (CGL) equation, and address the open problem of singularity formation in the 3D KS equation with logistic damping. We also provide a roadmap for extending our techniques to singularities involving multiple scales.

  In Chapter \ref{chap:3}, we develop and refine numerical approaches that facilitate deeper insights into singularity formation. We demonstrate that machine learning methods significantly enhance our capability to identify and characterize potential blowup solutions with high precision. We improve on existing Physics-Informed Neural Network (PINN) and Neural Operator (NO)  frameworks. Moreover, we present a novel machine learning paradigm, the Kolmogorov-Arnold Network (KAN) architecture, whose interpretability and excellent scaling properties are achieved through learnable nonlinearities inspired by the Kolmogorov-Arnold representation theorem.
  
  %Chapter \ref{chap:3} discusses the transition from numerical insights to rigorous analytical proofs of singular behaviors, particularly in fluid dynamics. We settle an important open problem regarding the nonuniqueness of Leray-Hopf solutions of the 3D incompressible NSE through a rigorous, computer-assisted verification framework. This involves employing interval arithmetic and a fixed-point argument to rigorously verify the existence of solutions and the unstable eigenvalues of associated linearized operators around numerical candidates. The method critically relies on decomposing the linearized operator into coercive and compact components, facilitating numerical verification on a finite-dimensional approximation. 

  Chapter \ref{chap:4} introduces Exponential Multiscale Finite Element Method (ExpMsFEM), developed to efficiently solve challenging multiscale PDEs beyond elliptic problems, such as the Helmholtz equation.  We construct adaptive local bases, proving exponential convergence theoretically and demonstrating superior computational performance in practice. Like KAN, ExpMsFEM exemplifies how insights from theory can guide the design of high-performance solvers with theoretical guarantees.
\end{abstract}

%% Uncomment the `iknowhattodo' option to dismiss the instruction in the PDF.
\begin{publishedcontent}[iknowwhattodo]
% List your publications and contributions here.
\nocite{hou2024blowup,hou20242,chen2024stability,liu2025finite,kan1,liu2024kan,kanbias,chen2024exponentially,chen2023exponentially,chen2021exponential,liu2022second,maust2022fourier,li2024scale,wang2025high,hou2025nonuniqueness,rigas2025initialization,lee2025kano}
\end{publishedcontent}

\tableofcontents
\listoffigures
\printnomenclature

\mainmatter

\chapter{Introduction}
Much of this thesis is motivated by the challenge of understanding singular behaviors in nonlinear partial differential equations, exemplified by the 3D incompressible Navier–Stokes equations (NSE):
\begin{equation}\label{NSE}
\mathbf{u}_{t}+\mathbf{u} \cdot \nabla \mathbf{u}=-\nabla \mathbf{p}+\nu \Delta \mathbf{u}, \quad \nabla \cdot \mathbf{u}=0\,.
\end{equation}
The 3D NSE is a cornerstone of mathematical physics and fluid dynamics, yet foundational questions—such as whether smooth initial data can develop singularities in finite time—remain open. Singular behaviors like blowup and turbulence not only raise deep theoretical issues but also complicate numerical simulation and physical interpretation.

This thesis is guided by a central principle: that rigorous analysis and numerical insight are most powerful when developed in tandem. Our work illustrates how numerical methods not only provide guidance and experimental evidence for conjectures but also serve as a foundation for rigorous proofs and inspire new analytical frameworks. In turn, the theoretical insights yield more robust, interpretable, and generalizable numerical methods.

More generally, we consider finite-time singularities in evolution equations of the form \begin{equation}
    \label{intro:blowup}
    a_t=F(a),\quad \limsup_{t \rightarrow T^{-}}\|g(a(t))\|_{L^{\infty}}=\infty, \quad T<+\infty\,,
\end{equation}
where $g(a)$ is the quantity of interest, such as $a$ itself or its gradient $\nabla a$. 

A canonical mechanism for finite-time blowup is via self-similar solutions, which take the form \begin{equation}
    \label{intro:ss ansatz}
    a(x,t)=(T-t)^\alpha \tilde{U}(y)\,,\quad y=x(T-t)^\beta\,.
\end{equation} This ansatz exploits the scaling symmetry of the equation, reducing the original PDE to a time-independent profile equation for 
$\tilde{U}$. Typically, $\tilde{U}$
  is smooth and bounded, making both analysis and computation more tractable. Substituting into the original equation and equating powers of 
$T-t$, one obtains a steady-state equation of the form
\begin{equation}
    \label{intro:profile eqn}
    \mathcal{F}(\tilde{\lambda},\tilde{U})=0\,,
\end{equation}
where the scalar $\tilde{\lambda}$ related to $\alpha$ and $\beta$ governs the blowup rate, and the operator $\mathcal{F}$ consists of a scaling operator plus the original operator $F$. Specific forms of $\mathcal{F}$ will be illustrated in the examples that follow.

The study of blowup in the self-similar regime thus reduces to analyzing the existence and stability of solutions to \eqref{intro:profile eqn}. Since the profile is now a `nice' function, one naturally asks: can numerical methods help identify such profiles, and more importantly, can they inform or even enable rigorous proofs? 

A common strategy is to solve \eqref{intro:profile eqn} directly by casting it as an optimization problem for $\tilde{U}$ and the associated scaling parameter $\tilde{\lambda}$, where neural networks become relevant due to their expressive power and seamless integration to modern computations at scale. However, one must build on techniques tailored to these specifically challenging problems. In particular, we build on \textbf{operator learning} and introduce Scale-Informed Neural Operator, which encodes scaling symmetries of PDEs, and  Fourier Continuation for (Physics-Informed) Neural Operators for PDEs on the whole space. We also build on \textbf{high-precision training of Physics-Informed Neural
Networks (PINNs)}, where exact enforcement of asymptotics and hard constraints becomes crucial for upgrading numerical candidates to rigorous proofs. Motivated by the goal of symbolic profile recovery, we propose the \textbf{Kolmogorov-Arnold Network (KAN)} architecture, which offers greater interpretability and compositional expressivity compared to conventional neural networks.

Alternatively, one can introduce artificial time marching and solve the equation forward in time with appropriate initial data,  with the hope of converging to the steady profile. This dynamical approach, however, requires precise control over unstable directions, often arising from continuous symmetries such as scaling or translation. Ruling out these instabilities forms a key aspect of both our theoretical framework and our numerical design.

To upgrade from numerical candidates with residue to a rigorous proof, perturbative argument and stability analysis are required. One can either use a fixed-point argument to establish solutions to the nonlinear equation \eqref{intro:profile eqn} of the perturbation around the approximate profile, or analyze the convergence to the steady state of its time-dependent counterpart. 

Specifically, we consider a perturbative ansatz of the form $\tilde{\lambda}=\bar{\lambda}+\lambda,\tilde{U}=\bar{U}+U,$ and substitute it into  \eqref{intro:profile eqn}. This yields a perturbed equation of the form \begin{equation}
    \label{intro:perturb eqn}
    0=\mathcal{L}(\lambda,U)+\mathcal{N}(\lambda,U)+\mathcal{E}\,,
\end{equation}
where $\mathcal{L}$ is the linearized operator associated with $\mathcal{F}$ at $(\bar{\lambda},\bar{U})$, $\mathcal{N}$ is the nonlinear part, and $\mathcal{E}$ is the residue from the approximate profile, which can be made arbitrarily small via high-precision computation and verified with computer-assisted proofs.

A clear characterization of stability is therefore crucial in either approach. A common strategy is to introduce a carefully chosen energy norm $E$, typically induced by an inner product on a Hilbert space, under which the linear, nonlinear, and residual terms satisfy: \begin{equation}
    \label{intro:framework}(\mathcal{L},U)_E\leq -c_1\|U\|^2_{E}\,,\quad(\mathcal{N},U)_E\leq c_2\|U\|^3_{E}\,,\quad\|\mathcal{E}\|_{E}\leq c_3\,,
\end{equation}
with positive constants. Here, $c_3$ is very small with $c_1^2>4c_2c_3$. Then we can construct an energy ball in either approach and demonstrate the existence and stability of a solution to \eqref{intro:perturb eqn}. Of course, the construction of the norm is extremely challenging and requires a case-by-case analysis. Nonetheless, we provide a general recipe for a range of equations using \textbf{singular weights}, capturing localized behaviors of the solution.

In practice, one needs to deal with potentially unstable directions. Some of those are induced by scaling symmetries of the equation, and one can enforce \textbf{local orthogonality constraints} to those directions via the introduction of extra scaling parameters in the artificial time formulation, perturbing the symmetries, which constitutes one of our key theoretical and numerical contributions. Full stability and exact asymptotics can be inferred or precisely computed without prior knowledge, amenable to the case of implicit profiles and computer-assisted proofs. In other cases where only finite codimensional stability is expected, we couple the argument with a \textbf{topological argument} handling the unstable modes. For even more challenging scenarios where damping is hard to extract directly for the linearized operator $\mathcal{L}$, one can \textbf{decompose it into a coercive one plus a compact perturbation}. Applying finite-rank approximations to the compact part, one can formally regard it as a nonlinear component and only needs to verify boundedness inverting on the finite-dimensional space, where computer-assisted proofs can again be applied. We rigorously established the open problem of nonuniqueness of Leray solutions in the 3D incompressible Navier-Stokes equation \cite{hou2025nonuniqueness}. Intuitively, the onset of nonuniqueness occurs after the singularity time, and building upon previous works, we reduce the problem to constructing a forward self-similar profile whose associated linearized operator admits a positive eigenvalue. We numerically compute the approximate profile and the eigenpair first, before rigorously establishing the existence of the exact solutions by a perturbative argument.  

The remainder of this chapter is organized around the central themes introduced above. We highlight our contributions in developing both theoretical frameworks and numerical tools, including ideas that extend beyond self-similarity. Beyond singularity formation, we also point to broader connections with scientific computing, AI for science, and machine learning, where the methodologies developed here may provide useful perspectives and technical ingredients for related problems.
\section{Numerics Inspire Proofs: Blowup via Local Modulations}\label{sec:num1}
We will elaborate on the general theoretical framework for stability analysis around an approximate profile, based on the perturbative formulation in \eqref{intro:perturb eqn} and the damping/nonlinearity/residue bounds in \eqref{intro:framework}.
As mentioned before, the linear damping estimate constitutes the most crucial part. We built upon the series of seminal works developed by Chen-Hou-Huang \cite{chen2021finite,chen2019finite2,chen2020singularity,chen2022stable} to introduce singular weights in $L^2$ estimates and extract damping. As an illustrative example, consider the following rescaled Riccati-type equation:
\begin{equation}
\label{intro:ric}
\hat{u}_\tau=-\hat{u}-\frac{1}{2}z\hat{u}_z+\hat{u}^2\,,
\end{equation}
where it can be interpreted as the Riccati ODE $a_t=a^2$ in the self-similar ansatz $$a(x,t)=(T-t)^{-1}\hat{u}(x(T-t)^{-1/2},-\log(T-t))\,.$$
Equation \eqref{intro:ric} admits a family of explicit steady states $\bar{u}=(1+cz^2)^{-1}$. These profiles are known to be stable under even perturbations that vanish to fourth order at the origin.

 Towards rigorously proving it, we can establish linear dynamical stability for even perturbations of $O(z^4)$ at the origin by a $L^2$ estimate with singular weight $|z|^{-\alpha}$ capturing the vanishing order at the origin. To be specific,  we compute the damping estimate by an integration by parts near the origin as $$(\mathcal{L}u,u\rho)\approx(-1+\frac{(\rho z)_z}{4\rho}+2\bar{u})(u,u\rho)\approx\frac{5-\alpha}{4}(u,u\rho)\,.$$
For $\alpha>5$, which corresponds to the vanishing order of $u$ greater than $2$, we thus have damping near the origin. Notice that in this simple case, the initial vanishing order of the perturbation is preserved throughout the dynamics. To control the nonlinear terms, we complement the singularly-weighted $L^2$ estimate with a high-order $H^k$ estimate for sufficiently large $k$.

Building on this idea, we established the finite time singularity of a modified Hou-Li model \cite{hou2024blowup}. The Hou–Li model was introduced in \cite{hou2008dynamic} as a 1D reduction of the 3D axisymmetric Navier–Stokes equations along the symmetry axis. To be specific, the axisymmetric NSE (in cylindrical coordinates) takes the form \begin{eqnarray*}
&&u_{1, t}+u^{r} u_{1, r}+u^{z} u_{1, z} =2 u_{1} \psi_{1, z}+\nu\Delta u_1\,, \\
&&\omega_{1, t}+u^{r} \omega_{1, r}+u^{z} \omega_{1, z} =\left(u_{1}^{2}\right)_{z}+\nu\Delta \omega_1\,, \\
&&-\left[\partial_{r}^{2}+(3 / r) \partial_{r}+\partial_{z}^{2}\right] \psi_{1} =\omega_{1}\,,
%&& u^{r} =-r \psi_{1, z}\,, \quad u^{z}=2 \psi_{1}+r \psi_{1, r}\,.
\end{eqnarray*}
where $u_1 = u^\theta/r,\; \omega_1 = \omega^\theta,\; \psi_1 = \psi^\theta/r, $ and $u^\theta$, $\omega^\theta$, and $\psi^\theta$ are the angular velocity, angular vorticity, and angular stream function, respectively. The Hou-Li model, of the following form:
\begin{equation*}
\begin{aligned}
u_{1, t}+2 \psi_1 u_{1, z} &=2 \psi_{1, z} u_1 +\nu u_{1,zz}\,, \\
\omega_{1, t}+2 \psi_1 \omega_{1, z} &=\left(u_1^{2}\right)_{z}+\nu \omega_{1,zz}\,, \\
- \psi_{1,zz} &=\omega_1\,,
\end{aligned}
\end{equation*}
is a reduction in the sense that if ($\omega_1$, $u_1$, $\psi_1$) is an exact solution of the 1D model, we can obtain an exact solution of the 3D Navier-Stokes equations by using a constant extension in $r$. The model became particularly relevant in the investigation of potential interior singularities for NSE, and the authors in \cite{hou2008dynamic} established its well-posedness in $C^m$.

Since the discovery and proof of the Hou-Luo scenario of boundary singularity in the Euler equations \cite{luo2013potentially-2,luo2014potentially,chen2022stable}, Hou discovered new numerical evidence that the 3D axisymmetric Euler and Navier-Stokes equations develop potential singular solutions at the origin \cite{hou2022potential, hou2022potentially}. Hou’s numerical simulations revealed that the axial velocity  $u^{z}=2 \psi_{1}+r \psi_{1, r}$ near the peak of $u_1$ is significantly weaker than $2\psi_1$. This weakening is attributed to the shift between the maxima of  $\psi_1$ and $u_1$: $\psi_1$ peaks at $r=r_\psi$, while $u_1$ peaks at $r=r_u$ with $r_\psi < r_u$, implying $\psi_{1, r}<0$ near the maximal of $u_1$. Thus the axial velocity $u^z$ is actually weaker than  its value at ${r=0}$, highlighting that the original Hou–Li model—being confined to the symmetry axis—fails to capture this subtle, yet critical, three-dimensional effect. To better understand this phenomenon and its relation to singularity formation, we introduce the following 1D weak advection model.
\begin{equation}
\label{1dwp1}
\begin{aligned}
u_{ t}+2 a\psi u_{ z} &=2 u \psi_{ z}+\nu u_{zz}\,, \\
\omega_{ t}+2 a\psi\omega_{ z} &=\left(u^{2}\right)_{z}+\nu \omega_{zz}\,, \\
- \psi_{zz} &=\omega\,,
\end{aligned}
\end{equation}
where $a$ is a parameter that measures the relative strength of advection in the Hou-Li model. We established the following theorems in \cite{hou2024blowup}:
\begin{theorem}
For the weak advection model \eqref{1dwp1} in the inviscid case $\nu=0$, there exists a constant $\delta>0$ such that for $a\in(1-\delta,1)$, the weak advection model \eqref{1dwp} develops a finite time singularity for some $C^{\infty}$ initial data. Moreover, there exists a self-similar profile $({\omega}_{\infty},{\omega}_{\infty},{\omega}_{\infty})$ corresponding to a blowup that is neither expanding nor focusing. More precisely, the blowup solution to \eqref{1dwp} has the form
$$\omega(x,t)=\frac{1}{1+c_{u,\infty}t}{\omega}_{\infty}\,,u(x,t)=\frac{1}{1+c_{u,\infty}t}u_{\infty}\,,\psi(x,t)=\frac{1}{1+c_{u,\infty}t}{\psi}_{\infty}\,,$$
for some negative constant $c_{u,\infty}$ with a blowup time given by $T= \frac{-1}{c_{u,\infty}}$.
\end{theorem}
\begin{theorem}
Consider the Hou-Li model \eqref{1d} in the inviscid case $\nu=0$. For any $\alpha<1$, \eqref{1d} develops a finite time singularity for some $C^{\alpha}$ initial data. Moreover, there exists a $C^{\alpha}$ self-similar profile corresponding to a blowup that is neither expanding nor focusing.
\end{theorem}
\begin{theorem}
Consider the weak advection model \eqref{1dwp} with viscosity. There exists a constant $\delta_1>0$ such that for $a\in(1-\delta_1,1)$, the weak advection model \eqref{1dwp} develops a finite time singularity for some $C^{\infty}$ initial data.
\end{theorem}
We leave the details and the proofs concerning the Hou-Li model to Appendix \ref{append:hou-li}.

In the remainder of this section, we elaborate on our framework for proving finite-time blowup in PDEs beyond classical self-similar or fully stable scenarios, using local modulation and singularly-weighted energy estimates. In many cases—especially when viscous terms are present—the blowup deviates from exact self-similarity. To capture these behaviors, we introduce additional modulation parameters that enforce vanishing conditions locally, allowing us to infer both the blowup rate and its stability directly from the modulation dynamics. Those modulations, stemming from computational concerns of numerical stability, turn out to be keys to the theoretical stability. 
We illustrate this framework across a sequence of increasingly complex examples. In Subsection \ref{subsec:nlh}, we discuss the semilinear heat equation, which resembles the Riccati toy model \eqref{intro:ric}. Subsection \ref{subsec:cgl} follows as a much more technical resolution of the complex Ginzburg-Landau equation with its full stability without symmetry assumptions. We solve the open problem of the singularity formation of the 3D Keller-Segel equation with logistic damping in Subsection \ref{subsec:3dks}, where we combine modulation analysis with a topological argument to handle finite codimensional instability.  We emphasize key conceptual ideas here for clarity, and refer to Chapter \ref{chap:2} for a complete treatise of the semilinear heat equation based on \cite{hou20242}, while leaving the technically challenging parts of complex Ginzburg-Landau \cite{chen2024stability} and 3D Keller-Segel equations \cite{liu2025finite} to the original papers.
\subsection{Nonlinear heat equation}
\label{subsec:nlh}
We consider the semilinear heat equation, one of the simplest and best-studied PDEs with singularities. 
$$a_t=\Delta a+a^2\,,$$
with the blowup profile\begin{equation*}
a(x, t) \sim \frac{1}{T-t} \bar{u}\left(\frac{x}{\sqrt{(T-t)|\log (T-t)|}}\right)\,, \quad \bar{u}(\xi)=\frac{1}{1+|\xi|^2 / 8}.
\end{equation*}

One can see that its blowup scaling resembles the Riccati-type equation, except for a log-correction. Moreover, in its blowup scaling, the diffusion term becomes asymptotically small and to its leading order coincides with \eqref{intro:ric}. This log-correction is of course subtle to compute and capture, requiring special attention.

Inspired by \eqref{intro:ric}, we employ a similar stability argument using singularly-weighted $L^2$ norm and a high-order $H^k$ norm. Extra modulations are required to ensure the vanishing conditions of the perturbation, which are inspired by numerical computations aimed at capturing the steady-state profile. We thus introduce general rescaling beyond self-similarity as $$\hat{u}(z,\tau)=\hat{C}_u(\tau) a(\hat{C}_l(\tau)z,t(\tau))\,,$$ where 
\begin{equation*}
\hat{C}_u=\hat{C}_u(0)\exp{\Big(\int_0^\tau \hat{c}_u(s) ds\Big)},\;  \hat{C}_l=\exp{\Big(-\int_0^\tau\hat{c}_l(s) ds\Big)},\; t=\int_0^\tau \hat{C}_u(s) ds,
\end{equation*}
with $\hat c_u$ and $\hat c_l$ being determined. We introduced the extra degree of freedom $\hat{C}_u(0)$, which we will later choose to be small for the estimates of the viscous term.
The renormalized equation for $\hat u$ reads as
\begin{equation*}
    \hat{u}_\tau=\hat{c}_u \hat{u}-\hat{c}_l z\hat{u}_z+\hat{u}^2+\frac{\hat{C}_u}{\hat{C}_l^2}\hat{u}_{zz}.
\end{equation*}
We then consider the approximate profile $\bar{u}=(1+z^2/8)^{-1}$ that solves
    $$ \bar{c}_u\bar{u}-\bar{c}_l z\bar{u}_z+\bar{u}^2=0, \quad \textup{for}\;\; \bar{c}_u=-1\,,\bar{c}_l=1/2\,.$$
    If we enforce that $u$  is an even function satisfying $u(0, \tau) = u_{zz}(0, \tau) = 0$ for all time $\tau$, we have by the dynamic rescaling equation,
\begin{equation*}
    \hat{c}_u+\bar{u}(0)+\frac{\hat{C}_u \bar{u}_{zz}(0)}{\hat{C}_l^2 \bar{u}(0)}=0\,,\hat{c}_u-2\hat{c}_l+2\bar{u}(0)+\frac{\hat{C}_u (\bar{u}_{zzzz}(0)+{u}_{zzzz}(0))}{\hat{C}_l^2 \bar{u}_{zz}(0)}=0\,.
\end{equation*}
Defining 
$$\lambda (\tau)=\frac{\hat{C}_u (\tau)}{\hat{C}_l^2(\tau)}=\hat{C}_u(0)\exp{\Big(\int_0^\tau ({c_u(s)}+2c_l(s)) ds\Big)}\,,$$
we can simplify the normalization constraints into $${c_u}-\frac{1}{4}\lambda=0\,,{c_u}-2{c_l}-\big(\frac{ 3}{2}+4{u}_{zzzz}(0)\big)\lambda=0\,.$$
This gives 
\begin{equation*}
{c_u}=\frac{1}{4}\lambda\,, \quad {c_l}=-(\frac{ 5}{8}+2{u}_{zzzz}(0))\lambda\,,\end{equation*}
from which we {derive} the ODE for $\lambda$ as 
\begin{equation*}
    \lambda_\tau=\lambda(c_u+2c_l)=-(1+4{u}_{zzzz}(0))\lambda^2\,.
\end{equation*}
We can formally infer the log-correction from the asymptotics $\lambda\approx 1/\tau\approx|\log{(T-t)}|^{-1}$. Those extra modulations ensure that we can work in the singularly weighted space, and we generalize to high dimensions via different spatial scalings in different coordinates. Computations also corroborate our theoretical findings beyond the small perturbation regime; for details see Chapter \ref{chap:2}.
\subsection{Complex Ginzburg-Landau equation}
\label{subsec:cgl}
We consider the complex Ginzburg-Landau equation \begin{equation*} 
    \psi_t=(1+\imath\beta)\Delta \psi + (1+\imath \delta)|\psi|^{p-1}\psi- \gamma \psi, 
\end{equation*}
which reduces to the semilinear heat equation we have considered when $\beta=\delta=\gamma=0,p=2$, while connecting with the nonlinear Schr\"odinger equation in the limit $\beta, |\delta| \to \infty$. It enjoys a similar blowup law with log corrections, and we aim to establish full stability beyond the assumption of even symmetries.

Upon introducing the phase-amplitude decomposition, we obtain \begin{align*}
 &\partial_t u = \big[\Delta - |\nabla \theta|^2 \big]u - \beta\big(2 \nabla u \cdot \nabla \theta  + u \Delta \theta \big)  + u^p - \gamma u, \\
& u \partial_t \theta = \beta \big[\Delta - |\nabla \theta|^2   \big] u + 2 \nabla u \cdot \nabla \theta  + u \Delta \theta + \delta u^p .\end{align*}
One sees that to the leading order in $u$, it resembles the nonlinear heat equation again, and we can adopt a similar stability argument, with the following extra difficulties, however.

\begin{itemize}
\item We need to deal with full stability, corresponding to modulations of terms all the way to second-order at the origin, with $1+d+d(d+1)/2$ degree of freedom in $d$-dimension. We thus introduce a matrix rescaling in space as: $$U(z,\tau) = H(\tau) u\big( \mathbf{R}(\tau) z + V(\tau), t(\tau)\big)\,,$$$$\Theta(z, \tau) = 
%\red{\mu(\tau)} + 
 \theta\big( \mathbf{R}(\tau) z + V(\tau), t(\tau)\big), 
 \quad 
 t(\tau) = \int_0^\tau H^{p-1}(s) ds,
$$
where $ \mathbf{R}(\tau) \in \mathbb{R}^{d \times d}$ is upper triangular, $V(\tau) \in \mathbb{R}^d$ and $H(\tau) \in \mathbb{R}_+$.
 The modulation corresponds to the symmetries of the equation, with $d-1$ extra modulation parameters. Similarly to the semilinear heat equation, we now have a matrix $\mathcal{Q}=H^{p-1}\mathbf{R}^{-1}\mathbf{R}^{-1,T}$ capturing the log-correction and making the viscous terms formally small.

\item General nonlinearity and the phase equation necessitate a lower bound of $U$: we use the maximal principle and a weighted $L^\infty$ estimate.

\item Sharp decay estimates of $\nabla^i U$ require almost tight power for the weights and interpolation and embedding inequalities.

\item Coupling of amplitude and phase needs a top-order energy estimate with special algebraic structure to cancel out top-order terms in diffusion.
$$\begin{aligned}
    \mathscr{D}_U & =\Delta_\mathcal{Q} U-2\beta\langle\nabla U,\nabla\Theta\rangle_\mathcal{Q} -U\langle\nabla \Theta,\nabla\Theta\rangle_\mathcal{Q} -\beta U\Delta_\mathcal{Q} \Theta\,, \\
    \label{im_v1}
\mathscr{D}_\Theta & = \beta\frac{\Delta_\mathcal{Q}  U}{U}+2\frac{\langle\nabla U,\nabla\Theta\rangle_\mathcal{Q} }{U}-\beta \langle\nabla \Theta,\nabla\Theta\rangle_\mathcal{Q} +\Delta_\mathcal{Q} \Theta\,.
\end{aligned}$$
We construct top-order energy as 
$$(|\nabla^k W|^2,\rho_k)+(|\nabla^k \Phi|^2,U^2\rho_k)\,,$$
where $W,\Phi$ are the perturbations in $U=\bar{U}+W,\Theta=\bar{\Theta}+\Phi$.
 \end{itemize}
We refer to \cite{chen2024stability} for the technical details and estimates.
\subsection{3D Keller-Segel equation with logistic damping}
\label{subsec:3dks}
We now demonstrate the applicability of our method to singularity beyond full stability and without an explicit approximate profile. We focus on $3$D Keller-Segel equation, a classical chemotaxis model, with a quadratic logistic damping term $-\mu \rho^2$ modeling density-dependent mortality rate and show the existence of finite-time blowup solutions with nonnegative density and finite mass for a sharp range of $\mu \in \big[0,\frac{1}{3}\big)$. To be specific, consider the model 
$$
\begin{cases}
\partial_t \rho = \Delta \rho -  \nabla \cdot (\rho \nabla c) - \mu \rho^2, \\
\Delta c +\rho = 0.
\end{cases}$$

The logistic damping makes the model inherently nonlocal, even in the radial symmetric assumption, unlike the original Keller-Segel equation, where a partial mass formulation could be introduced to make it local. Nonlocality makes the construction of approximate profiles challenging, and we deploy a phase-portrait method to establish the existence of a profile, treating the viscous term again as asymptotically small. To be specific, the radial profile satisfies $$Q+\beta y \cdot \nabla_y Q = \nabla_y Q \cdot \nabla_y\Delta^{-1}_yQ + (1-\mu)Q^2,
$$ with asymptotics $Q=\frac{1}{1-\mu}-|y|^{2j_0}+O(|y|^{2j_0+2})$ near the origin, where $\beta=\frac{1}{3(1-\mu)}+\frac{1}{2j_0}$. The left-hand side of the profile equation corresponds to the blowup scaling. 

To establish the existence of the profile and make the viscous terms small, we require $$\frac{1}{3(1-\mu)} < \beta < \frac{1}{2},$$which coincides with the sharp range of $\mu$.  

We now conclude finite codimensional stability by introducing the rescaling $$y =\frac{x}{\lambda^{2\beta}}, \quad \frac{d\tau}{dt} = \frac{1}{\lambda^2}, \quad \frac{\lambda_\tau}{\lambda} = -\frac{1}{2},\quad \rho(t,x) = \frac{1}{\lambda^2} \Psi \left( \tau, y \right).$$
The equation in the rescaled variable is $$\partial_\tau \Psi = \lambda^{2-4\beta} \Delta \Psi - \Psi - \beta y \cdot \nabla \Psi + \nabla \Psi \cdot \nabla \Delta^{-1} \Psi + (1-\mu) \Psi^2.$$ 

Consider the perturbative ansatz $$\Psi=Q+\varepsilon^u(\tau,y) + \varepsilon^s(\tau,y),
\text{ with } \varepsilon^u(\tau,y) = \sum_{j=0}^K c_j(\tau) \chi(y) |y|^{2j}.$$
For sufficiently large $K$, we have stability of the perturbation $\varepsilon^s$ using singularly-weighted estimates, provided that $\varepsilon^s$ vanishes to $O(|y|^{2K+2})$ near the origin, which will be enforced by local modulations of the coefficients $c_j$.

We can solve the dynamic equations for $c_j$ as
$$\begin{cases}
    \dot c_j =\frac{j_0 -j}{j_0} c_j + \lambda^{2-4 \beta} [\Delta \Psi]_j +  [N(\varepsilon^u)]_{j}, & 0 \le j < j_0, \\
    \dot c_{j_0} =  \sigma_{0,j_0}c_0 +\lambda^{2-4 \beta} [\Delta \Psi]_{j_0}  +
    [N(\varepsilon^u)]_j
    , & j =j_0, \\
    \dot c_{j} =  \frac{j_0 -j}{j_0} c_j + \sum_{i=0}^{j-1} \sigma_{i,j} c_i + \lambda^{2-4 \beta} [\Delta \Psi]_{j}  + [N(\varepsilon^u)]_j, & j_0 < j \le K.
\end{cases}$$
One can perform a standard bootstrap argument coupled with topological arguments to derive a finite codimensional stability with dimension $j_0+1$. Notice that here $j_0 < j \le K$ corresponds to fake unstable directions. We leave the technical details to \cite{liu2025finite}.
\section{Numerics Provide Basis for Proofs: High Precision NNs and KANs}\label{sec:num2}
In this section, we discuss the numerical tools we have developed for efficient computations of the profile; see the prototype equation \eqref{intro:profile eqn}. We will highlight how machine learning tools can contribute to the computation, especially when the profile is unstable, where numerical time marching might struggle with stability issues. The tools developed have a much broader impact beyond the interest of singularity formation, on general applications in scientific computing,
AI for science, and general machine learning.

In Subsection \ref{subsec:no}, we discuss our contribution to operator learning, where we introduced Fourier continuation to deal with non-periodic problems, and also introduced Scale-Informed Neural Operator for scaling invariance of PDEs. We discuss the high-precision training of PINNs via exact enforcement of hard constraints and asymptotics in Subsection \ref{subsec:pinn}. KAN: Kolmogorov–Arnold Networks as a more interpretable and scaling-efficient neural network architecture, is introduced in Subsection \ref{subsec:kan}. 
We sketch the high-level ideas here, and refer to our articles \cite{maust2022fourier,li2024scale} for details on developments on operator learning, Appendix \ref{append:pinns} based on \cite{wang2025high} for high precision PINN training, and Chapter \ref{chap:3} based on \cite{kan1,liu2024kan,kanbias} for a complete overview of KAN.

\subsection{Neural Operator with Fourier continuation and scaling invariance}
\label{subsec:no}
Neural Operators (NOs) \cite{li2020fourier,kovachki2023neural} stand as an important idea to learn mappings directly between function spaces via parametrization of a nonlinear mapping in the spectral space by neural networks, thus offering the discretization-invariance property regardless of the physical mesh. In the context of PDE solving, NO aims to learn the solution operator directly, and can be coupled with the data loss and the physical loss, resulting in the Physics-Informed Neural Operator (PINO). Originally, PINO leverages differentiation in the Fourier spectral space and is best suited for problems defined on a periodic domain. We extended the methodology by introducing  Fourier continuation in FC-PINO \cite{maust2022fourier} to deal with problems in the whole space, particularly suited for the aforementioned problems of singularity formation. FC-PINO enables efficient learning of blowup profiles, and more generally, families of profiles across varying scaling parameters—thanks to the operator learning formulation. Related to the scale-invariance in PDEs, we introduced the Scale-Informed Neural Operator \cite{li2024scale}. Scale-consistency helps each model
extrapolate to unseen scales, including the challenging Helmholtz and Navier-Stokes simulations. We refer to the details in our papers \cite{maust2022fourier,li2024scale}.
Building on operator learning, which integrates supervised data with PDE constraints, we envision leveraging known profiles from similar equations to guide profile discovery in new regimes.

\subsection{High precision training of PINNs}
\label{subsec:pinn}
Motivated by the need for a numerical profile with high precision to be upgraded to a rigorous proof \eqref{intro:framework}, we developed a high precision training procedure of PINNs on an unbounded domain in \cite{wang2025high}. This approach emphasizes exact enforcement of asymptotic behavior, incorporation of hard constraints, and the use of advanced second-order optimization methods. We demonstrated an accuracy $4$-digits better than previous state-of-the-art by the important work \cite{wang2023asymptotic}, in terms of a 2D Boussinesq equation related to the 3D Euler singularity on the boundary \cite{chen2022stable}. We fully open-sourced our codebase to facilitate future developments and reproducibility. Full technical details are provided in Appendix \ref{append:pinns}. Beyond architectures and physical inductive biases, we also work on developing optimizers for scalable and high-precision training, including our Self-scaled SOAP optimizer, and work on more challenging PDEs with singularities, including the weak convection model introduced in \cite{liu2017spatial}.
\subsection{KAN}
\label{subsec:kan}
Partially motivated by an exact symbolic search for blowups, we proposed Kolmogorov–Arnold Networks (KANs) \cite{kan1}, as an interpretable alternative to modern machine learning architectures. The prevalent fully connected neural network, or Multi-Layer Perceptrons (MLPs), has learnable linear weights between layers and fixed nonlinear activation functions. It is the nonlinearity that renders them expressivity, albeit it often comes at the cost of a huge number of parameters to be trained. KANs leverage the  Kolmogorov-Arnold representation theorem, which reduces the representation of a multivariate function to a composition of univariate functions. This compositional construction aligns well with modern machine learning, and we generalize it to arbitrary depth. Better scaling laws are established under the compositionally smooth assumption empirically and theoretically, for a wide range of benchmarks of symbolic regression, PDE solving, and larger-scale problems of imaging. KANs demonstrate greater expressivity than MLPs and are particularly effective at capturing high-frequency components \cite{kanbias}. 

We further demonstrated the capacity of KANs to address a wide range of scientific problems \cite{liu2024kan}, highlighting their synergy with AI for science. We also developed effective initialization schemes for KAN training \cite{rigas2025initialization} and integrated KAN into Operator Learning, providing a more interpretable and expressive alternative to existing Neural Operator frameworks \cite{lee2025kano}. We leave the detailed discussions to Chapter \ref{chap:3}.

\section{Numerics with Provable Guarantees: EKHMC and ExpMsFEM}\label{sec:num3}
Finally, in the last section, we talk about how insights from theory can guide the design of competitive solvers with theoretical guarantees.

We introduce two such methods: a preconditioned second-order sampler, Ensemble Kalman Hamiltonian Monte Carlo (EKHMC), in Subsection \ref{subsec:EHMC}, and a state-of-the-art multiscale PDE solver, the Exponential Multiscale Finite Element Method (ExpMsFEM), in Subsection \ref{subsec:ExpMsFEM}.
We sketch the high-level ideas here and refer to Chapter \ref{chap:4} for a detailed exposition of ExpMsFEM, based on \cite{chen2021exponential, chen2023exponentially, chen2024exponentially}. A full treatment of EKHMC appears in Appendix \ref{append:ehmc}, based on \cite{liu2022second}.

\subsection{EKHMC}
\label{subsec:EHMC}
Consider the task of sampling from a target distribution $$\pi(q)=\frac{1}{Z_q}\exp(-\Phi(q))\,,$$
which arises naturally in Bayesian inverse problems, among other applications. The simplest approach is via simulation of the Langevin equation, which has the target density as its invariant distribution. Variants of Langevin dynamics have been proposed to boost convergence to the target distribution, among which are the Hamiltonian Monte Carlo with an auxiliary velocity variable introduced, and the preconditioning method to sample the target distribution with large condition numbers. Building upon the works of the Ensemble Kalman Sampler \cite{garbuno2020interacting}, where an interacting particle system was introduced to approximate the covariance as a preconditioner, we proposed a second-order sampler \cite{liu2022second} with empirically faster mixing rates.
\begin{equation*}
\begin{aligned}
\frac{\mathrm{d} q}{\mathrm{d} t}&= p\,, \\
\frac{\mathrm{d} p}{\mathrm{d} t}&=-\mathcal{C}_q(\rho)D \Phi(q)- {\gamma p+\sqrt{2\gamma \mathcal{C}_q(\rho)} \frac{\mathrm{d} W}{\mathrm{d} t}}\,.
\end{aligned}
\end{equation*}
We showed that EKHMC is affine-invariant and proposed a gradient-free version of the algorithm. In the context of Bayesian inverse problems with linear forward maps, we further demonstrated that the associated mean-field dynamics preserve Gaussianity and converge to the target distribution at a rate independent of the linear operator. We leave the detailed discussions to Appendix \ref{append:ehmc}.
\subsection{ExpMsFEM}
\label{subsec:ExpMsFEM}
Consider the model problem in a bounded domain $\Omega \subset \mathbb{R}^d$ with a Lipschitz boundary $\Gamma$. Here, $d=2$. For generality, the boundary can contain disjoint parts $\Gamma = \Gamma_1\cup \Gamma_2$ where $\Gamma_1$ corresponds to the Dirichlet boundary conditions and $\Gamma_2$ corresponds to the Neumann and Robin type boundary conditions. 
% Given the multiscale coefficient $A(x)$ and the perturbation with potential $V(x)$, 
The model equation is: 
\begin{equation}
\left\{
\begin{aligned}
-\nabla \cdot(A\nabla u)+Vu&=f, \ \text{in} \ \Omega\\
u&=0, \ \text{on} \ \Gamma_1\\
A\nabla u\cdot\nu&=\beta u, \ \text{on} \  \Gamma_2 \, .
\end{aligned}
\right.
\end{equation}
Here, $A, V,\beta$ are functions in $L^{\infty}(\Omega)$ and can be rough, which makes the solution oscillating and difficult to solve. The vector $\nu$ is the outer normal to the boundary. 
In particular, when $V=0$, the equation is the standard elliptic equation \cite{chen2021exponential}. If $Vu=-k^2u$ and $u$ is a complex-valued function, one obtains the Helmholtz equation \cite{chen2023exponentially} with wavenumber $k$.

Standard finite element methods often fail to resolve such problems accurately due to the limited regularity of the solution. The key challenge is to design numerical methods that adapt to the multiscale structure of the problem. The core idea is to construct local basis functions that adapt to the local coefficients of the equation. These bases allow efficient compression of the operator and can be reused to solve multiple problems with varying right-hand sides.

One key observation we proposed is to reduce the task of an accurate finite element solution to an accurate approximation. To be precise, we demonstrated via a posteriori estimates that once the bases can serve as good approximations to the solutions, the numerical solution automatically becomes a quasi-optimal approximation. Moreover, by a careful design of the local-to-global coupling, we showed that we only need to achieve a good approximation on each of the local patches.

Next, we discuss how to construct local approximations. This was achieved by a harmonic-bubble splitting. Roughly speaking, the harmonic part depends only on the information on the boundaries of the patch, and the bubble part depends only on the local right-hand side information. Even for challenging cases like the Helmholtz equation, we show that the bubble component remains small and that the harmonic part admits exponentially accurate approximation via oversampling. We rigorously proved exponential convergence and designed our algorithm to achieve such accuracy using only online basis functions, complemented by an offline modification of the right-hand side. We leave the detailed discussions and numerical experiments to Chapter \ref{chap:4}.

\chapter{Blowups via Local Modulations: Nonlinear Heat}
\label{chap:2}
In this chapter and the two subsequent chapters, we present our framework for inferring the precise law and stability of blowups beyond self-similarity via local modulations and singularly-weighted estimates, based on \cite{hou20242, chen2024stability, liu2025finite}. For problems with full stability, such as nonlinear heat (NLH) \cite{hou20242} or complex Ginzburg-Landau (CGL) \cite{chen2024stability} equations, we use a novel idea of enforcing stable normalizations for perturbations around the approximate profile and establish a weighted $H^k$ stability, thereby avoiding the use of a topological argument and the analysis of a linearized spectrum. This result generalizes the $L^2$-based stability argument to blowups that are not exactly self-similar and can be adapted to higher dimensions. Full anisotropic scaling can be introduced to establish full stability beyond any symmetry assumption. The log correction for the blowup rate is {automatically inferred via the local normalization conditions,} captured by the energy estimates and refined estimates of the modulation parameters.  

Crucially, this framework is applicable even when only numerical blowup profiles are available, providing a path toward rigorous analysis guided by computation. For cases involving finite-codimensional stability, our method can be combined with a topological argument. We illustrate those two points by solving the open problem of singularity formation in the 3D Keller–Segel (KS) equation with logistic damping \cite{liu2025finite}. Finally, numerical experiments confirm the effectiveness of our normalization strategy, even under large perturbations outside the strict theoretical regime.

For illustrative purposes, we present our framework in the setting of the NLH with even symmetry, covering both the fully stable and finite-codimensional cases in Section \ref{sec:slh1}, in this chapter. Finally, in Section \ref{subsec:future works}, we outline directions for extending this methodology to handle blowups involving multiple scales, and discuss prospects for turning the modulation approach into a robust computational algorithm beyond the search for singularities. The more technically involved cases, such as the CGL and 3D KS with logistic damping, are addressed in detail in the two following chapters respectively.

\section{Introduction}
We consider the semilinear heat equation \begin{equation}
    \label{sml}
    a_t=\Delta a + a^2,
\end{equation}
where $a(t): {\mathbb{R}}^n \to {\mathbb{R}}$, subject to the boundary condition $\lim_{|x| \to \infty} a(x,t) = 0$ and an initial data $a(0) = a_0$. 
Several blowup criteria were established in the past, dating back to the works of Kaplan \cite{Kapcpam63}, Fujita \cite{FUJsut66}, Levine \cite{LEVarma73},  Friedman-McLeod \cite{FMiunj85}, etc. We refer to the book \cite{quittner2019superlinear} for a comprehensive review on this subject.   Given our interest in the singularity formation, in particular in characterizing blowup solutions to \eqref{sml}, we only mention works in this direction in the following, where a precise law of the blowup can be identified.

Singularity formation in nonlinear PDEs is generally connected to a group of symmetries associated with the problem under consideration. For the classical nonlinear heat equation \eqref{sml}, it is invariant under the scaling transformation 
\begin{equation}\label{def:scaling}
\forall \lambda > 0, \quad a_\lambda(x,t) = \frac{1}{\lambda^2}a\Big( \frac{x}{\lambda}, \frac{t}{\lambda^2}\Big),
\end{equation}
in the sense that if $a$ is a solution to \eqref{sml}, so is $a_\lambda$ with the rescaled initial data $a_{0, \lambda} = \frac{1}{\lambda^2}a_{0}\big(\frac{x }\lambda \big)$. The earliest application of this scaling invariant property to equation \eqref{sml} that we are  aware of is the work by Berger-Kohn \cite{berger1988rescaling}, where the authors developed a so-called rescaling algorithm to capture the blowup profile 
\begin{equation}\label{stabProfile}
a(x, t) \sim \frac{1}{T-t} \bar{u}\left(\frac{x}{\sqrt{(T-t)|\log (T-t)|}}\right)\,, \quad \bar{u}(\xi)=\frac{1}{1+|\xi|^2 / 8}.
\end{equation}
This blowup behavior is in agreement with classification results rigorously established in the works of Filippas-Kohn \cite{filippas1992refined}, Filippas-Liu \cite{filippas1993blowup}, Herrero-Velazquez \cite{HVaihn93} and Velazquez \cite{velazquez1992higher} under the assumption of type-I blowup, namely that  $\limsup\limits_{t\to T}\|a(t)\|_{L^\infty} (T-t) < +\infty$, otherwise, the blowup is of type-II.
In particular, Herrero-Velazquez \cite{HVasnsp92} showed that the blowup behavior \eqref{stabProfile} is generic in dimension $1$, and they claimed the same for higher dimensions in an unpublished work. A rigorous construction was later established by Bricmont-Kupiainen \cite{bricmont1994universality} to provide concrete examples of initial data leading to blowup behaviors classified in the works mentioned above. The method developed in \cite{bricmont1994universality} was generalized in the work of Merle-Zaag \cite{merle1997stability} through spectral analysis and a topological argument to establish the existence and stability of blowup solutions to \eqref{sml} with the behavior \eqref{stabProfile}. 

It is worth mentioning that the scaling invariance \eqref{def:scaling} gives rise to the notion of energy-criticality in the sense that 
$$\|a_\lambda\|_{\dot{H}^1} = \lambda^{n - 6}\|a\|_{\dot{H}^1},$$
and the problem \eqref{sml} is called energy-critical if $n = 6$, energy-subcritical if $n \leq 5$ and energy-supercritical if $n \geq 7$. It is well known that the only type-I blowup occurs in the energy-subcritical case (i.e., $n \leq 5$) from the work of Giga-Kohn \cite{GKcpam85, GKiumj87, GKcpam89}, Giga-Matsui-Sasayama \cite{GMSiumj04} {(see also \cite{GMSmmas04} for the case of convex domains and Quittner \cite{Qduke21} for non-convex domains).}  The blowup in the energy-critical and supercritical cases is more subtle, where type-II blowup may also exist as predicted in \cite{HVcrasp94} and \cite{FHVslps00} through formal matching asymptotic expansions. {Concrete examples of initial data leading to type-II blowup for \eqref{sml} with a general nonlinearity $|a|^{p-1}a$ were exhibited in several works \cite{HVpre94}, \cite{Mma07}, \cite{Capde17},  \cite{CMRjams19},  \cite{PMWjfa21}, \cite{Sjfa12}, \cite{PMWam19}, \cite{PMWZdcds20}, \cite{Hihp20}, \cite{PMWZZarx20},  \cite{Hapde20} and references therein. In particular, the Type II blowup constructed in \cite{Hapde20} for the energy-critical case in dimension $n = 6$ corresponding to the exponent $p=2$ considered in this present work.} Partial classification of type-II blowup was provided in  \cite{MMcpam04, MMjfa09} and \cite{CMRcmp17}, even though a complete classification of all blowup patterns still remains open. 

In this chapter, we are interested in adopting the idea of dynamic rescaling to provide rigorous proofs for the semilinear heat equation with a clear notion of stability. Specifically, just like in numerical algorithms, we introduce proper rescaling conditions to ensure the stability of the perturbation around the approximate steady state, whose proof constitutes the main goal of this chapter.  We adopt a $L^2$-based stability analysis with properly chosen singular weights and normalization conditions, inspired by the line of work pioneered by \cite{chen2021finite, chen2019finite2}, and present our main result as Theorem \ref{thm1}. {We introduce the weighted Sobolev space $\mathcal{E}_k$ for $k = 2n + 10$: 
\begin{equation}\label{def:Ek}
\mathcal{E}_k = \Big\{u: \|u\|^2_{\mathcal{E}_k} = \| u\|^2_{\rho}+\mu\|\nabla^k u\|_{\rho_k}^2 < +\infty\Big\},
\end{equation}
where $\| \cdot\|_{\rho}$ stands for the weighted $L^2$-norm with the weight functions $\rho$, $\rho_k$ being defined in  \eqref{rho_def}, \eqref{rhok_def}, and the constant $\mu$ will be detailed later in \eqref{def_energy}.} 
\begin{theorem}\label{thm1} Let $k = 2n + 10$, there exist positive constants $C_0$ and $\lambda_0$ such that for $0<\lambda<\lambda_0$ and if the initial perturbation {$g$ is even\footnote{We call a multivariate function $a$ even if it is even in each one of the coordinates, namely if $a(\xi_1 x_1,\xi_2 x_2,\cdots,\xi_n x_n)=a( x_1,x_2,\cdots, x_n)$ for any $\xi_i\in\{-1,1\}$.} and satisfies $\|g\|_{\mathcal{E}_k} \leq C_0\lambda$,} the equation \eqref{sml} with initial data $$a(x,0)=\lambda^{-1}\big(\bar{u}(x)+g(x)\big),$$ 
{admits a solution $a(x,t)$ that blows up in some finite time $T$.}  Moreover, we have the following convergence in $\mathcal{E}_k$,
$$\lim _{t \rightarrow T}(T-t) a\left(((T-t)|\log (T-t)|)^{\frac{1}{2}} z, t\right)=\bar{u}(z).$$
\end{theorem}

\begin{remark}
    Using the scaling invariance of \eqref{sml}, we can introduce an initial rescaling in space (corresponding to introducing $\hat{C}_l(0)$  in the dynamic rescaling formulation \eqref{cl} in Section \ref{sec:dy}) to obtain a result comparable to the theorems in \cite{bricmont1994universality,merle1997stability} that characterize the blowup time precisely. Here we highlight obtaining the correct rate and, for the sake of simplicity, do not rescale in space at $t=0$.

\end{remark}\begin{remark} We work under the even assumption for illustration purposes of the dynamical rescaling technique developed in this chapter. We refer to our subsequent work \cite{chen2024stability} on the complex Ginzburg-Landau equation for a generalized dynamical rescaling technique to remove the even assumption, where translational and rotational modulations were introduced; see step 1 in Section 1.2 therein.
\end{remark}
Compared with most of the aforementioned works on semilinear equations that work in parabolic scaling, we work in variables that correspond to the true blowup scaling and obtain stability precisely with respect to the weighted $H^k$ norms we construct, instead of resorting to a topological argument to identify the existence of the initial data for blowup.
\subsection{Literature review and main contributions}
The idea of dynamic rescaling formulation or the modulation technique to study blowup was originally introduced in the numerical study of self-similar blowup of the nonlinear Schr\"odinger equation  \cite{weinstein1985modulational,soffer1990multichannel,soffer2006soliton,mclaughlin1986focusing,landman1988rate}. Later on, the formulation has been generalized to various dispersive problems, both as numerical techniques and as an analysis tool; see for example nonlinear Schr\"odinger equation \cite{kenig2006global,merle2005blow}, compressible Euler equations \cite{buckmaster2019formation},  and the nonlinear heat equation \cite{berger1988rescaling,merle1997stability}. Recently, this modulation technique has been adopted to establish self-similar singularity for incompressible Euler equations in \cite{ elgindi2021finite, chen2019finite2, chen2022stable}.

When the equation admits an analytical approximate profile for blowups, analyzing the spectrum of the linearized operator has proven to be useful for establishing the blowup in many cases; see for example, the semilinear heat equation \cite{bricmont1994universality, merle1997stability} and the 2D Keller-Segel equation \cite{raphael2014stability, collot2022refined}. While this methodology is powerful, it hinges on the fact that we are able to construct a simple and analytical approximate steady state and can analyze the spectrum of the linearized operator explicitly (for semilinear heat equations) or at least asymptotically (for Keller-Segel equations). In this chapter, we provide a proof of blowup for the semilinear heat equation without analyzing the eigenvalues or the eigenfunction of the linearized operator at all, and we rule out the unstable directions via a clear characterization of a singularly weighted Sobolev space, instead of using Brouwer's fixed-point theorem and a topological argument. {This framework can be adopted even if we
only have a numerical or implicit profile and do not have explicit information on the
spectrum of its linearized operator; see the follow-up work by the last author and collaborators \cite{liu2025finite} on 3D Keller-Segel equation with logistic damping.} 

On the other hand, a direct $L^2$ \cite{chen2021finite} or $L^\infty$-based \cite{chen2022stable} stability argument with appropriate normalization conditions has been proven successful, even if no explicit approximate steady state can be identified. In fact, they are often combined with a numerical profile and rigorous computer-assisted proofs. See  \cite{chen2021finite, chen2020singularity, hou2024blowup} for applications in various 1D models for the Euler equations, \cite{ chen2019finite2, chen2022stable} for 3D axissymmetric incompressible Euler equations, and { \cite{MRRSam22a} for the compressible Navier-Stokes equation (and the accompanying paper on the compressible Euler equation \cite{MRRSam22b}). The nature of the blowup in \cite{MRRSam22b, MRRSam22a} and the current work is quite similar in the sense that the dominant behavior of the rescaled equations is driven by the inviscid part, although in each case one must use different scaling to get the precise asymptotic behavior.} The methodology can be roughly summarized in the following two steps. Firstly, we link self-similar singularity with convergence to a steady state using the dynamic rescaling equation and obtain approximate steady states either analytically or numerically. Then, upon choosing appropriate normalization conditions, we can perform linear and nonlinear stability estimates to show that the perturbation around the approximate steady state will remain small. Therefore, we can obtain a self-similar blowup with rates prescribed by the normalizing constants.

%Up until now, this line of work has been somewhat limited to the self-similar setting {to the authors' awareness} since it was believed that one has to at least; maybe except for the work. 
This chapter adopts the $L^2$-based methodology to establish blowups beyond the self-similar setting; see also \cite{collot2023stable} on the 1D inviscid primitive equation and \cite{van2003formal} on the harmonic map heat flow, where log corrections were also observed. We show that one can obtain the correct blowup rate by imposing proper vanishing conditions on the perturbation, without a priori knowledge of a formal blowup rate. Compared with a self-similar blowup, the crucial difference is that now we have an algebraic, instead of exponential, convergence of the normalizing constants in the rescaled time $\tau$, inferred for example, by \eqref{cu}. 

Another contribution is that we introduce different spatial rescalings in $n$ different dimensions in Section \ref{highd}, giving enough degrees of freedom for the normalization conditions. Those different rescaling constants in different dimensions will indeed converge to the same rescaling constant close to the blowup time. This approach may shed some light on the generalization of the dynamic rescaling framework to higher dimensions for other problems. We believe our method is robust for type-I singularities, especially for problems having non-self-adjoint linear operators, see for example our follow-up work  \cite{chen2024stability} on the complex Ginzburg-Landau equation.
 Finally, we demonstrate our choice of normalization to be effective even beyond the regime of small perturbations, in Section \ref{sec:num} based on numerical experiments. 

 \subsection{Notations}
Throughout the chapter, we use $(\cdot,\cdot)$ to denote the inner product on $L^2(\mathbb{R}^n)$: $(f,g)=\int_{\mathbb{R}^n}fg$.     We use $C$ to denote absolute constants dependent only on the dimension $n$, which may vary from line to line. We use $A\lesssim B$ for positive $B$ to denote that there exists a constant $C>0$ such that $A\leq CB$.  We adopt the notation of the Japanese bracket as $\langle z\rangle=\sqrt{1+|z|^2}\,.$
\section{Dynamic Rescaling Formulation and Normalization Conditions}
\label{sec:dy}
In this section, we discuss our dynamic rescaling formulation. Via enforcing local vanishing modulation conditions, we derive the law of the blowup formally and motivate the choices of singular weights for stability analysis.
\subsection{1D case}
We focus on the 1D case first and generalize to higher dimensions in Section \ref{highd}.
For the semilinear heat equation \eqref{sml}, we introduce the dynamic rescaling formulation
$$\hat{u}(z,\tau)=\hat{C}_u(\tau) a(\hat{C}_l(\tau)z,t(\tau))\,,$$ with $$\hat{C}_u=\hat{C}_u(0)\exp{(\int_0^\tau \hat{c}_u d\tau)}\,, \hat{C}_l=\exp{(\int_0^\tau -\hat{c}_l d\tau)}\,,t=\int_0^\tau \hat{C}_u d\tau\,.$$
Here, we introduce an extra degree of freedom $\hat{C}_u(0)$ as in \cite{chen2020singularity, hou2024blowup}, which we will later choose to be small for the estimates of the viscous term.
We have $$\hat{u}_\tau=\hat{c}_u \hat{u}-\hat{c}_l z\hat{u}_z+\hat{u}^2+\frac{\hat{C}_u}{\hat{C}_l^2}\hat{u}_{zz}\,.$$
    We know there exists an approximate profile $\bar{u}=(1+z^2/8)^{-1}$ which solves $$ \bar{c}_u\bar{u}-\bar{c}_l z\bar{u}_z+\bar{u}^2=0\,, \bar{c}_u=-1\,,\bar{c}_l=1/2\,.$$
 By using the dynamic rescaling formulation, we reduce the problem of establishing a blowup in the physical variables and quantifying its blowup rate to the problem of establishing stability in the dynamic rescaling formulation. We want to show that $\hat{u}$ converges to the steady state $\bar{u}$ of the dynamic rescaling equation and the normalization constants also converge.
We put the ansatz  $$\hat{u}=\bar{u}+u\,, \hat{c}_u=\bar{c}_u+c_u\,, \hat{c}_l=\bar{c}_l+c_l\,.$$
We will elaborate on how to enforce normalization conditions $c_u$ and $c_l$ such that the perturbed solution $u$ of the dynamic rescaling equation is stable for all time. Namely, we want to show that $u$ remains small for all time, and thus $\hat{c}_u,\hat{c}_l$ will correspond to the correct blowup scaling.

 If we enforce that the even perturbation satisfies $u(0)$ and $u_{zz}(0)$ vanish for all time, by the dynamic rescaling equation we have 
\begin{equation}\label{cons_hat}
    \hat{c}_u+\bar{u}(0)+\frac{\hat{C}_u \bar{u}_{zz}(0)}{\hat{C}_l^2 \bar{u}(0)}=0\,,\hat{c}_u-2\hat{c}_l+2\bar{u}(0)+\frac{\hat{C}_u (\bar{u}_{zzzz}(0)+{u}_{zzzz}(0))}{\hat{C}_l^2 \bar{u}_{zz}(0)}=0\,.
\end{equation}

\noindent Define $$\lambda=\frac{\hat{C}_u}{\hat{C}_l^2}=\hat{C}_u(0)\exp{(\int_0^\tau {c_u}+2c_l d\tau)}\,,$$
we can simplify the normalization condition into $${c_u}-\frac{1}{4}\lambda=0\,,{c_u}-2{c_l}-(\frac{ 3}{2}+4{u}_{zzzz}(0))\lambda=0\,.$$
Therefore we solve \begin{equation}\label{cu}{c_u}=\frac{1}{4}\lambda\,,{c_l}=-(\frac{ 5}{8}+2{u}_{zzzz}(0))\lambda\,.\end{equation}
And thus we can simplify the ODE for $\lambda$ as \begin{equation}\label{lamb}
    \lambda_\tau=\lambda(c_u+2c_l)=-(1+4{u}_{zzzz}(0))\lambda^2\,.
\end{equation}
\begin{remark}
    Notice that formally, when the perturbation is small, we can further solve this ODE to obtain $$\lambda\approx1/\tau = \frac{1}{|\log{(T-t)}|}\,.$$ Therefore, the effect of the viscosity terms can be treated perturbatively. We will make this heuristic argument rigorous later on by choosing $\hat{C}_u(0)$ small.
\end{remark}
%We  
\begin{remark}\label{rmk2}
To motivate our choice of normalization conditions, we can plug in an ansatz $\rho=z^{-\alpha}$ for the singular weight we use in the $L^2$ estimate, and calculate linear damping for the evolution of $u$. Via an integration by parts, we know that up to the linear part near the origin, we have $$(u_\tau,u\rho)\approx(-1+\frac{1}{4}\frac{(\rho z)_z}{\rho}+2)(u,u\rho)=(1-\frac{\alpha-1}{4})(u,u\rho)\,.$$ We calculate that we need $\alpha>5$ to extract linear damping, and therefore we need to enforce the perturbation $u$ to vanish to higher orders.

 Of course, we need to take care of nonlinear estimates. Thus, the singular weights cannot be as simple as $\rho=z^{-6}$, but this serves as the starting point of our stability analysis.
\end{remark}

\subsection{nD case}\label{highd}
A crucial idea in the $n$-dimensional case is that we introduce $n$ different scaling parameters in different directions. This gives us more freedom to enforce the normalization conditions and obtain a perturbation with the same vanishing orders. Consider
$$\hat{u}(z,\tau)=\hat{C}_u(\tau) a(\hat{C}_l^1(\tau)z_1,\hat{C}_l^2(\tau)z_2,\cdots,\hat{C}_l^n(\tau)z_n,t(\tau))\,,$$ with the same $\hat{C}_u$ and $t(\tau)$ defined as before, and \begin{equation}
    \label{cl}\hat{C}_l^i=\exp{(\int_0^\tau -\hat{c}_l^i d\tau)}\,.
\end{equation}  The equation for $\hat{u}$ is \begin{equation}\label{drf}\hat{u}_\tau=\hat{c}_u \hat{u}-\sum_i\hat{c}_l^iz_i \hat{u}_{i}+\hat{u}^2+\sum_i\lambda_i\hat{u}_{ii}\,,\end{equation}
where we use the short-hand notation for partial derivatives: we denote $f_{i}=\partial_{z_i}f$ and $f_{ij}=\partial_{z_j}\partial_{z_i}f$.
Using the same radial approximate steady state $\bar{u},\bar{c}_l,\bar{c}_u$ and a similar ansatz
\begin{equation}
\label{ansatz}\hat{u}=\bar{u}+u\,, \hat{c}_u=\bar{c}_u+c_u\,, \hat{c}_l^i=\bar{c}_l+c_l^i\,,
\end{equation}
we can enforce the same normalization condition that $u$ is of $O(|z|^4)$. Notice that if we choose  $u$ to be an even perturbation, we only need to enforce $u(0)=0$ and ${u}_{ii}(0)=0$. Those $n+1$ constraints can be solved exactly to obtain \begin{equation}
\label{cun}
    {c_u}=\frac{1}{4}\sum_i\lambda_i\,,{c_l^i}=-\sum_j\lambda_j(\frac{ 1+4\delta_{ij}}{8}+2{u}_{iijj}(0))\,,
\end{equation}
where $\delta_{ij}=1$ if $i=j$, and $0$ otherwise. Here we define $$\lambda_i=\frac{\hat{C}_u}{(\hat{C}^i_l)^2}=\hat{C}_u(0)\exp{(\int_0^\tau {c_u}+2c^i_l d\tau)}\,,$$
and we obtain  the ODE for $\lambda$ as follows:\begin{equation}
\label{lambn}\partial_\tau\lambda_i=\lambda_i(c_u+2c^i_l)=-(\sum_j4{u}_{iijj}(0)\lambda_j+\lambda_i)\lambda_i\,.
\end{equation}
Notice that  \eqref{cun}, \eqref{lambn} are consistent with  \eqref{cu}, \eqref{lamb} in the 1D case.
 \section{Stability of Perturbation and Finite Time Blowup}
 Building upon the general strategy of a weighted $L^2$-based stability argument as in \cite{chen2021finite, chen2019finite2}, we will prove Theorem \ref{thm1} in this section. We denote $\lambda=\max_i{\lambda_i}$.

\subsection{$L^2$ stability analysis}

Plugging in the ansatz \eqref{ansatz} into the dynamic rescaling equation \eqref{drf} and using the fact that $\bar{u}$ is an approximate steady state, we write down the evolution equation for $u$ as follows:
\begin{equation*}
    {u}_\tau=L(u)+N(u)+\sum_i F_i(z, \tau)+ \sum_i \lambda_i V_i(u)\,,
\end{equation*}
where we reorganize the different terms into the linear, nonlinear, error, and viscous terms respectively as $$\quad L=(-1+{c_u}) {u}- \sum_i(\frac{1}{2}+{c_l^i})z_i {u}_{i}+2\bar{u}u\,,\quad N={u}^2\,,$$
$$ F_i=\frac{1}{4}\lambda_i\bar{u}-{c_l^i}z_i\bar{u}_i+\lambda_i (\bar{u}_{ii}+\sum_j\frac{1}{2}u_{iijj}(0)z_j^2\chi(|z|))\,,$$
$$ V_i={u}_{ii}-\sum_j\frac{1}{2}u_{iijj}(0)z_j^2\chi(|z|)\,.$$
Here $\chi(z)$ is a 1D even smooth cutoff function such that $\chi(z)=0$ for $|z|\geq 2$ and $\chi(z)=1$ for $|z|\leq 1$. We introduce such a cutoff function to make each one of the four terms integrable in the weighted $L^2$ space. We name and group the terms in such a way that is convenient for our analysis. The ``linearized operator $L$'' is obtained by treating the scaling parameters $c_u$ and $c_l$ as known parameters, although they actually depend on $u$. As a result, the ``linearized operator $L$'' actually contains nonlinear terms in the original physical variables.

To show that the dynamic rescaling equation is stable for even perturbations and converges to a steady state, we will perform a weighted $L^2$ estimate with a singular weight $\rho$ and a weighted $L^2$ norm\begin{equation}
\label{rho_def}\rho=|z|^{-5-n}+10^{-3}|z|^{1-n}\,,\quad \|f\|_{\rho}=(f^2,\rho)^{1/2}\,.
\end{equation}
We choose such a weight to extract damping near the origin as in Remark \ref{rmk2}, while also having good control of growth at infinity, to make $L^{\infty}$ and thus the nonlinear estimates easier.
Via an integration by parts\footnote{We can justify the integration by parts here, and similarly the subsequent ones, by a density argument. Notice that both terms are indeed integrable. For compactly supported smooth functions, we can, of course, do integration by parts since the boundary integral vanishes as the radius goes to infinity. For general cases, we can take a sequence of compactly supported smooth functions that approximates the functions in the weighted spaces and then take limits. By the Cauchy-Schwarz inequality, the integrals also converge, and we validate the integration by parts. Such approximate functions exist since we can first truncate the function in an annulus $\epsilon\leq|z|\leq1/\epsilon$ such that the norm outside of the annulus is small; then the weighted norm is equivalent to a regular $L^2$-norm and we can approximate by compactly supported smooth functions inside the annulus \cite{taylor2006measure} and zero extend to the whole space.}, we have a standard $L^2$ estimate for the linear part:
\begin{equation*}
    (L,u\rho)=([(-1+{c_u})+\frac{1}{2} \sum_i(\frac{1}{2}+{c_l^i})\frac{(z_i\rho)_i}{\rho}+2\bar{u}]u,u\rho)\,.
\end{equation*}
We plug in the singular weight \eqref{rho_def}, using $O(\lambda)$ notations due to the form \eqref{cun}, and simplify as
$$\begin{aligned}
    &(-1+{c_u})+\frac{1}{2} \sum_i(\frac{1}{2}+{c_l^i})\frac{(z_i\rho)_i}{\rho}+2\bar{u}\\=&O((1+|\nabla^4u(0)|)\lambda)-\frac{1}{4}+\frac{0.006}{4 (10^{-3}+z^{-6})}-\frac{2z^2}{8+z^2}\,.
\end{aligned}$$
By a straightforward computation and the AM-GM inequality we have $$0.006(8+z^2)-4 (10^{-3}+z^{-6})2z^2=0.048-0.002z^2-8z^{-4}\leq 0\,.$$
Therefore, we have the simple linear stability
\begin{equation*}
    (L,u\rho)\leq(-\frac{1}{4}+O((1+|\nabla^4u(0)|)\lambda))(u,u\rho)\leq(-\frac{1}{4}+C(1+|\nabla^4u(0)|)\lambda)\|u\|_{\rho}^2\,.
\end{equation*}

The estimate of the nonlinear term is straightforward:
$$(N,u\rho)\leq\|u\|_{\infty}\|u\|_{\rho}^2\,.$$

Now we compute the error terms. We compute $\bar{u}_{ii}=-\frac{\bar{u}^2}{4}+\frac{z_i^2\bar{u}^3}{8}$. As a consequence, we use Fubini's principle to arrive at $$\begin{aligned}
\sum_{i}F_i&=\frac{\sum_i\lambda_i}{4}(\bar{u}+\frac{1}{2}z\cdot\nabla\bar{u}-\bar{u}^2)+\sum_i\frac{\lambda_i}{2}(z_i\bar{u}_i+\frac{z_i^2\bar{u}^3}{4})\\&+\sum_j\sum_i\frac{\lambda_i u_{iijj}(0)}{2}(4z_j\bar{u}_j+z_j^2\chi(|z|))\,.
\end{aligned}$$

Using the fact that $\bar{u}$ is an approximate solution to the dynamic rescaling equation, we know that the first term vanishes. We can compute to simplify $$\sum_{i}F_i=-\sum_i\lambda_iz_i^2|z|^2\frac{\bar{u}^3}{64}+\sum_j\sum_i\frac{\lambda_i u_{iijj}(0)z_j^2}{2}(\chi(|z|)-\bar{u}^2)\,.$$
We know that the error term is $O(|z|^4)$ at $z=0$ and $O(|z|^{-2})$ at $\infty$; thus lies in the weighted space. We conclude that $$(\sum_{i}F_i,u\rho)\leq C(1+|\nabla^4u(0)|)\lambda\|u\|_{\rho}\,.$$

The viscous part is more subtle since we need to deal with the singularity carefully. Notice that $$|\nabla^2\rho|\lesssim|\rho/|z|^2|\,.$$ We do integration by parts twice to derive
\begin{equation*}
\begin{aligned}
    (V_i,u\rho)&=(-\frac{1}{2}u_{iijj}(0)z_i^2\chi(z)|z|,\frac{u}{|z|}\rho)-(u_i,u_i\rho)-(u_i,u\rho_i)\\&\leq-\|u_i\|_{\rho}^2+C(|\nabla^4u(0)|^2+\|\frac{u}{|z|}\|_{\rho}^2)\,.
\end{aligned}
\end{equation*}
Finally, we decompose the whole space into the near field $I=[-1,1]^d$ and its complement $I^c$, notice that $u=O(z^4)$ at $z=0$, we have the estimate $$\|\frac{u}{|z|}\|_{\rho}^2\lesssim\int_{I}\frac{u^2}{|z|^{7+n}}+\int_{I^c}{u^2}{\rho}\lesssim(\sup_{I}\frac{u}{|z|^4})^2+\|u\|_{\rho}^2\lesssim\|\nabla^4u\|_{\infty}^2+\|u\|_{\rho}^2\,.$$

Denote $E^2_0=(u,u\rho)$. We collect the $L^2$ estimate as\begin{equation}
\begin{aligned}
\label{L2-col}
    \frac{1}{2}\partial_\tau E^2_0&=(L+N+\sum_i(F_i+\lambda _iV_i),u\rho)\leq(-\frac{1}{4}+C(1+\|\nabla^4u\|_{\infty})\lambda+\|u\|_{\infty})E^2_0\\&+C\lambda(1+\|\nabla^4u\|_{\infty}) E_0+C\lambda\|\nabla^4u\|_{\infty}^2\,.
    \end{aligned}
\end{equation}
To close the $L^2$ estimate, we need higher-order estimates to control $L^{\infty}$ norms.
\subsection{Higher order stability analysis}
\label{hknorm}
Consider the weighted $H^k$ norm for $k=2n + 10$:
\begin{equation}\label{rhok_def}
    E^2_k(u)=(\nabla^ku,\nabla^ku\rho_k)\,,\quad \rho_k=1+10^{-3k}|z|^{2k+1-n}\,,
\end{equation} 
and we will estimate \begin{equation*}
    \frac{1}{2}\partial_\tau E^2_k=(\nabla^kL+\nabla^kN+\sum_i(\nabla^kF_i+\lambda _i\nabla^kV_i),\nabla^ku\rho_k)\,.
\end{equation*}
Before we start, we will state the following lemma concerning interpolation inequalities of lower order terms and $L^\infty$ estimates of a Morrey-type. We define the weighted auxiliary norms $D_j=\|\nabla^j u\langle z\rangle^{j+(1-n)/2}\|_{2}$. By \eqref{rho_def} and \eqref{rhok_def}, we know that $D_0\lesssim E_0$, $D_k\lesssim E_k$.
\begin{lemma}\label{lem:int}
    For any $\nu>0$, there exists a   $C(\nu)$ such that the inequalities hold:
    $$D_j\leq\nu D_k+C(\nu)D_0\,,\quad 0\leq j<k\,,$$
    $$\|\nabla^j u\langle z\rangle^{i+1/2}\|_{\infty}\lesssim\|\nabla^{j+n}{u}\langle z\rangle^{i+({n+1})/{2}}\|_2,\quad 0\leq j,i\,.$$
\end{lemma}
\begin{proof}
    We use an integration by parts to compute for $k>j>0$: $$D_j^2=-\sum_i\int(\partial_i^2\nabla^{j-1}u\cdot\nabla^{j-1}u\langle z\rangle^{2j+1-n}+\partial_i\nabla^{j-1}u\cdot\nabla^{j-1}u\partial_i\langle z\rangle^{2j+1-n})\,.$$
    Noticing that $(\partial_i\langle z\rangle^{2j+1-n})^2\lesssim\langle z\rangle^{2j+1-n}\langle z\rangle^{2j-1-n}$, by Cauchy-Schwarz inequalities, we have 
    $D_j^2\lesssim D_{j-1}(D_{j}+D_{j+1})$.
    By a weighted AM-GM inequality, we compute for any $\nu>0$, $D_j^2\leq\nu D^2_{j+1}+C(\nu)D^2_{j-1}$.

    Now we prove the first interpolation inequality. Since $\nu$ is arbitrary, we only need to prove for $j=k-1$, which we can use induction on $k$ and the obtained inequality $D_j\leq\nu D_{j+1}+C(\nu)D_{j-1}$ to conclude.

    For the second inequality, we borrow the idea of proof for the embedding (3.13) of Proposition 1  in our follow-up paper \cite{chen2024stability}. We can assume $u\in C^{\infty}_c$ by a density argument and consider WLOG $z\in\mathbb{R}^n$ with $z_i\geq0$ for all components. In the region $\Omega(z)=\{y\in\mathbb{R}^n,y_i\geq z_i\}$ we have $|y|\geq |z|$ for any $y\in\Omega(z)$. By Lebniz's rule and the Cauchy-Schwarz inequality, we have 
$$\begin{aligned}|\nabla^j{u}(z)|&\lesssim \int_{\Omega(z)}|\partial_1\partial_2\cdots\partial_n\nabla^j{u}(y)|dy\\&\lesssim \|\nabla^{j+n}{u}\langle y\rangle^{i+{(n+1)}/{2}}\|_2(\int_{|y|\geq |z|}\langle y\rangle^{-2i-n-1}dy)^{1/2}\,.\end{aligned}$$
    We thus collect the pointwise bound and conclude the proof of the inequality.
\end{proof}
With the lemma in mind, we
denote the terms as lower order terms (l.o.t. for short) if their $\rho_k$-weighted $L^2$-norms are bounded by $\nu E_k+C(\nu)E_0$ for any $\nu>0$. Notice that for $0<j\leq k$, $\nabla^{j}\bar{u}\nabla^{k-j}u$ are l.o.t. since we can estimate $$|\nabla^{j}\bar{u}|\rho_k^{1/2}\lesssim\langle z\rangle^{-2-j}\langle z\rangle^{k+(1-n)/2}\lesssim\langle z\rangle^{k-j+(1-n)/2}\,,$$
and thus their $\rho_k$-weighted $L^2$-norms are bounded by $D_{k-j}$.
Therefore we collect the linear estimate via an integration by parts and $O(\lambda)$ notations as 
$$\begin{aligned}(\nabla^kL,\nabla^ku\rho_k)&\leq ([-1-\frac{k}{2}+\frac{1}{4}\frac{(z\rho_k)_z}{\rho_k}+2\bar{u}]\nabla^ku,\nabla^ku\rho_k)\\&+C(1+\|\nabla^4u\|_{\infty})\lambda E_k^2+\nu E_k^2+C(\nu)E_0^2\,.\end{aligned}$$ We can compute the damping as $$-1-\frac{k}{2}+\frac{n}{4}+\frac{1}{4}\frac{(2k+1-n)10^{-3k}|z|^{2k+1-n}}{1+10^{-3k}|z|^{2k+1-n}}+\frac{2}{1+|z|^2/8}\leq-\frac{1}{2} \,,$$
 where the last inequality is equivalent to $$(1+|z|^2/8)(2k+2-n+10^{-3k}|z|^{2k+1-n})-8(1+10^{-3k}|z|^{2k+1-n})\geq0\,,$$ which can be implied by an AM-GM inequality via $$\begin{aligned}
     &3n+\frac{10^{-3k}}{8}|z|^{2k+3-n}\geq(2k+3-n)(\frac{3n}{2})^{\frac{2}{2k+3-n}}(\frac{10^{-3k}}{8(2k+1-n)})^{1-\frac{2}{2k+3-n}}|z|^{2k+1-n}\\&>10^{-3k}/8(10^{6k/(2k+3-n)})|z|^{2k+1-n}>8\times10^{-3k}|z|^{2k+1-n}\,.
 \end{aligned}$$
We collect the linear estimate by choosing a small enough $\nu$ to get
\begin{equation*}
(\nabla^kL,\nabla^ku\rho_k)\leq (-\frac{1}{4}+C(1+\|\nabla^4u\|_{\infty})\lambda)E_k^2+CE_0^2\,.
\end{equation*}

For the nonlinear term $\nabla^kN$, by Leibniz's rule, we know that it will be a linear combination of $\nabla^{k-j}u\nabla^{j}u$. For a typical term, assume WLOG that $j\leq k/2$. By the interpolation lemma, we have 
$$\|\nabla^{k-j}u\nabla^{j}u\|_{\rho_k}\leq D_{k-j}\|\nabla^ju\langle z\rangle^j\|_{\infty}\lesssim D_{k-j}D_{j+n}\lesssim(D_k+D_0)^2\,.$$
Therefore, we can collect the nonlinear estimate 
\begin{equation*}
(\nabla^kN,\nabla^ku\rho_k)\leq C(E_k+E_0)^2 E_k\,.
\end{equation*}

For the error term, notice that $\nabla^kF_i$ is  $O(z^{-2-k})$ at $\infty$. Therefore, it is square integrable with the weight $\rho_k$ and we can estimate  \begin{equation*}
(\nabla^kF_i,\nabla^ku\rho_k)\leq C\lambda(1+\|\nabla^4u\|_{\infty}) E_k\,.
\end{equation*}

We estimate the viscous term using integration by parts twice
\begin{equation*}
(\nabla^kV_i,\nabla^ku\rho_k)= - (\nabla^ku_i,\nabla^ku_i\rho_k)-(\nabla^ku_i,\nabla^ku(\rho_k)_i)\leq\frac{1}{2}(\nabla^ku,\nabla^ku(\rho_k)_{ii})\leq CE_k^2\,.
\end{equation*}
Finally, we gather our $H^k$ estimate as
\begin{equation}
\begin{aligned}
\label{Hk-col}
    &\frac{1}{2}\partial_\tau E^2_k\leq (-\frac{1}{4}+C(1+\|\nabla^4u\|_{\infty})\lambda)E_k^2+CE_0^2+C(E_k+E_0)^2 E_k\\&+C\lambda(1+\|\nabla^4u\|_{\infty}) E_k+C\lambda E_k^2\,.
    \end{aligned}
\end{equation}
Using again the interpolation lemma, we know that $\|\nabla^4u\|_{\infty}+\|u\|_{\infty}\lesssim E_k+E_0$.
Combined with \eqref{L2-col}, we know that there exists a constant $\mu_0$, such that for $0<\mu<\mu_0$, if we consider the energy \begin{equation}\label{def_energy}
    E^2=E_0^2+\mu E_k^2\,,
\end{equation} we have the estimate 
\begin{equation*}
    \frac{1}{2}\partial_\tau E^2\leq(-\frac{1}{10}+C(1+E)\lambda)E^2+CE^3+C\lambda (1+E)E+C\lambda E^2\,.
\end{equation*}
Namely that \begin{equation}\label{energy}\partial_\tau E\leq(-\frac{1}{10}+CE\lambda+CE)E+C\lambda +C\lambda E\,.\end{equation}
Notice that here $C\geq1$ is an absolute constant.
\subsection{Finite time blowup}\label{sec4}
Recall the ODE \eqref{lambn} for $\lambda_i$, and noticing that $\lambda=\max{\lambda_i}$, we define $\gamma=1/\lambda$. $\gamma(0)=1/\hat{C}_u(0)$ will be the constant we choose now. The ODE for $\gamma$ is  
$$
    \partial_\tau\gamma=-\frac{\partial_\tau\lambda_i}{\lambda_i^2}=1+4\sum_ju_{iijj}(0)\frac{\lambda_j}{\lambda_i}\,,\quad i=\text{argmax}{\lambda_i}\,.
$$ \textit{a priori} estimate yields $|\nabla^4u(0)|\leq CE$. We can assume WLOG that $C\geq 1$. Defining $G=E\gamma$, we have \begin{equation}\label{odegamma}|\partial_\tau\gamma-1|\leq 4nC\frac{G}{\gamma}\,.\end{equation} We will show that $E$ decays as $1/\tau$. We calculate an ODE for $G$:
\begin{equation}\label{odeg}\begin{aligned}
    \partial_\tau G&\leq (-\frac{1}{10}+CE\lambda+CE)G+C+C E+E(1+4nC\frac{G}{\gamma})\\&\leq(-\frac{1}{10}+2C\frac{1}{\gamma})G+C+8nCG^2(\frac{1}{\gamma}+\frac{1}{\gamma^2})\,.
\end{aligned}\end{equation}
We choose $\hat{C}_u(0)=1/\gamma(0)\leq1/(10000nC^2)$ small enough such that if we start from $G(0)< 100C$, we will have the bootstrap estimate \begin{equation}
    \label{bootstrapbound}G< 100C,\quad \gamma\geq \gamma(0)\geq10000nC^2
\end{equation} for all time via a standard bootstrap argument. 
\begin{proof}[Proof of the bootstrap bound \eqref{bootstrapbound}]
    In fact, we know by \eqref{odegamma} that $\partial_\tau\gamma(0)>0$.  Assume that the bootstrap estimate is false,   then by continuity, there exists a rescaled time $\tau_0>0$ such that \eqref{bootstrapbound} holds for $0<\tau<\tau_0$ and either $G(\tau_0)\geq 100C$ or $\gamma(\tau_0)\leq \gamma(0)$. We compute by \eqref{odegamma} that in $(0,\tau_0)$, $\partial_\tau\gamma\geq1-\frac{400nC^2}{10000nC^2}>0$ which rules out the latter case. As a consequence, we estimate the ODE for $G$ \eqref{odeg} in $(0,\tau_0)$ as $$\partial_\tau G\leq -\frac{1}{20} G+C+G\frac{1}{50}\,.$$By continuity, we estimate $\partial_\tau G(\tau_0)<0 $ and we conclude that the former case cannot hold either. We reach a contradiction and conclude the bootstrap.
\end{proof}
Based on the bootstrap estimates, we have the following estimate for   $\gamma$: $$|\partial_\tau\gamma-1|\leq\frac{400nC^2}{\gamma}\,.$$
Thus  $\gamma/\tau\to 1$ as $\tau\to\infty$.  Thus we have $$|(\hat{C}_u)_t+1|=|\frac{(\hat{C}_u)_t t_\tau}{\hat{C}_u}+1|=|\frac{(\hat{C}_u)_\tau}{\hat{C}_u}+1|=|\hat{c}_u+1|=|c_u|\leq\frac{n}{4\gamma}\,,\quad \tau_t=1/\hat{C}_u\,.$$
We can finally show that there exists a blowup time $T>0$, such that $$\lim_{t\to T}\frac{\hat{C}_u}{T-t}=1\,,\quad\lim_{t\to T}\frac{\tau}{|\log(T-t)|}=1\,.$$ %which implies $$\lim_{t\to T}\frac{\hat{C}_l}{\sqrt{(T-t)|\log(T-t)|}}= 1\,.$$ 
Moreover, defining $\kappa=\sum_i\frac{1}{\lambda_i}$, we compute the following ODE $$\partial_\tau\kappa=n+4\sum_i\sum_j u_{iijj}(0)\frac{\lambda_j}{\lambda_i}\leq n+4nC\frac{G}{\gamma^2}\kappa\leq n+\frac{400nC^2}{\gamma^2}\kappa\,. $$
Therefore for sufficiently large $\tau$, $\partial_\tau\kappa\leq n+800nC^2\frac{\kappa}{\tau^2}$. We integrate and get for sufficiently large $\tau$, $\kappa\leq n\tau+1600n^2C^2\log\tau$. Therefore since $\lambda=\max{\lambda_i}=1/\gamma\to1/\tau$, we have $n\leq\liminf\sum_i\frac{\lambda}{\lambda_i}\leq\limsup\sum_i\frac{\lambda}{\lambda_i}\leq n.$ Namely we get $\lambda_i\tau\to\frac{\lambda_i}{\lambda}\to1$, and thus we arrive at the law $$\lim_{t\to T}\frac{\hat{C}^i_l}{\sqrt{(T-t)|\log(T-t)|}}=\lim_{t\to T}\sqrt\frac{\hat{C}_u}{{(T-t)\lambda_i\tau}}= 1\,.$$ 

Therefore for $C_0=100C, \lambda_0=1/(10000nC^2)$ and for initial data satisfying the assumption of Theorem \ref{thm1}, we know that $\gamma(0)=1/\lambda$ and   $u=g$ defined in the rescaled formulation satisfy the bootstrap assumption.
We conclude Theorem \ref{thm1} based on the asymptotics of $\hat{C}_u, \hat{C}_l^i$.

\section{Numerical Experiments}\label{sec:num}
In this section, we conduct numerical experiments to corroborate our analysis that our choice of normalization in Section \ref{sec:dy} indeed preserves a stable blowup and therefore we are able to capture the log-correction numerically, both in the 1D case and in the case of higher dimensions with nonradial perturbations. We remark that our proofs in the work are derived independently of the numerical results in this section. 

\textbf{Data availability statement:} The data and the code will be available upon request.

In practice, we hope to compute the profile even if we do not have prior knowledge. Therefore in our numerical experiment, we solve  \eqref{drf} with initial data as a large perturbation to the approximate steady state. We will compute $\hat{u}$ dynamically and recall that  our choice of normalization $\hat{c}_l$, $\hat{c}_u$ in \eqref{cons_hat} ensures that $\hat{u}(0)$, $\hat{u}_{zz}(0)$ remain constants in time.  

\subsection{1D case}

In our numerical study, we choose the initialization that is more general than the assumption of our theorem as $$\hat{u}(0,z)=(1+z^2/8+z^4/10)^{-1}\,, \quad\hat{C}_u(0)=1\,, \quad\lambda=1\,.$$
At each time step $\tau_m$, we first determine the normalization constants as $$\hat{c}_u=-\hat{u}(0)-\frac{\lambda\hat{u}_{zz}(0)}{\hat{u}(0)}\,,\hat{c}_l=\frac{\hat{c}_u}{2}+\hat{u}(0)+\frac{\lambda\hat{u}_{zzzz}(0)}{2\hat{u}_{zz}(0)}\,.$$
Next, we can determine the time step $k$ via the standard numerical stability conditions for a convection-diffusion equation, and then we use the 4-th order Runge-Kutta scheme for the discretization in time and a cubic spline for the discretization in space to evolve the equation $$\hat{u}_\tau=\hat{c}_u \hat{u}-\hat{c}_l z\hat{u}_z+\hat{u}^2+\lambda\hat{u}_{zz}\,.$$
Finally, we update our $\lambda$ for time $\tau_{m+1}=\tau_m+k$ by a 4-th order Runge-Kutta discretization scheme  of the ODE $$(\log\lambda)_\tau=(2\hat{c}_l+\hat{c}_u)\,.$$

We use a fixed nonuniform mesh in space with even symmetry considered, and our computational domain is $[0,10^5]$ with $2000$ gridpoints in space. We report that after $10^9$ iterations in time, the rescaled time $\tau\approx3.9887\times10^5$ and $\log(\hat{C}_u)\approx -3.9886\times10^5$. This means that the amplitude of the solution in the physical space grows $\exp(3.9886\times10^5)$ times, which is impossible to compute if we do not use a dynamic rescaling formulation. We remark that the computation is very stable and we stopped after $10^9$ iterations only due to concerns of computational time. In theory, we can compute for an arbitrarily long time and witness an arbitrary growth of the amplitude in the physical space.

To see that the profile $\hat{u}$ converges indeed to the steady state $\bar{u}$, we plot the profile after $m=5 \times 10^4, 5 \times 10^5, 5 \times 10^6, 10^7,1.5  \times 10^7, 2 \times 10^7$ iterations and compare it with the steady state. We see that the profile converges fast; see Figure \ref{fig:profile}. Furthermore, we investigate the convergence rate of the profile. Define $\gamma(\tau)=\sup_{z}\{\hat{u}(\tau)-\bar{u}\}$. We plot $\gamma\tau$ after $2 \times 10^7$ until $5 \times10^7$ iterations, corresponding to $\tau\in[218,11638]$. We see that the residue is approximately of order $1/\tau$; see Figure \ref{fig:res}. However, we are only using a finite domain and as time becomes larger, the effect of the finite domain size becomes more obvious, and $\gamma\tau$ will increase slightly.

\begin{figure}
    \centering
    \includegraphics[width=0.8\linewidth]{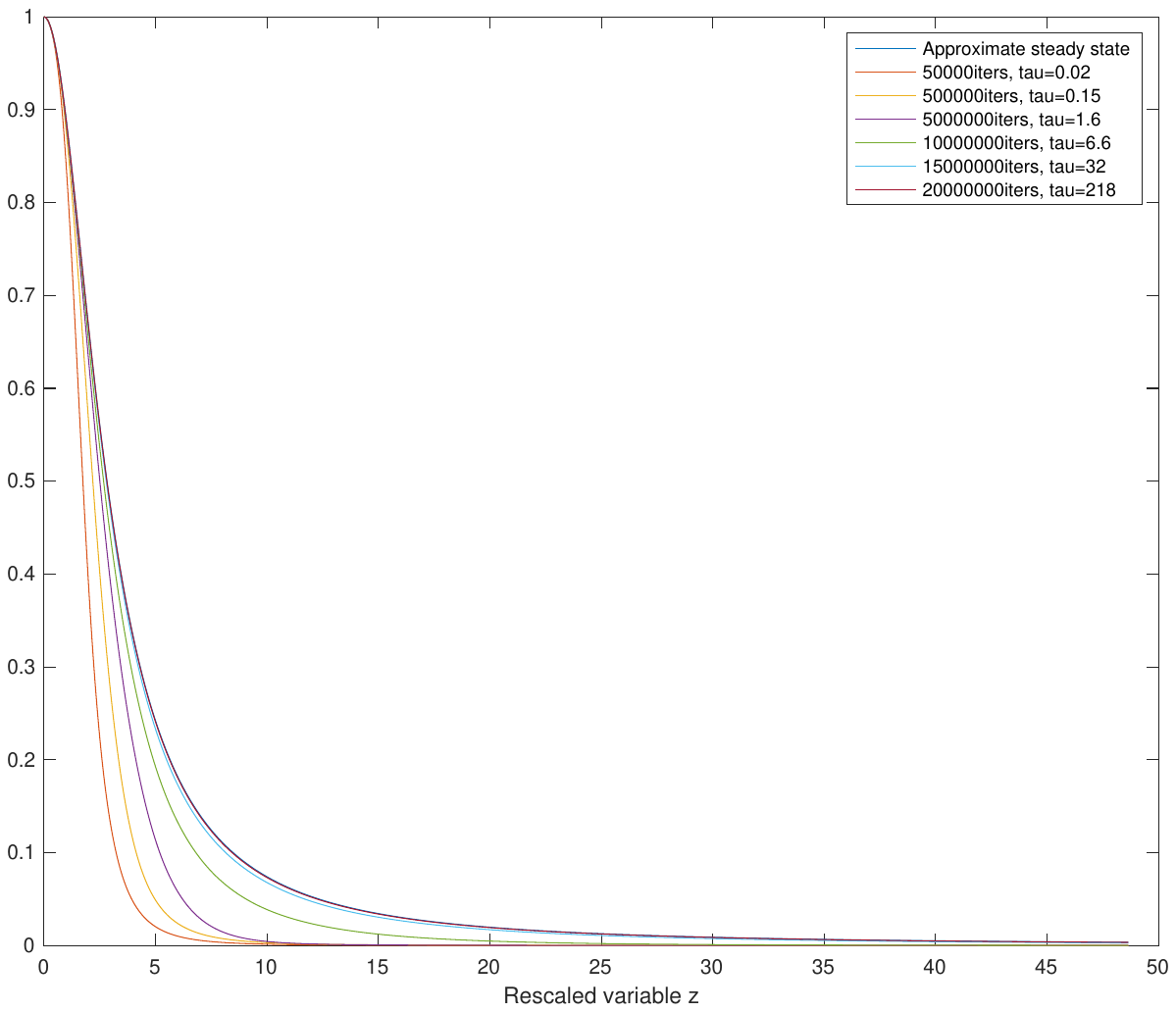}
    \caption{Comparison of the profile to the approximate steady state}
    \label{fig:profile}
\end{figure}

\begin{figure}
    \centering
    \includegraphics[width=0.7\linewidth]{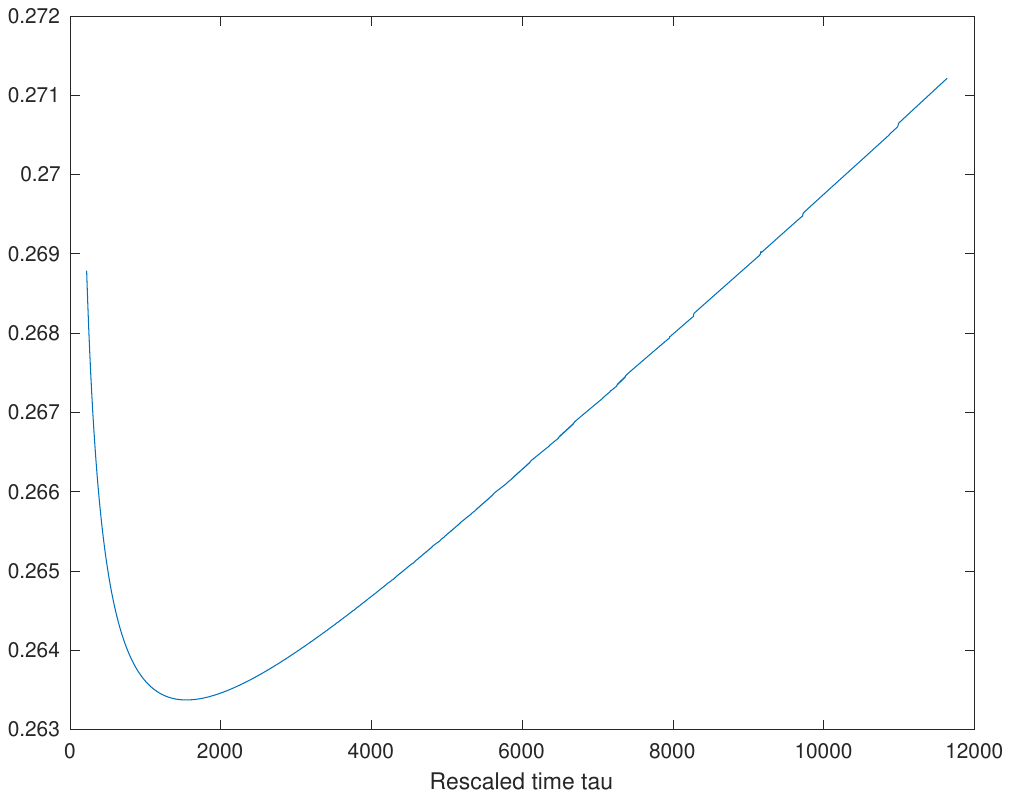}
    \caption{Plot of the  residue multiplied by the rescaled time}
    \label{fig:res}
\end{figure}
To see that we can recover the correct convergence rate, we plot $(1/2-\hat{c}_l)\tau$ and $(\hat{c}_u+1)\tau$ in time to see that they indeed converge to the correct constant $5/8$ and $1/4$ respectively and therefore will give the correct log-scaling; see for example indicated by \eqref{cu}. Again for visualization purposes, we only plot for the first $5\times 10^7$ iterations and we can see that they converge to the desired constants very fast; see Figure \ref{fig:cl}.
\begin{figure}[ht]

\centering
\includegraphics[width=6cm]{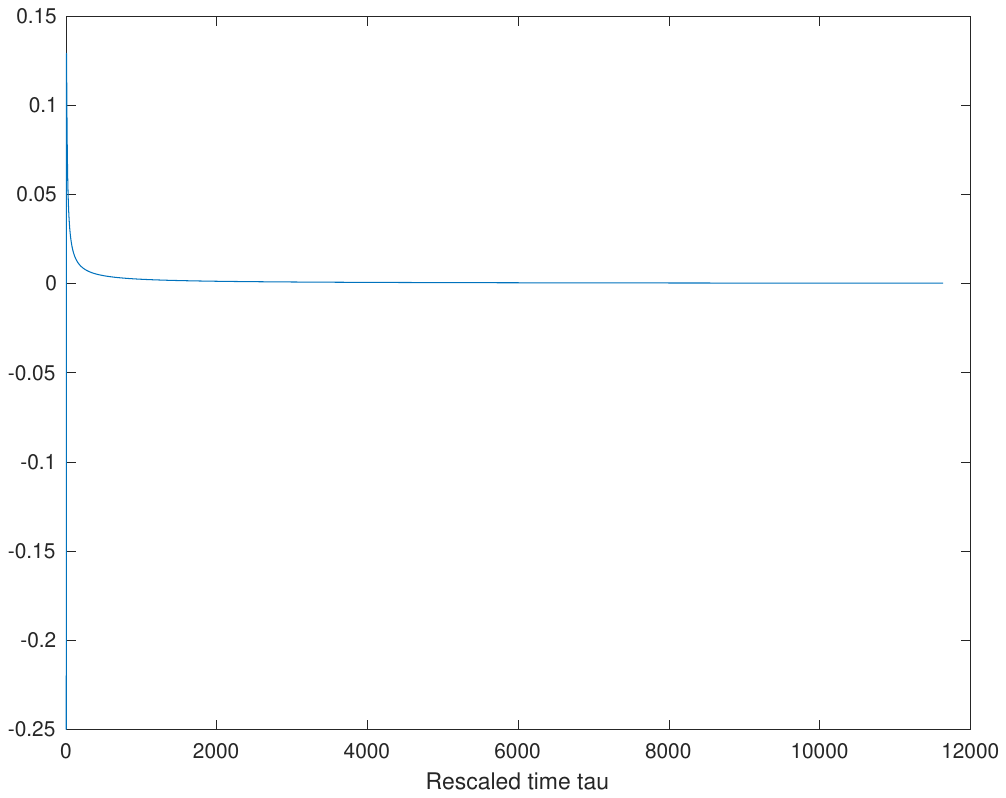}
\includegraphics[width=6cm]{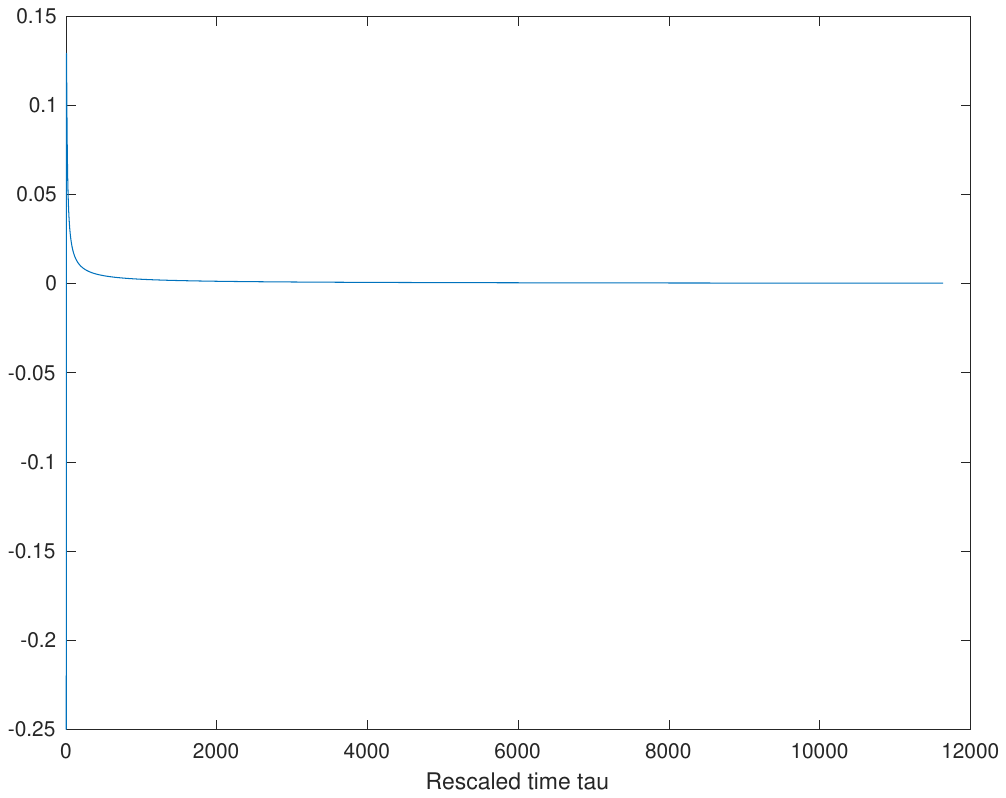}
\caption{Fitting the law of the normalization constants. Left: $(1/2-\hat{c}_l)\tau-5/8$ versus $\tau$; right: $(\hat{c}_u+1)\tau-1/4$ versus $\tau$.}
\label{fig:cl}
\end{figure}

\subsection{2D case}
For the 2D example, we choose a nonradial initialization  as $$\hat{u}(0,x,y)=(1+(x^2+y^2)/8+x^4/100)^{-1}\,, \quad\hat{C}_u(0)=1\,, \quad\lambda_1=\lambda_2=1\,.$$
At each time step $\tau_m$, we first determine the normalization constants as $$\hat{c}_u=-\hat{u}(0,0)-\frac{\lambda_1\hat{u}_{xx}(0,0)+\lambda_2\hat{u}_{yy}(0,0)}{\hat{u}(0,0)}\,,$$$$\hat{c}^1_l=\frac{\hat{c}_u}{2}+\hat{u}(0,0)+\frac{\lambda_1\hat{u}_{xxxx}(0,0)+\lambda_2\hat{u}_{xxyy}(0,0)}{2\hat{u}_{xx}(0,0)}\,,$$$$\hat{c}^2_l=\frac{\hat{c}_u}{2}+\hat{u}(0,0)+\frac{\lambda_1\hat{u}_{xxyy}(0,0)+\lambda_2\hat{u}_{yyyy}(0,0)}{2\hat{u}_{yy}(0,0)}\,.$$
Next, we can determine the time step $k$ via the standard numerical stability conditions for a convection-diffusion equation, and then we use the 4-th order Runge-Kutta scheme for the discretization in time and a cubic spline for the discretization in space to evolve the equation $$\hat{u}_\tau=\hat{c}_u \hat{u}-\hat{c}^1_l x\hat{u}_x-\hat{c}^2_l y\hat{u}_y+\hat{u}^2+\lambda_1\hat{u}_{xx}+\lambda_2\hat{u}_{yy}\,.$$
Finally, we update our $\lambda_1$, $\lambda_2$ for time $\tau_{m+1}=\tau_m+k$ by a 4-th order Runge-Kutta discretization scheme   of the ODE $$(\log\lambda_i)_\tau=(2\hat{c}^i_l+\hat{c}_u)\,,\quad i=1,2\,.$$

We use a fixed nonuniform mesh in space with even symmetry considered, and our computational domain is $[0,4000]$ with $200$ gridpoints in space in each direction. 
To see that we can recover the correct convergence rate,
we plot $R_i:=(1/2-\hat{c}_l^i)\tau$ and $R_u:=(\hat{c}_u+1)\tau$ as a function of $\tau$ after $10^7$ iterations to see that they indeed converge to the correct constant $3/4$ and $1/2$ respectively and therefore will give the correct log-scaling; see for example indicated by \eqref{cun}. We can see that they converge to the desired constants very fast; see Figure \ref{fig:clnd}.
\begin{figure}[ht]

\centering
\includegraphics[width=0.8\linewidth]{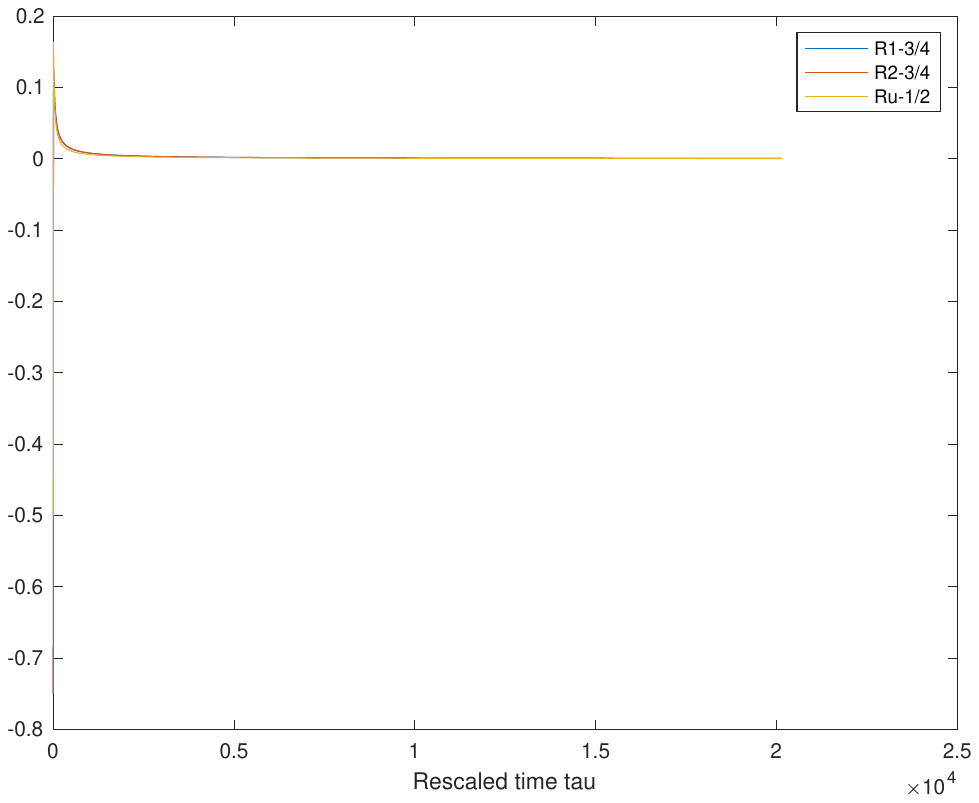}
\caption{Fitting the law of the normalization constants for 2D.}
\label{fig:clnd}
\end{figure}

\section{Finite Codimensional stability}\label{sec:slh1}

In this section, we will briefly sketch the high-level idea to study high-order vanishing type-I blowup for the 1D semilinear heat equation \eqref{sml} under \textit{radial (even symmetric)} setting.
\subsection{Self-similar renormalization and the approximate solution} For any fixed $m \in \mathbb{Z}_{> 1}$, we introduce the self-similar coordinate
\begin{equation}
    y = \frac{x}{\lambda^{\frac{1}{m}}}, \quad \frac{d\tau}{dt} = \frac{1}{\lambda^2}, \quad \tau\Big|_{t=0} =0, \quad 
\frac{\lambda_\tau}{\lambda} = -\frac{1}{2},
\label{ss coordinate: heat eq1}
\end{equation}

and corresponding renormalization
\begin{equation}
a(t,x) = \frac{1}{\lambda^2} U \left( \tau, y \right),
\label{ss renormalization: heat eq1}
\end{equation}
then $\lambda(\tau) = \lambda_0 e^{-\frac{1}{2} \tau}$ and $U$ solves the equation
\begin{equation}
\partial_\tau U = \lambda^{2-\frac{2}{m}} U_{yy} - U - \frac{1}{2m} y \cdot \nabla U + U^2,
\label{eq: heat renormalized setting1}
\end{equation}
where the diffusion term can be regarded as a perturbation since $2- \frac{2}{m} >0$. This in turn motivates us to find an approximate solution solving the equation
\begin{equation}
-U - \frac{1}{2m} y \cdot \nabla U + U^2=0.
\label{sec: PDE U*1}
\end{equation}
In particular, this equation can be explicitly solved by
\begin{equation}
U_* (y)= (1+ c y^{2m})^{-1}.
\label{sec3: profile1}
\end{equation}
Here $c>0$ is a constant, and $m>1$ describes the vanishing order of the next order expansion of $U_*$ near the origin.

\subsection{Linear stability}
We fix $m>1$ and $c=1$ in \eqref{sec3: profile1}, and plug the ansatz  $U = U_* + \epsilon$ into \eqref{eq: heat renormalized setting1}, it then follows that $\varepsilon$ solves
\[
\partial_\tau \varepsilon= \lambda^{2-\frac{2}{m}} U_{yy} + \mathcal{L} \varepsilon+ \varepsilon^2,
\]
where the linearized operator reads
\begin{equation}
\mathcal{L}\varepsilon=-\varepsilon-\frac{1}{2m}y \partial_y \varepsilon +2U_*\varepsilon.
\label{sec: L1}
\end{equation}
Next, we introduce a weighted $L_\Theta^2$ space with singular weight $\Theta(y) = y^{-4m-4}$ near the origin to extract damping, which is the essential step to close the nonlinear stability. Via the integration by parts, we have the coercivity near the origin
\begin{equation}
\left( \mathcal{L} \varepsilon, \varepsilon\right)_{L_\Theta^2}
= \left( \left( -1 + 2U_* + \frac{(\Theta y)_y}{4 m \Theta} \right) \varepsilon, \varepsilon\right)_{L_{\Theta}^2} 
\approx  - \frac{3}{4m} \left(\varepsilon, \varepsilon\right)_{L_{\Theta}^2}.
\label{sec3: cor L 11}
\end{equation}
In particular, with careful analysis, there is a small constant $0< \kappa = \kappa(m,U_*) \ll 1$ such that \eqref{sec3: cor L 11} can be extended to
\begin{equation}
\left( \mathcal{L} \varepsilon, \varepsilon\right)_{L_{\Theta+\kappa}^2}
 \leq - \frac{1}{4m} \left(\varepsilon, \varepsilon\right)_{L_{\Theta+\kappa}^2}.
\label{sec3: cor L 21}
\end{equation}

Additionally, we need to introduce the higher Sobolev norm $\dot H^{\bar K}$ to close the bootstrap argument. Precisely,
\begin{align}
\left( \mathcal{L} \varepsilon, \varepsilon\right)_{\dot H^{\bar K}}
& =\left( \left(-1-\frac{\bar K}{2m}+2U_*+\frac{1}{4m} \right)\partial_y^{\bar K}\varepsilon, \partial_y^{\bar K}\varepsilon \right)_{L^2} + O(\| \varepsilon\|_{H^{\bar K-1}} \| \varepsilon\|_{\dot H^{\bar K}}) \notag \\
& \leq - \frac{2\bar K-4m-1}{4m} \| \varepsilon\|_{\dot H^{\bar K}}^2 + O(\| \varepsilon\|_{H^{{\bar K}-1}} \| \varepsilon\|_{\dot H^{\bar K}}),
\label{sec3: cor L 31}
\end{align}
where the leading order enjoys damping once we choose $\bar K = \bar K(m) \gg 1$.

\subsection{Modulation ODEs and nonlinear stability} With the singular weight $\Theta(y) = |y|^{-4m-4}$ given previously, introducing a cutoff function $\chi$, we can further \textit{radially} decompose $\varepsilon$ into
\begin{equation}
\varepsilon(\tau,y)= \varepsilon_u(\tau,y) + \varepsilon_s(\tau,y), \quad \text{ with }  \quad \varepsilon_u = \sum_{j=0}^{m} c_j(\tau) \chi(y) y^{2j},
\end{equation}
such that $\varepsilon_s(\tau,y) = O(y^{2m+2})$ near the origin, which yields an ODE system for modulation parameters $\{ c_j \}_{j=0}^m$:
\begin{equation}
\begin{cases}
    \dot c_j = \left( 1- \frac{j}{m} \right) c_j + [\varepsilon_u^2]_j + \lambda^{2- \frac{2}{m}} [U_{yy}]_j, & 0 \leq j < m-1,
    \\
    \dot c_{m} = [\varepsilon_u^2]_m - 2c_0 + \lambda^{2- \frac{2}{m}}[U_{yy}]_m, & j=m.
\end{cases}
\end{equation}
Here $[h]_j$ is the $2j$-th order coefficient of Taylor expansion of $h(r)$ at the origin.
Additionally, $\varepsilon_s = O(y^{2m+2})$ solves the equation
\[
\partial_\tau \varepsilon_s =  \mathcal{L} \varepsilon_s  + 2 \varepsilon_u \varepsilon_s + \varepsilon_s + G[\lambda,U,\varepsilon_u],
\]
with the modulation term $G[\lambda,U,\varepsilon_u]= O(y^{2m+2})$ given by 
\[
G[\lambda,U,\varepsilon_u]= \left( \lambda^{2-\frac{2}{m}} U_{yy} + \mathcal{L}\varepsilon_u + \varepsilon_u^2 \right) - \sum_{j=0}^K \left[\lambda^{2-\frac{2}{m}} U_{yy} + \mathcal{L} \varepsilon_u + \varepsilon_u^2 \right]_j \chi  y^{2j}.
\]
Finally, we can use the standard topological argument together with \eqref{sec3: cor L 11} and \eqref{sec3: cor L 31} to derive the nonlinear stability with finite codimension $m+1$. 

\begin{remark}
    We expect that this nonlinear stability result can be improved to finite-codimension $m-1$, which is two dimensions lower than our previous findings. 
    The key underlying reason is the presence of two degrees of freedom, namely, the choice of the blowup time $T>0$ and the shrinking rate. These degrees of freedom can be utilized through a matching argument to recover the corresponding unstable directions as in \cite{Limodestability, lizhou2025nonradial}.

    Alternatively, one may employ a method of dynamical rescaling to establish stability with finite codimension $m-1$. 
    Specifically, we modify the coordinate \eqref{ss coordinate: heat eq1} as
    \[
    y = \frac{x}{\mu^\frac{1}{m}}, \quad  \frac{d\tau}{dt} = \frac{1}{\lambda^2}, \quad \tau\big|_{t=0}=0, \quad \frac{\lambda_\tau}{\lambda} = -\frac{1}{2} + \frac{1}{2} c_a, \quad \frac{\mu_\tau}{ \mu} = - \frac{1}{2}  - mc_s.
    \]
    We then define the corresponding renormalization
    \[
    u(t,x) = \frac{1}{\lambda^2} U (\tau,y).
    \]
    Under this coordinate transformation, $U$ satisfies
    \[
    \partial_\tau U = \lambda^2 \mu^{-\frac{2}{m}} U_{yy} + \left( -1 + c_a \right) U -\left( \frac{1}{2m} + c_s \right) y \partial_y U + U^2,
    \]
    where the parameters $(c_a,c_s)$ are determined by the modulation conditions
    \[
    U(\tau, 0) = U_*(0) \qquad \text{ and } \qquad  [U(\tau,0)]_{m}= [U_*]_{m},
    \] 
    and we eliminate neutral modes by fixing $c_0=c_m=0$.
    By applying a similar argument of modulation ODEs, we obtain the nonlinear stability with finite codimension $m-1$. Notably, introducing extra scaling parameters to perturb the scaling symmetry is crucial for extending the argument to the nonradial setting; see previous works by the second author and collaborators \cite{chen2024stability,hou20242}.
\end{remark}

\begin{remark}

Compared with the semilinear heat equation, analyzing the Keller-Segel equation with logistic damping involves several additional challenges. For example, 
the profile $U_*$ 
introduced in \eqref{sec3: profile1} is an explicit solution to 
the first-order and separable local equation \eqref{sec: PDE U*1}. 
In contrast, for Keller-Segel equation with logistic damping, the associated profile equatiois inherently nonlocal and cannot be trivially solved. 
This nonlocality requires a more delicate analysis.

    Moreover, since there is no explicit nontrivial solution to the profile equation, additional effort is required to derive quantitative properties of the profile $Q$. Combined with the nonlocal nature, these complexities make the establishment of linear coercivity of Keller-Segel equation more intricate than in the case of the semilinear heat equation.  Detailed strategies to handle these obstacles will be presented in Section 4 of \cite{liu2025finite}.

\end{remark}

\section{Future works}
\label{subsec:future works}
Having established our robust framework of blowups via local modulations and singularly weighted estimates, we outline some potential directions for generalizing our methods for more challenging singularities with multiple scales, and towards investigation as promising numerical methods.

An intriguing blowup phenomenon happens with a traveling-wave type singularity, say of the following $$\bar{w}(x,t)=(T-t)^{-3/2}\bar{\Omega}(\frac{x-1/2(T-t)^{1/2}}{(T-t)})\,.$$
%For example, in the above explicit example, $T=4, \bar{\Omega}(z)=-\frac{2}{1+16z^2}$.

One goal is to generalize the idea of local modulations to traveling-wave type singularities, where there are two scales at play: the inner scale that determines the blowup profile at $x=1/2(T-t)^{1/2}+O(T-t)$, and the outer scale that determines the blowup location at $x=O((T-t)^{1/2})$. We have demonstrated that via arguments of local modulations, the outer scale is shown to be stable, leveraging the precise vanishing order of the profile and making the perturbation vanish at a higher order. 

By studying this simple model, one can gain insights to establish singularity for more complicated models like Keller-Segel or potentially Navier-Stokes equations, which involve multiple blowup scales. Again, our framework of local modulations comes at the benefit of using only limited non-explicit information of the approximate profile, seamlessly amenable to computations of profiles and computer-assisted proofs.

On the other hand, numerical investigations showcased the efficiency of our method in the computation of the semilinear heat equation \cite{hou20242}. As indicated by our analysis, we enforced $\hat{u}(0,t)$ and $\nabla^2\hat{u}(0,t)$ to be constant for data with even symmetries. In practice, one can come up with more stable enforcement of those vanishing conditions, say with a splitting method to enforce those constant quantities, potentially improving the numerical stability of our method. Implementation of the numerical methods for general classes of blowups, including those with multiple scales, is of immediate future interest.

For general nonlinear steady-state equations, one approach is to introduce artificial time and perform time marching for convergence. Since we demonstrated that a clever nonlinear rescaling can rule out the potentially unstable directions and enhance convergence, we have reasons to believe that our methodology can be adopted broadly as numerical solvers beyond singularity formation.

\chapter{Kolmogorov-Arnold Network}
\label{chap:3}

In this chapter, we present Kolmogorov-Arnold Network (KAN), a novel deep learning architecture inspired by the Kolmogorov-Arnold representation theorem, mostly based on our works \cite{kan1,liu2024kan,kanbias}. Partially motivated by the need for symbolic computation of singularities, KANs offer a powerful tool for science, especially when interpretability is desired. Compared to prevalent Multi-Layer Perceptrons (MLPs) with  \textit{fixed} activation functions on \textit{nodes} (``neurons''), KANs have \textit{learnable} activation functions on \textit{edges} (``weights''). KANs learn interpretable 1D functions on their edges whose connection graph is also simple enough to be explained. Through examples in mathematics, KANs are shown to be useful ``collaborators'' helping scientists (re)discover mathematical and physical laws. Moreover, KANs are shown to be more accurate and have faster scaling laws than MLPs in function fitting and PDE solving, both theoretically and empirically.

\begin{figure}[hb]
    \centering
    \includegraphics[width=1.0\linewidth]{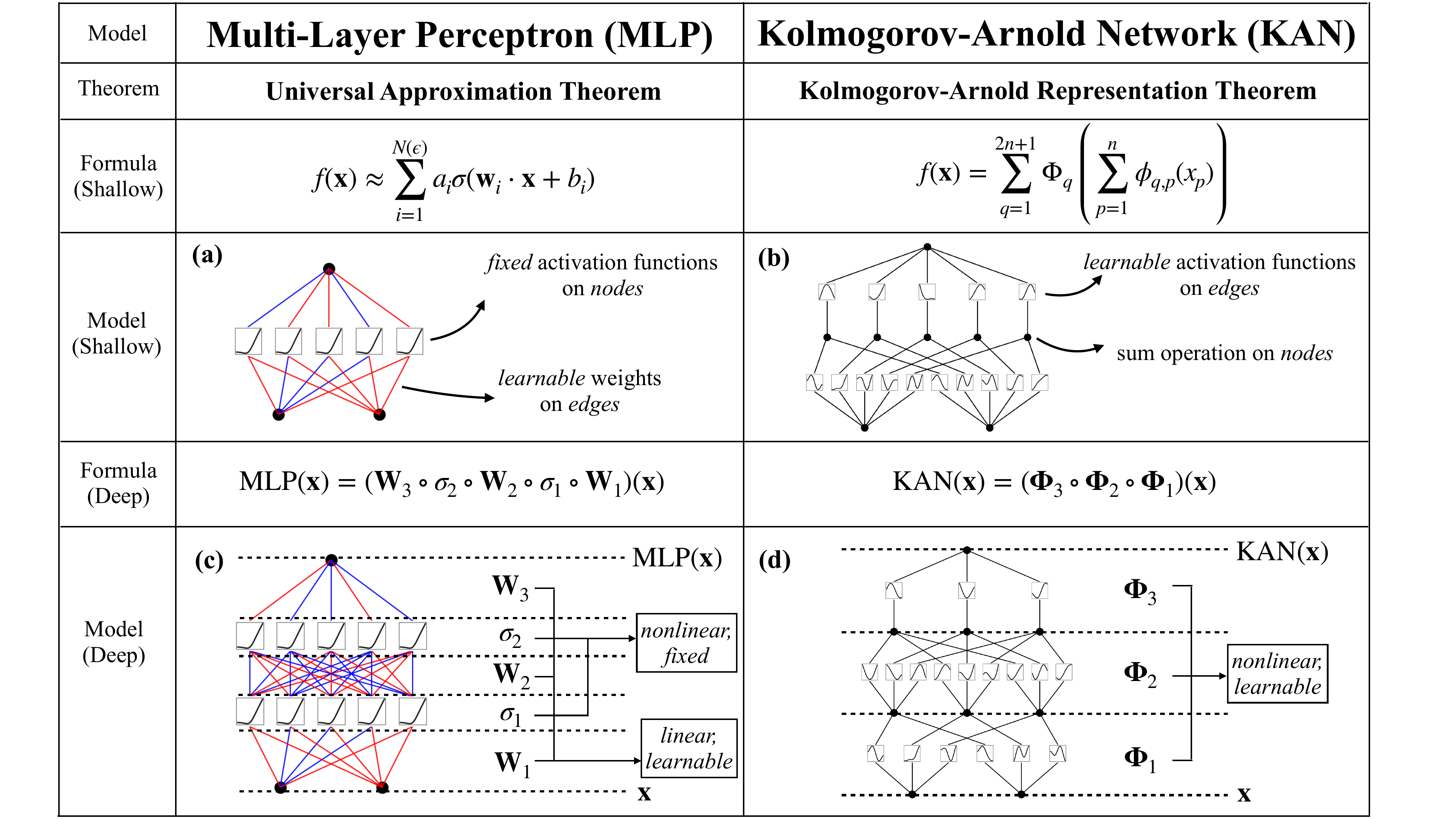}
    \caption{Multi-Layer Perceptrons (MLPs) vs. Kolmogorov-Arnold Networks (KANs)}
    \label{fig:kan_mlp}
\end{figure}

\section{Introduction}

Multi-layer perceptrons (MLPs)~\cite{haykin1994neural,cybenko1989approximation,hornik1989multilayer}, also known as fully-connected feedforward neural networks, are foundational building blocks of today's deep learning models. The importance of MLPs can never be overstated, since they are the default models in machine learning for approximating nonlinear functions, due to their expressive power guaranteed by the universal approximation theorem~\cite{hornik1989multilayer}. However, MLPs often lack interpretability, which makes them less useful for tasks when interpretability is key, e.g., when we want to extract symbolic formulas from datasets. In science, symbolic functions are prevalent, e.g., $E=mc^2$ (energy-mass relation), $r=\frac{a}{1+e{\cos}\theta}$ (ellipse), $p = e^{-\frac{E}{kT}}/Z$ (Boltzman distribution). Although MLPs can numerically approximate these functions to a reasonable accuracy, they cannot reveal symbolic structures of these equations.  %However, are MLPs the best nonlinear regressors we can build? Despite the prevalent use of MLPs, they have significant drawbacks. In transformers~\cite{vaswani2017attention} for example, MLPs consume almost all non-embedding parameters and are typically less interpretable (relative to attention layers) without post-analysis tools~\cite{cunningham2023sparse}.

Therefore, we need a representation theorem that is more aligned with symbolic representations than the universal approximation theorem. In our search, the good old Kolmogorov-Arnold representation theorem (KA theorem) came to our attention. Although the KA theorem has long been considered irrelevant for learning~\cite{girosi1989representation} because the theorem does not guarantee smoothness, we are more optimistic about the smoothness of deeper representations. For example, as we will show, $f(x_1,x_2,x_3,x_4)={\exp}({\sin}(x_1^2+x_2^2)+{\sin}(x_3^2+x_4^2))$ can be smoothly represented by a three-layer network, but a two-layer network that attempts to fit this function leads to pathological representations.

Unsurprisingly, the possibility of using Kolmogorov-Arnold representation theorem to build neural networks has been studied~\cite{sprecher2002space,koppen2002training,lin1993realization,lai2021kolmogorov,leni2013kolmogorov,fakhoury2022exsplinet,montanelli2020error}. However, most work has stuck with the original depth-2 width-($2n+1$) representation, and many did not have the chance to leverage more modern techniques (e.g., back propagation) to train the networks. Our contribution lies in generalizing the original Kolmogorov-Arnold representation to arbitrary widths and depths, revitalizing and  contextualizing it in today's deep learning world, as well as using empirical experiments to highlight its potential for AI + Science due to its accuracy and interpretability.

Named after the two great mathematicians, Andrey Kolmogorov and Vladimir Arnold, this new type of network is called the \textit{Kolmogorov-Arnold Network} (KAN). %Whereas MLPs are inspired by the universal approximation theorem, KANs are inspired by the Kolmogorov-Arnold representation theorem~\cite{kolmogorov, kolmogorov1957representation, braun2009constructive}. 
Like MLPs, KANs have fully-connected structures. However, while MLPs place fixed activation functions on \textit{nodes} (``neurons''), KANs place learnable activation functions on \textit{edges} (``weights''), as illustrated in Figure~\ref{fig:kan_mlp}.  %As a result, KANs have no linear weight matrices at all: instead, 
Each learnable weight parameter in an MLP is replaced by a learnable 1D function (parametrized as a spline) in a KAN. KANs' nodes simply sum incoming signals without applying any non-linearities. 

Although interpretability is our initial motivation to develop KANs, KANs demonstrate impressive accuracy and fast scaling laws as well, both theoretically and empirically. Despite their elegant mathematical interpretation, KANs are nothing more than combinations of splines and MLPs, leveraging their respective strengths and avoiding their respective weaknesses. Splines are accurate for low-dimensional functions % easy to adjust locally, and able to switch between different resolutions. However, splines have a 
but suffer from the curse of dimensionality (COD) problem. %, because of their inability to exploit compositional structures.  
MLPs, On the other hand, suffer less from COD thanks to their ability to learn features and compositional structure, but are less accurate than splines in low dimensions. %, because of their inability to optimize univariate functions. 
KANs have MLPs on the outside and splines on the inside, combining the best of two things into one.

The chapter is organized as follows: is organized as follows: In Section~\ref{sec:KAN}, we introduce the KAN architecture, analyze the network's approximation ability, and propose two training techniques to make KANs interpretable and accurate. In Section~\ref{sec:kan_interpretability_experiment}, we show that KANs are interpretable and can be used for scientific discoveries. We use a knot theory example from mathematics 
to demonstrate that KANs can be helpful ``collaborators'' for scientists. In Section~\ref{sec:kan_accuracy_experiment}, we show that KANs are more accurate than MLPs for data fitting and PDE solving with better scaling laws. We discuss the superiority of KANs learning high-frequency functions in Section~\ref{sec:sp}. We discuss related works, follow-up works inspired by our work, and draw conclusions in Section~\ref{sec:conclusions}.  

\section{Kolmogorov–Arnold Networks (KAN)}\label{sec:KAN}

Multi-Layer Perceptrons (MLPs) are inspired by the universal approximation theorem. We instead focus on the Kolmogorov-Arnold representation theorem, which can be realized by a new type of neural network called Kolmogorov-Arnold networks (KAN). We review the Kolmogorov-Arnold theorem in Section~\ref{subsec:kart}, to inspire the design of Kolmogorov-Arnold Networks in Section~\ref{subsec:kan_architecture}. Section~\ref{subsec:kan_scaling_theory} provides mathematical description of KANs' expressive power. Section~\ref{subsec:kan_grid_extension} and Section~\ref{subsec:kan_simplification} propose techniques to make KANs accurate and interpretable.

\subsection{Kolmogorov-Arnold representation theorem}\label{subsec:kart}

Vladimir Arnold and Andrey Kolmogorov established that if $f$ is a multivariate continuous function on a bounded domain, then $f$ can be written as a finite composition of continuous functions of a single variable and the binary operation of addition. More specifically, for a smooth $f:[0,1]^n\to\mathbb{R}$,
\begin{equation}\label{eq:KART}
    f(\mathbf{x}) = f(x_1,\cdots,x_n)=\sum_{q=1}^{2n+1} \Phi_q\left(\sum_{p=1}^n\phi_{q,p}(x_p)\right),
\end{equation}
where $\phi_{q,p}:[0,1]\to\mathbb{R}$ and $\Phi_q:\mathbb{R}\to\mathbb{R}$. In a sense, they showed that the only true multivariate function is addition, since every other function can be written using univariate functions and sum. One might naively consider this great news for machine learning: learning a high-dimensional function boils down to learning a polynomial number of 1D functions. However, these 1D functions can be non-smooth and even fractal, so they may not be learnable in practice~\cite{poggio2020theoretical}. Because of this pathological behavior, the Kolmogorov-Arnold representation theorem was regarded as theoretically sound but practically useless~\cite{poggio2020theoretical}. 

However, we are more optimistic about the usefulness of the Kolmogorov-Arnold theorem for machine learning. First of all, we need not stick to the original Eq.~(\ref{eq:KART}) which has only two-layer non-linearities and a small number of terms ($2n+1$) in the hidden layer: we will generalize the network to arbitrary widths and depths. Deeper and wider networks potentially have stronger expressive power with smooth functions. Moreover, most functions in science and daily life are often smooth and have sparse compositional structures~\cite{lin2017does}, potentially facilitating smooth Kolmogorov-Arnold representations.

\subsection{KAN architecture}\label{subsec:kan_architecture} 

Suppose we have a supervised learning task consisting of input-output pairs $\{x_i,y_i\}$, where we want to find $f$ such that $y_i\approx f(x_i)$ for all data points.
Eq.~(\ref{eq:KART}) implies that we are done if we can find appropriate univariate functions $\phi_{q,p}$ and $\Phi_q$. This inspires us to design a neural network which explicitly parametrizes Eq.~(\ref{eq:KART}). Since all functions to be learned are univariate functions, we can parametrize each 1D function as a B-spline curve, with learnable coefficients of local B-spline basis functions~\footnote{Details in Appendix~\ref{app:kan-implmentation-detail} and illustrated in Figure~\ref{fig:spline-notation} right.}. Now we have a prototype of KAN, whose computation graph is exactly specified by Eq.~(\ref{eq:KART}) and  illustrated in Figure~\ref{fig:kan_mlp} (b) (with the input dimension $n=2$), appearing as a two-layer neural network with activation functions placed on edges instead of nodes (simple summation is performed on nodes), and with width $2n+1$ in the middle layer.

As mentioned, such a network is known to be too simple to approximate any function arbitrarily well in practice with smooth splines! %Indeed, in the last subsection we mention the caveat that activation functions sometimes need to be non-smooth~\cite{schmidt2021kolmogorov,
%poggio2022deep}, which are not learnable in practice with gradient descent. 
We therefore generalize our KAN to be wider and deeper. %It is not immediately clear how to make KANs deeper, since Kolmogorov-Arnold representations correspond to two-layer KANs. To the best of our knowledge, there is not yet a ``generalized'' version of the theorem that corresponds to deeper KANs. 
The key insight comes from the analogy between MLPs and KANs. In MLPs, once we define a layer (which is composed of a linear transformation and nonlinearties), we can stack more layers to make the network deeper. To build deep KANs, we should first answer: ``what is a KAN layer?'' It turns out that a KAN layer with $n_{\text{in}}$-dimensional inputs and $n_{\text{out}}$-dimensional outputs can be defined as a matrix of 1D functions 
\begin{align}
    {\mathbf\Phi}=\{\phi_{q,p}\},\qquad p=1,2,\cdots,n_{\text{in}},\qquad q=1,2\cdots,n_{\text{out}},
\end{align}
where the functions $\phi_{q,p}$ have trainable parameters (parameterized as B-splines, see Appendix~\ref{app:kan-implmentation-detail}), as detaild below. In the Kolmogov-Arnold theorem, the inner functions form a KAN layer with $n_{\text{in}}=n$ and $n_{\text{ out}}=2n+1$, and the outer functions form a KAN layer with $n_{\text{in}}=2n+1$ and $n_{\text{out}}=1$. So the Kolmogorov-Arnold representations in Eq.~(\ref{eq:KART}) are simply compositions of two KAN layers. Now it becomes clear what it means to have deeper Kolmogorov-Arnold representations: simply stack more KAN layers!

The shape of a general KAN is represented by an integer array 
\begin{align}\label{arraykan}
    [n_0,n_1,\cdots,n_L],
\end{align}
where $n_i$ is the number of nodes in the $i^{\text{th}}$ layer of the computational graph. We denote the $i^{\text{th}}$ neuron in the $l^{\text{th}}$ layer by $(l,i)$, and the activation value of the $(l,i)$-neuron by $x_{l,i}$. Between layer $l$ and layer $l+1$, there are $n_ln_{l+1}$ activation functions: the activation function that connects $(l,i)$ and $(l+1,j)$ is denoted by 
\begin{align}
    \phi_{l,j,i},\quad l=0,\cdots, L-1,\quad i=1,\cdots,n_{l},\quad j=1,\cdots,n_{l+1}.
\end{align}
The pre-activation of $\phi_{l,j,i}$ is simply $x_{l,i}$; the post-activation of $\phi_{l,j,i}$ is denoted by $\tilde{x}_{l,j,i}\equiv \phi_{l,j,i}(x_{l,i})$. The activation value of the $(l+1,j)$ neuron is simply the sum of all incoming post-activations: 
\begin{equation}\label{eq:kanforward}
    x_{l+1,j} =  \sum_{i=1}^{n_l} \tilde{x}_{l,j,i} = \sum_{i=1}^{n_l}\phi_{l,j,i}(x_{l,i}), \qquad j=1,\cdots,n_{l+1}.
\end{equation}
In matrix form, this reads
\begin{equation}\label{eq:kanforwardmatrix}
    \mathbf{x}_{l+1} = 
    \underbrace{\begin{pmatrix}
        \phi_{l,1,1}(\cdot) & \phi_{l,1,2}(\cdot) & \cdots & \phi_{l,1,n_{l}}(\cdot) \\
        \phi_{l,2,1}(\cdot) & \phi_{l,2,2}(\cdot) & \cdots & \phi_{l,2,n_{l}}(\cdot) \\
        \vdots & \vdots & & \vdots \\
        \phi_{l,n_{l+1},1}(\cdot) & \phi_{l,n_{l+1},2}(\cdot) & \cdots & \phi_{l,n_{l+1},n_{l}}(\cdot) \\
    \end{pmatrix}}_{\mathbf{\Phi}_l}
    \mathbf{x}_{l},
\end{equation}
where ${\mathbf \Phi}_l$ is the function matrix corresponding to the $l^{\text{th}}$ KAN layer. A general KAN network is a composition of $L$ layers: given an input vector  $x_0\in\mathbb{R}^{n_0}$, the output of KAN is
\begin{equation}\label{eq:KAN_forward}
    {\text{KAN}}(\mathbf{x}) = (\mathbf{\Phi}_{L-1}\circ \mathbf{\Phi}_{L-2}\circ\cdots\circ\mathbf{\Phi}_{1}\circ\mathbf{\Phi}_{0})\mathbf{x}.
\end{equation}
We can also rewrite the above equation to make it more analogous to Eq.~(\ref{eq:KART}), assuming output dimension $n_{L}=1$, and define $f(\mathbf{x})\equiv {\text{KAN}}(\mathbf{x})$:
\begin{equation}
    f(\mathbf{x})=\sum_{i_{L-1}=1}^{n_{L-1}}\phi_{L-1,i_{L},i_{L-1}}\left(\sum_{i_{L-2}=1}^{n_{L-2}}\cdots\left(\sum_{i_2=1}^{n_2}\phi_{2,i_3,i_2}\left(\sum_{i_1=1}^{n_1}\phi_{1,i_2,i_1}\left(\sum_{i_0=1}^{n_0}\phi_{0,i_1,i_0}(x_{i_0})\right)\right)\right)\cdots\right),
\end{equation}
which is quite cumbersome. In contrast, our abstraction of KAN layers and their visualizations are cleaner and intuitive. The original Kolmogorov-Arnold representation Eq.~(\ref{eq:KART}) corresponds to a 2-Layer KAN with shape $[n,2n+1,1]$. Notice that all the operations are differentiable, so we can train KANs with back propagation. For comparison, an MLP can be written as interleaving of affine transformations $\mathbf{W}$ and non-linearities $\sigma$:
\begin{equation}
    {\text{MLP}}(\mathbf{x}) = (\mathbf{W}_{L-1}\circ\sigma\circ \mathbf{W}_{L-2}\circ\sigma\circ\cdots\circ\mathbf{W}_1\circ\sigma\circ\mathbf{W}_0)\mathbf{x}.
\end{equation}
It is clear that MLPs treat linear transformations and nonlinearities separately as $\mathbf{W}$ and $\sigma$, while KANs treat them all together in $\mathbf{\Phi}$. In Figure~\ref{fig:kan_mlp} (c) and (d), we visualize a three-layer MLP and a three-layer KAN, to clarify their differences. We use $k$-th order B-splines to parameterize the nonlinearities, and implementation details of KANs are left in Appendix~\ref{app:kan-implmentation-detail}.

{\textbf{Remark: Complexities}}. Assuming a KAN with depth $L$, width $N$, grid size $G$, spline order $k$. The model has $O(N^2GL)$ parameters. Suppose a training batch has size $B$, memory usage is $O(2^kBN^2GL)$, the number of operations is $O(2^kBN^2GL)$ both for forward and backward runs. The $2^k$ factor is due to the recursive computation of order $k$ splines.

\subsection{KAN's approximation abilities and scaling laws}\label{subsec:kan_scaling_theory}

Recall that in Eq.~\eqref{eq:KART}, the 2-Layer width-$(2n+1)$ representation may be non-smooth. However, deeper representations may bring the advantages of smoother activations. To facilitate an approximation analysis, we still assume smoothness of activations, but allow the representations to be arbitrarily wide and deep, as in Eq.~(\ref{eq:KAN_forward}). To emphasize the dependence of our KAN on the finite set of grid points, we use $\mathbf{\Phi}_l^G$ and $\Phi_{l,i,j}^G$ below to replace the notation $\mathbf{\Phi}_l$ and $\Phi_{l,i,j}$ used in Eq.~\eqref{eq:kanforward} and  \eqref{eq:kanforwardmatrix}.
\begin{theorem}[Approximation theory, KAN]\label{approx thm}
Let $\mathbf{x}=(x_1,x_2,\cdots,x_n)$.
    Suppose that a function $f(\mathbf{x})$ admits a representation  \begin{equation}\label{fkan}
    f = (\mathbf{\Phi}_{L-1}\circ\mathbf{\Phi}_{L-2}\circ\cdots\circ\mathbf{\Phi}_{1}\circ\mathbf{\Phi}_{0})\mathbf{x}\,,
\end{equation}
 as in Eq.~\eqref{eq:KAN_forward}, where each one of the $\Phi_{l,i,j}$ are  $(k+1)$-times continuously differentiable. Then there exists a constant $C$ depending on $f$ and its representation, such that we have the following approximation bound in terms of the grid size $G$: there exist $k$-th order B-spline functions $\Phi_{l,i,j}^G$ such that for any $0\leq m\leq k$, we have the bound \begin{equation}\label{appro bound}
    \|f-(\mathbf{\Phi}^G_{L-1}\circ\mathbf{\Phi}^G_{L-2}\circ\cdots\circ\mathbf{\Phi}^G_{1}\circ\mathbf{\Phi}^G_{0})\mathbf{x}\|_{C^m}\leq CG^{-k-1+m}\,.
\end{equation}
Here we adopt the notation of $C^m$-norm measuring the magnitude of derivatives up to order $m$: $$
\|g\|_{C^m}=\max _{|\beta| \leq m} \sup _{\mathbf{x}\in [0,1]^n}\left|D^\beta g(\mathbf{x})\right| .
$$
 
\end{theorem}

We leave the proof and an in-depth discussion on the implications of the theorem in Subsection~\ref{app:proof}. Asymptotically, provided that the assumption in Theorem \ref{approx thm} holds, KANs with finite grid size can approximate the function well with a residue rate independent of the dimension. This comes naturally since we only use splines to approximate 1D functions.  In particular, for $m=0$, we recover the accuracy in $L^\infty$ norm, which in turn provides a bound of RMSE on the finite domain, which gives a scaling exponent $k+1$. Of course, the constant $C$ is dependent on the representation; hence it will depend on the dimension. Notice that if the assumption in the theorem holds for a shallow KAN, it automatically holds for a deeper KAN by setting the remaining layers to identity. We also remark that: since the assumption in the theorem is a  strong one, the neural scaling law should not be expected to be universally applicable to all machine learning applications. 

KANs take advantage of the intrinsically low-dimensional compositional representation of underlying functions. This result shares an analogy to the rate in generalization error bounds of finite training samples, for a similar space studied for regression problems; see \cite{horowitz2007rate, kohler2021rate}, and also specifically for MLPs with ReLU activations \cite{schmidt2020nonparametric}. On the other hand, for general Sobolev or Besov spaces, sharp approximation rates have been obtained for ReLU-MLPs (and more generally MLPs with most piecewise polynomial activation functions) \cite{yarotsky2017error, bartlett2019nearly, siegel2023optimal}. These rates exhibit the curse of dimensionality, which is unavoidable due to the fact that Sobolev and Besov spaces with fixed smoothness are very large in high dimensions. By leveraging the representations of KANs using MLPs, we can provide a more general version of approximation theory in a larger function class, establishing KANs as universal approximators, as stated below.

We establish that MLPs can be represented using KANs of a comparable size. Specifically, we show that any MLP with the ReLU$^k$ activation function can be reparameterized as a KAN with a comparable number of parameters. This shows that the approximation and representation capabilities of KANs are at least as good as MLPs. 
\begin{theorem}\label{mlp-kan-representation-thm}
    Let $\Omega\subset \mathbb{R}^d$ be a bounded domain. Suppose that a function $f:\mathbb{R}^d\rightarrow \mathbb{R}$ can be represented by an MLP with width $W\geq 1$, depth $L\geq 1$, and activation function $\sigma_k = \max(0,x)^k$ for $k \geq 1$. Then there exists a KAN $g$ with width $W$, depth at most $2L$, and grid size $G = 2$ with $k$-th order B-spline functions such that
    \begin{equation}
        g(\mathbf{x}) = f(\mathbf{x})
    \end{equation}
    for all $\mathbf{x}\in \Omega$.
\end{theorem}

As a corollary, we can draw conclusions about the approximation capabilities of KANs by leveraging existing results about MLPs (see for instance \cite{barron1993universal,leshno1993multilayer,klusowski2018approximation,siegel2020approximation,siegel2022sharp,hon2022simultaneous,lu2021deep,shen2022optimal,yarotsky2018optimal,siegel2023optimal,yang2023nearly,yang2023optimal}). For example, we have the following result giving optimal approximation rates for very deep KANs on Sobolev spaces (see for instance \cite{adams2003sobolev} for the background on Sobolev spaces).
\begin{corollary}
\label{corollary}
    Let $\Omega\subset \mathbb{R}^d$ be a bounded domain with smooth boundary, $s > 0$ and $1\leq p,q\leq \infty$ be such that $1/q - 1/p < s/d$. This guarantees that the compact Sobolev embedding $W^s(L_q(\Omega))\subset\subset L_p(\Omega)$ holds. 
    
    Let $W_0 := W_0(d)$ be a fixed width (depending upon the input dimension $d$). Then for any $f\in W^s(L_q(\Omega))$ and any $L \geq 1$, there exists a KAN $g$ with width $W_0$, depth $L$, and grid size {$G=2$} with $k$-th order $B$-spline functions such that
    \begin{equation}
        \|f - g\|_{L_p(\Omega)} \leq CL^{-2s/d},
    \end{equation}
    where $C$ is a constant independent of $L$.
\end{corollary}
This result, which follows immediately from Theorem \ref{mlp-kan-representation-thm} and the approximation rates for ReLU (and more generally piecewise polynomial) neural networks derived in \cite{siegel2023optimal}, shows that very deep KANs attain an exceptionally good approximation rate on Sobolev spaces. In particular, in terms of the number of parameters $P$ they attain an approximation rate of $O(P^{-2s/d})$, while a classical (even non-linear) method of approximation can only attain a rate of $O(P^{-s/d})$ \cite{devore1998nonlinear}. This phenomenon, which is often called superconvergence \cite{devore2021neural}, also occurs for very deep ReLU$^k$ networks. However, it comes at the cost of parameters which are not encodable using a fixed number of bits and thus is not practically realizable \cite{yarotsky2020phase,siegel2023optimal}.

Now that the basic architecture of KANs is in place, we propose a few techniques to make KANs accurate and interpretable.

\subsection{Tricks for interpretability: pruning and symbolifying KANs}\label{subsec:kan_simplification}

\begin{comment}
\begin{figure}[t]
    \centering
    \includegraphics[width=1\linewidth]{kan/toy_interpretability_evolution.png}
    \caption{Caption}
    \label{fig:enter-label}
\end{figure}
\end{comment}

How do we choose the KAN shape? If we know that the dataset is generated via the symbolic formula $f(x,y) = {\exp}({\sin}(\pi x)+y^2)$, then we know that a $[2,1,1]$ KAN is able to express this function. However, in practice we do not know the shape a priori, so it would be nice to have approaches to determine this shape automatically. The idea is to start from a large enough KAN and train it with sparsity regularizations followed by pruning. One may even symbolify activation functions into symbolic functions like exp, sine, etc, to make KANs a useful tool for symbolic regression. The idea is to match learned spline functions with candidates in a symbolic function library specified by human users and replace the spline functions with the best-fitting ones.
\subsection{Tricks for accuracy: grid update and grid extension}\label{subsec:kan_grid_extension}
\textbf{Grid update} Since input data and (especially) hidden activations can have time-varying ranges in training, we update grids on the fly based on the statistics of input/activation ranges. The grid is initialized to be in [-1,1] (e.g., when $G=5$, the grid points are [-1, -0.6, -0.2, 0.2, 0.6, 1.0]), but once it receives input/activations, say, in the range [-3,3] (the maximum and minimum values are 3 and -3, respectively), the grid will be updated to [-3,3] (correspondingly, grid points become [-3,-1.8,-0.6,0.6,1.8,3.0]) to accommodate the whole range.

\textbf{Grid extension} A spline can be made arbitrarily accurate to a target function as the grid can be made arbitrarily fine-grained. This good feature can be inherited by KANs. By contrast, MLPs do not have the notion of ``fine-graining''. For KANs, one can first train a KAN with fewer parameters and then extend it to a KAN with more parameters by simply making its spline grids finer, without the need to retrain the larger model from scratch. The main idea of grid extension is: for each 1D function defined on a coarse grid, we determine the coefficient of a finer grid using least squares that minimize the difference between the two curves evaluated on data samples. Details of how to perform grid extension are included in Figure~\ref{fig:spline-notation}.

\subsection{Benefits of deep KANs}

\begin{figure}
    \centering
    \includegraphics[width=0.9\linewidth]{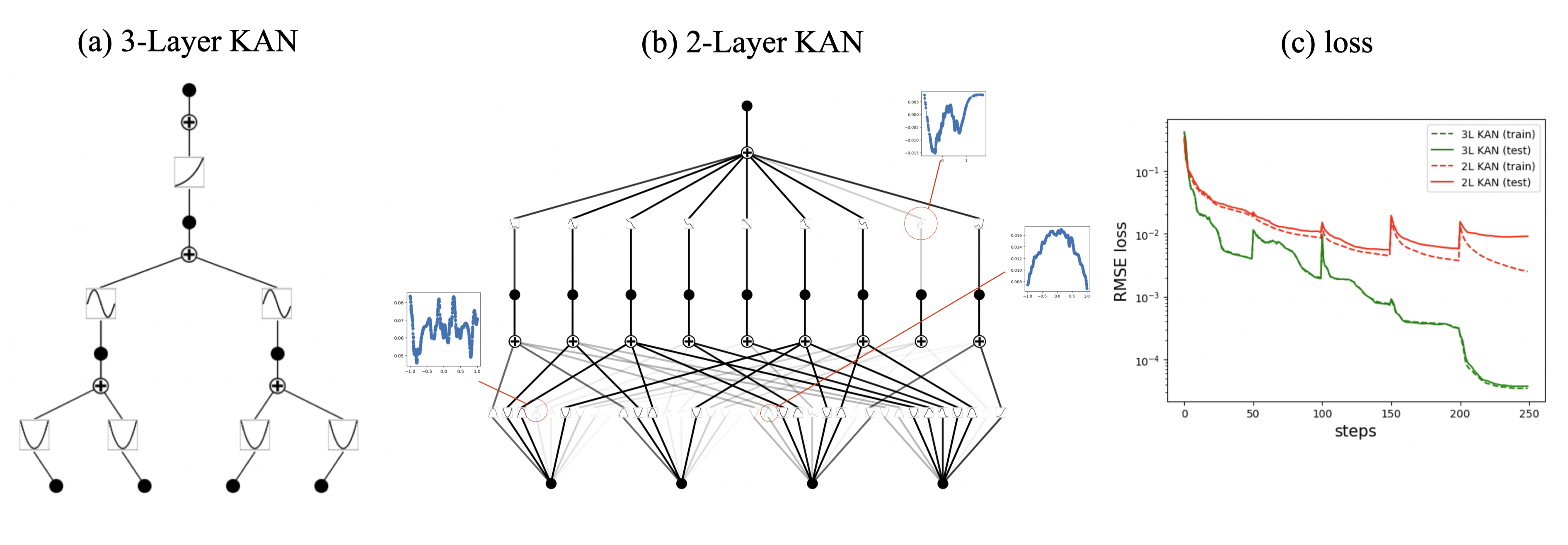}
    \caption{Fitting the function $f(x_1,x_2,x_3,x_4)={\exp}(\frac{1}{2}({\sin}(\pi(x_1^2+x_2^2))+{\sin}(\pi(x_3^2+x_4^2))))$. (a) 3-Layer KAN admits smooth representations. (b) The 2-Layer KAN learns highly oscillatory representations. (c) The 3-layer KAN achieves lower losses and has a smaller train-test gap than the 2-layer KAN.}
    \label{fig:deep-benefit}
\end{figure}

It is one of our major contributions to generalize the 2-layer KA representations to multiple layers. Although it is challenging to prove the benefits of deeper KANs theoretically, we want to present a concrete example where 3-layer KANs admit smooth representations while 2-layer KANs do not. We consider fitting a function $f(x_1,x_2,x_3,x_4)={\exp}(\frac{1}{2}({\sin}(\pi(x_1^2+x_2^2))+{\sin}(\pi(x_3^2+x_4^2))))$ where we draw samples (3000 training, 1000 training) uniformly from $[-1,1]^4$. We train a 3L KAN ([4,2,1,1]) and a 2L KAN ([4,9,1]) with the LBFGS optimizer for 250 steps, with increasing $G=3,5,10,20,50$ (50 steps for each $G$). As shown in Figure~\ref{fig:deep-benefit}, we see that the 3-layer KAN has smooth representations (as expected, since the parse tree of the symbolic formula has depth 3), while the 2-layer KAN learns highly oscillatory functions on some edges. The 3-layer KAN also achieves lower losses than the 2-layer KAN. While the 3-layer KAN has a small train-test gap, the 2-layer KAN starts to overfit at large grid sizes.

\section{KANs Are Interpretable}\label{sec:kan_interpretability_experiment}

In this section, we show that KANs can be interpretable on synthetic toy tasks and realistic research questions in math and physics. 

\textbf{Synthetic toy datasets}\label{subsec:supervised-interpretable}
\begin{figure}[tbp]
    \centering
    \includegraphics[width=0.9\linewidth]{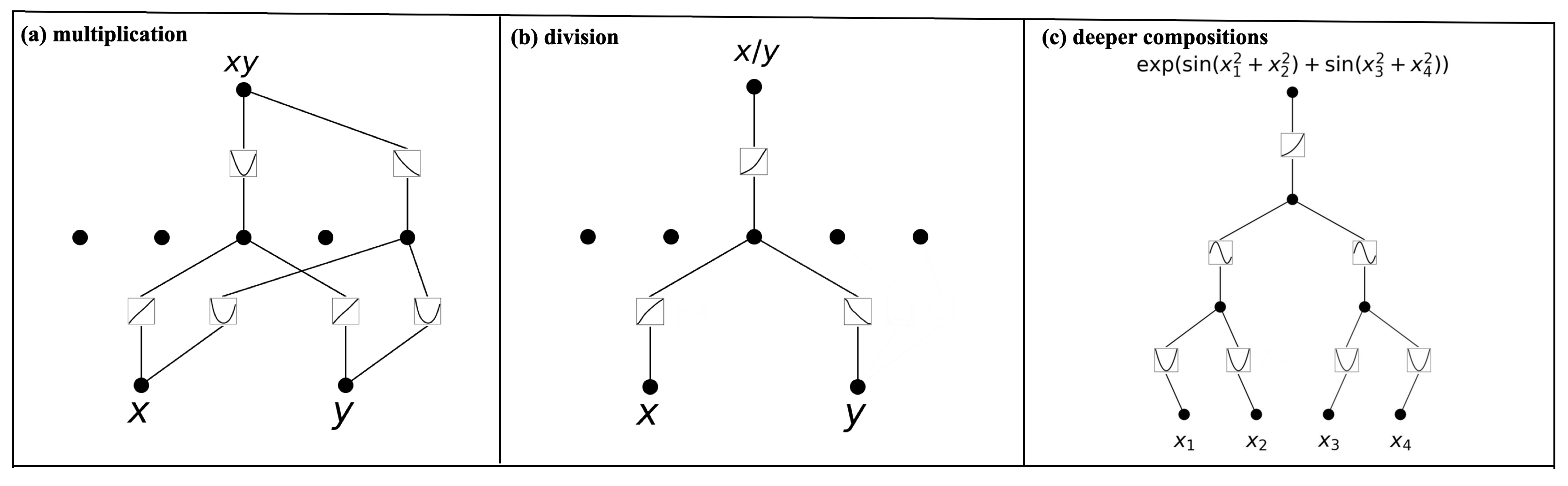}
    \caption{KANs are interpretable for simple symbolic tasks.}
    \label{fig:interpretable_examples}
\end{figure}
We first examine KANs' ability to reveal the compositional structures in symbolic formulas. Three examples are presented in  Figure~\ref{fig:interpretable_examples}. KANs are able to reveal the compositional structures present in these formulas, as well as learn the correct univariate functions. (1) Multiplication $f(x,y)=xy$. KAN computes it via the equation $2xy = (x+y)^2-(x^2+y^2)$. (2) Division of positive numbers $f(x,y)=x/y$. KAN computes it via ${\exp}({\log}x-{\log}y)$. (3) Deeper compositions $f(x_1,x_2,x_3,x_4)={\exp}({\sin}(x_1^2+x_2^2)+{\sin}(x_3^2+x_4^2))$.

\textbf{Application to Mathematics: Knot Theory} Knot theory is a subject in low-dimensional topology that sheds light on topological aspects of three-manifolds and four-manifolds and has a variety of applications, including in biology and topological quantum computing. In \cite{davies2021advancing}, supervised learning and human domain experts were utilized to arrive at a new theorem relating algebraic and geometric knot invariants. They use network attribution methods to find that the signature $\sigma$ is mostly dependent on meridinal distance $\mu$ (real $\mu_r$, imag $\mu_i$) and longitudinal distance $\lambda$. We show   that KANs can not only identify these important variables with much smaller networks and much more automation, but also present some interesting new results and insights.

\begin{figure}[t]
    \centering\includegraphics[width=0.9\linewidth]{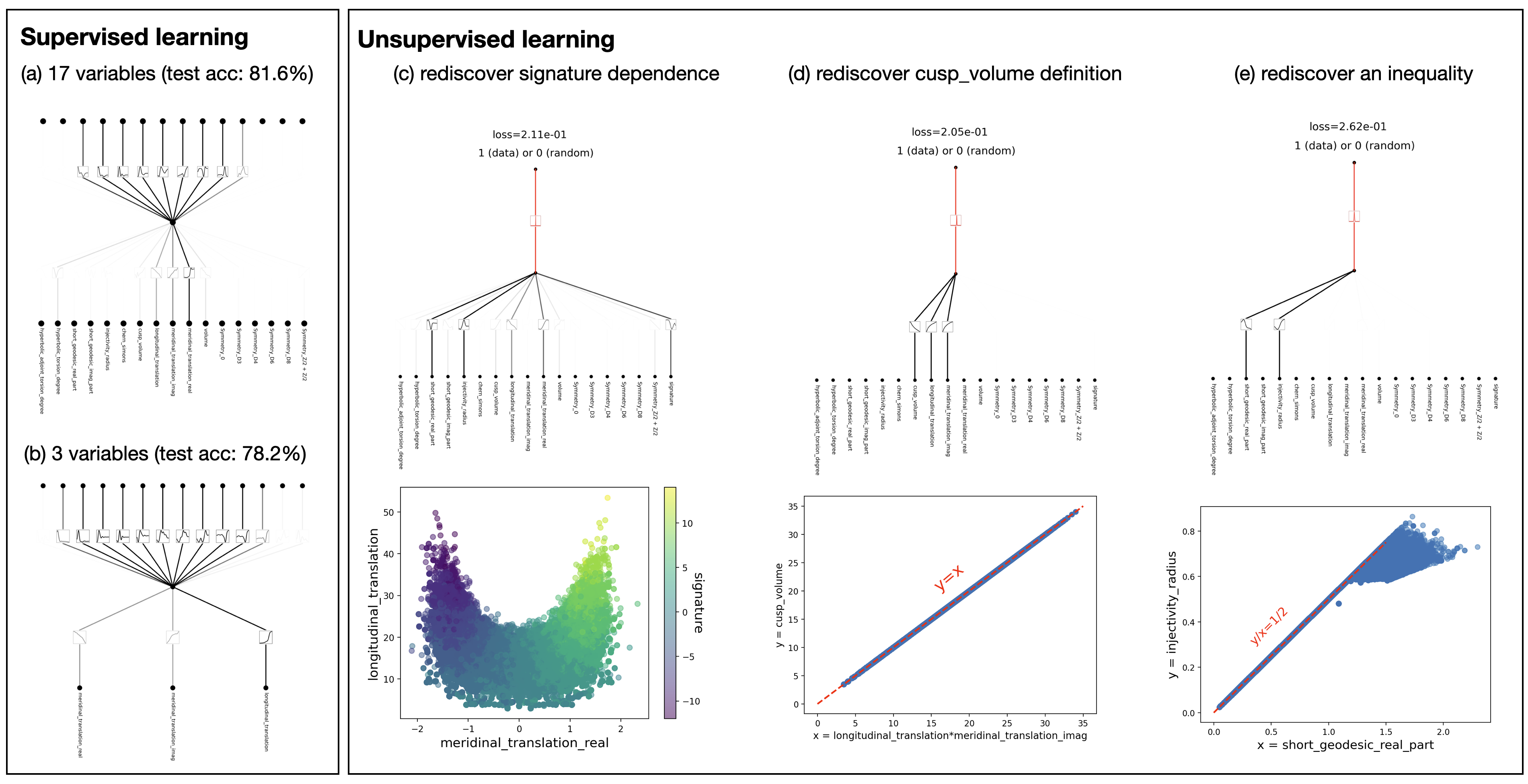}
    \caption{Knot dataset. Supervised mode (left): we rediscover DeepMind's three important variables. Unsupervised mode (right): we discover three ``new'' relations without supervision.}
    \label{fig:knot-supervised}
\end{figure}
  
We treat 17 knot invariants as inputs and signature as outputs. Similar to the setup in~\cite{davies2021advancing}, signatures (which are even numbers) are encoded as one-hot vectors and networks are trained with cross-entropy loss. We find that an extremely small $[17,1,14]$ KAN is able to achieve $81.6\%$ test accuracy (while DeepMind's 4-layer width-300 MLP achieves 78\% test accuracy). The $[17,1,14]$ KAN ($G=3$, $k=3$) has $\approx 200$ parameters, while the MLP has $\approx 3\times 10^5$ parameters. It is remarkable that KANs can be both more accurate and much more parameter efficient than MLPs at the same time. In terms of interpretability, we scale the transparency of each activation according to its magnitude, so it becomes immediately clear which input variables are important without the need for feature attribution (see Figure~\ref{fig:knot-supervised} left top): signature is mostly dependent on $\mu_r$, and slightly dependent on $\mu_i$ and $\lambda$, while dependence on other variables is small. We then train a $[3,1,14]$ KAN on the three important variables, obtaining test accuracy $78.2\%$ (Figure~\ref{fig:knot-supervised} left bottom). 

We attempt to make discoveries beyond DeepMind's in the unsupervised learning mode, where we treat all 18 variables (including signature) as inputs. We train 200 networks with different random seeds. They can be grouped into three clusters, with representative KANs displayed in Figure~\ref{fig:knot-supervised}. These three groups of dependent variables are (1) rediscovering DeepMind's relation in unsupervised learning. (2) cusp volume is by definition of the multiplication of two translations. (3) short geodesic $g_r$ is upper bounded by two times of injecitivy radius~\cite{petersen2006riemannian}. It is interesting that KANs' unsupervised mode can rediscover several known mathematical relations. The good news is that the results discovered by KANs are probably reliable; the bad news is that we have not discovered anything new yet. It is worth noting that we have chosen a shallow KAN for simple visualization, but deeper KANs can probably find more relations if they exist. We would like to investigate how to discover more complicated relations with deeper KANs in future work. 

\section{KANs Are Accurate}\label{sec:kan_accuracy_experiment}

In this section, we demonstrate that KANs are more accurate at representing functions than MLPs in various tasks (regression and PDE solving). When comparing two families of models, it is fair to compare both their accuracy (loss) and their complexity (number of parameters). All experiments reported in the work are reproducible on CPUs, usually within minutes, at most in a day. Codes are built based on PyTorch~\cite {paszke2019pytorch}.

\begin{figure}[ht]
    \centering
    \includegraphics[width=1.0\linewidth]{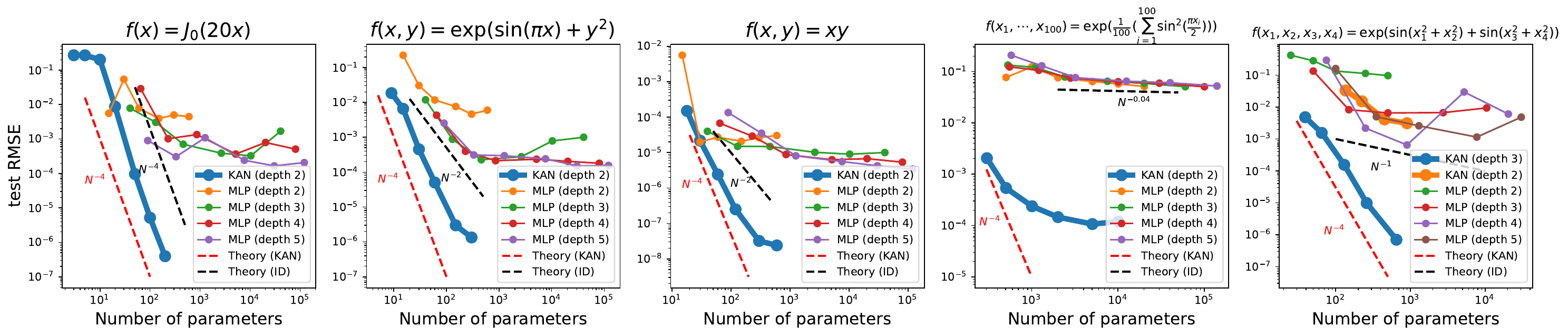}
    \caption{Compare KANs to MLPs on five toy examples. KANs can almost saturate the fastest scaling law predicted by our theory $(\alpha=4)$, while MLPs scales slowly and plateau quickly.}
    \label{fig:model_scaling}
\end{figure}

\textbf{Toy datasets} In Section~\ref{subsec:kan_scaling_theory}, our theory suggested that test RMSE loss $\ell$ scales as $\ell\propto N^{-(k+1)}=N^{-4} (k=3)$ with model parameters $N$. However, this relies on the existence of a smooth Kolmogorov-Arnold representation. As a sanity check, we construct five examples we know have smooth KA representations: (1) $f(x)=J_0(20x)$, which is the Bessel function. Since it is a univariate function, it can be represented by a spline, which is a $[1,1]$ KAN. (2) $f(x,y)={\exp}({\sin}(\pi x)+y^2)$. We know that it can be exactly represented by a $[2,1,1]$ KAN. (3) $f(x,y)=xy$. We know from Figure~\ref{fig:interpretable_examples} that it can be exactly represented by a $[2,2,1]$ KAN. (4) A high-dimensional example $f(x_1,\cdots,x_{100})={\exp}(\frac{1}{100}\sum_{i=1}^{100}{\sin}^2(\frac{\pi x_i}{2}))$ which can be represented by a $[100,1,1]$ KAN. (5) A four-dimensional example $f(x_1,x_2,x_3,x_4)={\exp}(\frac{1}{2}({\sin}(\pi(x_1^2+x_2^2))+{\sin}(\pi(x_3^2+x_4^2))))$ which can be represented by a $[4,4,2,1]$ KAN. The empirical scaling for KANs is quite aligned with theory and outperforms MLPs.

\textbf{Fitting Images} We task KANs with three images: (1) The Cameraman picture is the standard picture for the image fitting task. (2) The turbulence profile is taken from PDEBench~\cite{takamoto2022pdebench}, demonstrating high-frequency and fractal behavior typical in scientific computing. (3) Van Gogh's \textit{The Starry Night} is quite challenging because it contains fine-grained details as well. In addition to MLPs, We compare KANs with these stronger baselines: (A) MLP with random Fourier features (MLP\_RFF). Before feeding input coordinates $\mathbf{x}\equiv(x,y)$ to the MLP, we first augment them into a higher-dimensional feature space $\Phi(\mathbf{x})=(\mathbf{x},\Phi_1(\mathbf{x}),\cdots, \Phi_{N_f}(\mathbf{x}))$, where $\Phi_i(\mathbf{x})=({\cos}(\mathbf{s}_i\cdot \mathbf{x}), {\sin}(\mathbf{s}_i\cdot \mathbf{x})), i=1,\cdots, N_f$, and $\mathbf{s}_i\sim \mathcal{N}(0,s^2)$ ($s$ controls the frequency bias). We choose $N_f=50$ and $s=3, 30$. (B) SIREN~\citep{sitzmann2020implicit} uses sines as activation functions in MLPs and uses large initialization for the first layer (effectively creating high-frequency features). To compare KANs and baselines as fairly as possible, we try two control strategies (same shape or number of parameters) and report both performance (measured by PSNR) and efficiency (wall time). For all baseline models, 1 means their width is the same as KAN 1, while 2 means their number of parameters is (approximately) the same as KAN 1 ($\sqrt{G}$ times wider, where $G=10$ is the grid size used in KAN 1). We also explore KAN 2, which uses a finer grid ($G=100$ instead of $G=10$) for the first layer only (inspired by the idea of random Fourier features in the input layer). The whole image is treated as the training set and there is no test set. All models are trained with the Adam Optimizer for 15000 steps with learning rate decay (5000 steps for learning rate $10^{-3}$, $10^{-4}$ and $10^{-5}$), with batch size 1024, on a V100 GPU.

 We have a few observations from the results: (1) KANs are comparable to or even outperform baseline methods (including SIREN) in terms of PSNR, however with more training time. (2) Having random features in the inputs is useful for MLPs, especially high-frequency random features ($s=30$ outperforms $s=3$). We may also understand KANs' superior performance as being good at generating random features in early layers. By changing the grid size in the first layer from $G=10$ to $G=100$ (KAN 2), PSNR significantly increases with little additional overhead in training time. We show the turbulence profile in Figure~\ref{fig:turbulence-main}.

\begin{figure}
    \centering
    \includegraphics[width=1.0\linewidth]{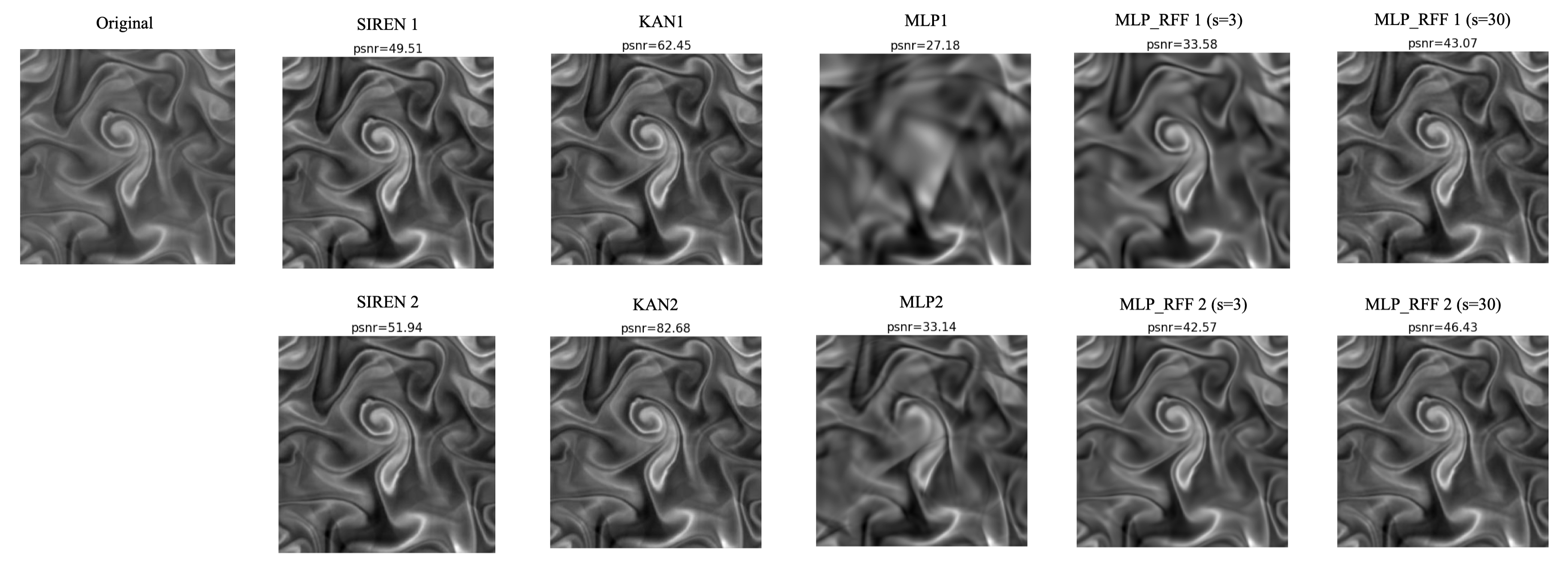}
    \caption{Image fitting task (a PDE solution from PDEBench~\cite{takamoto2022pdebench}). KAN outperforms baseline methods in terms of PSNR.}
    \label{fig:turbulence-main}
\end{figure}

\textbf{Solving partial differential equations (PDEs}
\begin{figure}[t]
    \centering
    \includegraphics[width=1.0\linewidth]{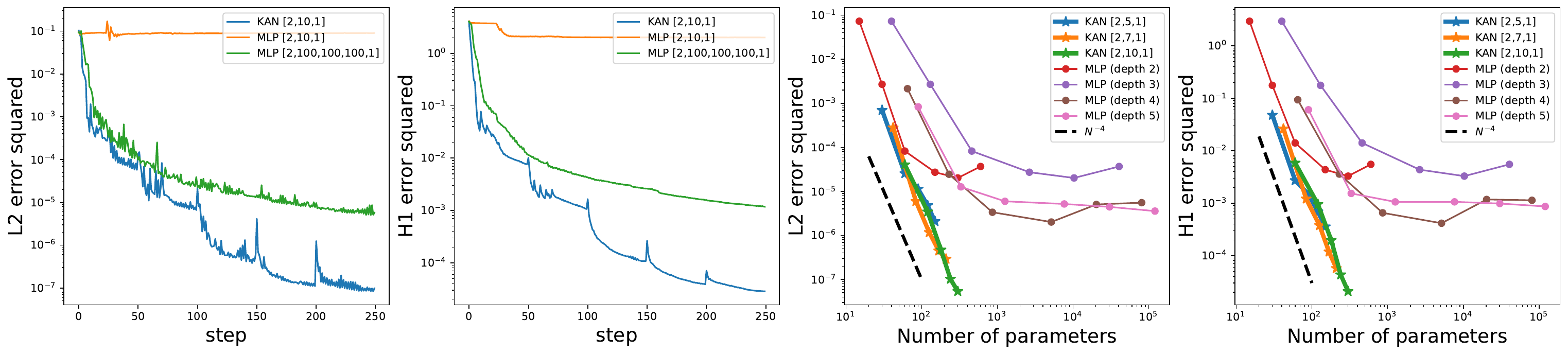}
    \caption{The PDE example. We plot L2 squared and H1 squared losses between the predicted solution and ground truth solution. First and second: training dynamics of losses. Third and fourth: scaling laws of losses against the number of parameters. KANs converge faster, achieve lower losses, and have steeper scaling laws than MLPs.}
    \label{fig:PDE}
\end{figure}
We consider a Poisson equation with zero Dirichlet boundary data. For $\Omega=[-1,1]^2$, consider the PDE 
$u_{xx}+u_{yy}=f$ with zero boundary condition. We consider the data $f=-\pi^2(1+4y^2)\sin(\pi x)\sin(\pi y^2)+2\pi\sin(\pi x)\cos(\pi y^2)$ for which $u=\sin(\pi x)\sin(\pi y^2)$ is the true solution. We use the framework of physics-informed neural networks (PINNs) \cite{raissi2019physics, karniadakis2021physics} to solve this PDE, with the loss function given by $\text{loss}_{\text{pde}}=\alpha\text{loss}_i+\text{loss}_b\coloneqq\alpha\frac{1}{n_i}\sum_{i=1}^{n_i}|u_{xx}(z_i)+u_{yy}(z_i)-f(z_i)|^2+\frac{1}{n_b}\sum_{i=1}^{n_b}u^2\,,$
where we use $\text{loss}_i$ to denote the interior loss, discretized and evaluated by a uniform sampling of $n_i$ points $z_i=(x_i,y_i)$ inside the domain, and similarly we use $\text{loss}_b$ to denote the boundary loss, discretized and evaluated by a uniform sampling of $n_b$ points on the boundary. $\alpha=0.01$ is the hyperparameter balancing the effect of the two terms. KANs are shown to have Pareto frontiers than MLPs for this simple example. 

\section{KANs Have Less Spectral Bias}\label{sec:sp}
In this section, from the perspective of learning and optimization, we study the spectral bias of KANs compared with MLPs. Standard MLPs with ReLU activations (or even tanh) are known to suffer from the spectral bias \cite{rahaman2019spectral,xu2019frequency,xu2019training}, in the sense that they will fit low-frequency components first. This is in contrast to traditional iterative numerical methods like the Jacobi method that learn high frequencies first \cite{xu2019frequency}. Although the spectral bias acts as a regularizer that improves performance for machine learning applications \cite{rahaman2019spectral,zhang2021understanding,poggio2018theory,zhang2020rethink,fridovich2022spectral}, for scientific computing applications, it may be necessary to learn high-frequencies as well. 
To alleviate the spectral bias, high-frequency information has to be encoded using methods like Fourier feature mapping
\cite{sitzmann2020implicit,tancik2020fourier,benbarka2022seeing,novello2024taming}, or one needs to use nonlinear activation functions more similar to traditional methods; see for example, the hat activation function \cite{hong2022activation} which resembles a finite element basis. We demonstrate that KANs are less biased toward low frequencies than MLPs. We highlight that the multi-level learning feature specific to KANs, i.e. grid extension of splines, improves the learning process for high-frequency components.  Detailed comparisons with different choices of depth, width, and grid sizes of KANs are made, shedding some light on how to choose the hyperparameters in practice. 
 We remark that although the spectral bias is considered a form of regularization which is desirable for machine learning tasks \cite{rahaman2019spectral,zhang2021understanding,poggio2018theory,zhang2020rethink,fridovich2022spectral}, for scientific computing applications it is typically important to capture all frequencies and so the spectral bias may negatively affect the performance of neural networks for such applications \cite{wang2022and,rathorechallenges,cai2019multi,hong2022activation}. 

\subsection{Spectral bias theory for shallow KANs}
We consider the spectral bias properties of KANs with a single layer. This theory is very similar to the theory developed in \cite{hong2022activation,zhang2023shallow,ronen2019convergence} for the spectral bias of single hidden layer MLPs. The key observation is that a single layer KAN is a linear model. In particular, we see that if $L = 1$ then the KAN applied to an input $\mathbf{x}\in \mathbb{R}^d$ is (for simplicity, we consider the case of a KAN without the SiLu non-linearity)
\begin{equation}\label{linear-KAN-representation}
{\text{KAN}}(\mathbf{x},\theta)_i=\sum_{j=1}^d\sum_{l=1}^{G+k-1} c_{ijl}B_l(x_j),
\end{equation}
where $\theta = \{c_{ijl}\}$ are the parameters of the KAN. Here $i=1,...,d'$ where $d'$ is the dimension of the output, $j = 1,...,d$, and $l = 1,...,G+k-1$. Note that the only parameter here is the grid size $G$, since the width is determined by the input and output dimensions.

Based upon this, we can analyze least squares fitting with shallow KANs. In particular, let $\Omega = [-1,1]^d$ be the (symmetric) unit cube in $\mathbb{R}^d$, let $f:\Omega\rightarrow \mathbb{R}^{d'}$ be a target function we are trying to learn, and consider the (continuous) least squares regression loss
\begin{equation}\label{least-squares-fitting}
    L(\theta) = \int_{\Omega}\|f(\mathbf{x}) - {\text{KAN}}(\mathbf{x},\theta)\|^2d\mathbf{x}.
\end{equation}
Due to the representation \eqref{linear-KAN-representation}, this loss function is quadratic in the parameters $\theta$. Let $M$ denote the corresponding Hessian matrix, i.e. so that $$L(\theta) = (1/2)\theta^TM\theta + b^T\theta.$$ This Hessian matrix (indexed by $i,j,l$) is given by
\begin{equation}\label{hessian-matrix-equation}
    M_{(i,j,l),(i',j',l')} = \begin{cases}
        \int_\Omega B_l(x_j)B_{l'}(x_{j'})d\mathbf{x} & i = i'\\
        0 & i\neq i'.
    \end{cases}
\end{equation}

The convergence of gradient descent on the least squares regression is determined by the eigendecomposition of the Hessian matrix $M$ which is estimated in the following theorem. {The theorem is essentially a generalization of the fact that the Gram matrix of the B-spline basis is well conditioned (see for instance \cite{devore1993constructive}, Theorem 4.2 in Chapter 5).}
\begin{theorem}\label{hessian-eigenvalue-bound-theorem}
    Given a single hidden layer KAN with grid size $G$, degree $k$ B-splines, input dimension $d$ and output dimension $d'$, let $M$ denote the Hessian matrix defined in \eqref{hessian-matrix-equation} corresponding to the least squares fitting problem \eqref{least-squares-fitting}. Then the eigenvalues $0\leq \lambda_1(M) \leq \cdots \leq \lambda_{N}(M)$ (here $N = (G+k-1)dd'$) satisfy
    \begin{equation}
        \frac{\lambda_{N}(M)}{\lambda_{d'(d-1) + 1}(M)} \leq Cd
    \end{equation}
    for a constant $C$ depending only on the spline degree $k$.
\end{theorem}
Theorem \ref{hessian-eigenvalue-bound-theorem} shows that away from $d'(d-1)$ eigenvectors the matrix $M$ is well conditioned. This means that gradient descent will converge at the same rate in all directions orthogonal to these $d'(d-1)$ eigenvectors. Note that since the number of eigenvectors we must remove is independent of the grid size $G$, we expect that when $G$ is relatively large most components of the KAN will converge at roughly the same rate toward the solution. Thus the KAN with a large number of grid points will not exhibit the same spectral bias toward low frequencies seen by MLPs. {We remark that in constrast the Hessian associated with a two layer ReLU MLP with width $n$ has a condition number which scales like $n^4$ \cite{hong2022activation}, which explains the strong spectral bias exhibited by ReLU MLPs.}
\begin{remark}
    We note that the $d'(d-1)$ eigenvectors which must be excluded is not an artifact of the proof. In fact, this is due to the fact that the KAN parameterization is not unique. Indeed, the constant function $f(x) = 1$ can be parameterized in $d$ different ways by using the B-splines in each of the $d$ different directions. This ambiguity gives rise to directions in parameter space where the function parameterized by the KAN doesn't change and this results in degenerate eigenvectors of the matrix $M$.  
\end{remark}
The analysis given is necessarily highly simplified and heuristic. In particular, we only analyze a single layer of the KAN network and consider the continuous least squares loss. Nonetheless, we argue that it gives an explanation for why we would expect KANs to have a significantly different spectral bias than MLPs, and in particular why we expect that they learn all frequencies roughly similarly. In the remainder of this section, we experimentally test this hypothesis and compare the spectral bias of KANs with MLPs on a variety of simple problems. We implement these numerical experiments using the pykan package version $0.2.5$.

\subsection{1D waves of different frequencies}
In the first example, we take the same setting as in \cite{rahaman2019spectral} and study the regression of a linear combination of waves of different frequencies.
Consider the function prescribed as 
$$f(x)=\sum A_i \sin \left(2 \pi k_i z+\varphi_i\right),\quad k=(5,10,\cdots,45,50).$$
The phases $\varphi_i$ are uniformly sampled from $[0,2\pi]$ and we consider two cases of amplitudes: one with equal amplitude $A_i=1$ and another with increasing amplitude $A_i=0.1i$. We use a neural network, either ReLU MLP or KAN, to regress $f$ sampled at $200$ uniformly spaced points in $[0,1]$, with full batch ADAM iteration as the optimizer with a learning rate of $0.0003$. For MLPs, we train with $80000$ iterations as in \cite{rahaman2019spectral}; for KANs, we only train with 8000 iterations. Normalized magnitudes of discrete Fourier
transform at frequencies $k_i$ are computed as $|\tilde{f}_{k_i}/A_i|$ and averaged over 10 runs of different phases.

We plot the evolution of $|\tilde{f}_{k_i}/A_i|$ during training across all frequencies; see Figures \ref{fig:1D_eqn}, \ref{fig:1D_inc} for comparisons of MLPs and KANs with different sizes for equal and increasing amplitudes respectively. KANs suffer significantly less than MLPs from spectral biases. Once the size of KANs, especially the grid size and depth is large enough, KANs almost learn all frequencies at the same time, while even very deep and wide MLPs still have difficulties learning higher frequencies, even with $10$x epochs!

\begin{figure}[htbp]
    \centering
    \begin{minipage}[b]{0.32\textwidth}
        \centering
        \includegraphics[width=\textwidth]{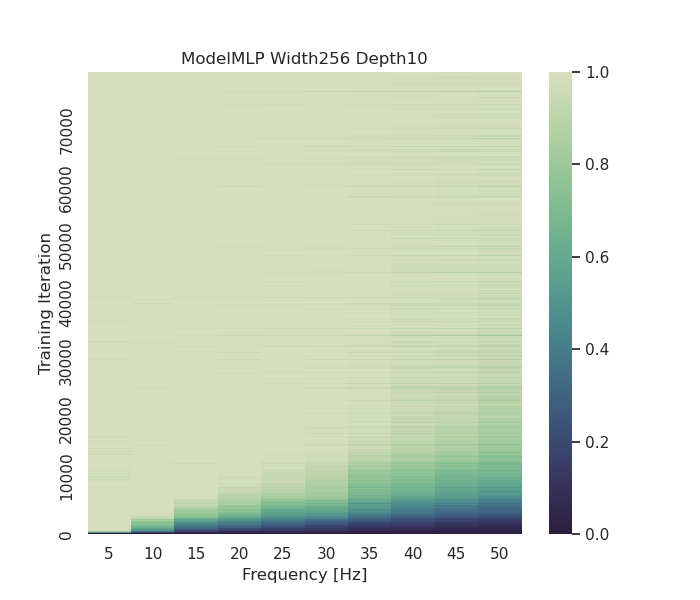}
    
    \end{minipage}
    \hfill
    \begin{minipage}[b]{0.32\textwidth}
        \centering
        \includegraphics[width=\textwidth]{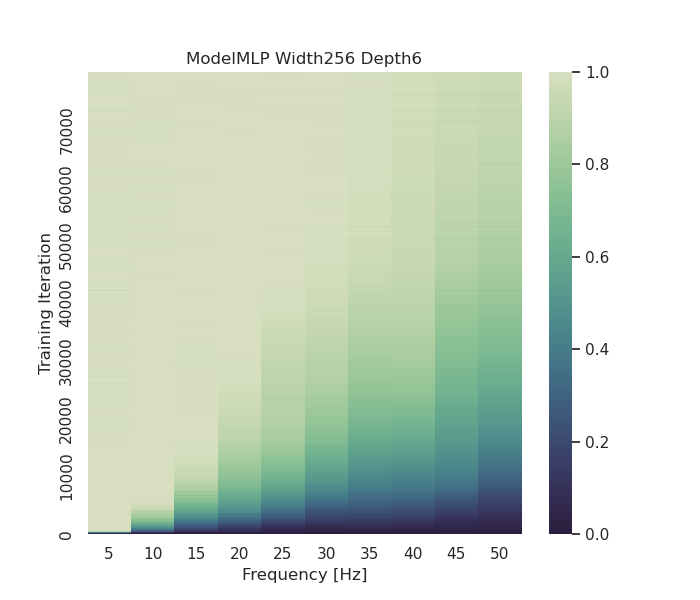}
    \end{minipage}
    \hfill
    \begin{minipage}[b]{0.32\textwidth}
        \centering
        \includegraphics[width=\textwidth]{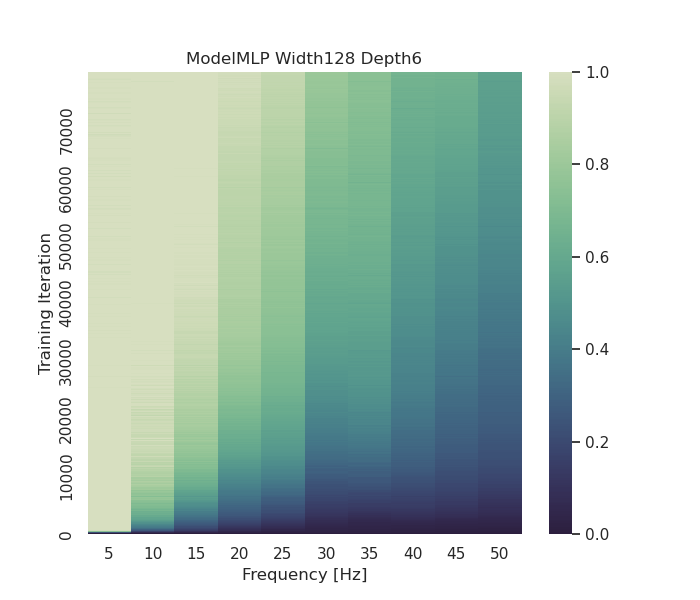}
    \end{minipage}

    \vspace{1em}

    \begin{minipage}[b]{0.32\textwidth}
        \centering
        \includegraphics[width=\textwidth]{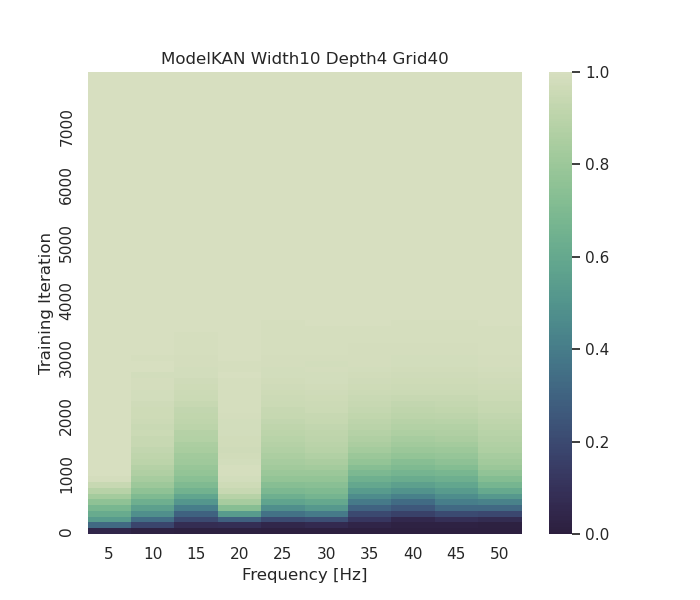}
    \end{minipage}
    \hfill
    \begin{minipage}[b]{0.32\textwidth}
        \centering
        \includegraphics[width=\textwidth]{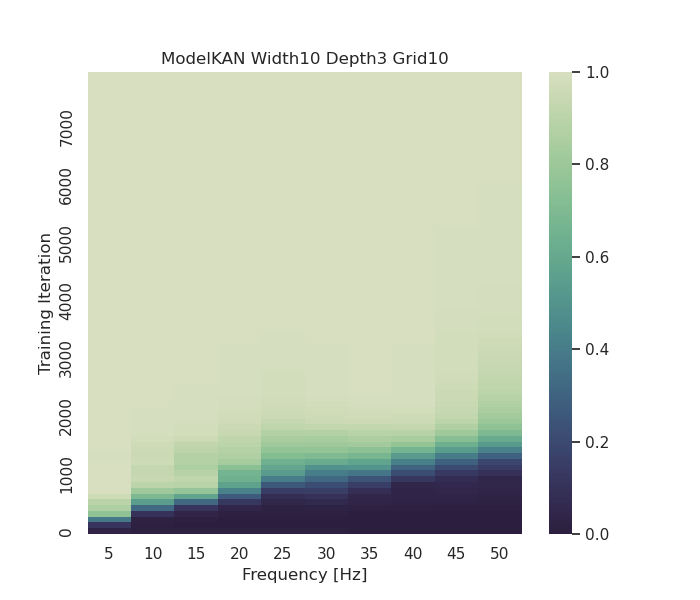}
    \end{minipage}
    \hfill
    \begin{minipage}[b]{0.32\textwidth}
        \centering
        \includegraphics[width=\textwidth]{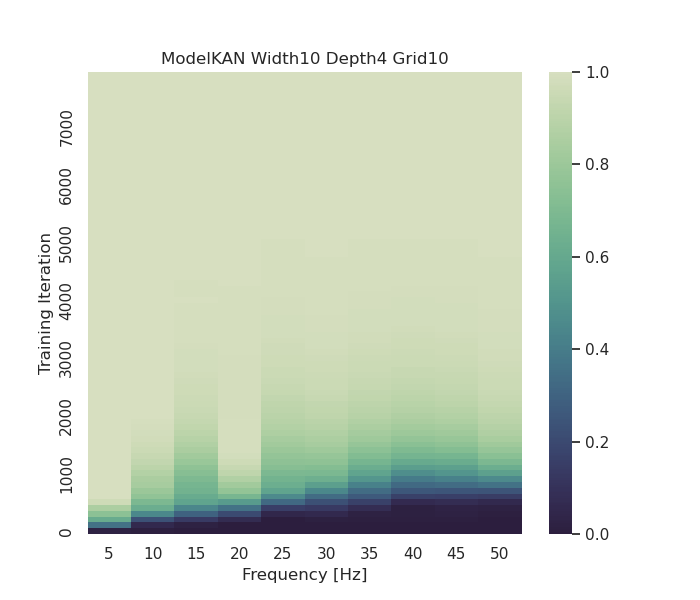}
    
    \end{minipage}
    
    \caption{1D wave dataset, where the target function has equal amplitudes of different frequency modes. Under various hyperparameters, MLPs manifest strong spectral biases (top), while KANs do not (bottom). Note that the y axis (training steps) of MLP is 10 times that of KAN.}
    \label{fig:1D_eqn}
\end{figure}

\begin{figure}[htbp]
    \centering
    \begin{minipage}[b]{0.32\textwidth}
        \centering
        \includegraphics[width=\textwidth]{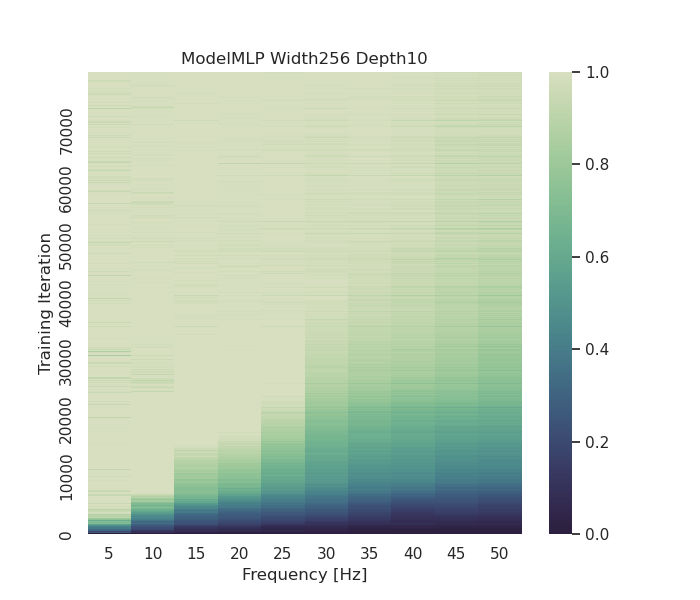}
    
    \end{minipage}
    \hfill
    \begin{minipage}[b]{0.32\textwidth}
        \centering
        \includegraphics[width=\textwidth]{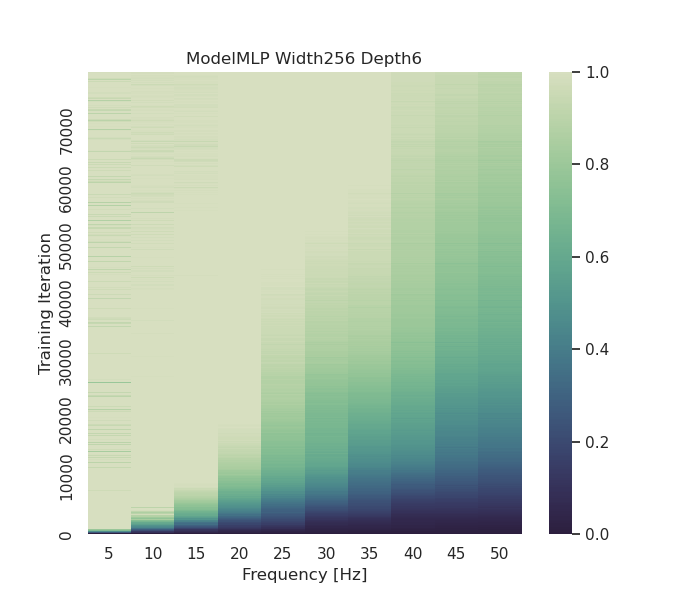}
    \end{minipage}
    \hfill
    \begin{minipage}[b]{0.32\textwidth}
        \centering
        \includegraphics[width=\textwidth]{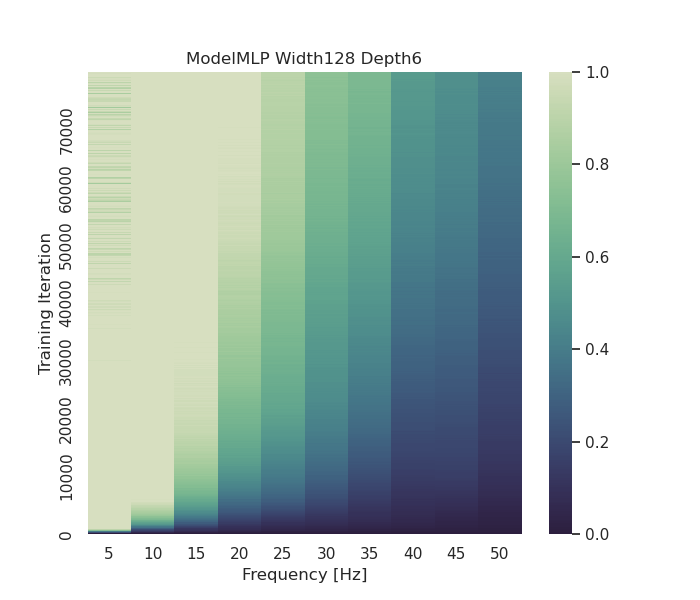}
    \end{minipage}

    \vspace{1em}

    \begin{minipage}[b]{0.32\textwidth}
        \centering
        \includegraphics[width=\textwidth]{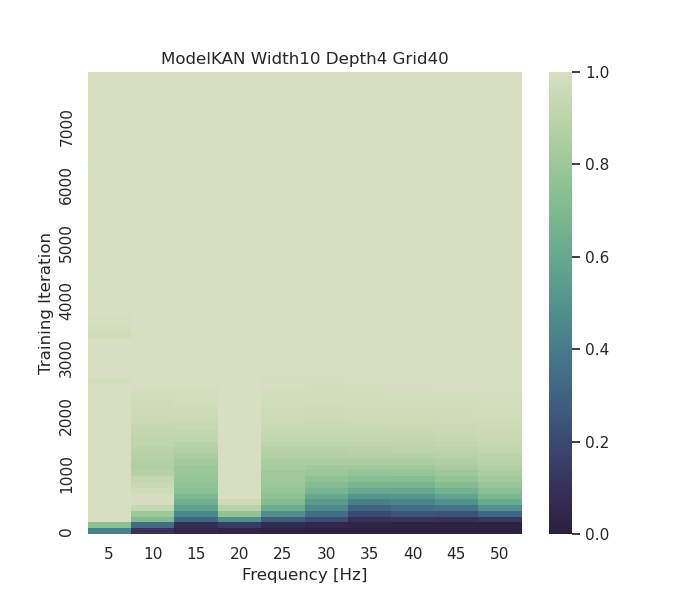}
    \end{minipage}
    \hfill
    \begin{minipage}[b]{0.32\textwidth}
        \centering
        \includegraphics[width=\textwidth]{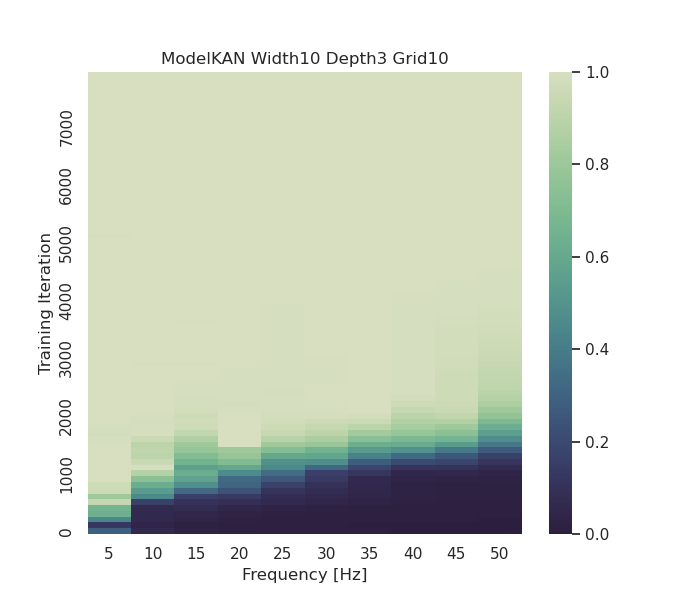}
    \end{minipage}
    \hfill
    \begin{minipage}[b]{0.32\textwidth}
        \centering
        \includegraphics[width=\textwidth]{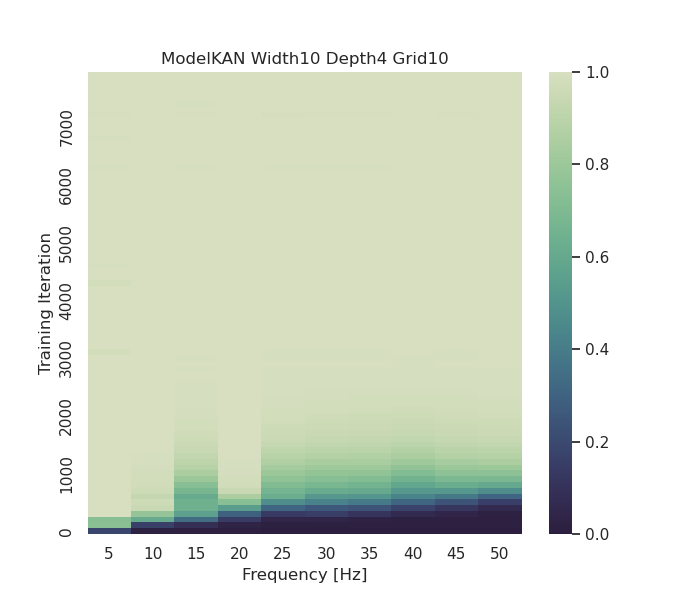}
    
    \end{minipage}
    
    \caption{1D wave dataset, where the target function has increasing amplitudes of different frequency modes. Under various hyperparameters, MLPs manifest severe spectral biases (top), while KANs do not (bottom). Note that the y axis (training steps) of MLP is 10 times that of KAN.}
    \label{fig:1D_inc}
\end{figure}

% \begin{figure}[htbp]
%     \centering
%     \begin{minipage}[b]{0.32\textwidth}
%         \centering
%         \includegraphics[width=\textwidth]{kan/1D_frequency/width10depth3grid10modelKAN.png}
    
%     \end{minipage}
%     \hfill
%     \begin{minipage}[b]{0.32\textwidth}
%         \centering
%         \includegraphics[width=\textwidth]{kan/1D_frequency/width3depth3grid10modelKAN.png}
%     \end{minipage}
%     \hfill
%     \begin{minipage}[b]{0.32\textwidth}
%         \centering
%         \includegraphics[width=\textwidth]{kan/1D_frequency/width10depth2grid10modelKAN.png}
%     \end{minipage}

%     \vspace{1em}

%     \begin{minipage}[b]{0.32\textwidth}
%         \centering
%         \includegraphics[width=\textwidth]{kan/1D_frequency/width10depth4grid10modelKAN.png}
%     \end{minipage}
%     \hfill
%     \begin{minipage}[b]{0.32\textwidth}
%         \centering
%         \includegraphics[width=\textwidth]{kan/1D_frequency/width10depth3grid20modelKAN.png}
%     \end{minipage}
%     \hfill
%     \begin{minipage}[b]{0.32\textwidth}
%         \centering
%         \includegraphics[width=\textwidth]{kan/1D_frequency/width10depth3grid40modelKAN.png}
    
%     \end{minipage}
    
%     \caption{We can see that as we increase the layers, we can learn higher and higher frequency.}
%     \label{fig:1D_eqn_kan}
% \end{figure}
\subsection{Gaussian random field}
In this example, we consider fitting functions sampled from a Gaussian random field.
The target function $f$ is sampled from a $d$-dimensional Gaussian random field with mean zero and covariance $\exp(-|x-y|^2/(2\sigma^2))$.  Here small $\sigma$ corresponds to rough functions and large $\sigma$ corresponds to smooth functions. 

To approximate the Gaussian random field, we sample $f$ using the KL expansion \cite{karhunen1947under}. We sample $N=5000$ points uniformly from $[-1,1]^d$ and calculate the (empirical) covariance matrix $K$. Then we truncate its first $m<N$ eigenpairs $\lambda_i,\phi_i$, with the cutoff threshold $\lambda_{m+1}<0.1\lambda_{1}\leq\lambda_{m}$ and sample $f$ approximately via $$f=\sum_{i\leq m}\lambda_i\xi_i\phi_i,$$ where $\xi_i$ are i.i.d standard Gaussians $N(0,1)$. For $f$ with different scales $\sigma$ and dimensions $d$, we split the points into $80\%$ training and $20\%$ testing points. We use MLPs and KANs with different sizes to regress on the training set, with the mean squared loss as the loss function. For MLPs, we use $500$ iterations of LBFGS iteration, and for KANs, we use the grid extension technique, with grid sizes $(10,20,30,40,50)$, each trained with $100$ iterations of LBFGS.

We plot the loss curves here and compare the losses of different scales $\sigma$ and dimensions $d$, using an MLP of $6$ layers and $256$ neurons in each hidden layer, and KANs with $10$ neurons in each hidden layer and $2,3,4$ layers; see Figure \ref{grf:train} for the regression loss on the training set with dimensions $2,3,4$ and scales $2^i$, $i=0,-1,-2,-3$. We see that for larger scale and smoother functions, MLP performs better, while for smaller scale and rough functions, KANs perform better without suffering much from spectral biases, and grid extension is especially helpful. We remark that one can choose smaller grid sizes of KANs for smoother functions and obtain more accurate regressions.

Precisely since KANs are not susceptible to spectral biases, they are likely to overfit to noises. As a consequence,
we notice that KANs are more subject to overfitting on the training data regression when the task is very complicated; see the second line of Figure \ref{grf:test}. On the other hand, we can increase the number of training points to alleviate the overfitting; see the last line of Figure \ref{grf:test} where we increased the number of training and test samples by $10$x. We remark that the current implementation of grid extension is prone to oscillation after refining grids during the undersampled regime, as observed in \cite{rigas2024adaptive}, and we will improve it in future works. 
\begin{figure}[htbp]
    \centering
    \begin{minipage}[b]{0.24\textwidth}
        \centering
        \includegraphics[width=\textwidth]{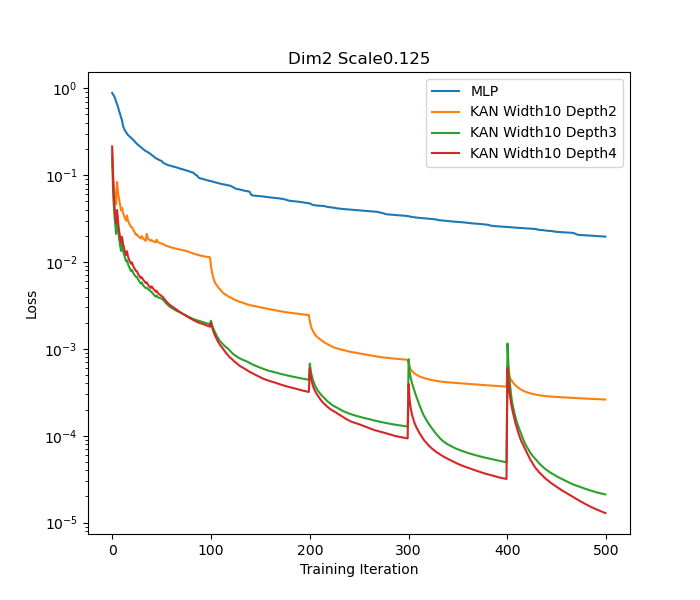}
    
    \end{minipage}
    \hfill
    \begin{minipage}[b]{0.24\textwidth}
        \centering
        \includegraphics[width=\textwidth]{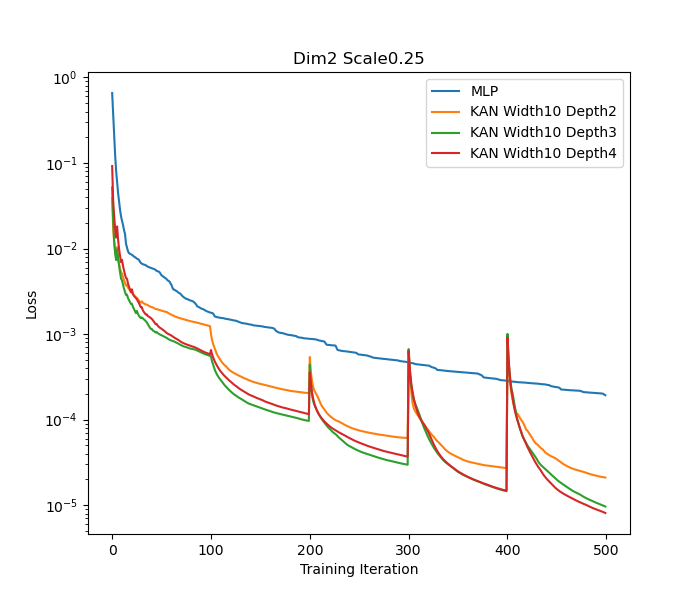}
    \end{minipage}
    \hfill
    \begin{minipage}[b]{0.24\textwidth}
        \centering
        \includegraphics[width=\textwidth]{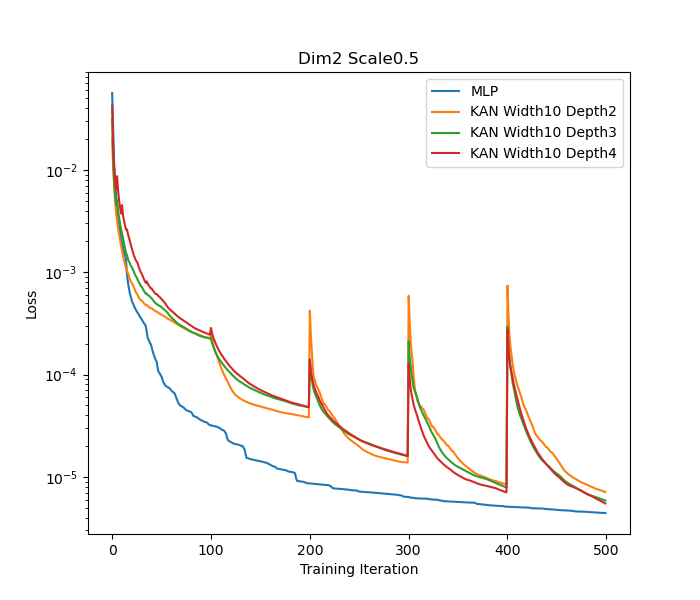}
    \end{minipage}
    \hfill
    \begin{minipage}[b]{0.24\textwidth}
        \centering
        \includegraphics[width=\textwidth]{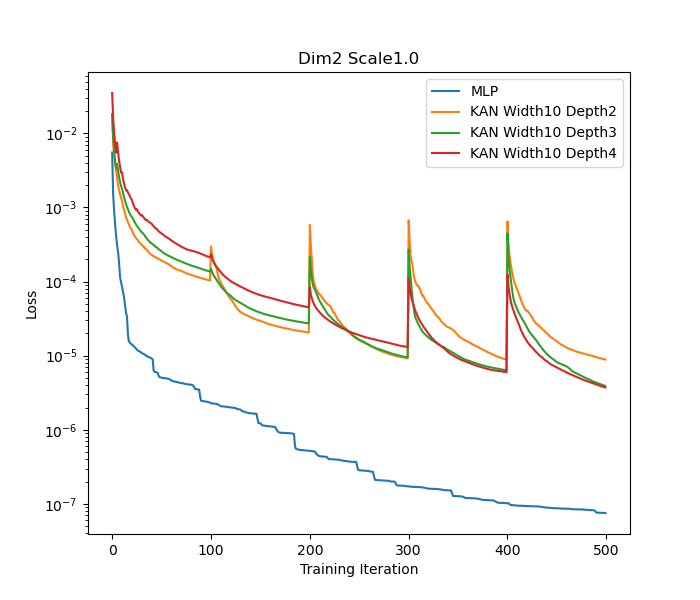}
    \end{minipage}
    \vspace{1em}
\begin{minipage}[b]{0.24\textwidth}
        \centering
        \includegraphics[width=\textwidth]{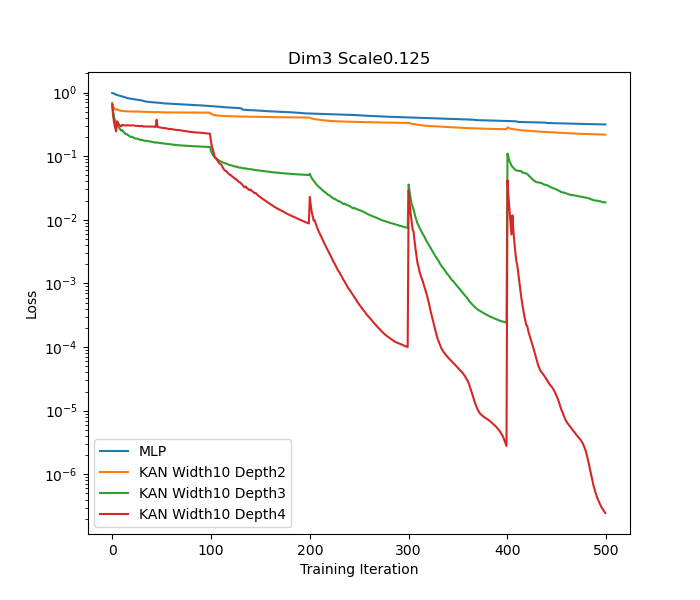}
    
    \end{minipage}
    \hfill
    \begin{minipage}[b]{0.24\textwidth}
        \centering
        \includegraphics[width=\textwidth]{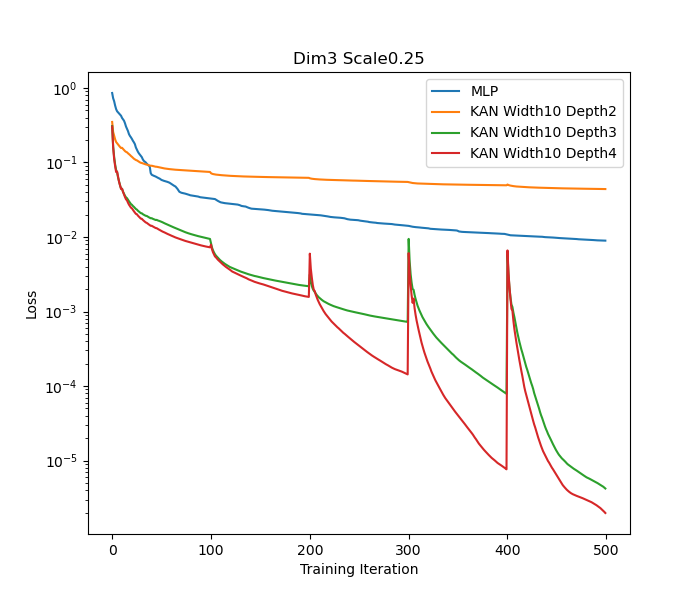}
    \end{minipage}
    \hfill
    \begin{minipage}[b]{0.24\textwidth}
        \centering
        \includegraphics[width=\textwidth]{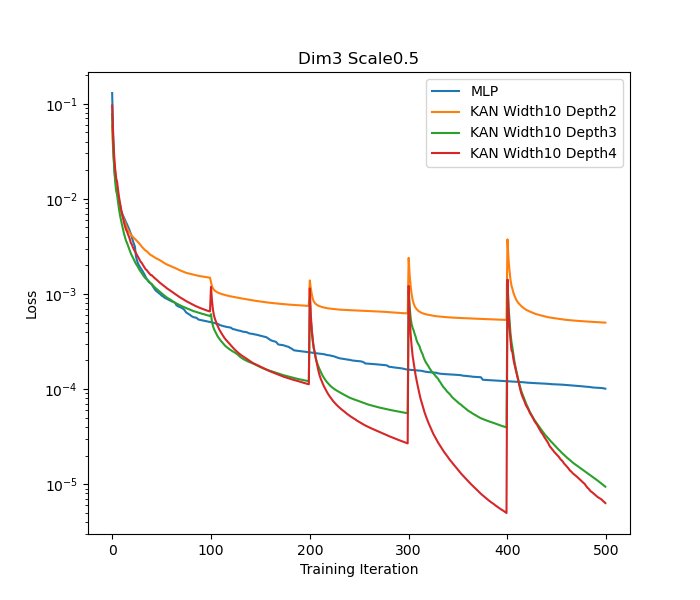}
    \end{minipage}
    \hfill
    \begin{minipage}[b]{0.24\textwidth}
        \centering
        \includegraphics[width=\textwidth]{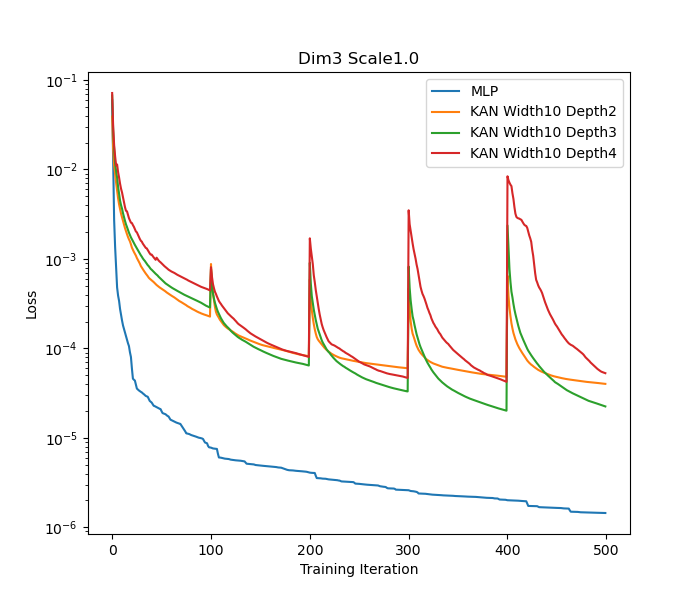}
    \end{minipage}
    \vspace{1em}\begin{minipage}[b]{0.24\textwidth}
        \centering
        \includegraphics[width=\textwidth]{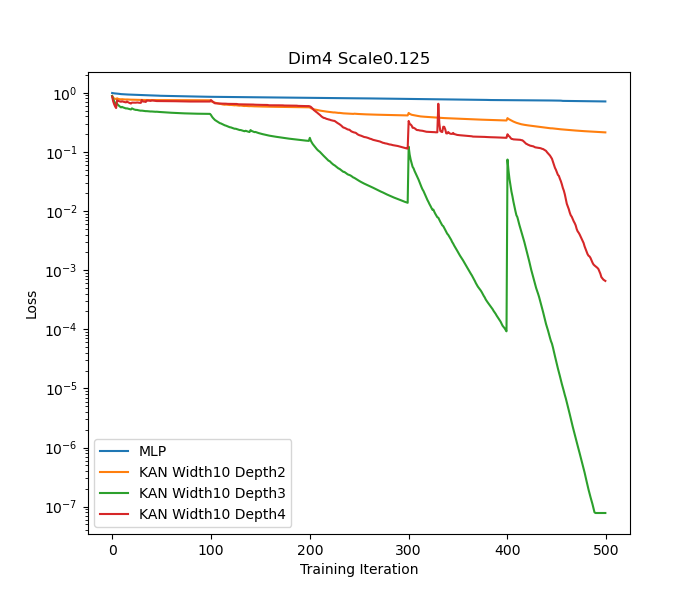}
    
    \end{minipage}
    \hfill
    \begin{minipage}[b]{0.24\textwidth}
        \centering
        \includegraphics[width=\textwidth]{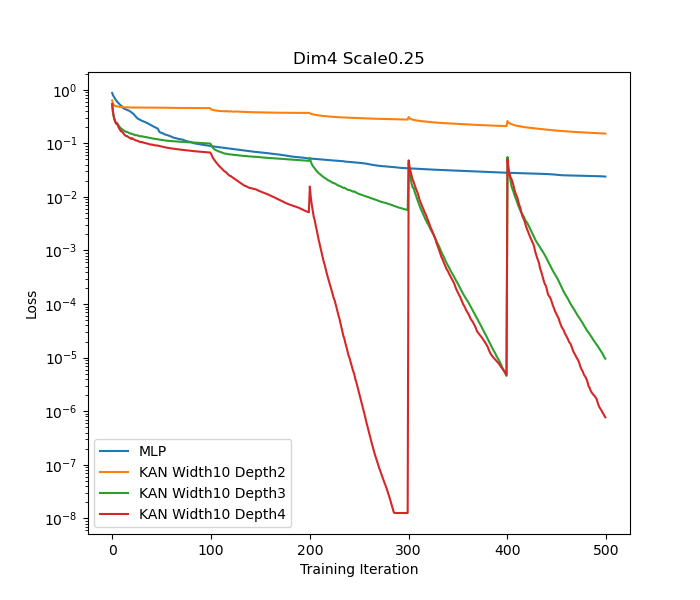}
    \end{minipage}
    \hfill
    \begin{minipage}[b]{0.24\textwidth}
        \centering
        \includegraphics[width=\textwidth]{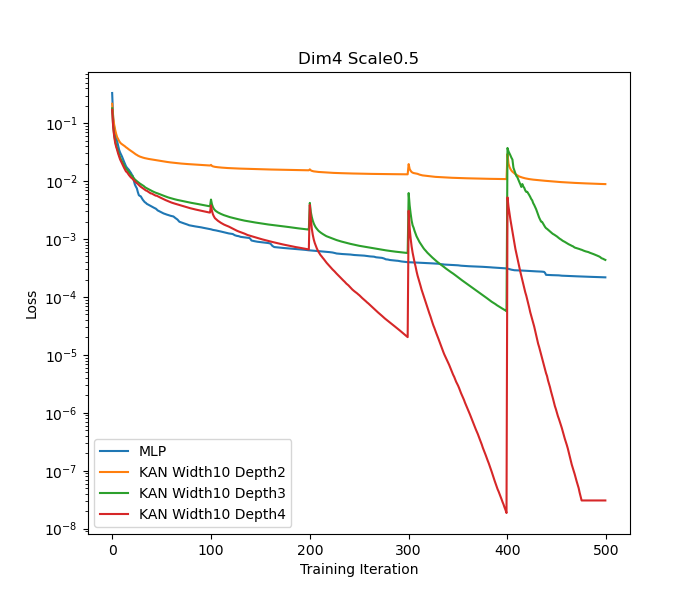}
    \end{minipage}
    \hfill
    \begin{minipage}[b]{0.24\textwidth}
        \centering
        \includegraphics[width=\textwidth]{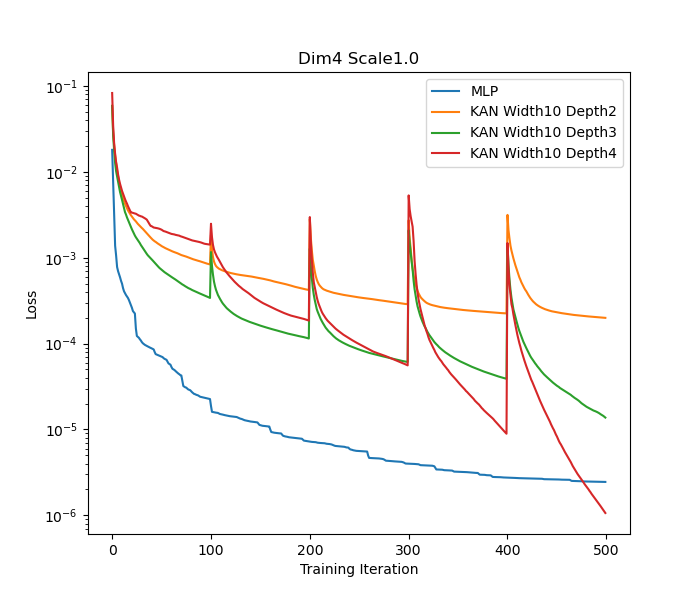}
    \end{minipage}
    \vspace{1em}

    \caption{The Gaussian random field dataset. Training losses of MLP and KANs, with different scales and dimensions.}
    \label{grf:train}
\end{figure}

\begin{figure}[htbp]
    \centering
    \begin{minipage}[b]{0.24\textwidth}
        \centering
        \includegraphics[width=\textwidth]{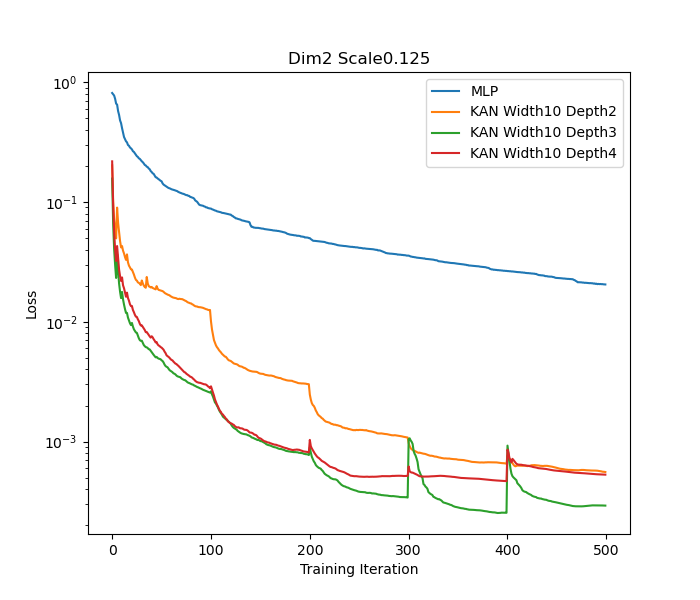}
    
    \end{minipage}
    \hfill
    \begin{minipage}[b]{0.24\textwidth}
        \centering
        \includegraphics[width=\textwidth]{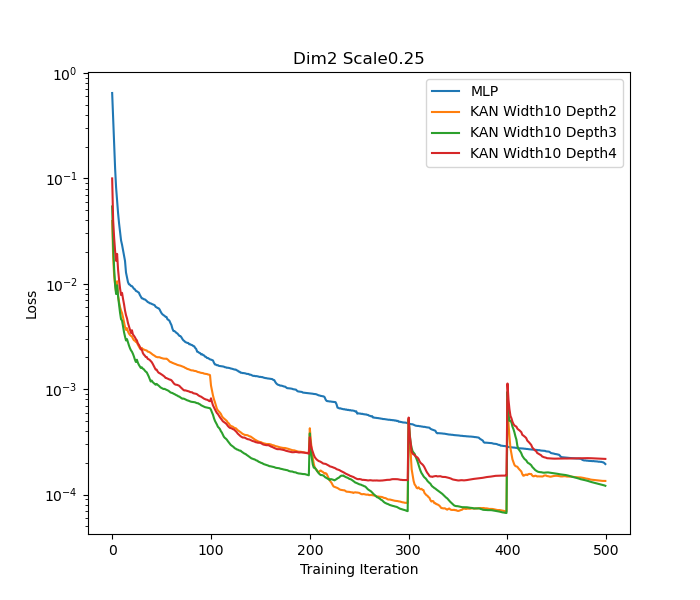}
    \end{minipage}
    \hfill
    \begin{minipage}[b]{0.24\textwidth}
        \centering
        \includegraphics[width=\textwidth]{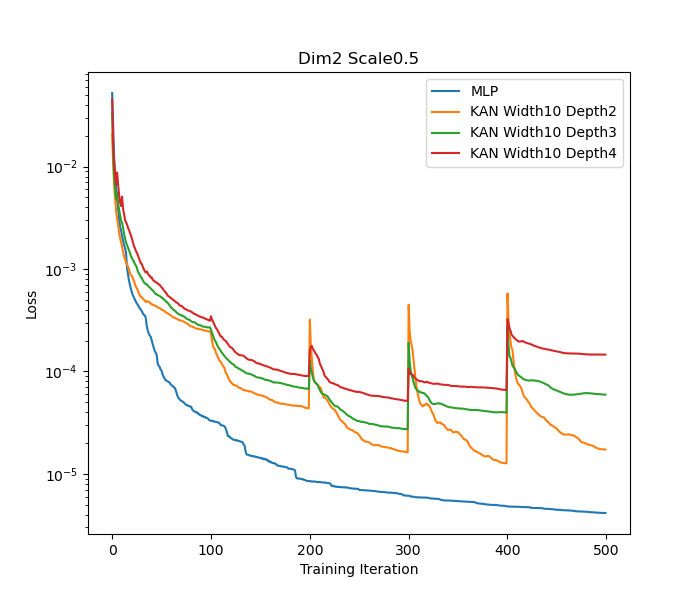}
    \end{minipage}
    \hfill
    \begin{minipage}[b]{0.24\textwidth}
        \centering
        \includegraphics[width=\textwidth]{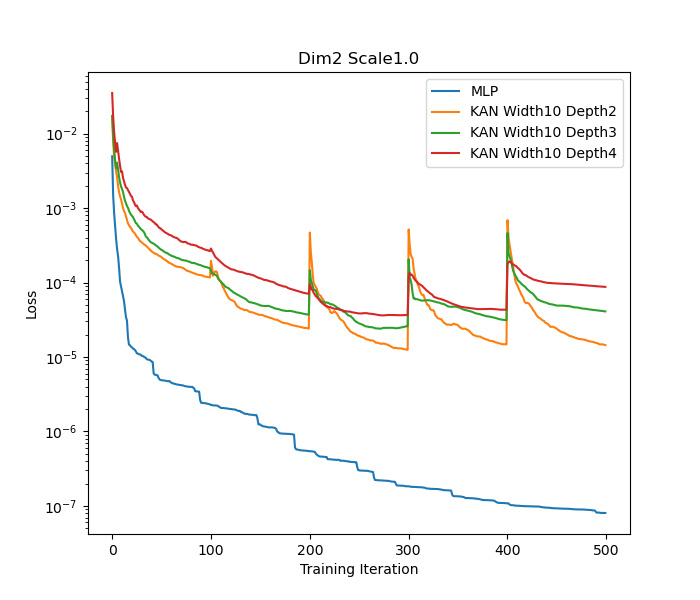}
    \end{minipage}
    \vspace{1em}
\begin{minipage}[b]{0.24\textwidth}
        \centering
        \includegraphics[width=\textwidth]{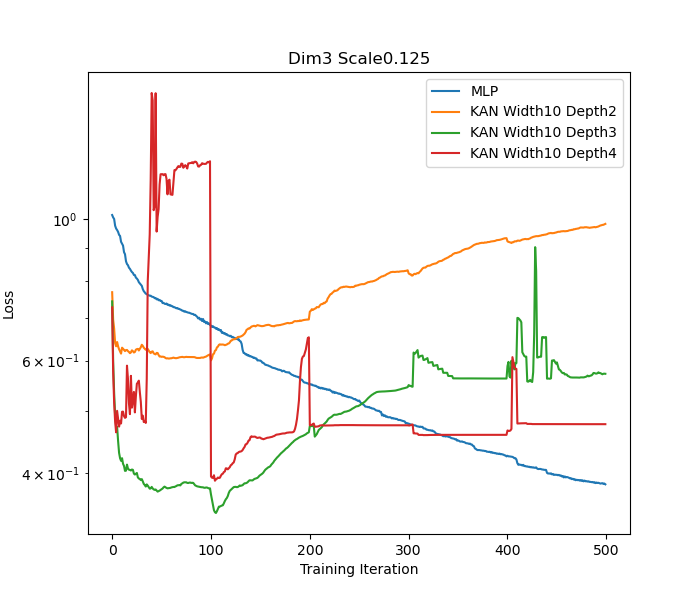}
    
    \end{minipage}
    \hfill
    \begin{minipage}[b]{0.24\textwidth}
        \centering
        \includegraphics[width=\textwidth]{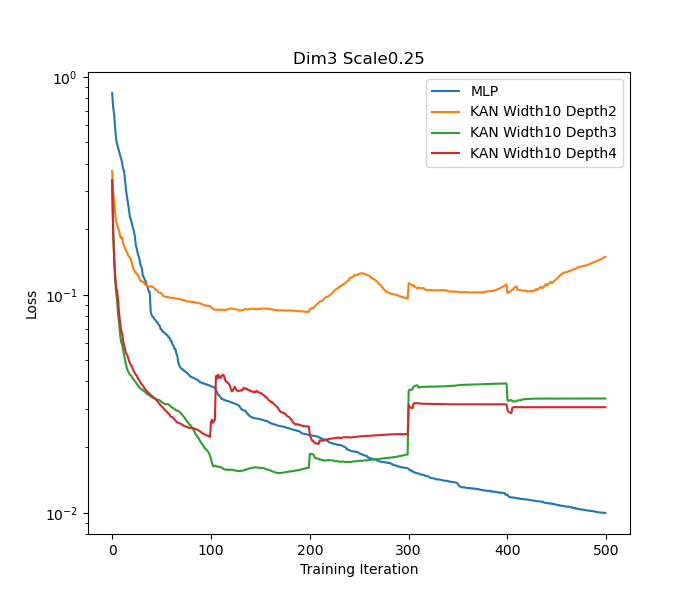}
    \end{minipage}
    \hfill
    \begin{minipage}[b]{0.24\textwidth}
        \centering
        \includegraphics[width=\textwidth]{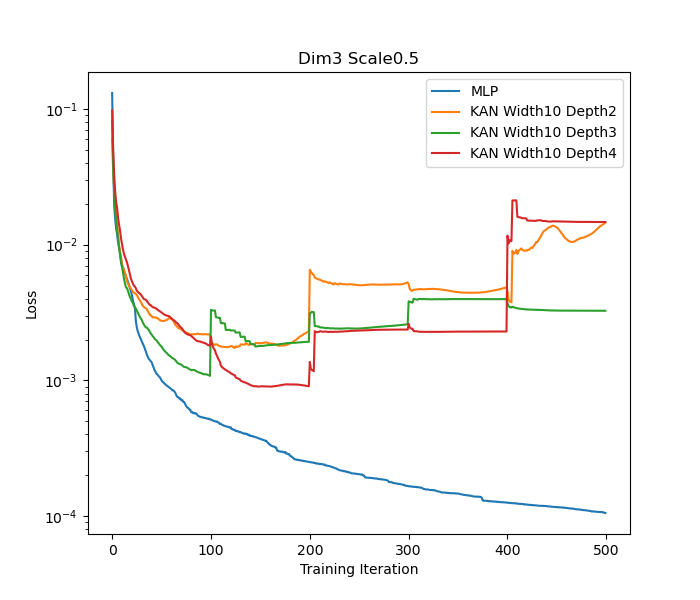}
    \end{minipage}
    \hfill
    \begin{minipage}[b]{0.24\textwidth}
        \centering
        \includegraphics[width=\textwidth]{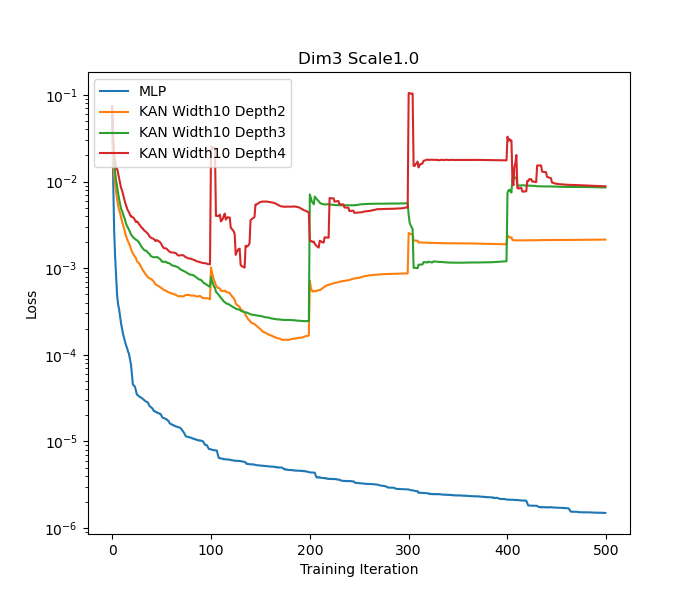}
    \end{minipage}
    \vspace{1em}\begin{minipage}[b]{0.24\textwidth}
        \centering
        \includegraphics[width=\textwidth]{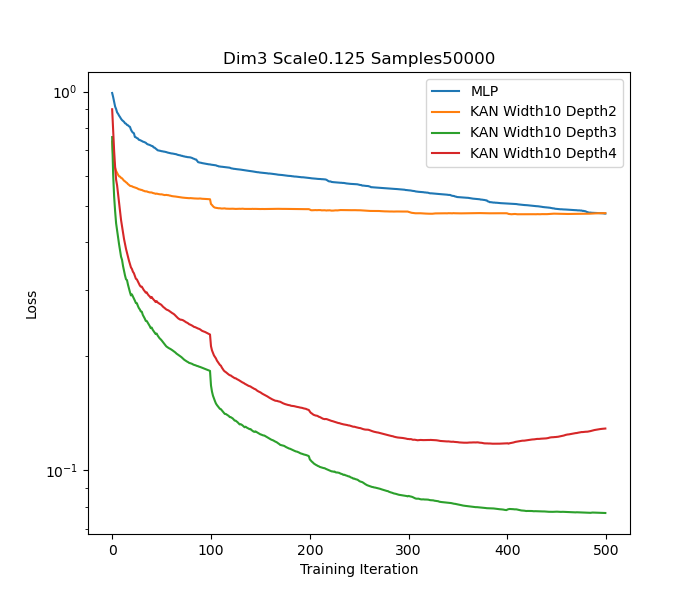}
    
    \end{minipage}
    \hfill
    \begin{minipage}[b]{0.24\textwidth}
        \centering
        \includegraphics[width=\textwidth]{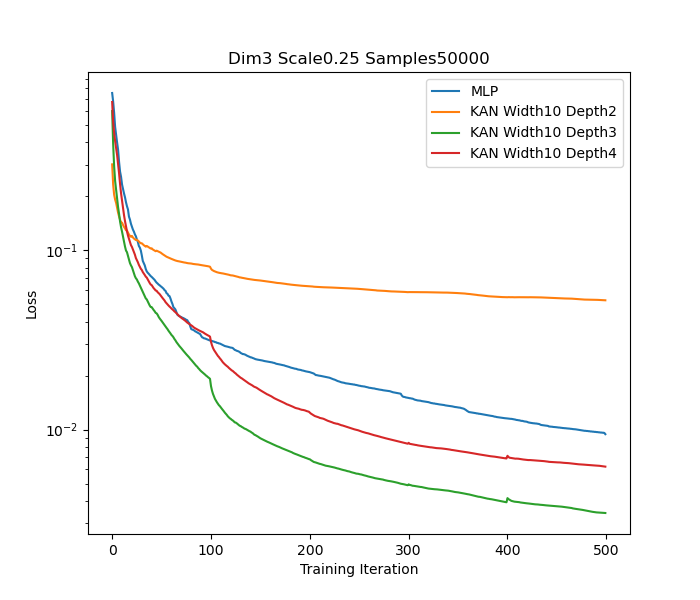}
    \end{minipage}
    \hfill
    \begin{minipage}[b]{0.24\textwidth}
        \centering
        \includegraphics[width=\textwidth]{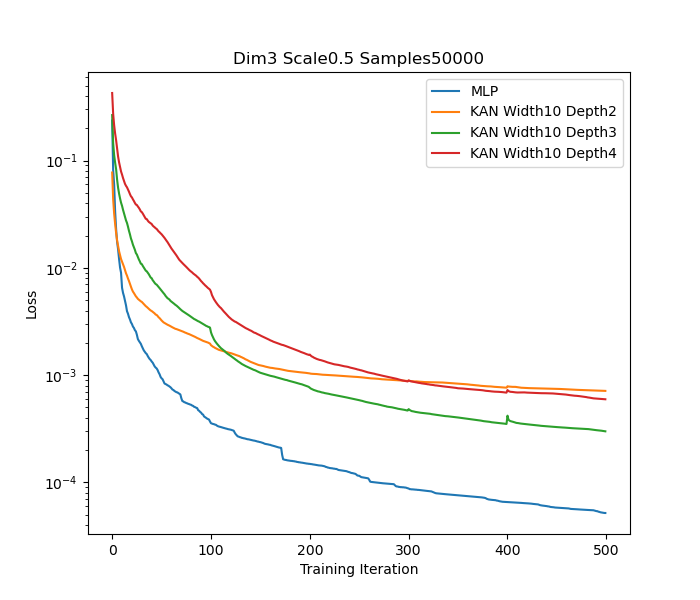}
    \end{minipage}
    \hfill
    \begin{minipage}[b]{0.24\textwidth}
        \centering
        \includegraphics[width=\textwidth]{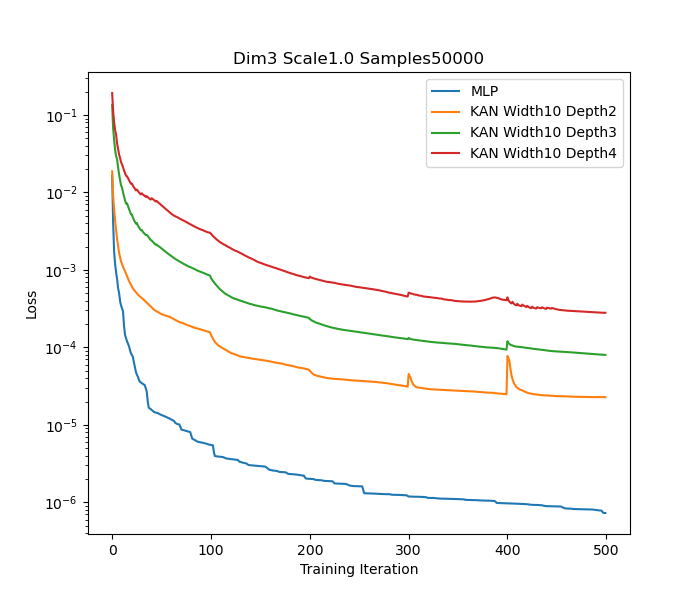}
    \end{minipage}
    \vspace{1em}

    \caption{The Gaussian random field dataset. Test losses of MLP and KANs, with different scales and dimensions. Increasing the number of samples by $10$x helps overfitting.}
    \label{grf:test}
\end{figure}
\subsection{PDE example}
In this example, we solve the 1D Poisson equation with a high-frequency solution, similar to \cite{xu2019frequency}. To be precise, consider the equation with zero Dirichlet boundary condition\begin{equation}
    -u_{xx}=f\,\,\,\, \text{in} [-1,1]\,,\quad u(-1)=u(1)=0\,.
\end{equation}
Here for a frequency $k\in \mathbb{N}$, the right-hand-side and the associated true solution are $$f=\pi^2\sin(\pi x)+\pi^2{k}\sin(k\pi x),\quad u=\sin(\pi x)+\frac{1}{k}\sin(k\pi x).$$
The different frequencies are normalized in a way that for $k> 1$, the ground truth has the same energy ($H^1$) norm. We use the variational form of the elliptic equation and the associated Deep Ritz Method \cite{yu2018deep}. Parametrizing $u$ by a neural network, we minimize the loss $$\lambda \int_{-1}^1 \left(\frac{1}{2}u_x^2-fu\right) dx+u^2(-1)+u^2(1).$$

For frequencies $k=2,4,8,16,32$, we use $2000$ uniformly spaced sample points and the neural network using an MLP of 6 layers and 256 neurons in each hidden layer and a KAN of 2 layers with 10 neurons in the hidden layer. We choose the hyperparameter $\lambda=0.01$ balancing the energy and boundary loss and perform LBFGS iterations. For MLPs, we use 200 iterations, and for KANs, we use grid sizes $(20,40)$, each trained with 100 iterations. We plot the relative $L^2$ and $H^1$ losses compared to the ground truth in Figure \ref{fig:drm}. We can see that KANs perform consistently better, and the residue barely deteriorates when the frequency increases, whereas it becomes extremely hard for MLPs to optimize when $k=16,32.$

\begin{figure}[htbp]
    \centering
    \begin{minipage}[b]{0.32\textwidth}
        \centering
        \includegraphics[width=\textwidth]{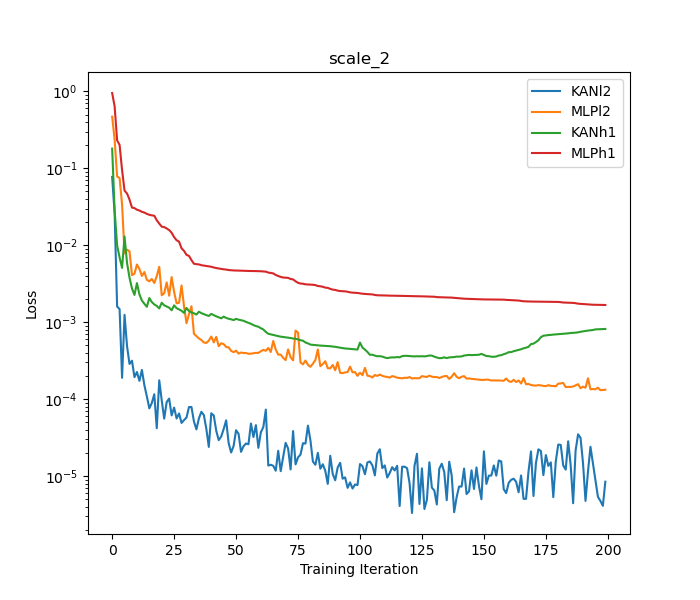}
    
    \end{minipage}
    \hfill
    \begin{minipage}[b]{0.32\textwidth}
        \centering
        \includegraphics[width=\textwidth]{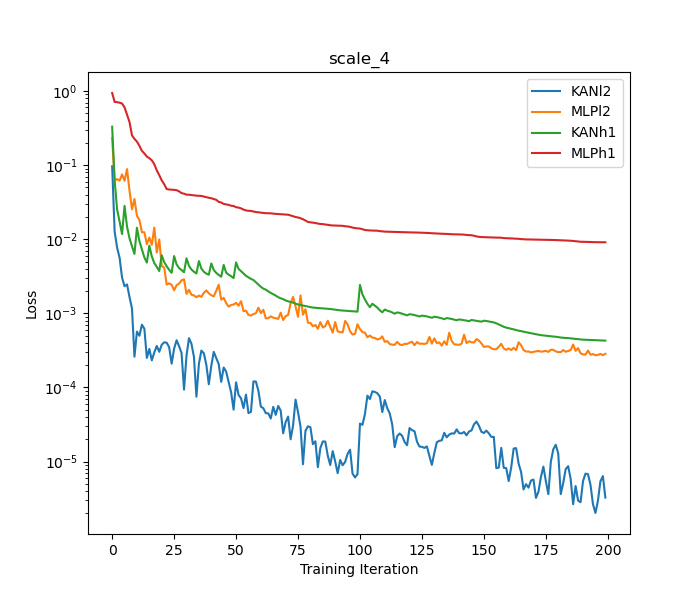}
    \end{minipage}
    \hfill
    \begin{minipage}[b]{0.32\textwidth}
        \centering
        \includegraphics[width=\textwidth]{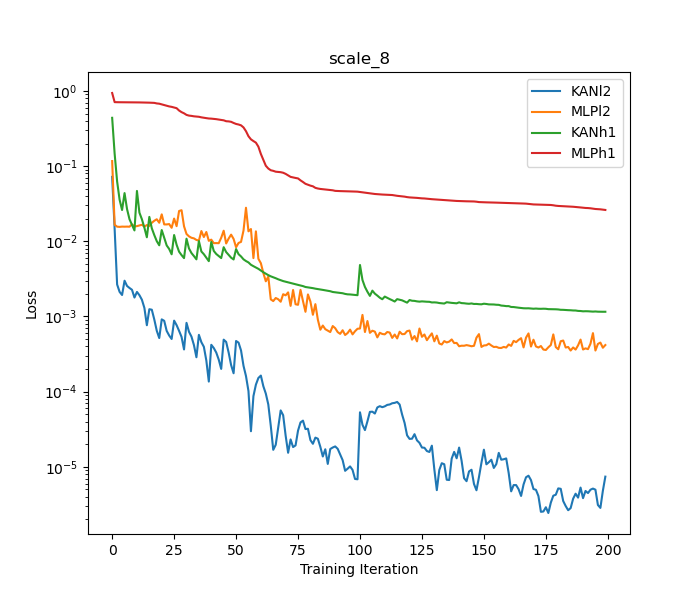}
    \end{minipage}

    \vspace{1em}

    \begin{minipage}[b]{0.48\textwidth}
        \centering
        \includegraphics[width=\textwidth]{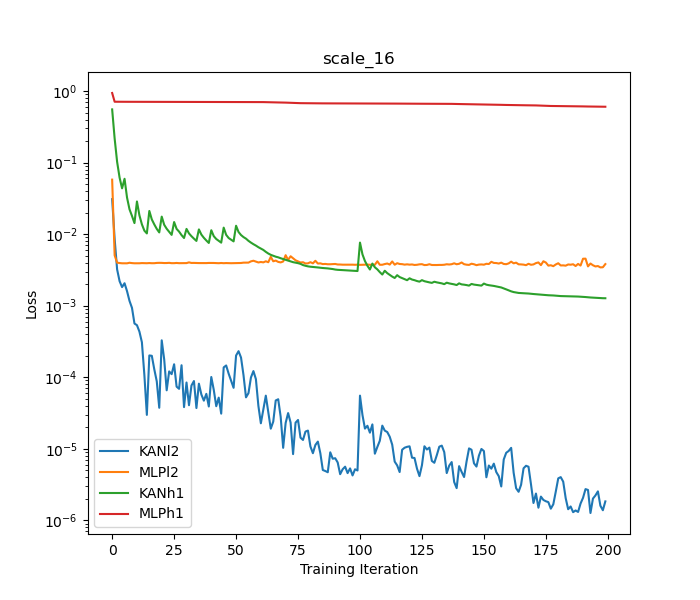}
    \end{minipage}
    \hfill
    \begin{minipage}[b]{0.48\textwidth}
        \centering
        \includegraphics[width=\textwidth]{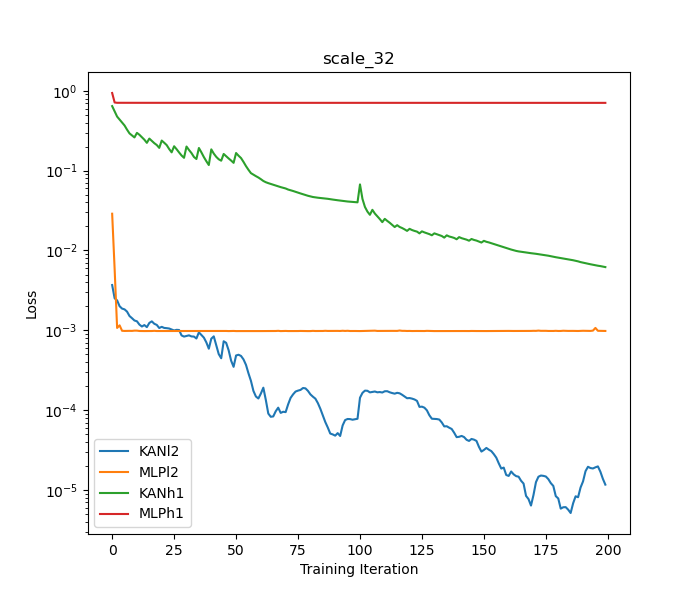}
    
    \end{minipage}
    
    \caption{Solving PDEs. $L^2$ and $H^1$ losses of MLP and KAN with different frequencies of the solution.}
    \label{fig:drm}
\end{figure}
\section{Conclusions and Discussions}\label{sec:conclusions}
Inspired by the Kolmogorov-Arnold representation theorem, we propose the Kolmogorov-Arnold Networks (KANs) as promising alternatives to MLPs. %, especially when one cares about interpretability -- an example is extracting symbolic formulas from datasets, which is a common task in science. 
Our contributions are three-fold: (1) we put the KA theorem in the perspective of modern machine learning, relating to MLPs, and generalize the representation from two-layer to multiple layers via the KAN layers introduced, greatly enhancing expressive power. (2) we show that KANs are interpretable, serving as a useful tool for scientific discoveries. (3) we show that KANs are accurate and have nice scaling laws via theory and experiments. The major limitation of this work, however, is that our numerical examples focus on various aspects of science and are relatively small-scale. The scalability and extensibility of KANs for large-scale machine-learning tasks are left as future work. Especially we want to highlight the potential of applications of KANs to AI for Science tasks, since KANs can extract interpretable information from data to provide scientific insights, and reversely from scientific prior to build better models \cite{liu2024kan}.

\subsection{Related works}\label{app:related_works}

\textbf{Kolmogorov-Arnold theorem and neural networks.} The connection between the Kolmogorov-Arnold theorem (KAT) and neural networks is not new in the literature ~\cite{poggio2022deep,schmidt2021kolmogorov,sprecher2002space,koppen2002training,lin1993realization,lai2021kolmogorov,leni2013kolmogorov,fakhoury2022exsplinet,ismayilova2024kolmogorov,poluektov2023new}, but the pathological behavior of inner functions makes KAT appear unpromising in practice~\cite{poggio2022deep}. Most of these prior works stick to the original 2-layer width-($2n+1$) networks, which were limited in expressive power and many of them are even predating back-propagation. Therefore, most studies were built on theories with rather limited or artificial toy experiments. More broadly speaking, KANs are also somewhat related to generalized additive models (GAMs)~\cite{agarwal2021neural}, graph neural networks~\cite{zaheer2017deep} and kernel machines~\cite{song2018optimizing}. The connections are intriguing and fundamental but might be out of the scope of the current work.

Our contribution lies in generalizing the Kolmogorov network to arbitrary widths and depths, revitalizing and contexualizing them in today's deep learning stream, as well as highlighting its potential role as a foundation model for AI + Science.  

There are also subsequent works exploring other parametrizations of activation functions in the KAN formulation, including special polynomials \cite{aghaei2024fkan, seydi2024exploring, ss2024chebyshev}, rational functions \cite{aghaei2024rkan}, radial basis function \cite{li2024kolmogorov,ta2024bsrbf}, Fourier series \cite{xu2024fourierkan}, and wavelets \cite{bozorgasl2024wav,seydi2024unveiling}. Active follow-up research focuses on applying KANs to various domains, such as partial differential equations \cite{wang2024kolmogorov,shukla2024comprehensive, rigas2024adaptive} and operator learning \cite{abueidda2024deepokan,shukla2024comprehensive,nehma2024leveraging}, graphs \cite{bresson2024kagnns,de2024kolmogorov,kiamari2024gkan,zhang2024graphkan}, time series \cite{vaca2024kolmogorov,genet2024tkan,xu2024kolmogorov,genet2024temporal}, computer vision \cite{cheon2024kolmogorov,azam2024suitability,li2024u,cheon2024demonstrating,seydi2024unveiling,bodner2024convolutional}, and various scientific problems \cite{liu2024ikan,liu2024initial,yang2024endowing,herbozo2024kan,kundu2024kanqas, li2024coeff,ahmed24graphkan,liu2024complexityclaritykolmogorovarnoldnetworks,peng2024predictive,pratyush2024calmphoskan}. 

\textbf{Neural Scaling Laws (NSLs).} NSLs are the phenomena where test losses behave as power laws against model size, data, compute etc~\cite{kaplan2020scaling,henighan2020scaling,gordon2021data,hestness2017deep,sharma2020neural,bahri2021explaining,michaud2023the,song2024resource}. The origin of NSLs still remains mysterious, but competitive theories include intrinsic dimensionality~\cite{kaplan2020scaling}, quantization of tasks~\cite{michaud2023the}, resource theory~\cite{song2024resource}, random features~\cite{bahri2021explaining}, compositional sparsity~\cite{poggio2022deep}, and maximum  arity~\cite{michaud2023precision}. This work contributes to this space by showing that a high-dimensional function can surprisingly scale as a 1D function (which is the best possible bound one can hope for) if it has a smooth Kolmogorov-Arnold representation. Our work brings fresh optimism to neural scaling laws. We have shown in our experiments that this fast neural scaling law can be achieved on synthetic datasets, but future research is required to address the question  whether this fast scaling is achievable for more complicated tasks (e.g., language modeling): Do KA representations exist for general tasks? If so, does our training find these representations in practice?

\textbf{Mechanistic Interpretability (MI).} MI is an emerging field that aims to mechanistically understand the inner workings of neural networks~\cite{olsson2022context,meng2022locating,wang2023interpretability,elhage2022toy,nanda2023progress,zhong2023the,liu2023seeing,elhage2022solu,cunningham2023sparse}. MI research can be roughly divided into passive and active MI research. Most MI research is passive in focusing on understanding existing neural networks trained with standard methods. Active MI research attempts to achieve interpretability by designing intrinsically interpretable architectures or developing training methods to explicitly encourage interpretability~\cite{liu2023seeing,elhage2022solu}. Our work lies in the second category, where the model and training method are by design interpretable.   

\textbf{Learnable activations.} The idea of learnable activations in neural networks is not new in machine learning. Trainable activations functions are learned in a differentiable way~\cite{goyal2019learning, fakhoury2022exsplinet, ramachandran2017searching, zhang2022neural} or searched in a discrete way~\cite{bingham2022discovering}. Activation function are parametrized as polynomials~\cite{goyal2019learning}, splines~\cite{fakhoury2022exsplinet,bohra2020learning,aziznejad2019deep}, sigmoid linear unit~\cite{ramachandran2017searching}, or neural networks~\cite{zhang2022neural}. KANs use B-splines to parametrize their activation functions.

\textbf{Symbolic Regression.} There are many off-the-shelf symbolic regression methods based on genetic algorithms (Eureka~\cite{Dubckov2011EureqaSR}, GPLearn~\cite{gplearn}, PySR~\cite{cranmer2023interpretable}), neural-network based methods (EQL~\cite{martius2016extrapolation}, OccamNet~\cite{dugan2020occamnet}), physics-inspired method (AI Feynman~\cite{udrescu2020ai,udrescu2020ai2}), and reinforcement learning-based methods~\cite{mundhenk2021symbolic}. KANs are most similar to neural network-based methods, but differ from previous works in that our activation functions are continuously learned before symbolic snapping rather than manually fixed~\cite{Dubckov2011EureqaSR,dugan2020occamnet}.

\textbf{Physics-Informed Neural Networks (PINNs) and Physics-Informed Neural Operators (PINOs).}
In Section~\ref{sec:kan_accuracy_experiment} PDE, we demonstrate that KANs can replace the paradigm of using MLPs for imposing PDE loss when solving PDEs. We refer to Deep Ritz Method \cite{yu2018deep}, PINNs \cite{raissi2019physics, karniadakis2021physics} for PDE solving, and Fourier Neural operator \cite{li2020fourier}, PINOs \cite{li2021physics, kovachki2023neural, maust2022fourier}, DeepONet \cite{lu2021learning} for operator learning methods learning the solution map. There is potential to replace MLPs with KANs in all the aforementioned networks. 

\textbf{AI for Mathematics.} AI has recently been applied to several problems in Knot theory, including detecting whether a knot is the unknot~\cite{Gukov:2020qaj,kauffman2020rectangular} or a ribbon knot~\cite{gukov2023searching}, and predicting knot invariants and uncovering relations among them~\cite{hughes2020neural,Craven:2020bdz,Craven:2022cxe,davies2021advancing}. For a summary of data science applications to datasets in mathematics and theoretical physics see e.g.~\cite{Ruehle:2020jrk,he2023machine}, and for ideas how to obtain rigorous results from ML techniques in these fields,  see~\cite{Gukov:2024aaa}.

\section{Appendix}
\subsection{Implementation details of KAN}\label{app:kan-implmentation-detail}

\begin{figure}[ht]
    \centering
    \includegraphics[width=0.9\linewidth]{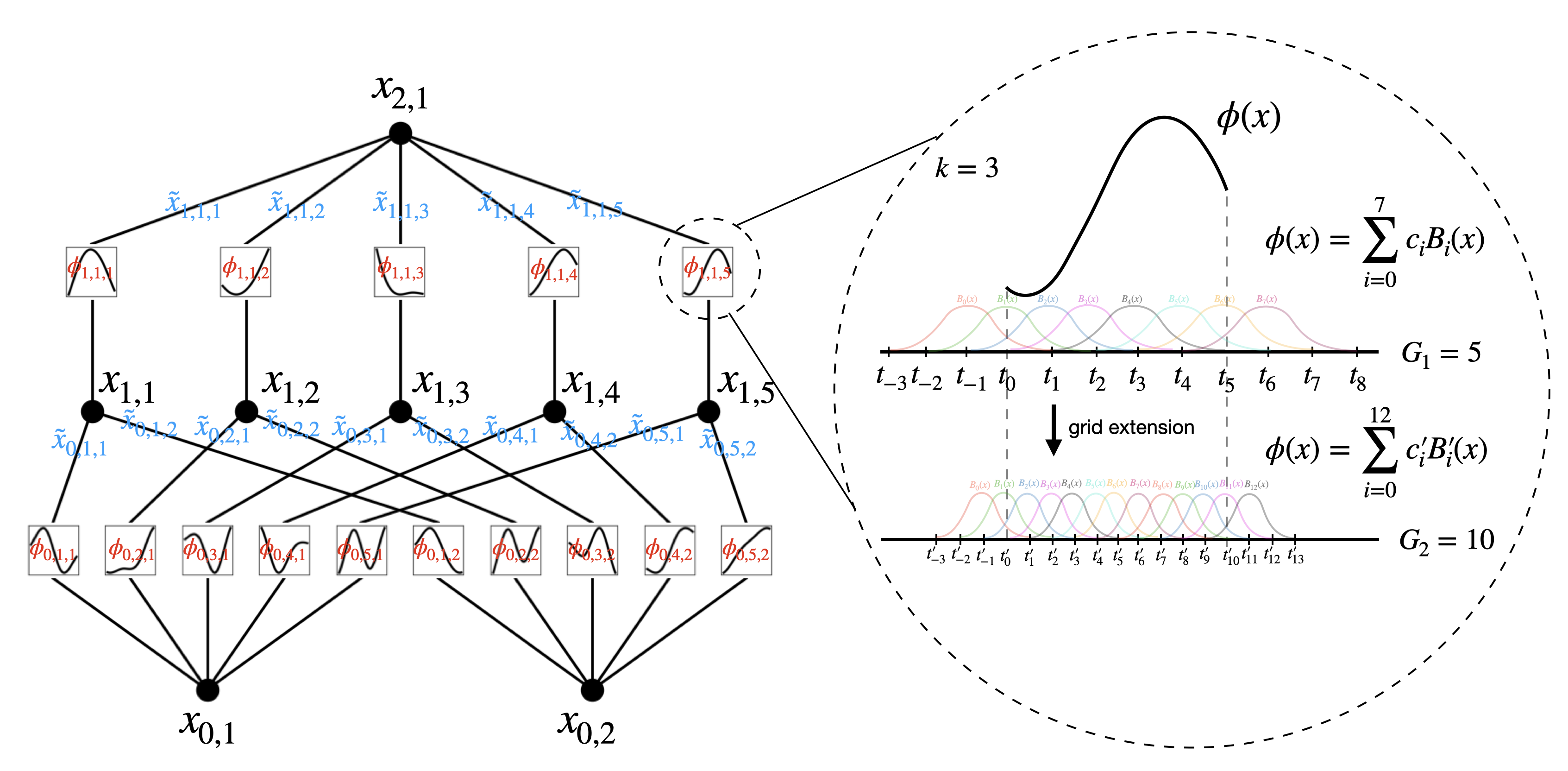}
    \caption{Left: Notations of activations that flow through the network. Right: an activation function is parameterized as a B-spline, which allows switching between coarse-grained and fine-grained grids.}
    \label{fig:spline-notation}
\end{figure}

\textbf{Implementation details.}
Although a KAN layer Eq.~(\ref{eq:kanforward}) looks extremely simple, it is non-trivial to make it well optimizable. The key tricks are: %(1) Parametrization. We use B-splines to parametrize the activation functions; the coefficients of B-spline basis are the trainable parameters. 
%(1) Parallelism. Activation functions are the most computationally expensive part of KAN (compared to efficient vector-matrix multiplication for MLPs); we attempt to parallelize activation computations. 
\begin{enumerate}[(1)]
    \item  Residual activation functions. We include a basis function $b(x)$ (similar to residual connections) such that the activation function $\phi(x)$ is the sum of the basis function $b(x)$ and the spline function:
    \begin{align}\label{KAN-activation-function}
        \phi(x)=w_{b} b(x)+w_{s}{\text{spline}}(x).
    \end{align}
    We set
    \begin{align}
        b(x)={\text{silu}}(x)=x/(1+e^{-x})
    \end{align}
    in most cases. ${\text{spline}}(x)$ is parametrized as a linear combination of B-splines such that
    \begin{align}
        {\text{spline}}(x) = \sum_i c_iB_i(x)
    \end{align}
    where $c_i$s are trainable (see Figure~\ref{fig:spline-notation} for an illustration). In principle $w_b$ and $w_s$ are redundant since it can be absorbed into $b(x)$ and ${\text{spline}}(x)$. However, we still include these factors (which are by default trainable) to better control the overall magnitude of the activation function.
    \item Initialization scales. Each activation function is initialized to have $w_s=1$ and $\text{spline}(x)\approx 0$~\footnote{This is done by drawing B-spline coefficients $c_i\sim\mathcal{N}(0,\sigma^2)$ with a small $\sigma$, typically we set $\sigma=0.1$.}. $w_b$ is initialized according to the Xavier initialization, which has been used to initialize linear layers in MLPs.
    \item Update of spline grids. We update each grid on the fly according to its input activations, to address the issue that splines are defined on bounded regions but activation values can evolve out of the fixed region during training~\footnote{Other possibilities are: (a) the grid is learnable with gradient descent, e.g., \cite{xu2015nonlinear}; (b) use normalization such that the input range is fixed. We tried (b) at first but its performance is inferior to our current approach.} Grid updates (grid size $G_1\to G_1$) use the same least square method as grid extensions (grid size $G_1\to G_2>G_1$). 
\end{enumerate}

\textbf{Parameter count.} For simplicity, let us assume a network 
\begin{enumerate}[(1)]
    \item of depth $L$,
    \item with layers of equal width $n_0=n_1=\cdots=n_{L}=N$,
    \item with each spline of order $k$ (usually $k=3$) on $G$ intervals (for $G+1$ grid points).
\end{enumerate}
Then there are in total $O(N^2L(G+k))\sim O(N^2LG)$ parameters. In contrast, an MLP with depth $L$ and width $N$ only needs $O(N^2L)$ parameters, which appears to be more efficient than KAN. Fortunately, KANs usually require much smaller $N$ than MLPs, which not only saves parameters, but also achieves better generalization (see e.g., Figure~\ref{fig:model_scaling} and~\ref{fig:PDE}) and facilitates interpretability. 
We remark that for 1D problems, we can take $N=L=1$ and the KAN network in our implementation is nothing but a spline approximation. For higher dimensions, we characterize the generalization behavior of KANs with a theorem below.

\subsection{Proofs}
\label{app:proof}

\begin{proof}[Proof of Theorem~\ref{approx thm}]
    By the classical 1D B-spline theory \cite{de1978practical} and the fact that $\Phi_{l,i,j}$ as continuous functions can be uniformly bounded on a bounded domain, we know that there exist finite-grid B-spline functions $\Phi_{l,i,j}^G$ such that for any $0\leq m\leq k$, $$\|(\Phi_{l,i,j}\circ\mathbf{\Phi}_{l-1}\circ\mathbf{\Phi}_{l-2}\circ\cdots\circ\mathbf{\Phi}_{1}\circ\mathbf{\Phi}_{0})\mathbf{x}-(\Phi_{l,i,j}^G\circ\mathbf{\Phi}_{l-1}\circ\mathbf{\Phi}_{l-2}\circ\cdots\circ\mathbf{\Phi}_{1}\circ\mathbf{\Phi}_{0})\mathbf{x}\|_{C^m}\leq C_0G^{-k-1+m}\,,$$
     with a  constant $C_0$  independent of $G$. We fix those B-spline approximations. Therefore we have  that the residue $R_l$ defined via $$R_l\coloneqq (\mathbf{\Phi}^G_{L-1}\circ\cdots\circ\mathbf{\Phi}^G_{l+1}\circ\mathbf{\Phi}_{l}\circ\mathbf{\Phi}_{l-1}\circ\cdots\circ\mathbf{\Phi}_{0})\mathbf{x}-(\mathbf{\Phi}_{L-1}^G\circ\cdots\circ\mathbf{\Phi}_{l+1}^G\circ\mathbf{\Phi}_{l}^G\circ\mathbf{\Phi}_{l-1}\circ\cdots\circ\mathbf{\Phi}_{0})\mathbf{x}$$
satisfies $$\|R_l\|_{C^m}\leq C_1G^{-k-1+m}\,,$$
with another constant independent of $G$. Finally notice that $$f-(\mathbf{\Phi}^G_{L-1}\circ\mathbf{\Phi}^G_{L-2}\circ\cdots\circ\mathbf{\Phi}^G_{1}\circ\mathbf{\Phi}^G_{0})\mathbf{x}=R_{L-1}+R_{L-2}+\cdots+R_1+R_0\,,$$
we know that \eqref{appro bound} holds for another constant $C$ independent of $G$.
\end{proof}

    \textbf{Remark}: We can be more precise about the dependence of the constant $C$ in the theorem. Define the compositionally smooth function class $\mathcal{C}^{n,W,L,k}$as the class of functions in the form of \eqref{fkan} such that the input dimension equals $n$, the width or $\max_{0\leq i\leq L}n_i$ in the definition \eqref{arraykan} equals $W\geq n$, depth equals $L$, smoothness equals $k$. Then $C$ only depends on $W,L,k$ and $\max \|{\phi_{l,i,j}}\|_{C^m}$.

\begin{proof}[Proof of Theorem \ref{mlp-kan-representation-thm}]
We will show that each layer of an MLP with the activation function $\sigma_k$ can be represented by a KAN with two hidden layers, width $W$ and grid size $G = 2$ with degree $k$ B-splines. By composing such layers, we obtain the desired result. 

For a single layer of MLP, we consider the linear part and the non-linear activation separately. We first observe that on any compact subset of $\mathbb{R}^W$ the linear function
\begin{equation}
    x_i = \sum_{j=0}^W a_{ij}x^{in}_j + b_i
\end{equation}
can be represented with a single KAN layer of width $W$ by setting $\phi_{ij}$ to the linear function
\begin{equation}
    \phi_{i,j}(x) = a_{ij}x + \frac{b_i}{n}.
\end{equation}
We claim that this linear function can be exactly represented on any interval $[-R,R]$ in the form \eqref{KAN-activation-function}. To do this, we first set $w_b = 0$ and choose the grid points for the B-splines to be $$\{-(2k-1)R, -(2k-3)R,...,-R,R,...,(2k-3)R, (2k-1)R\}.$$ 
Note that based upon the KAN architecture, this corresponds to the extension of the uniform grid $t_0 = -R, t_1 = R$ which has grid size $G = 1$. It is also easy to verify that there are $(k+1)$ B-splines supported on this grid, whose restriction to $[-R,R]$ span the space of polynomials of degree $k$. Thus, in particular, any linear function on $[-R,R]$ can be represented as a linear combination of these B-splines.

Next, we consider the non-linear activation, which is given by the coordinatewise application of $\sigma_k$, i.e.
\begin{equation}
    x_i^{out} = \sigma(x_i).
\end{equation}
This can be represented by a single hidden layer KAN by setting
\begin{equation}
    \phi_{i,j}(x) = \begin{cases}
        \sigma_k(x) & i=j\\
        0 & i \neq j.
    \end{cases}
\end{equation}
We claim that the functions $\sigma_k$ can be represented in the form \eqref{KAN-activation-function} on any finite interval $[-R,R]$. To to this, we again set $w_b = 0$ and choose the grid points for the B-splines to be
$$\{-kR, -(k-1)R,...,-R,0,R,...,(k-2)R, (k-1)R\}.$$
This grid is the grid extension of the uniform grid $t_0 = -R, t_1 = 0, t_2 = R$ which has grid size $G = 2$. It is easy also to verify that there are $(k+2)$ B-splines supported on this grid and that any piecewise polynomial on $[-R,R]$ with a single breakpoint at $0$ which is $C^{k-1}$ is a linear combination of these B-splines. Hence the function $\sigma_k$ can be represented on $[-R,R]$ in the form \eqref{KAN-activation-function} using this grid.

The proof is now completed by composing these layers and choosing $R$ sufficiently large so that for any input $x\in \Omega$ (which is bounded) the inputs and outputs of every neuron in the original MLP lie in the interval $[-R,R]$.
\end{proof}

\begin{proof}[Proof of Theorem \ref{hessian-eigenvalue-bound-theorem}]
    We first observe from \eqref{hessian-matrix-equation} that the matrix $M$ is block diagonal with $d'$ identical blocks. Denoting these $(G+k-1)d\times (G+k-1)d$ blocks by $B$, it thus suffices to prove that
    \begin{equation}
        \frac{\lambda_{(G+k-1)d}(B)}{\lambda_{d}(B)} \leq C.
    \end{equation}
    To do this, we analyze the blocks $B$ and note that they take the form
    \begin{equation}
        B = \begin{pmatrix}
            C & D & \cdots & D\\
            D & C & \cdots & D\\
            \vdots & \vdots & \ddots & \vdots\\
            D & D & \cdots & C
        \end{pmatrix}.
    \end{equation}
    Here the diagonal sub-blocks $C\in \mathbb{R}^{(G+k-1)\times (G+k-1)}$ are the Gram matrix of the one-dimensional B-spline basis, i.e.
    \begin{equation}
        C_{ij} = \int_0^1 B_i(x)B_j(x)dx,
    \end{equation}
    and the off-diagonal sub-blocks $D\in \mathbb{R}^{(G+k-1)\times (G+k-1)}$ are rank one matrices
    \begin{equation}
        D = vv^T,
    \end{equation}
    where the vector $v\in \mathbb{R}^{G+k-1}$ is given by
    \begin{equation}
        v_{i} = \int_0^1 B_i(x)dx.
    \end{equation}
    It is well-known that the Gram matrix $C$ is well-conditioned uniformly in $G$ for a fixed $k$, i.e. $\lambda_{G+k-1}(C)/\lambda_1(C) \leq K$ for a fixed constant $K$ depending only upon $k$. See for instance \cite{devore1993constructive}, Theorem 4.2 in Chapter 5, where it is shown that the $L_2$-norm of a spline and the properly scaled $\ell_2$-norm of its B-spline coefficients are equivalent up to a constant depending only on $k$. This is equivalent to the well-conditioning of the Gram matrix $C$.
    
    In addition, we can easily verify using Jensen's inequality (or Cauchy-Schwartz) that $D \preceq C$. Indeed, letting $w\in \mathbb{R}^{G+k-1}$ we see that
    \begin{equation}
        w^TDw = \left(\int_0^1 f(x)dx\right)^2 \leq \int_0^1 f(x)^2dx = w^TCw,
    \end{equation}
    where the function $f(x) = \sum_{i=1}^{G+k-1} w_iB_i(x)$.
    
    Let $\mathbf{1}\in \mathbb{R}^d$ be the vector of ones and note that
    \begin{equation}
        (v\otimes \mathbf{1})(v\otimes \mathbf{1})^T = \begin{pmatrix}
            D & D & \cdots & D\\
            D & D & \cdots & D\\
            \vdots & \vdots & \ddots & \vdots\\
            D & D & \cdots & D
        \end{pmatrix}
    \end{equation}
    so that $B-(v\otimes \mathbf{1})(v\otimes \mathbf{1})^T$ is a block diagonal matrix with diagonal blocks $C - D$. We proceed to upper bound the largest eigenvalue of $B$ by
    \begin{equation}\begin{split}
        \lambda_{(G+k-1)d}(B) &= \max_{\|w\| = 1} w^TBw \\
        &= \max_{\|w\| = 1} w^T\begin{pmatrix}
            C-D & 0 & \cdots & 0\\
            0 & C-D & \cdots & 0\\
            \vdots & \vdots & \ddots & \vdots\\
            0 & 0 & \cdots & C-D
        \end{pmatrix}w + (w^T(v\otimes \mathbf{1}))^2.
    \end{split}
    \end{equation}
    Writing $w = (w_1,...,w_d)$ with $w_i\in \mathbb{R}^{G+k-1}$ and $\sum_{i=1}^{G+k-1} \|w_i\|^2 = 1$ and using that $D = vv^T$, we get the bound
    \begin{equation}
    \begin{split}
        \lambda_{(G+k-1)d}(B) &\leq \max_{\|w_1\|^2 + \cdots + \|w_d\|^2 = 1} \sum_{i=1}^d w_i^TCw_i + \left(\sum_{i=1}^d v^Tw_i\right)^2 - \sum_{i=1}^d (v^Tw_i)^2\\
        &\leq \max_{\|w_1\|^2 + \cdots + \|w_d\|^2 = 1} \sum_{i=1}^d w_i^TCw_i + (d-1)\sum_{i=1}^d (v^Tw_i)^2
    \end{split}
    \end{equation}
    Since $D \preceq C$ we have $(v^Tw_i)^2 \leq w_i^TCw_i$ which gives the bound
    \begin{equation}
        \lambda_{(G+k-1)d}(B) \leq d\max_{\|w_1\|^2 + \cdots + \|w_d\|^2 = 1} \sum_{i=1}^d w_i^TCw_i = d\lambda_{G+k-1}(C).
    \end{equation}
    Next, we lower bound the $d$-th eigenvalue of $B$. For this, we use the Courant-Fisher minimax theorem to see that
    \begin{equation}
        \lambda_{d}(B) = \max_{W_{d}}\min_{w\in W_{d},~\|w\| = 1} w^TBw,
    \end{equation}
    where the maximum is taken over all subspaces of $W_d$ of codimension $<d$. We consider the specific subspace
    \begin{equation}
        W_d = \{(w_1,...,w_d);~v^Tw_i = 0~\text{for all $i=1,...,d$}\} \oplus \text{span}(v\otimes \mathbf{1})
    \end{equation}
    and observe that for any $(w_1,...,w_d)$ with $v^Tw_i = 0$ we have
    \begin{equation}
        w^TBw = \sum_{i=1}^d w_i^TCw_i \geq \lambda_1(C)\sum_{i=1}^d \|w_i\|^2 = \lambda_1(C)\|w\|^2,
    \end{equation}
    while for the $w = v\otimes \mathbf{1}$ (which is orthogonal) we have
    \begin{equation}
        w^TBw = \sum_{i=1}^d v^TCv + (d-1)\sum_{i=1}^d \|v\|^2 \geq \lambda_1(C)d\|v\|^2 = \lambda_1(C)\|w\|^2.
    \end{equation}
    Thus, $\lambda_d(B) \geq \lambda_1(W)$. Combining these bounds and using the well-conditioning of the Gram matrix $C$, we get
    \begin{equation}
    \frac{\lambda_{(G+k-1)d}(B)}{\lambda_{d}(B)} \leq d\frac{\lambda_{G+k-1}(C)}{\lambda_1(C)} \leq K
    \end{equation}
    for a constant $K$ which only depends upon $k$. This completes the proof.
\end{proof}

\chapter{Exponentially Convergent Multiscale Finite Element Method}
\label{chap:4}
In this chapter, we present the exponentially convergent multiscale finite element method (ExpMsFEM), developed for efficient model reduction of PDEs in heterogeneous media without scale separation and for high-frequency wave propagation problems. This method, based on a series of articles \cite{chen2021exponential, chen2023exponentially, chen2024exponentially}, systematically enriches the approximation space within a non-overlapping domain decomposition framework to achieve nearly exponential convergence with respect to the number of basis functions. We provide a concise overview based on our review paper \cite{chen2024exponentially}, with technical details and full theoretical analysis found in the cited works.

A central challenge in achieving exponential convergence is overcoming the algebraic Kolmogorov 
n-width barrier, which requires basis functions that depend on the right-hand side of the equation. ExpMsFEM addresses this by introducing a two-part function representation: an offline stage, where basis functions independent of the right-hand side are constructed and used to assemble the Galerkin system; and an online stage, where problem-dependent correctors are computed efficiently and in parallel. This decomposition enables high-accuracy solutions and computational reuse in multi-query scenarios.

\section{Introduction}
Multiscale methods provide an efficient way to solve challenging PDEs. A few local basis functions adapted to the problem are constructed offline to provide an effective model reduction of the equation. One can then use the reduced model to compute the solution online, possibly with different right-hand sides and in a way much faster than solving the original equation. This property is beneficial in multi-query scenarios such as optimal design and inverse problems. Moreover, multiscale methods are inevitable for challenging problems in rough media and high-frequency wave propagation since standard numerical methods suffer from a vast number of degrees of freedom. See examples of the failure of finite element methods (FEMs) in elliptic equations with rough coefficients \cite{babuvska2000can} and the pollution effect in the Helmholtz equation \cite{babuska1997pollution}.

In this chapter, we present the framework of ExpMsFEM, the exponentially convergent multiscale finite element method. It is a generalization of the classical MsFEM \cite{hou_multiscale_1997}. The main contribution of ExpMsFEM is the systematic improvement over MsFEM to achieve exponentially convergent accuracy regarding the number of basis functions. Also, unlike most generalizations of MsFEM in the literature, ExpMsFEM does not rely on the partition of unity functions to connect local and global approximation spaces. Instead, ExpMsFEM uses edge localization and coupling intrinsic to the non-overlapped domain decomposition to communicate the local and global approximations.

In the literature, exponentially convergent multiscale methods have been pioneered in the work of optimal basis \cite{babuska2011optimal} based on the partition of unity functions; see also the developments in \cite{smetana2016optimal,buhr2018randomized,chen2020randomized,babuvska2020multiscale, schleuss2020optimal, ma2021error, ma2021novel}. The work demonstrates the importance of Caccioppoli's inequality in establishing exponential convergence; more precisely, the inequality implies the \textit{low approximation complexity} of the restriction operator acting on harmonic-type functions. The theory of ExpMsFEM is also based on some arguments using Caccioppoli's inequality. Additionally, since no partition of unity functions is used, technical tools such as $C^{\alpha}$ estimates and trace theorems are needed to analyze ExpMsFEM. ExpMsFEM was the first method to achieve exponential convergence on Helmholtz equations. We will comment on the similarities and differences between the optimal basis work and ExpMsFEM at the end of the chapter. We focus on articulating the main ideas and the computational framework in the case of 2D stationary problems with homogeneous boundary data. We provide references for the detailed analysis in the corresponding papers \cite{chen2021exponential, chen2023exponentially}.

\subsection{Organization} In Section \ref{sec: model problem}, we present the model problem that is the focus of this chapter. In Section \ref{sec: The ExpMsFEM Framework}, we present the motivation and framework of the ExpMsFEM. We provide numerical experiments to demonstrate the effectiveness of the ExpMsFEM framework in Section \ref{sec: Numerical Experiments}. In Section \ref{sec: Discussions}, we discuss related literature, future possibilities, and open questions. 

\section{Model Problem}
\label{sec: model problem}
Consider the model problem in a bounded domain $\Omega \subset \mathbb{R}^d$ with a Lipschitz boundary $\Gamma$. Here, $d=2$. For generality, the boundary can contain disjoint parts $\Gamma = \Gamma_1\cup \Gamma_2$ where $\Gamma_1$ corresponds to the Dirichlet boundary conditions and $\Gamma_2$ corresponds to the Neumann and Robin type boundary conditions. 

The model equation is: 
\begin{equation}
\label{eqn: model}
\left\{
\begin{aligned}
-\nabla \cdot(A\nabla u)+Vu&=f, \ \text{in} \ \Omega\\
u&=0, \ \text{on} \ \Gamma_1\\
A\nabla u\cdot\nu&=\beta u, \ \text{on} \  \Gamma_2 \, .
\end{aligned}
\right.
\end{equation}
Here, $A, V,\beta$ are functions in $L^{\infty}(\Omega)$ and can be rough, which makes the solution oscillating and difficult to solve. The vector $\nu$ is the outer normal to the boundary. 

In particular, when $V=0$, the equation is the standard elliptic equation \cite{chen2021exponential}. If $Vu=-k^2u$ and $u$ is a complex-valued function, one obtains the Helmholtz equation \cite{chen2023exponentially} with wavenumber $k$. 

The weak formulation of \eqref{eqn: model} is given by
\begin{equation}
\label{eqn: weak form}
a(u, v) :=(A\nabla u, \nabla {v})_{\Omega}+(V u,  {v})_{\Omega}  -( \beta u,  {v})_{\Gamma_2} =( f,  {v})_{\Omega} , \quad \forall v \in \mathcal{H}(\Omega)\, ,
\end{equation}
where $(\cdot,\cdot)_X$ is the standard $L^2$ inner product on the set $X$.
The space for $v$ is $\mathcal{H}(\Omega):=\{w \in H^1(\Omega): w|_{\Gamma_1}=0 \}$ and the solution $u \in \mathcal{H}(\Omega)$. The energy norm $\|\cdot\|_{\mathcal{H}(\Omega)}$ is defined as 
\[\|w\|_{\mathcal{H}(\Omega)}^2:=(A\nabla w, \nabla {w})_{\Omega}+(|V| w,  {w})_{\Omega}\,.\]
Here, we adopt an abuse of notation that the space can be real-valued or complex-valued, depending on the context.

A generic assumption for $A$ is $0<A_{\min}\leq A(x)\leq A_{\max} < \infty$. We will present more detailed assumptions on $V,\beta$ later in specific problems that our theory in \cite{chen2021exponential,chen2023exponentially} covers. 
Indeed, the theory can encompass the case for very general $V$, provided that $|(Vu,u)|_{\Omega}\leq V_0(u,u)_{\Omega}$ for some constant $V_0$ and the PDE satisfies good stability estimates; see for example the rough Helmholtz example in \cite{chen2023exponentially}.
In this review, we mainly focus on the \textit{conceptual algorithmic framework} of solving the equation \eqref{eqn: model} via ExpMsFEM rather than a detailed analysis of the equation and the method.

\section{The ExpMsFEM Framework}
\label{sec: The ExpMsFEM Framework}
In subsection \ref{sec: Solving PDEs as function approximation}, we discuss the general recipe for solving PDEs as a function approximation problem. This motivates us to find accurate function representations to be used in the Galerkin method. We explain how ExpMsFEM manages to get exponentially convergent representations in subsections \ref{sec: Harmonic-bubble splitting}, \ref{sec: edge loc}, \ref{sec: Exponentially convergent SVD} and \ref{sec: alg ExpMsFEM}.
\subsection{Solving PDEs as function approximation}
\label{sec: Solving PDEs as function approximation}
By the standard finite element theory (e.g., \cite{brenner2008mathematical}), when using the Galerkin method to solve \eqref{eqn: weak form}, a key step is to find a function representation, or a space of basis functions that can approximate the solution accurately. More precisely, suppose the space is $S$, then, one usually wants
 \begin{equation}
    \label{apro proxy}
        {\eta(S)}:=\sup _{f \in L^{2}(\Omega) \backslash\{0\}} \inf _{v \in S} \frac{\left\|N(f)-v\right\|_{\mathcal{H}(\Omega)}}{\|f\|_{L^{2}(\Omega)}}
    \end{equation}
to be small. Here, $N: f \to u$ is the solution operator\footnote{Sometimes, $N$ is chosen to be the solution operator of the adjoint equation; for example see \cite{melenk2010convergence}.} of \eqref{eqn: model}.

For example, consider the elliptic equation with $V=0$ and $\Gamma_2=\emptyset$. In such cases, the Galerkin method provides an optimal approximation of the solution in the space of basis functions with respect to the energy norm \cite{brenner2008mathematical, chen2021exponential}, due to the Galerkin orthogonality. Therefore, a small $\eta(S)$ directly implies a small error in the solution. For the Helmholtz equation, similar arguments hold based on the G\aa rding-type inequality, which leads to the quasi-optimality of the solution; see, for example, \cite{melenk2010convergence, chen2023exponentially}. The failure of many finite element methods in elliptic equations with rough coefficients \cite{babuvska2000can} and Helmholtz's equations \cite{babuska1997pollution} is due to the poor approximation property. $\eta(S)$ is typically not small if $S$ is the standard finite element space, such as the space of tent functions.

Conceptually, ExpMsFEM finds an exponentially convergent function representation of the solution through the following three steps: (1) harmonic-bubble splitting, (2) edge localization, (3) oversampling and exponentially convergent singular value decomposition (SVD). We will detail the three steps and discuss relevant rigorous results at the end of subsections \ref{sec: Harmonic-bubble splitting}, \ref{sec: edge loc}, and \ref{sec: Exponentially convergent SVD}. Then, we summarize the algorithm in subsection \ref{sec: alg ExpMsFEM}.
\subsection{Harmonic-bubble splitting}
\label{sec: Harmonic-bubble splitting}
Consider a shape regular and uniform partition of the domain $ \Omega $ into finite elements with a mesh size $H$. The collection of elements is denoted by $\mathcal{T}_H=\{T_1, T_2,..., T_r\}$. Let $\mathcal{E}_H=\{e_1,e_2,...,e_q\}$ be the collection of edges in the interior of $\Omega$. We use $\mathcal{N}_H=\{x_1,x_2,...,x_p\}$ to denote the collection of interior nodes. We also use $E_H$ to denote the collection of interior edges as a set, i.e., $E_H=\bigcup_{e \in \mathcal{E}_H} e \subset \Omega$. A more detailed explanation of the mesh structure can be found in \cite{chen2021exponential,chen2023exponentially}.

In each element $T\in\mathcal{T}_H$,
we decompose the solution $u$ into $u=u_{T}^{\mathsf{h}}+u_T^{\mathsf{b}}$ such that
\begin{equation}
\label{eqn:decomposed}
\begin{aligned}
    &\left\{
    \begin{aligned}
    -\nabla \cdot (A \nabla u_T^\mathsf{h} )+V u^\mathsf{h}_T&=0, \ \text{in} \  T\\
    u_T^\mathsf{h}&=u, \ \text{on} \  \partial T \setminus (\Gamma_1\cup\Gamma_2)\\
    u_T^\mathsf{h}&=0, \ \text{on} \  \partial T \cap \Gamma_1\\
    A\nabla u_T^\mathsf{h}\cdot\nu&=\beta u_T^\mathsf{h},\  \text{on} \ \partial T \cap \Gamma_2\, ,
    \end{aligned}
    \right.
    \\
    &\left\{
    \begin{aligned}
    -\nabla \cdot (A \nabla u^\mathsf{b}_T )+V u^\mathsf{b}_T&=f,\  \text{in} \  T\\
    u^\mathsf{b}_T&=0, \ \text{on} \  \partial T\setminus (\Gamma_1\cup\Gamma_2)\\
    u^\mathsf{b}_T&=0, \ \text{on} \  \partial T\cap \Gamma_1\\
   A\nabla u^\mathsf{b}_T\cdot\nu&=
  \beta u^\mathsf{b}_T, \ \text{on} \ \partial T \cap \Gamma_2 \, .
    \end{aligned}
    \right.
    \end{aligned}
    \end{equation}
In short, $u^{\mathsf{h}}_T$ incorporates the interior boundary value of $u$ on the element, while $u^{\mathsf{b}}_T$ contains information of the right-hand side. All equations in \eqref{eqn:decomposed} should be understood in the standard weak sense as in \eqref{eqn: weak form}.

We can further define a global decomposition $u=u^{\mathsf{h}}+u^{\mathsf{b}}$, such that for each $T$, it holds that $u^{\mathsf{h}}(x)=u^{\mathsf{h}}_T(x)$, $u^{\mathsf{b}}(x)=u^{\mathsf{b}}_T(x)$ when $x \in T$. 
 Here, the component $u^{\mathsf{h}}_T$ (resp. $u^{\mathsf{h}}$) is called the local (resp. global) \textit{harmonic part}, $u^{\mathsf{b}}_T$ (resp. $u^{\mathsf{b}}$) is the local (resp. global) \textit{bubble part}, of the solution $u$. Here, the harmonic part $u^\mathsf{h}$ is not necessarily a harmonic function due to the existence of $A$ and $V$, but it has a similar low complexity property that a harmonic function has, due to the iterative argument of Caccioppoli's inequality first proposed in \cite{babuska2011optimal}. We will discuss this low complexity property in subsection \ref{sec: Exponentially convergent SVD}.
 
 Now, in the representation $u=u^{\mathsf{h}}+u^{\mathsf{b}}$, the part $u^{\mathsf{b}}$ can be directly computed by solving local problems in parallel since the local boundary conditions are all known. We are left to deal with the part $u^\mathsf{h}$.
\begin{remark}
\label{rmk: harmonic bubble splitting}
    We discuss several theoretical concerns and possible generalizations below:
\begin{itemize}
    \item A sufficient condition for the local components in \eqref{eqn:decomposed} to be well-defined is that the operator $u \to -\nabla \cdot (A \nabla u) + Vu$ (as well as the corresponding boundary conditions) is elliptic in each local element, implied by the Poincar\'e inequality. In \cite{chen2021exponential}, we consider elliptic equations with $V = 0$ and $\Gamma_2 = \emptyset$, so this condition is satisfied. In \cite{chen2023exponentially}, we consider the Helmholtz equation where $V <0, |V|=O(k^2)$ and $\operatorname{Re}\beta =0, \operatorname{Im}\beta = O(k)$. For such a case, the elliptic property is guaranteed when $H = O(1/k)$.
    \item For the global components $u^\mathsf{h}$, $u^{\mathsf{b}}$ to be well-defined, we need the condition that the solution $u$ is continuous. This can be guaranteed by the $C^{\alpha}$ estimates of the equation \eqref{eqn: model} under the assumptions mentioned earlier; see discussions in \cite{chen2021exponential,chen2023exponentially}.
    \item We can generalize the above decomposition to PDEs with inhomogeneous boundary conditions. To achieve so, we incorporate these boundary data into the equation for $u^\mathsf{b}$; see also Section 5.3 in \cite{chen2023exponentially} for a concrete example of problems with inhomogeneous boundary data.
\end{itemize}
\end{remark}

\subsection{Edge localization} 
\label{sec: edge loc}
The next step is to find some local basis functions that accurately approximate $u^\mathsf{h}$. ExpMsFEM uses the idea of \textit{edge localization} to localize this approximation task.

First, we define the ``harmonic extension'' operator $Q_{E_H}$ that maps the edge values $\tilde{u}^\mathsf{h} = u^\mathsf{h}|_{E_H} \in H^{1/2}(E_H)$ to $ u^\mathsf{h} \in H^1(\Omega)$, through the relation in the first set of equation in \eqref{eqn:decomposed}. Here, we adopt the convention that if we write a tilde on the top of a function, it is the restriction of this function on the edge set. We have that $u^\mathsf{h} = Q_{E_H}\tilde{u}^\mathsf{h} = Q_{E_H}\tilde{u}$, since $u^\mathsf{h}$ and $u$ have the same edge values.

Then, let $C(E_H)$ be the space of continuous functions on $E_H$. We consider the edge interpolation operator $I_H: H^{1/2}(E_H) \cap C(E_H) \to H^{1/2}(E_H) \cap C(E_H)$ such that \[I_H \tilde{u} = \sum_{x_i \in \mathcal{N}_H} \tilde{u}(x_i)\tilde{\psi}_i\] where the edge function $\tilde{\psi}_i$ is linear on $E_H$ and satisfies $\tilde{\psi}_i(x_j) = \delta_{ij}$. Note that by the convention of our notation we have $\psi_i = Q_{E_H}\tilde{\psi}_i \in H^1(\Omega)$. It is worth noting that $\psi_i's$ are the basis functions used in the vanilla MsFEM.

With the interpolation operator, we can write 
\[Q_{E_H}\tilde{u} = Q_{E_H}(\tilde{u}-I_H\tilde{u}) + \sum_{x_i \in \mathcal{N}_H} u(x_i) \psi_i\, . \]

Now, the residue $\tilde{u}-I_H\tilde{u}$ is zero at each interior node. This property allows us to localize the residue to each edge. Indeed, by an abuse of notation, we can write
\begin{equation}
    Q_{E_H}(\tilde{u}-I_H\tilde{u}) = \sum_{e \in \mathcal{E}_H} Q_{E_H}(\tilde{u}-I_H\tilde{u})|_{e}\, ,
\end{equation}
where we equate the function $(\tilde{u}-I_H\tilde{u})|_{e}$ that is defined on $e$ to its zero extension to $E_H$, so that $(\tilde{u}-I_H\tilde{u})|_{e}\in H^{1/2}(E_H)$ and thus $Q_{E_H}(\tilde{u}-I_H\tilde{u})|_{e}$ makes sense.

Therefore, we localize the approximation task of $u^\mathsf{h}$ to $Q_{E_H}(\tilde{u}-I_H\tilde{u})|_{e}$, which is defined for each edge $e$.

\begin{remark}
\label{rmk: theory for edge localization}
    Again, we discuss several theoretical concerns below:
    \begin{itemize}
        \item Once the condition in Remark \ref{rmk: harmonic bubble splitting} is satisfied, the extension operator $Q_{E_H}$ is well-defined because the local equation is elliptic.
        \item According to the comment in Remark \ref{rmk: harmonic bubble splitting}, the solution $u$ is continuous, so the nodal interpolation $I_H \tilde{u}$ is well-defined.
        \item One can rigorously show that if we can approximate each local term with \[\|Q_{E_H}(\tilde{u}-I_H\tilde{u})|_{e} - w_e\|_{\mathcal{H}(\Omega)} \leq \epsilon_e \, ,\]
        then the global approximation error satisfies
        \[\|Q_{E_H}(\tilde{u}-I_H\tilde{u}) - \sum_{e \in \mathcal{E}_H} w_e\|_{\mathcal{H}(\Omega)}^2 \leq C_{\mathrm{mesh}}\sum_{e \in \mathcal{E}_H}\epsilon_e^2\, , \]
        where $C_{\mathrm{mesh}}$ is a constant dependent on the mesh structure only. In our previous work \cite{chen2021exponential,chen2023exponentially}, we formalize the approximation in the edge space via the $H_{00}^{1/2}(e)$ norm, which is equivalent to the $\mathcal{H}(\Omega)$ norm here after the extension by $Q_{E_H}$; see Proposition 2.5 and Theorem 2.6 in \cite{chen2021exponential}. In this review chapter, we explain the ideas using $Q_{E_H}$ rather than $H_{00}^{1/2}(e)$, since the former is more concise in an algorithm-focused exposition.
        
        We call the step from local approximation to global approximation \textit{edge coupling}.
    \end{itemize}
\end{remark}
% \yc{$Q_{E_H}$: this notation is not consistent with our previous works. I can understand that this might be useful for an algorithm-focused exposition.}
\subsection{Exponentially convergent SVD}
\label{sec: Exponentially convergent SVD}
Recall that by using the harmonic-bubble splitting and edge localization, we get the representation

\begin{equation}
    u = u^{\mathsf{h}}+u^{\mathsf{b}} = \sum_{e \in \mathcal{E}_H} Q_{E_H}(\tilde{u}-I_H\tilde{u})|_{e} + \sum_{x_i \in \mathcal{N}_H} u(x_i) \psi_i + u^{\mathsf{b}}\, .
\end{equation}
ExpMsFEM then relies on oversampling and local SVD to get an exponentially convergent approximation of each $Q_{E_H}(\tilde{u}-I_H\tilde{u})|_{e}$. For each $e$, consider an oversampling domain $w_e \supset e$. Any domain containing $e$ in the interior may be used, and as an illustrative example, we set 
\begin{equation*}
    \label{eqn: os domain 1 layer}
        \omega_e=\overline{\bigcup \{T\in \mathcal{T}_H: \overline{T} \cap e \neq \emptyset\}}\, .
    \end{equation*} 
An illustration of this choice for a quadrilateral mesh is given in Figure \ref{fig:os domain}. 
         \begin{figure}[ht]
        \centering
\tikzset{every picture/.style={line width=0.75pt}} %set default line width to 0.75pt    
\begin{tikzpicture}[x=0.75pt,y=0.75pt,yscale=-1,xscale=1]
%uncomment if require: \path (0,236); %set diagram left start at 0, and has height of 236

%Shape: Grid [id:dp4502413607423936] 
\draw  [draw opacity=0][dash pattern={on 4.5pt off 4.5pt}] (60,64) -- (201.5,64) -- (201.5,163) -- (60,163) -- cycle ; \draw  [dash pattern={on 4.5pt off 4.5pt}] (70,64) -- (70,163)(109,64) -- (109,163)(148,64) -- (148,163)(187,64) -- (187,163) ; \draw  [dash pattern={on 4.5pt off 4.5pt}] (60,74) -- (201.5,74)(60,113) -- (201.5,113)(60,152) -- (201.5,152) ; \draw  [dash pattern={on 4.5pt off 4.5pt}]  ;
%Straight Lines [id:da8503725141214521] 
\draw    (109,113) -- (148,113) ;
%Straight Lines [id:da21825982078459039] 
\draw    (70,74) -- (187,74) ;
%Straight Lines [id:da29433241828016854] 
\draw    (187,74) -- (187,152) -- (70,152) -- (70,74) ;
%Shape: Grid [id:dp9605623982728824] 
\draw  [draw opacity=0][dash pattern={on 4.5pt off 4.5pt}] (274,71) -- (399.5,71) -- (399.5,162) -- (274,162) -- cycle ; \draw  [dash pattern={on 4.5pt off 4.5pt}] (274,71) -- (274,162)(313,71) -- (313,162)(352,71) -- (352,162)(391,71) -- (391,162) ; \draw  [dash pattern={on 4.5pt off 4.5pt}] (274,71) -- (399.5,71)(274,110) -- (399.5,110)(274,149) -- (399.5,149) ; \draw  [dash pattern={on 4.5pt off 4.5pt}]  ;
%Straight Lines [id:da6107855326361947] 
\draw    (313,71) -- (313,110) ;
%Straight Lines [id:da9454074405226199] 
\draw    (274,71) -- (274,149) -- (352,149) -- (352,71) -- cycle ;

% Text Node
\draw (124,102) node [anchor=north west][inner sep=0.75pt]   [align=left] {$\displaystyle e$};
% Text Node
\draw (166,76) node [anchor=north west][inner sep=0.75pt]   [align=left] {$\displaystyle \omega _{e}$};
% Text Node
\draw (88,171) node [anchor=north west][inner sep=0.75pt]   [align=left] {interior edge};
% Text Node
\draw (253,171) node [anchor=north west][inner sep=0.75pt]   [align=left] {edge connected to boundary};
% Text Node
\draw (314,89) node [anchor=north west][inner sep=0.75pt]   [align=left] {$\displaystyle e$};
% Text Node
\draw (276,130) node [anchor=north west][inner sep=0.75pt]   [align=left] {$\displaystyle \omega _{e}$};
\end{tikzpicture}
        \caption{Illustration of oversampling domains. On the right, we use an edge connected to the upper boundary as an illustrating example.}
        \label{fig:os domain}\index{figures}
    \end{figure}
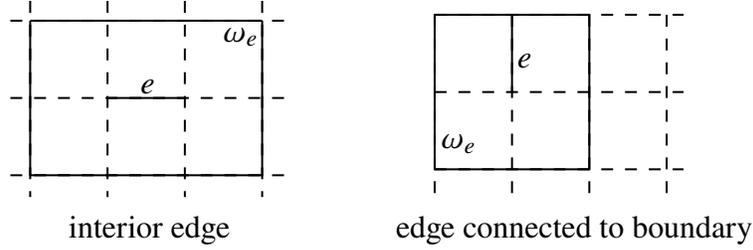

We can view $(\tilde{u}-I_H\tilde{u})|_{e}$ as the image of an operator acting on $u|_{\omega_e} \in H^1(\omega_e)$. We denote this operator by $R_e$ such that $Q_{E_H}(\tilde{u}-I_H\tilde{u})|_{e} = Q_{E_H}R_e (u|_{\omega_e})$. Now, we apply the harmonic-bubble splitting in subsection \ref{sec: Harmonic-bubble splitting} to the domain $\omega_e$, which leads to $u|_{\omega_e} = u_{\omega_e}^\mathsf{h} + u_{\omega_e}^\mathsf{b}$. It follows that
\begin{equation}
    Q_{E_H}(\tilde{u}-I_H\tilde{u})|_{e} = Q_{E_H}R_e u^\mathsf{h}_{\omega_e} + Q_{E_H}R_e u^\mathsf{b}_{\omega_e}\, .
\end{equation}
The term $R_e u^\mathsf{h}_{\omega_e}$ is a restriction of a harmonic part. As we mentioned at the beginning of this chapter, one can prove that the restriction operator acting on harmonic-type functions is of \textit{low approximation complexity}. More precisely, consider the space of harmonic parts in $\omega_e$, defined via 
\begin{equation}
    % \label{eqn-Helmholtz-harmonic-space}
    \begin{aligned}
        U(\omega_e):= \{v \in \mathcal{H}(\omega_e): &-\nabla \cdot (A \nabla v)+Vv=0, \text{ in } \omega_e\\
        % & v=0, \text{ on } \Gamma_D \cap \partial \omega_e, \\ 
        &A\nabla v\cdot\nu=\beta v, \text{ on } \Gamma_1 \cap \partial \omega_e\} 
        \, .
    \end{aligned}
    \end{equation}
The space is equipped with the norm $\|\cdot\|_{\mathcal{H}(\omega_e)}$. Then, one can show that the left singular values (in descending order) of the local operator 
\[Q_{E_H}R_e: (U(\omega_e),\|\cdot\|_{\mathcal{H}(\omega_e)}) \to (\mathcal{H}(\Omega), \|\cdot\|_{\mathcal{H}(\Omega)})\]
decays as $\lambda_{e,m} \leq C\exp(-bm^{\frac{1}{d+1}})$ in dimension $d$, for some generic constant $C,b$ independent of $m$ and $H$. Equivalently, if we write the left singular vectors as $v_{e,m} \in H^1(\Omega)$, which is local and supported in the neighboring elements of the edge $e$, then there exists some coefficient $b_{e,j}$ such that
\begin{equation}
    \|Q_{E_H}R_e u^\mathsf{h}_{\omega_e} - \sum_{1\leq j \leq m} b_{e,j} v_{e,j} \|_{\mathcal{H}(\Omega)} \leq C\exp(-bm^{\frac{1}{d+1}})\|u^\mathsf{h}_{\omega_e}\|_{\mathcal{H}(\omega_e)}\, .
\end{equation}
For more details, see Theorem 3.10 in \cite{chen2023exponentially}. Then, summing these local errors up, we get 
\begin{equation}
\begin{aligned}
    \sum_{e \in \mathcal{E}_H} \|u^\mathsf{h}_{\omega_e}\|_{\mathcal{H}(\omega_e)}^2 &\leq 2\sum_{e \in \mathcal{E}_H} (\|u|_{\omega_e}\|_{\mathcal{H}(\omega_e)}^2 + \|u_{\omega_e}^\mathsf{b}\|_{\mathcal{H}(\omega_e)}^2)\\
    & = O(\|u\|^2_{\mathcal{H}(\Omega)}+\|f\|^2_{L^2(\Omega)})\, ,
\end{aligned}
\end{equation}
where we used the fact that $\|u_{\omega_e}^\mathsf{b}\|_{\mathcal{H}(\omega_e)} = O(\|f\|_{L^2(\omega_e)})$ by the elliptic estimate.

Combining the above  estimates with edge coupling in Remark \ref{rmk: theory for edge localization}, we get the representation
\begin{equation}
\label{eqn: exp representation}
\begin{aligned}
    u = u^{\mathsf{h}}+u^{\mathsf{b}} = &\sum_{e \in \mathcal{E}_H} \sum_{1\leq j \leq m} b_{e,j} v_{e,j} + \sum_{x_i \in \mathcal{N}_H} u(x_i) \psi_i + u^{\mathsf{n}} \\
    &+ O\left(\exp(-bm^{\frac{1}{d+1}})(\|u\|_{\mathcal{H}(\Omega)}+\|f\|_{L^2(\Omega)})\right)\, ,
\end{aligned}
\end{equation}
where $u^\mathsf{n} := u^\mathsf{b} + \sum_{e \in \mathcal{E}_H} Q_{E_H}R_e u^\mathsf{b}_{\omega_e}$ is a part that depends on $f$ locally.
\begin{remark}
    We discuss several theoretical aspects and the implications of the above representation.
    \begin{itemize}
        
        \item The proof of the exponentially decaying singular values of $Q_{E_H}R_e$ is based on two steps. The first step is the iterative argument of Caccioppoli's inequality, first proposed in \cite{babuska2011optimal} and then refined in \cite{ma2021novel}. It shows that the singular values of the restriction operator on $U(\omega_e)$, which restricts a function from the original domain $\omega_e$ to a subdomain $\omega^* \supset e$, decay nearly exponentially fast. The second step is based on a stability estimate of the operator $Q_{E_H}R_e$ acting on $U(\omega^*)$; see Lemma 3.10 in \cite{chen2021exponential} or Lemma 6.1, 6.2 in \cite{chen2023exponentially}. 
        \item We can understand that the oversampling technique is used to take advantage of the low complexity property of the restriction operator. Historically, the idea of oversampling was proposed in \cite{hou_multiscale_1997} to reduce the resonance error in MsFEM.
        \item The remarkable thing about the representation in \eqref{eqn: exp representation} is the exponentially decaying error bound. 
        
        First, for elliptic equations with rough coefficients, the error bound implies that these basis functions can capture the behavior of the solution, which is a hard task for FEMs. Therefore, ExpMsFEM overcomes the difficulty of rough coefficients.
        
        Second, for the Helmholtz equation, the stability constant of the solution operator can depend on $k$; indeed, this is the main cause of the pollution effect \cite{babuska1997pollution}. Denote the stability constant by $C_{\text{stab}}(k)$ such that $\|u\|_{\mathcal{H}(\Omega)} \leq C_{\text{stab}}(k)\|f\|_{L^2(\Omega)}$. A prevalent and reasonable assumption on the constant is that of polynomial growth, namely $C_{\mathrm{stab}}(k)\leq C(1+k^\gamma)$ for some constants $\gamma$ and $C$; see, for example, \cite{lafontaine2019most}. In such case, we can further bound the error by
        \[\exp(-bm^{\frac{1}{d+1}})(\|u\|_{\mathcal{H}(\Omega)}+\|f\|_{L^2(\Omega)}) \leq  \exp(-bm^{\frac{1}{d+1}})(C(1+k^\gamma)+1)\|f\|_{L^2(\Omega)}\, .\]
        Therefore, once the number of basis functions per edge $m \sim \log^{d+1}(k)$ (logarithmically on $k$ only), the approximation error can be uniformly small for all $k$. It implies that the quantity $\eta(S)$ in \eqref{apro proxy} is small, which is important in determining the error of Galerkin's methods. In this sense, ExpMsFEM overcomes the difficulty of the pollution effect by using basis functions whose number scales at most $\log^{d+1}(k)$.
        \item The exponentially accurate representation in \eqref{eqn: exp representation} will not be possible if we do not use terms dependent on the right-hand side. Indeed, using basis functions independent of $f$, the optimal approximation error rate will be algebraic if the right-hand side is in $L^2(\Omega)$ only, due to well-known results in approximation theory (the Kolmogorov $n$-width \cite{pinkus2012n, melenk2000n}); see also the complexity analysis of the Green function of Helmholtz's equation \cite{engquist2018approximate}. From this perspective, we can understand that ExpMsFEM breaks the Kolmogorov barrier by using \textit{nonlinear model reduction} \cite{peherstorfer2022breaking}, i.e., the basis functions can depend on the input of the model, here the right-hand side.
    \end{itemize}
\end{remark}

\subsection{The solver based on ExpMsFEM}
\label{sec: alg ExpMsFEM}
Now, we can use the representation in \eqref{eqn: exp representation} to solve the equation efficiently. First, we form $\psi_i, v_{e,j}$ by computing the local extension $Q_{E_H}\tilde{\psi}_i$ for each node and the top-$m$ left singular vectors $v_{e,j}, 1\leq j \leq m$ of the local operator $Q_{E_H}R_e$ for each $e$; problems on different nodes and edges are independent and parallelizable. These become our offline basis functions.

For any right-hand side $f$, we compute the online part $u^\mathsf{n}$ by solving local linear equations involving $f$. This step can be parallelized. 

Then, we form an effective equation for $u-u^\mathsf{n}$ as
\begin{equation}
\label{eqn: effective eqn for u h}
    a(u-u^\mathsf{n},v)=(f,v)_{\Omega}-a(u^\mathsf{n},v)\, ,
\end{equation}
for any $v \in \mathcal{H}(\Omega)$. We solve the equation for $u-u^\mathsf{n}$ using a Galerkin method. As an example, using the Ritz-Galerkin method, we choose 
\[S = \mathrm{span}~\{\psi_i\ \text{for } x_i \in \mathcal{N}_H, \ v_{e,j} \  \text{for } 1\leq j \leq m, e \in \mathcal{E}_H \}\, ,\] 
and find a numerical solution $u_S \in S$
that satisfies
\begin{equation}
\label{eqn: Galerkin}
    a(u_S,v)=(f,v)_{\Omega}-a(u^\mathsf{n},v)\, ,
\end{equation}
for any $v \in S$. The final numerical solution is given by $u_S + u^\mathsf{n}$. We call $u^\mathsf{n}$ the online part and $u_S$ the offline part since $u_S$ lies in a space that is independent of $f$.

Note that in the Galerkin method for solving $u_S$, the stiffness matrix only needs to be assembled once and can be used for different $f$ afterward. We can understand \eqref{eqn: effective eqn for u h} as a reduced model of the original equation.
\begin{remark}
    We discuss several theoretical aspects regarding the effectiveness of the above method.
    \begin{itemize}
        \item The accuracy of the numerical solution is due to the quasi-optimality property mentioned earlier in subsection \ref{sec: Solving PDEs as function approximation}: once $\eta(S)$ is small, the solution error is of the same order compared to the optimal approximation using the basis functions, which is exponentially small according to the representation \eqref{eqn: exp representation}.
        \item When the solution is complex-valued, such as in the Helmholtz equations, we can use both the Ritz and Petrov versions of the Galerkin methods; for the former, if $\overline{S} \neq S$, we need to replace $S$ by $S+\overline{S}$; see discussions in \cite{chen2023exponentially}.
        \item One thing worth noting is that $\|u^\mathsf{n}\|_{\mathcal{H}(\Omega)}$ is of order $O(H)$, due to the standard elliptic estimate \cite{chen2021exponential,chen2023exponentially}. Therefore, if we aim for $O(H)$ accuracy only, we can ignore this part, and simply setting $u^\mathsf{n}$ = 0 in the above algorithm will lead to a solution accurate up to $O(H)$.
    \end{itemize}
\end{remark}

\section{Numerical Experiments}
\label{sec: Numerical Experiments}
In this section, we present some numerical experiments to demonstrate the effectiveness of ExpMsFEM. For all the experiments, we consider the domain $\Omega=[0,1]\times [0,1]$ and discretize it by a uniform two-level quadrilateral mesh; see a fraction of this mesh in Figure \ref{fig:mesh1}, where we also show an edge $e$ and its oversampling domain $\omega_e$ in solid lines. 
\begin{figure}[htbp]
\centering

\tikzset{every picture/.style={line width=0.75pt}} %set default line width to 0.75pt        

\begin{tikzpicture}[x=0.75pt,y=0.75pt,yscale=-1,xscale=1]
%uncomment if require: \path (0,303); %set diagram left start at 0, and has height of 303
%Shape: Grid [id:dp32131165057233446] 
\draw  [draw opacity=0][dash pattern={on 4.5pt off 4.5pt}] (170.5,75) -- (365.5,75) -- (365.5,231) -- (170.5,231) -- cycle ; \draw  [color={rgb, 255:red, 0; green, 0; blue, 0 }  ,draw opacity=1 ][dash pattern={on 4.5pt off 4.5pt}] (209.5,75) -- (209.5,231)(248.5,75) -- (248.5,231)(287.5,75) -- (287.5,231)(326.5,75) -- (326.5,231) ; \draw  [color={rgb, 255:red, 0; green, 0; blue, 0 }  ,draw opacity=1 ][dash pattern={on 4.5pt off 4.5pt}] (170.5,114) -- (365.5,114)(170.5,153) -- (365.5,153)(170.5,192) -- (365.5,192) ; \draw  [color={rgb, 255:red, 0; green, 0; blue, 0 }  ,draw opacity=1 ][dash pattern={on 4.5pt off 4.5pt}] (170.5,75) -- (365.5,75) -- (365.5,231) -- (170.5,231) -- cycle ;
%Shape: Grid [id:dp5565750274689868] 
\draw  [draw opacity=0][dash pattern={on 0.84pt off 2.51pt}] (170.5,75) -- (365.5,75) -- (365.5,231) -- (170.5,231) -- cycle ; \draw  [color={rgb, 255:red, 0; green, 0; blue, 0 }  ,draw opacity=1 ][dash pattern={on 0.84pt off 2.51pt}] (183.5,75) -- (183.5,231)(196.5,75) -- (196.5,231)(209.5,75) -- (209.5,231)(222.5,75) -- (222.5,231)(235.5,75) -- (235.5,231)(248.5,75) -- (248.5,231)(261.5,75) -- (261.5,231)(274.5,75) -- (274.5,231)(287.5,75) -- (287.5,231)(300.5,75) -- (300.5,231)(313.5,75) -- (313.5,231)(326.5,75) -- (326.5,231)(339.5,75) -- (339.5,231)(352.5,75) -- (352.5,231) ; \draw  [color={rgb, 255:red, 0; green, 0; blue, 0 }  ,draw opacity=1 ][dash pattern={on 0.84pt off 2.51pt}] (170.5,88) -- (365.5,88)(170.5,101) -- (365.5,101)(170.5,114) -- (365.5,114)(170.5,127) -- (365.5,127)(170.5,140) -- (365.5,140)(170.5,153) --
(365.5,153)(170.5,166) -- (365.5,166)(170.5,179) -- (365.5,179)(170.5,192) -- (365.5,192)(170.5,205) -- (365.5,205)(170.5,218) -- (365.5,218) ; \draw  [color={rgb, 255:red, 0; green, 0; blue, 0 }  ,draw opacity=1 ][dash pattern={on 0.84pt off 2.51pt}] (170.5,75) -- (365.5,75) -- (365.5,231) -- (170.5,231) -- cycle ;
%Straight Lines [id:da38546691028387126] 
\draw    (209.5,114) -- (209.5,192) -- (326.5,192) -- (326.5,114) -- cycle ;

%Straight Lines [id:da5883609541572945] 
\draw [color={rgb, 255:red, 0; green, 0; blue, 0 }  ,draw opacity=1 ][line width=0.75]    (248.5,153) -- (287.5,153) ;
%Shape: Grid [id:dp6978138592522893] 
\draw  [draw opacity=0][dash pattern={on 0.84pt off 2.51pt}] (407.5,157) -- (507.5,157) -- (507.5,218) -- (407.5,218) -- cycle ; \draw  [color={rgb, 255:red, 0; green, 0; blue, 0 }  ,draw opacity=1 ][dash pattern={on 0.84pt off 2.51pt}] (417.5,157) -- (417.5,218)(430.5,157) -- (430.5,218)(443.5,157) -- (443.5,218)(456.5,157) -- (456.5,218)(469.5,157) -- (469.5,218)(482.5,157) -- (482.5,218)(495.5,157) -- (495.5,218) ; \draw  [color={rgb, 255:red, 0; green, 0; blue, 0 }  ,draw opacity=1 ][dash pattern={on 0.84pt off 2.51pt}] (407.5,167) -- (507.5,167)(407.5,180) -- (507.5,180)(407.5,193) -- (507.5,193)(407.5,206) -- (507.5,206) ; \draw  [color={rgb, 255:red, 0; green, 0; blue, 0 }  ,draw opacity=1 ][dash pattern={on 0.84pt off 2.51pt}]  ;
%Shape: Grid [id:dp3571739898377386] 
\draw  [draw opacity=0][dash pattern={on 4.5pt off 4.5pt}] (406.5,73) -- (508.5,73) -- (508.5,135) -- (406.5,135) -- cycle ; \draw  [color={rgb, 255:red, 0; green, 0; blue, 0 }  ,draw opacity=1 ][dash pattern={on 4.5pt off 4.5pt}] (416.5,73) -- (416.5,135)(455.5,73) -- (455.5,135)(494.5,73) -- (494.5,135) ; \draw  [color={rgb, 255:red, 0; green, 0; blue, 0 }  ,draw opacity=1 ][dash pattern={on 4.5pt off 4.5pt}] (406.5,83) -- (508.5,83)(406.5,122) -- (508.5,122) ; \draw  [color={rgb, 255:red, 0; green, 0; blue, 0 }  ,draw opacity=1 ][dash pattern={on 4.5pt off 4.5pt}]  ;

% Text Node
\draw (263.5,140) node [anchor=north west][inner sep=0.75pt]   [align=left] {$\displaystyle e$};
% Text Node
\draw (303.5,115) node [anchor=north west][inner sep=0.75pt]   [align=left] {$\displaystyle \omega _{e}$};
% Text Node
\draw (414,136) node [anchor=north west][inner sep=0.75pt]   [align=left] {coarse mesh};
% Text Node
\draw (424,218) node [anchor=north west][inner sep=0.75pt]   [align=left] {fine mesh};

\end{tikzpicture}

\caption{Two level mesh: a fraction}
\label{fig:mesh1}\index{figures}
\end{figure}
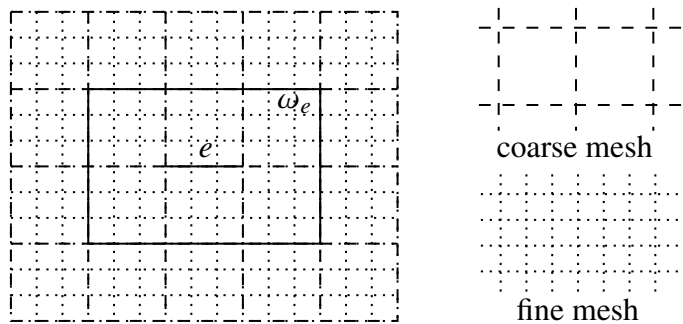
 The coarse and fine mesh sizes are denoted by $H$ and $h$, respectively. 
 
For a given equation,  we compute the reference solution $u_{\text{ref}}$ using the classical FEM on the fine mesh with a sufficiently small $h$, which we choose to be $h=1/1024$. By \textit{a posteriori} estimates, we can check that the fine mesh indeed resolves the corresponding problems; thus, the associated fine mesh solutions could serve as accurate reference solutions for all of our numerical examples. In our numerical computation, we solve local problems that are required in the ExpMsFEM framework using the fine mesh. For detailed implementation, we refer to \cite{chen2021exponential,chen2023exponentially}.
 \begin{remark}[Accuracy on the discrete level]
For simplicity of presentation, we do not provide error analysis of ExpMsFEM on the fully discrete level, where the accuracy of the local problems can depend on the resolution of the fine grid. For a detailed error estimate on the fully discrete level in the context of partition of unity methods, see, for example, \cite{ma2021error, ma2021wavenumber}.
\end{remark}
The accuracy of a numerical solution $u_{\mathrm{sol}}$ is computed by comparing it with the reference solution $u_{\text{ref}}$ on the fine mesh. The accuracy will be measured both in the $L^2$ norm and energy norm: 
    \begin{equation}
\label{rel_error}
\begin{aligned}
e_{L^{2}}&=\frac{\|u_{\text{ref}}-u_{\mathrm{sol}}\|_{L^{2}(\Omega)}}{\|u_{\text{ref}}\|_{L^{2}(\Omega)}}\, ,\\
e_{\mathcal{H}}&=\frac{\|u_{\text{ref}}-u_{\mathrm{sol}}\|_{\mathcal{H}(\Omega)}}{\|u_{\text{ref}}\|_{\mathcal{H}(\Omega)}}\, . 
\end{aligned}
\end{equation}

In subsection \ref{sec: exp period multiscale}, we consider an elliptic equation where the coefficient $A(x)$ is periodic but contains multiple scales. This example demonstrates the exponential accuracy of ExpMsFEM. In subsection \ref{sec: exp high contrast}, we consider an elliptic equation where $A(x)$ is of high contrast. This example shows the robustness of ExpMsFEM regarding the high contrast. In subsection \ref{sec: exp Helmholtz}, an instance of Helmholtz's equation with rough media and mixed boundary conditions is presented. This example illustrates the effectiveness of ExpMsFEM in solving general indefinite Helmholtz equations.
\subsection{A periodic example with multiple spatial scales}
\label{sec: exp period multiscale}
In the first example, we consider an elliptic problem ($V=0$) with multiple spatial scales. We choose the coefficient $A$ with five scales as follows: 
   \begin{equation}
\begin{aligned} 
A(x)=\frac{1}{6}\left(\frac{1.1+\sin \left(2 \pi x_1 / \epsilon_{1}\right)}{1.1+\sin \left(2 \pi x_2 / \epsilon_{1}\right)}+\frac{1.1+\sin \left(2 \pi x_2 / \epsilon_{2}\right)}{1.1+\cos \left(2 \pi x_1 / \epsilon_{2}\right)}+\frac{1.1+\cos \left(2 \pi x_1 / \epsilon_{3}\right)}{1.1+\sin \left(2 \pi x_2 / \epsilon_{3}\right)}\right.\\\left.+\frac{1.1+\sin \left(2 \pi x_2 / \epsilon_{4}\right)}{1.1+\cos \left(2 \pi x_1 / \epsilon_{4}\right)}+\frac{1.1+\cos \left(2 \pi x_1 / \epsilon_{5}\right)}{1.1+\sin \left(2 \pi x_2 / \epsilon_{5}\right)}+\sin \left(4 x_1^{2} x_2^{2}\right)+1\right) \, ,\end{aligned}
\end{equation}
where $x=(x_1,x_2)$, $\epsilon_1=1/5$, $\epsilon_2=1/13$, $\epsilon_3=1/17$, $\epsilon_4=1/31$, $\epsilon_5=1/65$. We choose homogeneous Dirichlet boundary conditions, i.e., $\Gamma_2=\emptyset$. We set $f=-1$.

In this example, we illustrate the exponential accuracy and the convergence rate with respect to the coarse mesh size $H$. We take $H=2^{-i}$, $i=3,4,...,7$ and take $m=1,2,...,6$ for each $H$. The numerical results are shown in Figure \ref{fig:eg1}, where $N_c=1/H$. 
\begin{figure}[ht]
    \centering
    \includegraphics[width=6cm]{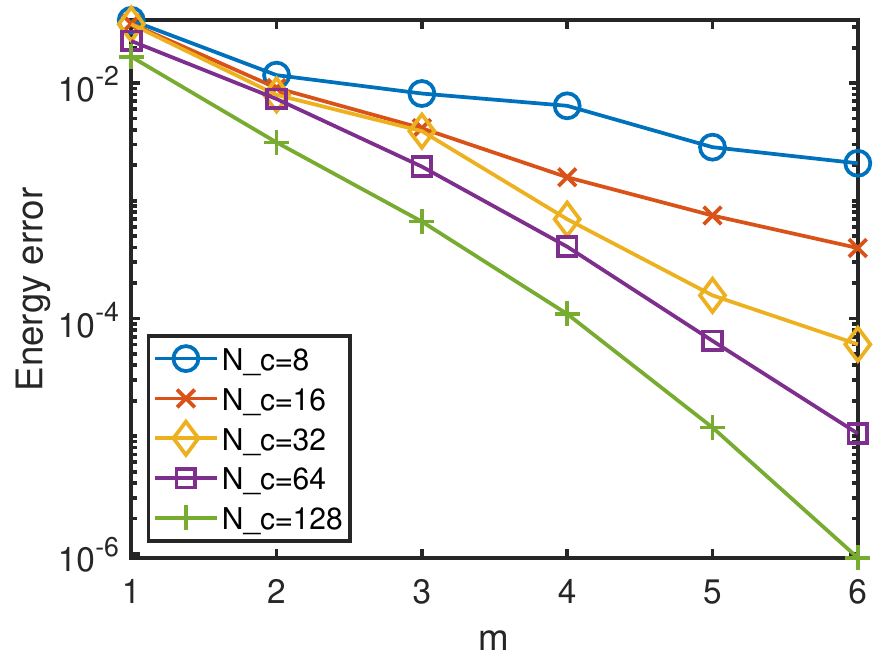}
    \includegraphics[width=6cm]{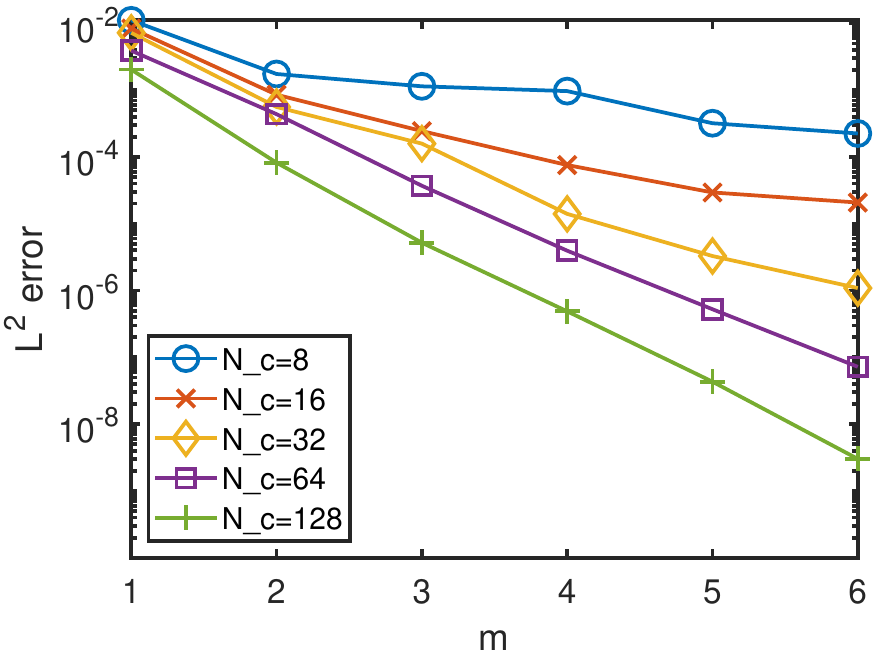}
    \caption{Numerical results for the periodic example. Left: $e_{\mathcal{H}}$ versus $m$; right: $e_{L^2}$ versus $m$.}
    \label{fig:eg1}\index{figures}
\end{figure}
We can see an exponential decay of errors for every coarse mesh size $H$. For smaller $H$, the convergence is faster. This can be understood as a finite-resolution effect. For example, when $H=1/128$, there are only $H/h-1=7$ total degrees of freedom on each edge, so of course, $m=6$ basis per edge would result in a very accurate solution.

\subsection{An example with high contrast channels}
\label{sec: exp high contrast}
In the second example, we consider an elliptic problem ($V=0$) with high contrast channels. Let 
\[X:=\{(x_1,x_2) \in [0,1]^2, x_1,x_2 \in \{0.2,0.3,...,0.8\}\} \subset [0,1]^2 \, , \]
and the coefficient is defined as
\begin{equation*}
   A(x)=\left\{ 
    \begin{aligned} 1, \quad &\text{if} \ \operatorname{dist}(x,X)\geq 0.015\\
    M, \quad &\text{else}\, .
    \end{aligned}
    \right.
\end{equation*}
Here, $M$ is a parameter controlling the contrast. We visualize $\log_{10} A$ in the left plot of  Figure $\ref{fig:contour_A}$ for $M=10^6$.
\begin{figure}[ht]
    \centering
    \includegraphics[width=6cm]{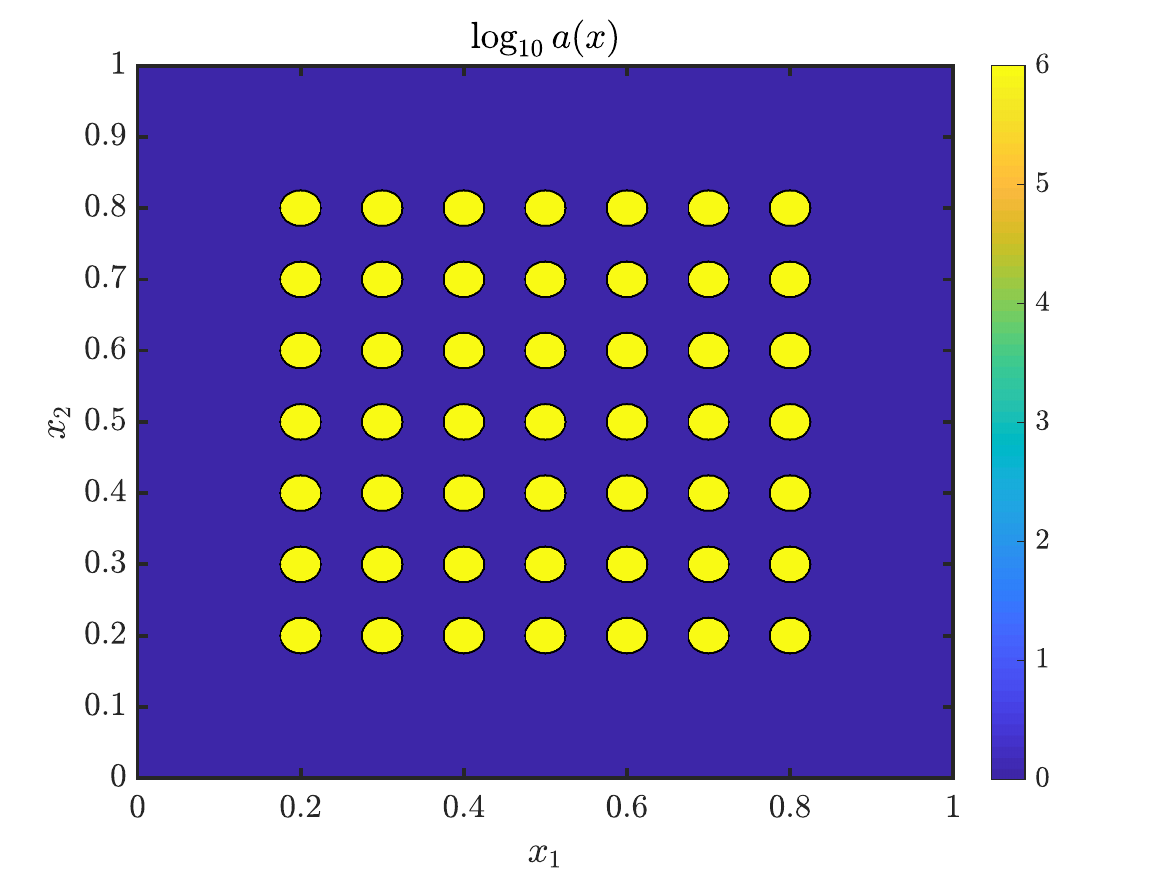}
    \includegraphics[width=6cm]{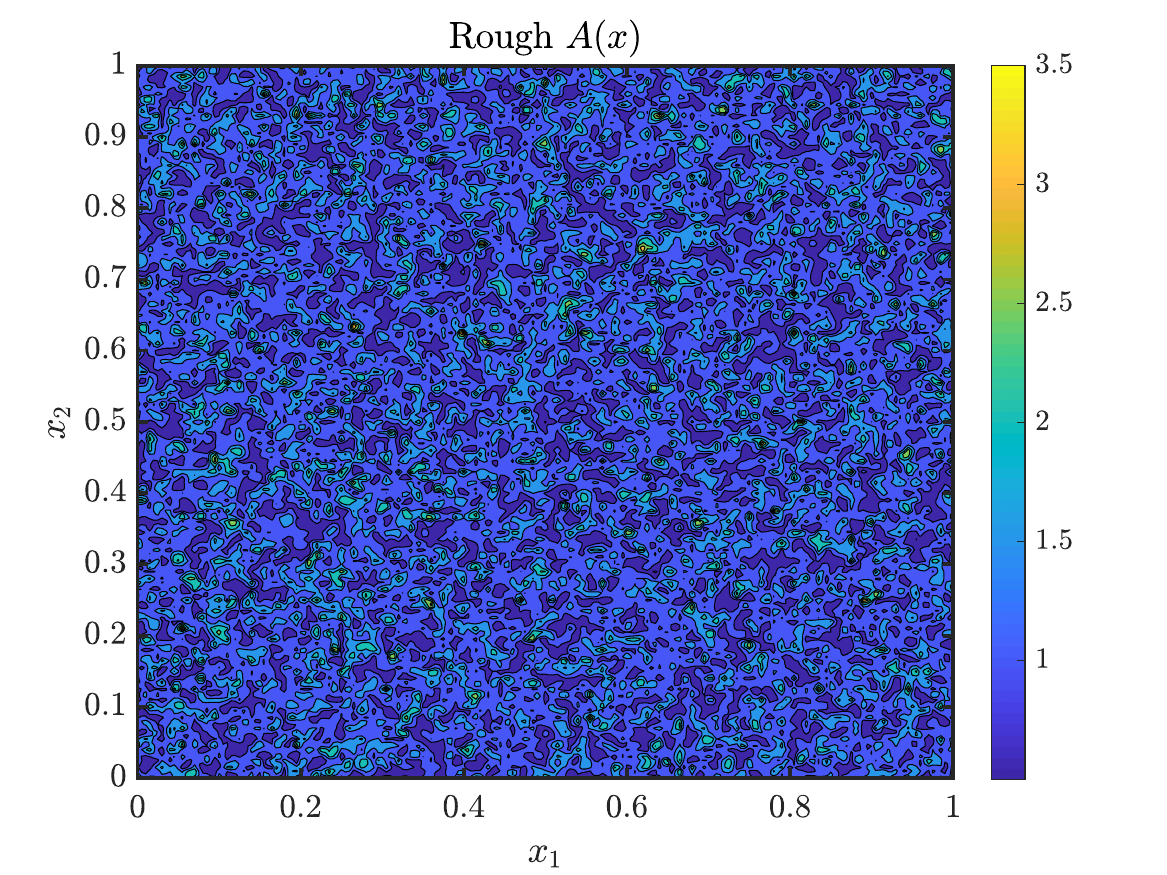}
    \caption{Left: the contour of $\log_{10} A$ for the high contrast example; right: the contour of $A$ for the rough media example.}\index{figures}
    \label{fig:contour_A}
\end{figure}
Again, we choose homogeneous Dirichlet boundary conditions, i.e., $\Gamma_2=\emptyset$, with a non-constant right-hand side $f(x)=x_1^4-x_2^3+1$. 

In this example, we illustrate the convergence rate w.r.t the contrast $M$. We take different $M$ using the coarse mesh size $H=2^{-5}$ and  $m=1,2,...,7$. The numerical results are shown in Figure \ref{fig:eg2}. 
\begin{figure}[ht]
    \centering
    \includegraphics[width=6cm]{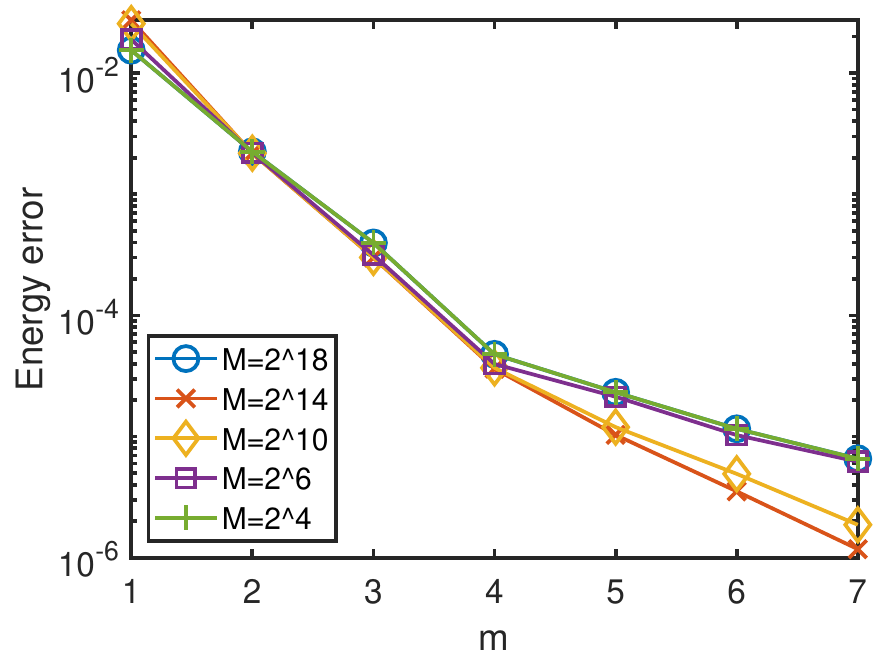}
    \includegraphics[width=6cm]{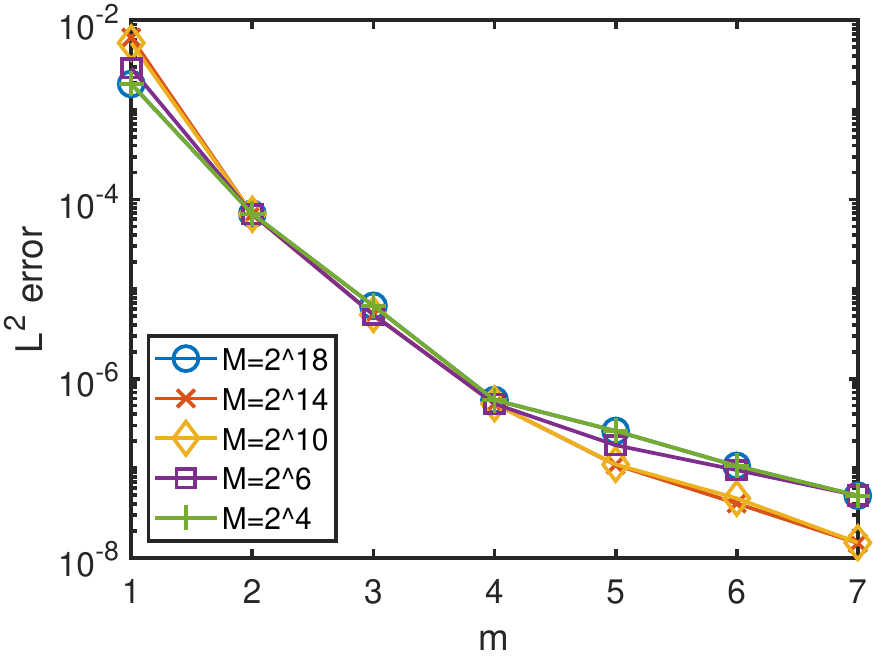}
    \caption{Numerical results for the high contrast example. Left: $e_{\mathcal{H}}$ versus $m$; right: $e_{L^2}$ versus $m$.}
    \label{fig:eg2}\index{figures}
\end{figure}
We observe a consistently exponential error decay independent of the contrast. Thus, our method demonstrates robustness with respect to the contrast $A(x)$. An intuitive explanation for this robustness could be that every step in ExpMsFEM is adaptive to $A(x)$. For example, the singular value decay of the operator $Q_{E_H}R_e$ would have some robustness regarding high contrasts in $A(x)$ because both of the norms in the domain and image of the operator are $A(x)$-weighted. We leave the theoretical analysis of deriving an $A(x)$-adapted estimate for future study. 

Also, we would like to mention that the size $h=1/1024$ of the fine mesh can actually resolve contrasts $M=2^4$ and $2^6$ only; for higher contrasts, a posterior error analysis shows the reference solution on the fine mesh is not very accurate. However, we consistently observe a small error in our solution compared to the fine mesh solution, even in the regime where the fine mesh solution itself is not accurate. This implies that ExpMsFEM admits a very accurate dimension reduction of the equation on the fine mesh.

\subsection{An example of Helmholtz equation with rough field and mixed boundary}
\label{sec: exp Helmholtz}
In the last example, we consider the Helmholtz equation. This example is the same as Example 3 in \cite{chen2023exponentially}. We present it here to demonstrate that our methods are effective for complicated coefficients and mixed boundary conditions.

We impose the
homogeneous Dirichlet boundary condition on $(x_1,0), x_1 \in [0,1]$, the homogeneous Neumann boundary condition on $(x_1,1), x_1 \in [0,1]$, and the homogeneous Robin boundary condition on the other two parts of $\partial \Omega$. 
We choose $A(x)$ to be a realization of some random field; more precisely, we set
\begin{equation}
A(x)=|\xi(x)|+0.5\, ,
\end{equation}
        where the field $\xi(x)$ satisfies \[\xi(x)=a_{11}\xi_{i,j}+a_{21}\xi_{i+1,j}+a_{12}\xi_{i,j+1}+a_{22}\xi_{i+1,j+1}, \ \text{if}\ x \in [\frac{i}{2^{7}},\frac{i+1}{2^{7}})\times [\frac{j}{2^{7}},\frac{j+1}{2^{7}})\, .\]
        Here, $\{\xi_{i,j}, 0\leq i,j \leq 2^7 \}$ are i.i.d. standard Gaussian random variables. In addition, $a_{11}=(i+1-2^7x_1)(j+1-2^7x_2)$, $a_{21}=(2^7x_1-i)(j+1-2^7x_2)$, $a_{12}=(i+1-2^7x_1)(2^7x_2-j)$, $a_{22}=(2^7x_1-i)(2^7x_2-j)$ are interpolating coefficients to make $\xi(x)$ piecewise linear. A sample from this field is displayed in the right plot of Figure \ref{fig:contour_A}.

Moreover, we also take $V/k^2$ and $\beta/ik$ as independent samples drawn from this random field. We choose the wavenumber $k=2^5$, the right-hand side 
{$f(x_1,x_2)=x_1^4-x_2^3+1$}, and the coarse mesh $H=2^{-5}$. Again, we take $m=1,2,...,7$ and present the numerical results in Figure \ref{fig:eg4}. 
\begin{figure}[ht]
    \centering
    \includegraphics[width=6cm]{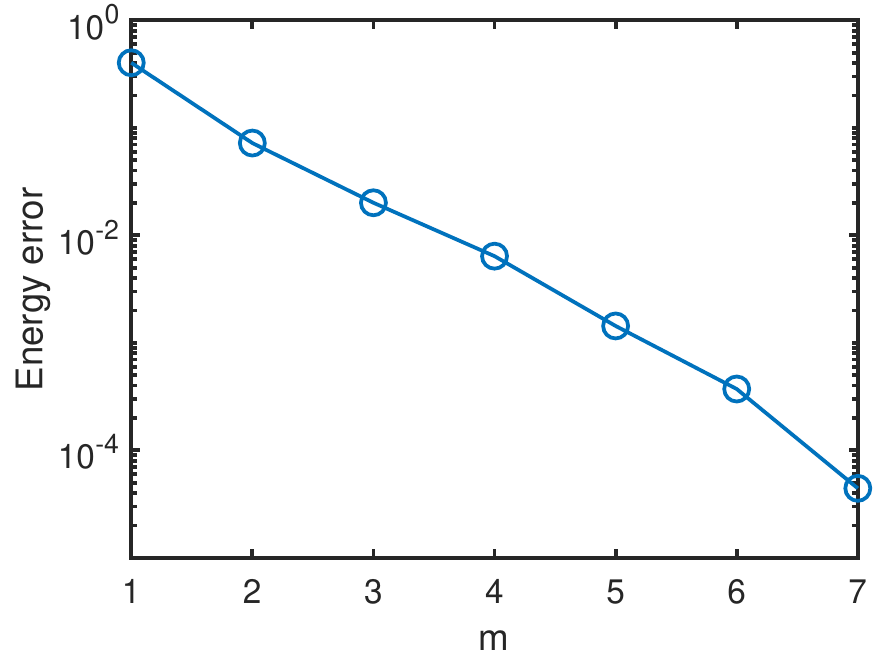}
    \includegraphics[width=6cm]{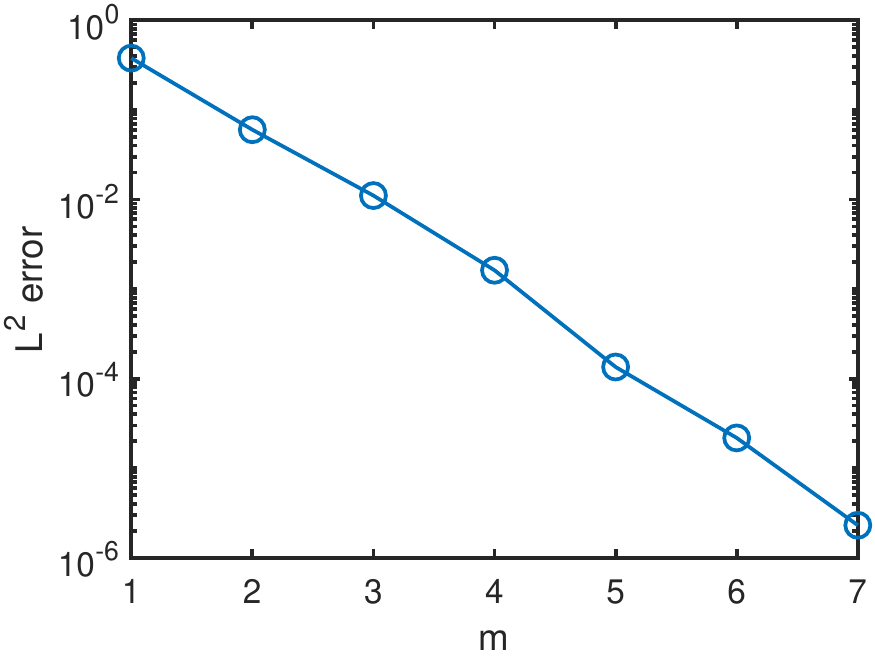}
    \caption{Numerical results for the mixed boundary and rough field example. Left: $e_{\mathcal{H}}$ versus $m$; right: $e_{L^2}$ versus $m$.}\index{figures}
    \label{fig:eg4}
\end{figure}
Clearly, a nearly exponential rate of convergence is still observed for this challenging example.

\section{Discussions}
\label{sec: Discussions}
In this section, we discuss related multiscale methods in the literature; for a more specific review under the context of the elliptic and Helmholtz equations, see \cite{chen2021exponential,chen2023exponentially}. We also outline future possibilities and open questions about ExpMsFEM at the end of this section.
\subsection{Related literature}
There is a vast amount of literature on multiscale methods and numerical homogenization. 

Earlier work mainly focuses on structured $A(x)$ such as in periodic media and with scale separation; some examples include the generalized finite element methods (GFEM) \cite{babuvska1983generalized}, the multiscale finite element method (MsFEM) \cite{hou_multiscale_1997, hou1999convergence, efendiev2000convergence}, the variational multiscale methods (VMS) \cite{hughes_variational_1998}, and the heterogeneous multiscale method (HMM) \cite{abdulle2012heterogeneous}. 

Later on, people are interested in multiscale methods that can address more general rough coefficients that lie in $L^{\infty}(\Omega)$ only; see, for example, the work of optimal basis using partition of unity functions \cite{babuska2011optimal,babuvska2020multiscale, ma2021error, ma2021novel}, harmonic coordinates \cite{owhadi2007metric}, local orthogonal decomposition (LOD) \cite{malqvist_localization_2014,henning2013oversampling,kornhuber2018analysis, hauck2021super,maier2021high}, Gamblets related approaches \cite{owhadi2011localized,owhadi2014polyharmonic,owhadi2015bayesian,owhadi_multigrid_2017,hou_sparse_2017,owhadi2019operator, chen2020multiscale}, and generalizations of MsFEM \cite{hou2015optimal,chung2018constraint,li2019convergence,fu2019edge}. Different methods differ in how to find an accurate function representation. In deriving the function representation in ExpMsFEM, the solution is first decomposed into a harmonic part and a bubble part. For elliptic equations, this decomposition is the same as the orthogonal decomposition in previous work of MsFEM \cite{hou2015optimal} and approximate component mode synthesis \cite{hetmaniuk2010special,hetmaniuk2014error}.

To the best of our knowledge, among all the previous work, the optimal basis framework using partition of unity functions (and its variant) is the only one that achieves nearly exponential accuracy regarding the number of basis functions. Our ExpMsFEM \cite{chen2021exponential,chen2023exponentially} is motivated by the argument of Caccioppoli's inequality used in the optimal basis framework. ExpMsFEM is the first framework that achieves exponential accuracy without using partition of unity functions and is a direct generalization of MsFEM.

We comment in more detail on the differences and similarities between the optimal basis framework and ExpMsFEM. In the optimal basis framework, the exponentially accurate representation is obtained through the partition of unity functions rather than the edge localization and coupling in ExpMsFEM. More precisely, one can write 
\begin{equation}
\label{eqn: representation PUM}
    u = \sum_{i} \eta_i u = \sum_{i} \eta_i u^\mathsf{h}_{\omega_i} + \sum_{i} \eta_i u^\mathsf{b}_{\omega_i}\, ,
\end{equation}
where $\{\eta_i\}_i$ are partition of unity functions subordinate to an overlapped domain decomposition $\{\omega_i\}_i$ and $u_{\omega_i}^\mathsf{h}, u_{\omega_i}^\mathsf{b}$ are obtained by the harmonic-bubble splitting in $\omega_i$. The part $\eta_i u^\mathsf{h}_{\omega_i}$ can be seen as a ``restriction'' of harmonic-type functions. Thus, the argument using Caccioppoli's inequality implies that this part can be approximated by basis functions with a nearly exponential convergence rate.

Compared to \eqref{eqn: exp representation}, the representation \eqref{eqn: representation PUM} admits better geometric flexibility since by using partition of unity functions, such representation can work for problems in general dimensions. The representation \eqref{eqn: exp representation} produced by ExpMsFEM is tied to the mesh structure. When $d=2$, we have nodal and edge basis functions in the representation \eqref{eqn: exp representation}. When $d \geq 3$, we need facial basis functions and so on to represent the solution; for details see section 7 in \cite{chen2023exponentially}. In this sense, ExpMsFEM removes the partition of unity functions in the overlapped domain decomposition but pays the design cost of using a more complicated geometric structure in the non-overlapped domain decomposition. Nevertheless, the benefit of non-overlapped domain decomposition is that the basis functions are more localized since the local domain is smaller. Also, ExpMsFEM does not have the additional parameter of the partition of unity functions. Some basic numerical comparisons between ExpMsFEM and optimal basis using partition of unity functions are presented in \cite{chen2023exponentially}. We need a more in-depth comparison between the two approaches to identify their trade-offs more clearly.

\subsection{Future directions} To now, ExpMsFEM has been successfully applied to solve elliptic and Helmholtz equations. Moving forward, one can extend this idea to advection-dominated diffusion problems, time-dependent problems such as Schr\"odinger's equations, and many other linear equations. Extension to nonlinear equations appears to be nontrivial since the decomposition used in ExpMsFEM requires linearity of the equation. It could be interesting to explore the combination of ExpMsFEM and linearization to provide nonlinear homogenization of these equations.

For the current ExpMsFEM framework, we observe its robustness regarding the high contrast in the media numerically (subsection \ref{sec: exp high contrast}), but a rigorous understanding of such robustness is still lacking. Moreover, a discrete-level analysis of ExpMsFEM could be helpful for its practical use.

In essence, both ExpMsFEM and optimal basis using partition of unity functions take advantage of the low approximation complexity structures of the restriction operator on harmonic-type functions. Finding other novel low complexity structures is crucial to advance multiscale computation and model reduction. 

ExpMsFEM and optimal basis using partition of unity functions imply that nonlinear model reduction can break the Kolmogorov barrier and achieve remarkable exponential convergence. Embedding this idea to data-driven model reduction or operator learning also represents an exciting avenue for future work.

\printbibliography[heading=bibintoc]

\appendix

\appendix
\chapter{Blowup Analysis for a Quasi-exact 1D Model of 3D Euler and Navier-Stokes}
\label{append:hou-li}
We study the singularity formation of a quasi-exact 1D model proposed by Hou-Li in \cite{hou2008dynamic}. This model is based on an approximation of the axisymmetric Navier-Stokes equations in the $r$ direction. The solution of the 1D model can be used to construct an exact solution of the original 3D Euler and Navier-Stokes equations if the initial angular velocity, angular vorticity, and angular stream function are linear in $r$. This model shares many intrinsic properties similar to those of the 3D Euler and Navier-Stokes equations. It captures the competition between advection and vortex stretching as in the 1D De Gregorio \cite{de1990one, de1996partial} model. We show that the inviscid model with weakened advection and smooth initial data or the original 1D model with H\"older continuous data develops a self-similar blowup. We also show that the viscous model with weakened advection and smooth initial data develops a finite time blowup. %Inspired by a recent work by Chen \cite{chen2021regularity}, we perform linear stability analysis around an equilibrium solution of the 1D model by using the dynamic rescaling formulation. 
To obtain sharp estimates for the nonlocal terms, we perform an exact computation  for the low-frequency Fourier modes and  {extract damping in leading order estimates}  for the high-frequency modes using singularly weighted norms {in the energy estimates}. The analysis for the viscous case is more subtle since the viscous terms produce some instability if we just use singular weights. We establish the blowup analysis for the viscous model by carefully designing an energy norm that combines a singularly weighted energy norm and a {sum of} high-order Sobolev norms.

\section{Introduction}

Whether the 3D incompressible Euler and Navier-Stokes equations can develop a finite time singularity from smooth initial data is one of the most outstanding open questions in nonlinear partial differential equations. An essential difficulty is that the vortex stretching term has a quadratic nonlinearity in terms of vorticity. A simplified 1D model was proposed by the Constantin-Lax-Majda model (CLM model for short) \cite{constantin1985simple} to capture the effect of nonlocal vortex stretching. The CLM model can be solved explicitly and can develop a finite time singularity from smooth initial data. Later on, De Gregorio incorporated the advection term into the CLM model to study the competition between advection and vortex stretching \cite{de1990one,de1996partial}, see  {\cite{constantin1986} for singularity formation in the distorted Euler equations with transport neglected and also   \cite{hou2008dynamic,lei2009stabilizing} for a related study on the stabilizing effect of advection for the 3D Euler and Navier-Stokes equations. There have been recent studies on the effect of advection and vortex stretching in other related models; see \cite{sarria2013blow} for the generalized inviscid Proudman-Johnson equation, \cite{elgindi2022strong} with a Riesz transform added to the vorticity formulation of 2D Euler equation, and \cite{miller2023finite} with advection term dropped in the vorticity formulation of 3D Euler equation.}  In \cite{okamoto2008generalization},
Okamoto, Sakajo, and Wunsch further introduced a parameter for the advection term to measure the relative strength of the advection in the De Gregorio model. These simplified 1D models have inspired  many subsequent studies. Interested readers may consult the excellent surveys \cite{constantin2007euler,gibbon2008three,kiselev2018,majda2002vorticity} and the references therein. {Very recently, the authors in \cite{huang2023self} established self-similar blowup for the whole family of gCLM models with $a\leq1$ using a fixed-point argument.}
On the other hand, these 1D scalar models are phenomenological in nature and cannot be used to recover the solution of the original 3D Euler equations. 

{For the line of research on the singularity formation for the 3D Euler equations,}  Luo-Hou \cite{luo2013potentially-2} presented in 2014 convincing numerical evidence that the 3D axisymmetric Euler equations with smooth initial data and boundary develop a potential finite time singularity. %In the same paper \cite{luo2013potentially-2}, the authors proposed the Hou-Luo model along the boundary $r=1$. This model captures many essential features observed in the Hou-Luo blowup scenario for the axisymmetric Euler equations. In \cite{choi2014on}, the authors proved the blowup of the Hou-Luo model. In \cite{chen2021HL}, Chen-Hou-Huang proved the asymptotically self-similar blowup of the Hou-Luo model by extending the method of analysis for the finite time blowup of the De Gregorio model by the same authors in \cite{chen2021finite}. 
Inspired by Elgindi's recent breakthrough for finite time singularity of the axisymmetric Euler with no swirl and $C^{1,\alpha}$ velocity \cite{ elgindi2021finite}, Chen and Hou proved the finite time blowup of the 2D Boussinesq and 3D Euler equations with $C^{1,\alpha}$ initial velocity and boundary \cite{chen2019finite2}. {For other recent works on singularity formation of 3D Euler with limited regularity, see also \cite{chen2023remarks,cordoba2023finite} for initial data that is smooth except at the origin, \cite{cordoba2023blow} for more smooth data but with a 
 $C^{1/2-\epsilon}$ force, and \cite{elgindi2021incompressible, elgindi2019finite} for settings with nonsmooth boundary.} Very recently, Chen and Hou proved stable and nearly self-similar blowup of the 2D Boussinesq and 3D Euler with smooth initial data and boundary using computer assistance \cite{chen2022stable}. 

In 2008, Hou and Li \cite{hou2008dynamic} proposed a new 1D model for the 3D axisymmetric Euler and Navier-Stokes equations. This model approximates the 3D axisymmetric Euler and Navier-Stokes equations along the symmetry axis based on an approximation in the $r$ direction. The solution of the 1D model can be used to construct an exact solution of the original 3D Navier-Stokes equations if the initial angular velocity, angular vorticity, and angular stream function are linear in $r$. This model shares many intrinsic properties similar to those of the 3D Navier-Stokes equations. Thus, it captures some essential nonlinear features of the 3D Euler and Navier-Stokes equations.
%and is very different from the Hou-Luo model, which is defined along the boundary $r=1$, instead of the symmetry axis $r=0$. 
In the same paper  \cite{hou2008dynamic}, the authors proved the global regularity of the Hou-Li model by deriving a new Lyapunov functional, which captures the exact cancellation between advection and vortex stretching. 

The purpose of this chapter is to study the singularity formation of a weak advection version of the Hou-Li model for smooth data. We introduce a parameter $a$ to characterize the relative strength between advection and vortex stretching, just like the gCLM model. Both inviscid and viscous cases are considered. We also prove the finite time singularity formation of the original inviscid Hou-Li model ($a=1$ and $\nu=0$) with $C^\alpha$ initial data. Inspired by the recent work of Chen \cite{chen2021regularity} for the De Gregorio model, we consider the case of $a < 1$ and treat $1-a$ as a small parameter. For the $C^\alpha$ initial data, we consider the original Hou-Li model with $a=1$ and $1-\alpha$ small. 
By using the dynamic rescaling formulation and analyzing the stability of the linearized operator around an approximate steady state of the original Hou-Li model ($a=1$), we prove finite time self-similar blowup. 

We follow a general strategy that we have established in our previous works \cite{chen2021finite,chen2019finite2}.
Establishing linear stability of the approximate steady state is the most crucial step in our blowup analysis. To obtain sharp estimates for the nonlocal terms, we carry out an exact computation for the low-frequency Fourier modes and {extract damping in leading order estimates}  for the high-frequency modes using singularly weighted norms {in the energy estimates}.  The blowup analysis for the viscous model is more subtle since the viscous terms do not provide damping and produce some bad terms if we use a singularly weighted norm. We establish the blowup analysis for the viscous model by carefully designing an energy norm that combines a singularly weighted energy norm and a {sum of} high-order Sobolev norms.

\subsection{Problem setting}
In \cite{hou2008dynamic},
Hou-Li introduced the following reformulation of the axisymmetric Navier-Stokes equation:
\begin{eqnarray}
\label{3d}
&&u_{1, t}+u^{r} u_{1, r}+u^{z} u_{1, z} =2 u_{1} \psi_{1, z}+\nu\Delta u_1\,, \\
&&\omega_{1, t}+u^{r} \omega_{1, r}+u^{z} \omega_{1, z} =\left(u_{1}^{2}\right)_{z}+\nu\Delta \omega_1\,, \\
&&-\left[\partial_{r}^{2}+(3 / r) \partial_{r}+\partial_{z}^{2}\right] \psi_{1} =\omega_{1}\,,
%&& u^{r} =-r \psi_{1, z}\,, \quad u^{z}=2 \psi_{1}+r \psi_{1, r}\,.
\end{eqnarray}
where $u_1 = u^\theta/r,\; \omega_1 = \omega^\theta,\; \psi_1 = \psi^\theta/r, $ and $u^\theta$, $\omega^\theta$, and $\psi^\theta$ are the angular velocity, angular vorticity, and angular stream function, respectively. 
%In \cite{liuwang2006}, Liu and Wang showed that  $u^\theta$, $\omega^\theta$ and $\psi^\theta$ must be an odd function of $r$ if the solution of the axisymmetric Euler or Navier-Stokes equations is smooth. Based on this observation, $u^{r} =-r \psi_{1, z}$, $u^{z}=2 \psi_{1}+r \psi_{1, r} $, and we can recover the Euler equation if $\nu=0$. 
By the well-known Caffarelli-Kohn-Nirenberg partial regularity result \cite{caffarelli1982partial}, the axisymmetric Navier-Stokes equations can develop a finite time singularity only along the symmetry axis $r=0$. To study the potential singularity or global regularity of the axisymmetric Navier-Stokes equations,
Hou-Li \cite{hou2008dynamic} proposed the following 1D model along the symmetry axis $r=0$:
\begin{equation}
\label{1d}
\begin{aligned}
u_{1, t}+2 \psi_1 u_{1, z} &=2 \psi_{1, z} u_1 +\nu u_{1,zz}\,, \\
\omega_{1, t}+2 \psi_1 \omega_{1, z} &=\left(u_1^{2}\right)_{z}+\nu \omega_{1,zz}\,, \\
- \psi_{1,zz} &=\omega_1\,.
\end{aligned}
\end{equation}
Such a reduction is exact in the sense that if ($\omega_1$, $u_1$, $\psi_1$) is an exact solution of the 1D model, we can obtain an exact solution of the 3D Navier-Stokes equations by using a constant extension in $r$. This corresponds to the case when the physical quantities $u^{\theta}=ru_1$, $\omega^{\theta}=r\omega_1$ are linear in $r$.
We assume that the solutions are periodic in $z$ on $[0,2\pi]$.
We already know from the original Hou-Li paper that this system is well-posed for $C^m$ initial data with $m \geq 1$.  In \cite{hou2008dynamic}, the authors also used the well-posedness of the Hou-Li model to construct globally smooth solutions to the 3D equations with large dynamic growth.

In two recent papers by Hou \cite{hou2022potential, hou2022potentially}, the author presented new numerical evidence that the 3D axisymmetric Euler and Navier-Stokes equations develop potential singular solutions at the origin. This new blowup scenario is very different from the Hou-Luo blowup scenario, which occurs on the boundary. In this computation, the author observed that the axial velocity $u^{z}=2 \psi_{1}+r \psi_{1, r}$ near the maximal point of $u_1$ is significantly weaker than $2\psi_1$. This is due to the fact that $\psi_1$ reaches the maximum at a position $r=r_\psi$ that is smaller than the position $r=r_u$ in which  $u_1$ achieves its maximum, i.e. $r_\psi < r_u$. Therefore $\psi_{1, r}$ is negative near the maximal position of $u_1$. Thus the axial velocity $u^z$ is actually weaker than $2 \psi_1$, which corresponds to $u^z |_{r=0}$. Thus, the original Hou-Li model along $r=0$ does not capture this subtle phenomenon, which is three-dimensional in nature. To gain some understanding of this potentially singular behavior, we introduce the following 1D weak advection model.
\begin{equation}
\label{1dwp}
\begin{aligned}
u_{ t}+2 a\psi u_{ z} &=2 u \psi_{ z}+\nu u_{zz}\,, \\
\omega_{ t}+2 a\psi\omega_{ z} &=\left(u^{2}\right)_{z}+\nu \omega_{zz}\,, \\
- \psi_{zz} &=\omega\,,
\end{aligned}
\end{equation}
where $a$ is a parameter that measures the relative strength of advection in the Hou-Li model.
\begin{remark}
For simplicity, we drop the subscript $1$ in the above weak advection model.
The proposed model \eqref{1dwp} in the inviscid case $\nu=0$  resembles the generalized Constantin-Lax-Majda model (gCLM) \cite{okamoto2008generalization} 
$$\omega_t+au\omega_x=u_x\omega\,,\qquad u_x=H\omega\,,$$
        where
     $$
H \omega(x)=\frac{1}{\pi} p.v. \int_{\mathbb{R}} \frac{\omega(y)}{x-y} \mathrm{~d} y
$$
is the Hilbert transform. They share similar structures of competition between advection and vortex stretching. The case when $a=1$ corresponds to the De Gregorio (DG) model.
We obtain an explicit steady-state to the inviscid Hou-Li model \eqref{1d} $(\omega,u,\psi)=(\sin x,\sin x, \sin x)$, similar to the steady state $(\omega,u)=(-\sin x,\sin x)$ of the DG model on $S^1$. Many of the results we present in this chapter have analogies for the gCLM model; see in particular \cite{chen2021regularity,chen2021slightly}.
\end{remark}

\subsection{Main results}
We summarize the main results of the chapter below and devote the subsequent sections to proving these results.
Our first result is on the finite-time blowup of the weak inviscid advection model; for its proof see Section \ref{linearest}  and \ref{nonlinearest}. 
\begin{theorem}
\label{t1}
For the weak advection model \eqref{1dwp} in the inviscid case $\nu=0$, there exists a constant $\delta>0$ such that for $a\in(1-\delta,1)$, the weak advection model \eqref{1dwp} develops a finite time singularity for some $C^{\infty}$ initial data. Moreover, there exists a self-similar profile $({\omega}_{\infty},{\omega}_{\infty},{\omega}_{\infty})$ corresponding to a blowup that is neither expanding nor focusing. More precisely, the blowup solution to \eqref{1dwp} has the form
$$\omega(x,t)=\frac{1}{1+c_{u,\infty}t}{\omega}_{\infty}\,,u(x,t)=\frac{1}{1+c_{u,\infty}t}u_{\infty}\,,\psi(x,t)=\frac{1}{1+c_{u,\infty}t}{\psi}_{\infty}\,,$$
for some negative constant $c_{u,\infty}$ with a blowup time given by $T= \frac{-1}{c_{u,\infty}}$.
\end{theorem}
\begin{remark}
    Such self-similar blowup that is neither expanding nor focusing is observed numerically for $a\in[0.6,0.9]$. See also a similar phenomenon observed for the gCLM model in \cite{lushnikov2021collapse} for $a\in[0.68,0.95]$. The blowup result for the gCLM model has been proved in \cite{chen2021slightly} for $a$ sufficiently close to $1$. We remark that for $a$ very close to $1$, since {we can show that $c_{u,\infty}=2(a-1)+o(a-1)$}, the blowup time becomes very large due to the very small coefficient {$1-a$} in the vortex stretching term {which slightly dominates the advection term}. It would be extremely difficult to compute such singularity numerically since it takes an extremely long time for the singularity to develop. For $a$ below a critical value $a_0$, i.e. $a < a_0$, we observe that the weak advection Hou-Li model develops a focusing singularity.
\end{remark}
The second result is on the blowup of the original Hou-Li model with $C^\alpha$ initial data; for its proof see Section \ref{sec hol}. In \cite{elgindi2020effects}, the authors made an important observation that advection can be weakened by $C^\alpha$ data. Intuitively if $u =O(x^\alpha)$ in the origin, since $\psi$ is $C^2$, we have that $\psi u_x\approx \alpha\psi_x u$ near the origin, the vortex stretching term is stronger than the advection term if $\alpha < 1$. See \cite{chen2021finite,chen2021regularity} on results of  blowup of the DG model with H\"older continuous data.
{
\begin{theorem}
\label{t2}
Consider the Hou-Li model \eqref{1d} in the inviscid case $\nu=0$. There exists a constant $\delta_0>0$ such that for $\alpha\in(1-\delta_0,1)$, \eqref{1d} develops a finite time singularity for some $C^{\alpha}$ initial data. Moreover, there exists a $C^{\alpha}$ self-similar profile corresponding to a blowup that is neither expanding nor focusing, similar to the setting in Theorem \ref{t1}.
\end{theorem}
\begin{remark}
    This theorem establishes blowups of type $C^{\alpha}$ for any $\alpha$ close to 1, which of course implies blowups in less regular classes since $C^{\alpha}\subset C^{\alpha_1}$ for $\alpha_1<\alpha$. The regularity of the profile determines the speed of the blowup since our constructed $C^{\alpha}$ profile has blowup time $T=O(-1/(2(\alpha-1)))$. We remark that however, we do not have blowup for data intrinsically in a low regularity class $C^{\epsilon}$ for $\epsilon$ close to $0$; that is, data that is $C^{\epsilon}$ but not in any higher $C^\alpha$ classes. We conjecture that such blowup might be focusing, which is beyond the scope of this chapter.
\end{remark}
\begin{remark}
    The above two theorems imply that the result of the wellposedness in \cite{hou2008dynamic} of the Hou-Li model for $C^1$ initial data is sharp. As long as the advection is weakened or slightly less smooth data is allowed, we would have a self-similar blowup.
\end{remark}}
The third result is on the finite-time blowup of the weak advection model with viscosity. The dynamic rescaling formulation implies that the viscous terms are asymptotically small. Thus, we can build on Theorem \ref{t1} to establish Theorem \ref{t3}. {We remark that there is no exact self-similar profile due to the viscous term.}
We will provide more details of the blowup analysis for the viscous case in Section \ref{sec5}.
\begin{theorem}
\label{t3}
Consider the weak advection model \eqref{1dwp} with viscosity. There exists a constant $\delta_1>0$ such that for $a\in(1-\delta_1,1)$, the weak advection model \eqref{1dwp} develops a finite time singularity for some $C^{\infty}$ initial data.
\end{theorem}

We use the framework of the dynamic rescaling formulation to establish the blowups. This formulation was first introduced by McLaughlin, Papanicolaou, and co-workers in their study of self-similar blowup of the nonlinear Schr\"odinger equation  \cite{mclaughlin1986focusing,landman1988rate}.  This formulation was later developed into an effective modulation technique, which has been applied to analyze the singularity formation for the nonlinear Schr\"odinger equation \cite{kenig2006global,merle2005blow}, compressible Euler equations \cite{buckmaster2019formation},  the nonlinear heat equation \cite{merle1997stability}, the generalized KdV equation \cite{martel2014blow}, and other dispersive problems. Recently this approach has been applied to prove singularity in various gCLM models \cite{chen2021finite,chen2021regularity,chen2021slightly} and in Euler equations \cite{ elgindi2021finite, chen2019finite2, chen2022stable}. Our blowup analysis consists of several steps.  First, we use the dynamic rescaling formulation to link a self-similar singularity to the (stable) steady state of the dynamic rescaling formulation. Secondly, we identify, either analytically or numerically, an approximate steady state to the dynamic rescaling formulation. Thirdly, we perform energy estimates using a singularly weighted norm to establish linear and nonlinear stability of the approximate steady state. Finally, we establish exponential convergence to the steady state in the rescaled time. 

The crucial ingredient of the framework is the linear stability of the approximate steady state, and we usually adopt a singularly weighted $L^2$-based estimate. {To avoid an overestimate in the linear stability analysis, we expand the perturbation in terms of the orthonormal basis with respect to the weight $L^2$ norm and reduce the linear stability estimate into an estimate of a quadratic form for the Fourier coefficients. We further extract the damping effect of the linearized operator by establishing a lower bound on the eigenvalues of an infinite-dimensional symmetric matrix. We prove the positive-definiteness of this quadratic form by performing an exact computation of the eigenvalues of a small number of Fourier modes with rigorous computer-assisted bounds}, and treat the high-frequency Fourier modes as a small perturbation by using the asymptotic decay of {the quadratic form in} the high-frequency Fourier coefficients.

\subsection{Organization of the chapter and notations}
In Section \ref{linearest}, we introduce our dynamic rescaling formulation and link the blowup of the physical equation to the steady state of the dynamic rescaling formulation. The linear stability of the approximate steady state is established.
In Section \ref{nonlinearest}, we establish the nonlinear stability of the approximate steady state and the exponential convergence to the steady state, which proves Theorem \ref{t1} and the blowup for the weak advection model. In Section \ref{sec hol}, we prove Theorem \ref{t2} and establish blowup for the original model with H\"older continuous data. In Section \ref{sec5}, we prove Theorem \ref{t3} by designing a special energy norm to estimate the viscous terms. We provide the crucial linear damping estimates in  the Appendix using computer assistance.

Throughout the article, we use $(\cdot,\cdot)$ to denote the inner product on $S^1$: $(f,g)=\int_{-\pi}^{\pi}fg$. We use $C$ to denote absolute constants, which may vary from line to line, and we use $C(k)$  to denote some constant that may depend on specific parameters $k$  we choose. We use $A\lesssim B$ for positive $B$ to denote that there exists an absolute constant $C>0$ such that $A\leq CB$.
\section{Dynamic Rescaling Formulation and Linear Estimates}
\label{linearest}
\subsection{Dynamic rescaling formulation}
\label{drf sec}
We will establish the singularity formation of the weak advection model by using the dynamic rescaling formulation. We first consider the inviscid case with $\nu=0$. For solutions to the system \eqref{1dwp}, we introduce $$\tilde{u}(x, \tau)=C_{u}(\tau)  u( x, t(\tau))\,,\quad \tilde{\omega}(x, \tau)=C_{u}(\tau)  \omega( x, t(\tau))\,,\quad \tilde{\psi}(x, \tau)=C_{u}(\tau)  \psi(x, t(\tau))\,,$$ where $$
C_{u}(\tau)=\exp \left(\int_{0}^{\tau} c_{u}(s) d s\right)\,, \quad t(\tau)=\int_{0}^{\tau} C_{u}(s) d s\,.$$
We can show that the rescaled variables solve the following dynamic rescaling equation \begin{equation}
\label{1drf}
\begin{aligned}
\tilde{u}_{\tau}+2a \tilde{\psi}\tilde{u}_{ x} &=2 \tilde{u} \tilde{\psi}_{ x}+c_u \tilde{u}\,, \\
\tilde{\omega}_{ \tau}+2a \tilde{\psi}\tilde{\omega}_{ x} &=\left(\tilde{u}^{2}\right)_{x}+c_{u}\tilde{\omega}\,, \\
-\tilde{\psi}_{xx} &=\tilde{\omega}\,.
\end{aligned}
\end{equation}

\begin{remark}
We do not rescale the spatial variable $x$, since we are interested in a blowup solution that is neither focusing nor expanding within a fixed period. The scaling factors for $u$, $\omega$, $\psi$ are thus the same.
\end{remark}
When we establish a self-similar blowup, it suffices to show the dynamic stability of equation \eqref{1drf} close to an approximate steady state  with scaling parameter $c_u<-\epsilon<0$ uniformly in time for a small constant $\epsilon$; see also \cite{chen2021finite}. In fact, it's easy to see that if $(\tilde{u},\tilde{\omega},\tilde{\psi},c_u)$ converges to a steady-state $({u}_{\infty},{\omega}_{\infty},{\psi}_{\infty},c_{u,\infty})$ of \eqref{1drf}, then $$\omega(x,t)=\frac{1}{1+c_{u,\infty}t}{\omega}_{\infty}\,,u(x,t)=\frac{1}{1+c_{u,\infty}t}u_{\infty}\,,\psi(x,t)=\frac{1}{1+c_{u,\infty}t}{\psi}_{\infty}\,,$$
is a self-similar solution of \eqref{1dwp}. %Therefore we can identify the steady-state solution of \eqref{1drf} with t
    
From now on, we will primarily work in the dynamic rescaling formulation and use the notations that $\tilde{u} =\bar{u}+\hat{u}$, where $\bar{u}$ is the approximate steady state that we perturb around and $\hat{u}$ is the perturbation. Notations for variables $\tilde{\omega}$ and $\tilde{\psi}$ are similar. 
\subsection{Equations governing the perturbation}
We use the steady state corresponding to the case of $a=1$ to construct an approximate steady state for \eqref{1drf}.
$$\bar{\omega}=\sin x\,,\quad\bar{u}=\sin x\,,\quad\bar{\psi}=\sin x\,,\quad \bar{c}_u=2(a-1)\bar{\psi}_x(0)=2(a-1)\,.$$
We consider odd perturbations $\hat{u}$, $\hat{\omega}$, $\hat{\psi}$. The parities are preserved in time by equation \eqref{1drf}.  We use the normalization condition as $c_u=2(a-1)\hat{\psi}_x(0)$. This normalization ensures that $\bar{u}_x(0)+\hat{u}_x(0)$ is conserved in time. 

To simplify our presentation, we will drop the $\hat{}$ in the perturbation $\hat{u}$ and use $u$ for $\hat{u}$, $\omega$ for $\hat{\omega}$, $\psi$ for $\hat{\psi}$.   Now the perturbations satisfy the following system
\begin{equation}
\label{1drfr}
\begin{aligned}
{u}_{ \tau} &=-2a \sin x u_x-2a \cos x \psi+ 2 u \cos x+2\sin x \psi_x+\bar{c}_u u+ c_u\bar{u}+N_1+F_1\,, \\
{\omega}_{ \tau}&=-2a \sin x \omega_x-2a \cos x\psi +2 u \cos x+2\sin x u_x+\bar{c}_u \omega+ c_u\bar{\omega}+N_2+F_2\,, \\
-{\psi}_{xx} &={\omega}\,,
\end{aligned}
\end{equation}
where $N_1, \; N_2$ and $F_1, \; F_2$ are the nonlinear terms and error terms defined below:
$$N_1=(c_u+2\psi_x)u-2a\psi u_x\,,\quad N_2=c_u\omega+2uu_x-2a\psi \omega_x\,,$$
$$F_1=(\bar{c}_u+2\bar{\psi}_x)\bar{u}-2a\bar{\psi} \bar{u}_x=2(a-1)\sin x(1-\cos x)\,,\quad F_2=\bar{c}_u\bar{\omega}+2\bar{u}\bar{u}_x-2a\bar{\psi} \bar{\omega}_x=F_1\,.$$
We further organize the system \eqref{1drfr} into the main linearized term and a smaller term containing a factor of $a-1$:
\begin{equation}
\label{1drc}
\begin{aligned}
{u}_{ \tau} &=L_1+(a-1)L'_1+N_1+F_1\,, \\
{\omega}_{ \tau}&=L_2+(a-1)L'_2+N_2+F_2\,, \\
-{\psi}_{xx} &={\omega}\,.
\end{aligned}
\end{equation}
where
$$L_1=-2 \sin x u_x-2 \cos x \psi+ 2 u \cos x+2\sin x \psi_x\,,$$ $$ L'_1=-2\sin x u_x-2 \cos x \psi+2u+2\psi_x(0)\sin x\,,$$ $$L_2=-2 \sin x \omega_x-2 \cos x\psi +2 u \cos{x}+2\sin x u_x\,,$$ $$ L'_2=-2 \sin x \omega_x-2 \cos x\psi+2\omega+2\psi_x(0)\sin x \,.$$

To show that the dynamic rescaling equation is stable and converges to a steady state, we will perform a weighted-$L^2$ estimate with a singular weight $\rho$ and a weighted $L^2$ norm$$\rho=\frac{1}{2\pi(1-\cos x)}\,,\quad \|f\|_{\rho}=(f^2,\rho)^{1/2}\,.$$ For initial perturbation with $u_x(0,0)=0$, we have $u_x(0,\tau)=0$ for all time and 
$$E^2(\tau)=\frac{1}{2}((u_x^2,\rho)+(\omega^2,\rho))\; ,$$ 
is well-defined. We will first show that the dominant parts $L_1$ and $L_2$ provide damping.  The following lemma is crucial and motivates the choice of $\rho$.
\begin{lemma}
\label{dampl}We have the following identity
$$(\sin x f_x,f\rho)=\frac{1}{2}(f^2,\rho)\,,$$
which can be verified directly by using integration by parts.
\end{lemma}
\subsection{Stability of the main parts in the linearized equation}
 \label{stab main}

In order to extract the maximal amount of damping, we will expand the perturbed solution in the Fourier series and perform exact calculations. We first explore the orthonormal basis in $L^2(\rho)$.
\begin{lemma}
For the space of odd periodic functions on $[0,2\pi]$, we describe a complete set of orthonormal basis $\{o^k\}$ in $L^2(\rho)$  $$o^k=\sin(kx)-\sin((k-1)x)\,,\quad k=1,2,\cdots\,.$$
Similarly, for the space of even periodic functions that lie in $L^2(\rho)$, we describe a complete set of orthonormal basis $\{e^k\}$
$$e^k=\cos(kx)-\cos((k+1)x)\,,\quad k=0,1,\cdots\,.$$
\end{lemma}
Now we are now ready to establish linear stability. 
\begin{proposition}\label{prop d} The following energy estimate holds for the leading linearized operators
$$dE_1\coloneqq((L_1)_x,u_x\rho)+(L_2,\omega\rho)\leq -0.16[(u_x,u_x\rho)+(\omega,\omega\rho)]\,.$$
\end{proposition}
\begin{proof}
Consider the expansion of $\omega$, $u_x$, and $u$ in the orthonormal basis
$$\omega=\sum_{k\geq 1} a_k o^k\,, \quad u=\sum_{k\geq 1} b_k o^k\,, \quad u_x=\sum_{k\geq1} c_k e^k\,.$$
Note the summation index for $u_x$ satisfies $k \geq 1$ since we can easily see that $$ (u_x,e^0\rho)=\frac{1}{2\pi}\int_0^{2\pi} u_x=0\,.$$We first express $b_k$ in terms of $c_k$. If we insert the expression of the basis into $u$, take derivative and compare the coefficients with the expansion of $u_x$, we get $$c_i=\sum_{k=1}^i b_k-ib_{i+1}\,.$$
Therefore we can solve $$b_{i+1}=b_1-\sum_{k=1}^{i-1}\frac{c_k}{k(k+1)}-\frac{c_{i}}{i}\,.$$
Moreover, we have the compatibility condition $u_x(0)=\sum_{k \geq 0} b_k=0$. Therefore we can solve $b_1$ and obtain 
\begin{equation}\label{bc}
    b_{i}=\sum_{k\geq i}\frac{c_k}{k(k+1)}-\frac{c_{i-1}}{i}\,,
\end{equation}
where we define $c_0=0$.

Now we write out the terms explicitly using the expansions
\begin{equation*}
\begin{aligned}
dE_1&=2(-u\sin x-u_{xx}\sin x,u_x\rho)+2(\sin x\psi+\sin x\psi_{xx},u_x\rho)\\&+2(- \sin x \omega_x- \cos x\psi,\omega \rho) +2 (u \cos{x}+\sin x u_x,\omega \rho)\\&=-[(u_x,u_x\rho)+(\omega,\omega\rho)+(u,u\rho)]+2[-( \cos x\psi,\omega \rho)+(\sin x\psi,u_x\rho)+ (u \cos{x},\omega \rho)]\,.    
\end{aligned}    
\end{equation*}

Here we use the crucial Lemma \ref{dampl} to extract damping on the local terms and the Biot-Savart law $-\psi_{xx}=\omega$ to cancel the effect of the nonlocal terms {$\sin x\psi_{xx}$ in $(L_1)_x$ and $\sin xu_{x}$ in $L_2$}. Next, we calculate the remaining nonlocal terms explicitly.
$$-2( \cos x\psi,\omega \rho)=(2(1-\cos x)\psi,\omega \rho)-2(\psi,\omega \rho)=\frac{1}{\pi}(\psi,\omega)-2(\psi,\omega \rho)\,.$$
We express $\omega$ and $\psi$ both in terms of orthonormal basis $o^k$ corresponding to the weighted norm and the canonical basis $\sin(kx)$ corresponding to the (normalized by $\frac{1}{\pi}$) $L^2$ norm.
$$\omega=\sum_{k \geq 1} (a_k-a_{k+1}) \sin(kx)\,,$$
where we denote $a_0=0$. Therefore 
$$\psi=\sum_{k \geq 1} \frac{a_k-a_{k+1}}{k^2} \sin(kx)\,.$$
Furthermore, we collect
$$\psi=\sum_{k\geq 1} \sum_{j\geq k}\frac{a_j-a_{j+1}}{j^2} o^k\,.$$
Therefore we can compute explicitly that 
$$-2( \cos x\psi,\omega \rho)=\sum_{k \geq 1} \frac{(a_k-a_{k+1})^2}{k^2}-2\sum_{k\geq 1} a_k\sum_{j\geq k}\frac{a_j-a_{j+1}}{j^2}\,.$$

We use integration by parts similar to Lemma \ref{dampl} to obtain 
$$2[(\sin x\psi,u_x\rho)+ (u \cos{x},\omega \rho)]=2(\cos{x}\omega+\psi-\sin x\psi_x,u\rho)\coloneqq2(T_u,u\rho)\,.$$
We further have $$\begin{aligned}
T_u&=\sum_{k \geq 1} (a_k-a_{k+1})[\frac{k+1}{2k}\sin((k-1)x)+\frac{k-1}{2k}\sin((k+1)x)+\frac{1}{k^2}\sin(kx)]\\&=\sum_{k \geq 1} \sin(kx) [\frac{k+2}{2(k+1)}(a_{k+1}-a_{k+2})+\frac{a_{k}-a_{k+1}}{k^2}+\frac{k-2}{2(k-1)}(a_{k-1}-a_{k})]\\&=\sum_{k\geq 1}[\frac{k-2}{2(k-1)}a_{k-1}+(\frac{1}{2k(k-1)}+\frac{1}{k^2})a_{k}+(\frac{k+1}{2k}+\frac{1}{(k+1)^2}-\frac{1}{k^2})a_{k+1}\\&+\sum_{j>k+1}(\frac{1}{j^2}-\frac{1}{(j-1)^2})a_j]o^k\,,  
\end{aligned}$$
where the terms involving $\frac{1}{k-1}$ in the summand is regarded as $0$ for $k=1$. Therefore we collect explicitly that 
$$\begin{aligned}
dE_1&=\sum_{k\geq 1}\{ -(a_k^2+b_k^2+c_k^2)+ \frac{(a_k-a_{k+1})^2}{k^2}-2 a_k\sum_{j\geq k}\frac{a_j-a_{j+1}}{j^2}+2b_k[\frac{k-2}{2(k-1)}a_{k-1}\\&+(\frac{1}{2k(k-1)}+\frac{1}{k^2})a_{k}+(\frac{k+1}{2k}+\frac{1}{(k+1)^2}-\frac{1}{k^2})a_{k+1}+\sum_{j>k+1}(\frac{1}{j^2}-\frac{1}{(j-1)^2})a_j]\}\,.\end{aligned}$$
Substituting \eqref{bc} into the above and we can simplify
$$
\begin{aligned}
dE_1&=-\sum_{k \geq 1}\{ a_k^2(1+\frac{1}{k^2}-\frac{1}{(k-1)^2})+c_k^2(1+\frac{1}{k(k+1)})+ 2a_k a_{k+1}\frac{1}{(k+1)^2}\\&+2 a_k \sum_{j>k+1}a_j(\frac{1}{j^2}-\frac{1}{(j-1)^2})+2a_kc_k\frac{1+2k-k^2}{2k^2(k+1)}+2a_{k+1}c_k\frac{k^2-k-1}{2k^2(k+1)^2}\\&-2a_{k+2}c_k\frac{k+2}{2(k+1)^2}+\sum_{j>k}2a_k c_j\frac{1}{j(j+1)}\}\,.\end{aligned}$$
After this explicit computation, we notice that the damping estimate in Proposition \ref{prop d} can be cast into an estimate of a quadratic form; see \eqref{dample}, which is equivalent to a lower bound on the eigenvalues of an infinite-dimensional symmetric matrix. 
\begin{equation}
    \label{dample}\begin{aligned}
 F(a,c)&\coloneqq \sum_{k\geq1}\{ a_k^2(0.84+\frac{1}{k^2}-\frac{1}{(k-1)^2})+c_k^2(0.84+\frac{1}{k(k+1)})+ 2a_k a_{k+1}\frac{1}{(k+1)^2}\\&+2 a_k \sum_{j>k+1}a_j(\frac{1}{j^2}-\frac{1}{(j-1)^2})+2a_kc_k\frac{1+2k-k^2}{2k^2(k+1)}+2a_{k+1}c_k\frac{k^2-k-1}{2k^2(k+1)^2}\\&-2a_{k+2}c_k\frac{k+2}{2(k+1)^2}+\sum_{j>k}2a_k c_j\frac{1}{j(j+1)}\}\geq0\,.\end{aligned}\end{equation}

We notice that the entries decay fast. {Therefore the strategy to prove \eqref{dample} is to combine a computer-assisted estimate of the eigenvalues of its finite truncation with a decay estimate of the remaining part. We will defer the proof of \eqref{dample} to the Appendix, see Lemma \ref{cap} and the proof}. Thereby we conclude the linear estimate.
\end{proof}

\section{Nonlinear Estimates and Convergence to Self-similar Profile}
\label{nonlinearest}
\subsection{Nonlinear stability}
By Proposition \ref{prop d} and equation \eqref{1drc}, we have \begin{equation}
\label{nonlinear}
    \begin{aligned}\frac{1}{2}\frac{d}{d\tau}E^2(\tau)&\leq -0.16E^2(\tau)+(a-1)[((L'_1)_x,u_x\rho)+(L'_2,\omega\rho)]\\&+((N_1)_x,u_x\rho)+(N_2,\omega\rho)+((F_1)_x,u_x\rho)+(F_2,\omega\rho)\,.\end{aligned}
\end{equation}
We first provide some estimates about the weighted $L^2$ norm and $L^{\infty}$ norm of some lower-order terms.
\begin{lemma}
\label{lotb}
The following estimates hold
\begin{enumerate}
    \item Weighted $L^2$ norm: $$\|\psi\|_{\rho},\|\psi_x-\psi_x(0)\|_{\rho},\|u\|_{\rho}\lesssim E\,.$$
    \item $L^{\infty}$ norm: $$\|\psi_x\|_{\infty},\|\frac{\psi}{\sin x}\|_{\infty},\|u\|_{\infty}\lesssim E\,.$$
\end{enumerate}
   
\end{lemma}
\begin{proof}
    For (1), we use the setting of the Fourier series approach as in the proof of Proposition \ref{prop d} and pick up the notation there.
    $$\begin{aligned}
    \|\psi_x-\psi_x(0)\|_{\rho}^2&=\sum_{k \geq 1} (\sum_{j\geq k}\frac{a_j-a_{j+1}}{j})^2\leq\sum_{k \geq 1} \sum_{j\geq k}a^2_j(\frac{1}{k^2}+\sum_{j>k}(\frac{1}{j}-\frac{1}{j+1})^2)\\&\lesssim\sum_{j \geq 1} a_j^2\sum_{k \geq 1} \frac{1}{k^2}\lesssim\|\omega\|_{\rho}^2\,,\end{aligned}$$
    where we have used the Cauchy-Schwarz inequality. We can similarly estimate $\|\psi\|_{\rho}$. Then we get
    $$\|u\|_{\rho}^2=\sum_{j \geq 1} b_j^2=\sum_{j \geq 1} \frac{1}{j(j+1)}c_j^2\leq\sum_{j \geq 1} c_j^2=\|u_x\|_{\rho}^2\,.$$
    For (2), we first compute using Fourier series similar to (1) $$\psi_x(0)=\sum_{j \geq 1}\frac{a_j-a_{j+1}}{j}\lesssim\|\omega\|_{\rho}\,.$$
    Next, we estimate $$\|\psi_x-\psi_x(0)\|_{\infty}\lesssim\|\psi_{xx}\|_{1}\lesssim\|\omega\|_{\rho}\,.$$
Similarly, we obtain the estimate for $\|u\|_{\infty}$.
    For $\|\frac{\psi}{\sin x}\|_{\infty}$, since $\psi$ is odd and periodic, we have $\psi(\pi)=\psi(0)=0$ and only need to estimate this norm in $[0,\pi]$. Since $\sin x\geq \frac{2}{\pi}\min\{x,\pi-x\}$ in $[0,\pi]$,
    we have $\|\frac{\psi}{\sin x}\|_{\infty}\lesssim\|\psi_x\|_{\infty}$ by Lagrange's mean value theorem.
\end{proof}
Combined with the damping in Lemma \ref{dampl}, we further obtain 
$$\begin{aligned}((L'_1)_x,u_x\rho)&=2(-\cos xu_x-\sin xu_{xx}+\sin x\psi-\cos x\psi_x+u_x+\psi_x(0)\cos x,u_x\rho)\\&\lesssim E^2+(\|\psi\|_{\rho}+\|\psi_x-\psi_x(0)\|_{\rho})E\lesssim E^2\,,\end{aligned}$$
$$(L'_2,\omega\rho)=2(-\sin x\omega_x-\cos x\psi+\omega+\psi_x(0)\sin x,\omega\rho)\lesssim E^2+(\|\psi\|_{\rho}+|\psi_x(0)|)E\lesssim E^2\,.$$
$$((F_1)_x,u_x\rho)\lesssim |a-1|E\,,\quad(F_2,\omega\rho)\lesssim |a-1|E\,,$$
$$\begin{aligned}((N_1)_x,u_x\rho)&=2(-\omega u+(1-a)(\psi_x-\psi_x(0))u_{x}-a\psi u_{xx},u_x\rho)\\&\lesssim E^2(\|\psi_x\|_{\infty}+\|u\|_{\infty})+|(u_x^2,(\psi\rho)_x)|\\&\lesssim E^2(\|\psi_x\|_{\infty}+\|u\|_{\infty}+\|\frac{\psi}{\sin x}\|_{\infty})\lesssim E^3\,,\end{aligned}$$
$$(N_2,\omega\rho)=2((a-1)\psi_x(0)\omega+uu_x-a\psi\omega_x,\omega\rho)\lesssim E^2(\|\psi_x\|_{\infty}+\|u\|_{\infty}+\|\frac{\psi}{\sin x}\|_{\infty})\lesssim E^3\,.$$
Therefore we have \begin{equation}
\label{nonl-collected}
    \frac{d}{d\tau}E(\tau)\leq -(0.16-C|a-1|)E+C|a-1|+CE^2\,.
\end{equation}
We can perform the standard bootstrap argument to show that there exist absolute constants $\delta,C>0$ such that if $|a-1|<\delta$ and $E(0)<C|a-1|$, then we have $E(\tau)<C|a-1|$ for all time. In particular $c_u=O(|a-1|^2)$ and $c_u+\bar{c}_u<0$. Therefore we prove that the solution blows up in finite time.
\subsection{Estimates using a higher-order Sobolev norm}
\label{h2sec}
In order to establish convergence of the solution to a steady state, we need to estimate weighted norms of $u_t$ and $\omega_t$. As was pointed out in \cite{chen2021slightly}, we need to provide stability estimates of the equation in higher-order Sobolev norms to close the estimate. In particular, we choose $$K^2(\tau)=\|D_xu_x\|_{\rho}^2+\|D_x\omega\|_{\rho}^2\,,$$
where we denote $D_x$ to be the operator $\sin x\partial x$.

\begin{remark}
This choice of weighted  norms is again motivated by the local linear damping estimates. We recall that the leading order terms of the local terms in the linearized operators $(L_1)_x$, $L_2$ are $-2D_xu_x$ and $-2D_xw$, and we have $2(D_x f,f\rho)=(f,f\rho)$. Therefore in this new weighted norm, the combined terms would again give damping \begin{equation}\label{highD}
    (-2D_x D_x u_x,D_x u_x\rho)+(-2D_xD_x w,D_x w\rho)=-K^2\,.
\end{equation}
\end{remark}
We now obtain
$$\begin{aligned}\frac{1}{2}\frac{d}{d\tau}K^2(\tau)&\leq (D_x(L_1)_x,D_xu_x\rho)+(D_xL_2,D_x\omega\rho)+(D_x(N_1)_x,D_xu_x\rho)\\&+(D_xN_2,D_x\omega\rho)+(a-1)[(D_x(L'_1)_x,D_xu_x\rho)+(D_xL'_2,D_x\omega\rho)]\\&+(D_x(F_1)_x,D_xu_x\rho)+(D_xF_2,D_x\omega\rho)\,.\end{aligned}$$
We will denote the terms that have $\|\cdot\|_{\rho}$ norm bounded by $E$ as $\textit{l.o.t.}$. The bound
$$\|D_x[fg]\|_{\rho}\lesssim (\|f_x\|_{2}+\|f\|_{2})\quad \text{for}\, g=1,\cos x,\sin x$$
combined with the oddness of $\psi$ and $u$ would imply that $D_x[fg]$ is $\textit{l.o.t.}$ for $f=\psi,\psi_x,u$ and $g=\sin x,\cos x,1$.
Therefore combined with \eqref{highD}, we have the following estimate for the main term $$dK_1\coloneqq(D_x(L_1)_x,D_xu_x\rho)+(D_xL_2,D_x\omega\rho)\,,$$ $$\begin{aligned}dK_1&\leq -K^2-2(D_x[\sin x \omega],D_xu_x\rho)+2(D_xD_xu,D_x\omega\rho)+CEK\\&=-K^2-(\sin 2x\omega,D_xu_x\rho)+(\sin 2x u_x,D_x\omega\rho)+CEK\\&\leq-K^2+CEK\,,\end{aligned}$$
where we have again used a crucial cancellation in the equality, similar to that of $dE_1$ in Subsection \ref{stab main}. We estimate the rest of the terms similar to the nonlinear stability estimates in \eqref{nonlinear}.
$$(D_x(L'_1)_x,D_xu_x\rho)+(D_xL'_2,D_x\omega\rho)\lesssim K^2+EK\,,$$
$$(D_x(F_1)_x,D_xu_x\rho)+(D_xF_2,D_x\omega\rho)\lesssim|a-1|K\,,$$
$$\begin{aligned}(D_x(N_1)_x,D_xu_x\rho)&\lesssim EK^2+|(-2wD_xu+2(1-a)D_x\psi_x u_x-2a\psi D_xu_{xx},D_xu_x\rho)|\\&\lesssim EK^2+\|\sin xu_x\|_{\infty}EK+|(u^2_{xx},(\psi\sin^2x\rho)_x)|\\&\lesssim EK(K+E)+K^2\|\frac{\psi}{\sin x}\|_{\infty}\lesssim EK(K+E)\,,\end{aligned}$$
$$(D_xN_2,D_x\omega\rho)\lesssim EK^2+|(2D_xu u_x-2a\psi D_x\omega_{x},D_x\omega\rho)|\lesssim EK(K+E)\,,$$
where we have used integration by parts and the estimate $$\|\sin xu_x\|_{\infty}\lesssim\|\sin xu_{xx}\|_{1}+\|\cos xu_{x}\|_{1}\lesssim\|D_xu_{x}\|_{\rho}+\|u_{x}\|_{\rho}\lesssim E+K\,.$$

We can finally prove that $$\frac{d}{d\tau}K(\tau)\leq -(1-C|a-1|)K+CE+C|a-1|+CE(E+K)\,.$$
Therefore combined with \eqref{nonl-collected}, we can find an absolute constant $\mu>1$ such that $$\frac{d}{d\tau}(K+\mu E)\leq-(0.1-C|a-1|)(K+\mu E)+C|a-1|+C(K+\mu E)^2\,.$$
By using a standard bootstrap argument, there exist absolute constants $\delta_0<\delta,C>0$, if $|a-1|<\delta_0$ and $K(0)+\mu E(0)<C|a-1|$, then $K(\tau)+\mu E(\tau)<C|a-1|$ for all time.
\subsection{Convergence to the steady state}
\label{sec ct}
We estimate the weighted norm of $\omega_\tau$ and $u_{\tau,x}$ and then use the standard convergence in time argument as in \cite{chen2021finite,chen2021slightly}.
$$J^2(\tau)=\frac{1}{2}((u_{\tau,x}^2,\rho)+(\omega_{\tau}^2,\rho))\,.$$
Applying the estimates of $\frac{d}{d\tau}E$ to $\frac{d}{d\tau}J$, we can get damping for the linear parts, and the small error terms corresponding to $\bar{\omega}$ and $\bar{u}$ vanishes. Therefore we yield
$$\frac{1}{2}\frac{d}{d\tau}J^2\leq-(0.16-C|a-1|)J^2+((N_1)_{\tau,x},u_{\tau,x}\rho)+((N_2)_\tau,\omega_t\rho)\,.$$ 
 Using estimates similar to Lemma \ref{lotb} and nonlinear estimates in \eqref{nonlinear}, we get
$$((N_1)_{\tau,x},u_{\tau,x}\rho)\lesssim EJ^2+(\psi_\tau u_{xx},u_{\tau,x}\rho)\lesssim EJ^2+J^2\|u_{xx}\sin x\|_{\rho}\leq(E+K)J^2\,.$$
$$((N_2)_{\tau},\omega_\tau\rho)\lesssim EJ^2+(\psi_\tau \omega_{x},\omega_{\tau}\rho)\lesssim EJ^2+J^2\|\omega_{x}\sin x\|_{\rho}\leq(E+K)J^2\,.$$
Combined with the a priori estimates on $E+K$, we can establish exponential convergence of $J$ to zero. Then we can use the same argument as in \cite{chen2021finite,chen2021slightly} to establish exponential convergence to the steady state and conclude the proof of Theorem \ref{t1}.
\section{Blowup of the Original Model with H\"older Continuous Data}
\label{sec hol}
In this section, we follow the strategy of the linear and nonlinear estimates of the weak advection model in Sections \ref{linearest} and \ref{nonlinearest}, and establish blowup of $C^{\alpha}$ data for the original model \eqref{1dwp} with $a=1$. Here $\alpha<1$ is close to $1$. Many of the ideas are drawn from the paper \cite{chen2021regularity} and we only outline the most important steps. Intuitively, $C^{\alpha}$ regularity of the profile weakens the advection and therefore contributes to a blowup in finite time.
\subsection{Dynamic rescaling formulation around the approximate steady state}
Before we start, we will solve the Biot-Savart law of recovering $\psi$ from $\omega$ with odd symmetry.
\begin{lemma}
\label{alpha est}
    Suppose that $\omega$, $\psi$ are odd and periodic on $[-\pi,\pi]$, with $-\psi_{xx}=\omega$. Then we solve $\psi_x(0)=-\frac{1}{2\pi}\int_{0}^{2\pi}y{\omega}(y)$ and obtain \begin{equation}\label{bs-lawsolved}
        \psi=\int_{0}^{x}(y-x){\omega}(y)dy+x\psi_x(0)\,.
    \end{equation}
\end{lemma}
The proof of this lemma is straightforward by integration in $x$.

We construct the following approximate steady state with $C^{\alpha}$ regularity for \eqref{1drf}.
$$\bar{\omega}_{\alpha}=sgn(x)|\sin x|^{\alpha}\,,\quad\bar{u}_{\alpha}=sgn(x)|\sin x|^{\frac{1+\alpha}{2}}\,,\quad \bar{c}_{u, {\alpha}}=(\alpha-1)\bar{\psi}_{{\alpha},x}(0)\,,$$
where $\bar{\psi}_{\alpha}$ is related to $\bar{\omega}_{\alpha}$ via \eqref{bs-lawsolved}.
We consider odd perturbations $u$, $\omega$, $\psi$. The odd symmetry of the solution is preserved in time by equation \eqref{1drf}.  We will use the normalization condition as $c_u=(\alpha-1){\psi}_x(0)$, which ensures that $u$ vanishes to a higher order at all times so that we can use the same singular weight $\rho$. In fact we compute using \eqref{1drf} and the normalization conditions that $$\lim_{x\to 0}\frac{\bar{u}_{{\alpha}}(x)+u(x,\tau)}{x^{\frac{1+\alpha}{2}}}=\lim_{x\to 0}\frac{\bar{u}_{{\alpha}}(x)+u(x,0)}{x^{\frac{1+\alpha}{2}}}\,.$$ 
Therefore if we make the initial perturbation $u(x,0)$ vanish to order $1+\alpha$ around the origin,  $u(x,\tau)$ will also vanish to order $1+\alpha$ for all time. 

Now similar to what we obtain in \eqref{1drfr}, the perturbations satisfy the following system
\begin{equation}
\label{1drfr-alpha}
\begin{aligned}
{u}_{ \tau} &=L_1+R_{1,\alpha}+N_{1,\alpha}+F_{1,\alpha}\,, \\
{\omega}_{ \tau}&=L_2+R_{2,\alpha}+N_{2,\alpha}+F_{2,\alpha}\,, \\
-{\psi}_{xx} &={\omega}\,,
\end{aligned}
\end{equation}
where we extract the same leading order linear parts  as in \eqref{1drc}, while the nonlinear and error terms change and the residual error terms $R_{1,\alpha}$, $R_{2,\alpha}$ model the discrepancy between our approximate profile with $C^{\alpha}$ regularity and the steady state profile $\bar{\omega}=\bar{u}=\bar{\psi}=\sin x$. Define $\psi_{\text{res}}=\bar{\psi}_{\alpha}-\bar{\psi},\omega_{\text{res}}=\bar{\omega}_{\alpha}-\bar{\omega},u_{\text{res}}=\bar{u}_{\alpha}-\bar{u}$. We can express $R_{i,\alpha}$ and $F_{i,\alpha}$ as follows
$$R_{1,\alpha}=-2 \psi_{\text{res}} u_x-2 u_{\text{res},x} \psi+ 2 u \psi_{\text{res},x}+2u_{\text{res}} \psi_x+\bar{c}_{u,\alpha} u+ c_u\bar{u}_\alpha
\,.\,$$ 
$$R_{2,\alpha}=-2\psi_{\text{res}} \omega_x-2\omega_{\text{res},x}\psi +2 u u_{\text{res},x}+2u_{\text{res}} u_x+\bar{c}_{u,\alpha} \omega+c_u\bar{\omega}_{\alpha}\,,$$
$$N_{1,\alpha}=(c_u+2\psi_x)u-2\psi u_x\,,\quad N_{2,\alpha}=c_u\omega+2uu_x-2\psi \omega_x\,,$$
$$F_{1,\alpha}=(\bar{c}_{u,\alpha}+2\bar{\psi}_{\alpha,x})\bar{u}_{\alpha}-2\bar{\psi}_{\alpha} \bar{u}_{\alpha,x}\,,\quad F_{2,\alpha}=\bar{c}_{u,\alpha}\bar{\omega}_{\alpha}+2\bar{u}_{\alpha}\bar{u}_{\alpha,x}-2\bar{\psi}_{\alpha} \bar{\omega}_{\alpha,x}\,.$$
Before we perform our energy estimates, we will obtain some basic estimates of the residues.
\begin{lemma}
\label{residue-lemma}
The following estimates hold for $\kappa=\frac{7}{8}<\frac{9}{10}<\alpha<1$.
\begin{enumerate}
    \item Pointwise estimates of the residues:
    $$|\partial_x^i\omega_{\text{res}}|+|\partial_x^i u_{\text{res}}|\lesssim |\alpha-1||\sin x|^{\kappa-i}\,,\quad i=0,1,2,3,$$ $$\|\psi_{\text{res}}\|_{\infty}+\|\psi_{\text{res},x}\|_{\infty}\lesssim |\alpha-1|\,.$$
    \item Refined estimates using cancellations:
    $$|\frac{\alpha-1}{2}\bar{u}_{\alpha,x}-\sin xu_{\text{res},xx}|+|\sin x\partial x[\frac{\alpha-1}{2}\bar{u}_{\alpha,x}-\sin xu_{\text{res},xx}]|\lesssim |\alpha-1|^{1/2}|x||\sin x|^{\alpha-1}\,.$$
\end{enumerate}
\end{lemma}
\begin{proof}
The first part of (1) and (2) can be proved by using direct calculations, and we refer to Lemma 6.1 in \cite{chen2021regularity} for details. There seems to be a typo in (6.11) in \cite{chen2021regularity} where $\bar{\omega}_{\alpha}$ should have been $\bar{\omega}_{\alpha,x}$.
    Furthermore, by the expression in Lemma  \ref{alpha est} we get the second part of (1).
\end{proof}
Similar to the weak advection case, we define the energy $E^2(\tau)=\frac{1}{2}((u_x^2,\rho)+(\omega^2,\rho))\,.$ We will estimate the growth of $E(\tau)$.
The leading order linear estimates $L_1, L_2$ can be obtained in Proposition \ref{prop d}.
The estimates for the nonlinear terms $N_{1,\alpha}, N_{2,\alpha}$ follow almost exactly the same as the weak advection case by using Lemma \ref{lotb}.

\subsection{Nonlinear stability}
By the computation in the previous subsection, we get 
\begin{equation*}
\label{nonlinearalpha}
    \frac{1}{2}\frac{d}{d\tau}E^2(\tau)\leq -0.16E^2+CE^3+((R_{1,\alpha})_x,u_x\rho)+(R_{2,\alpha},\omega\rho)+((F_{1,\alpha})_x,u_x\rho)+(F_{2,\alpha},\omega\rho)\,.
\end{equation*}
Further, we get
$$\begin{aligned}
    ((R_{1,\alpha})_x,u_x\rho)&\leq (-2\psi_{\text{res}}u_{xx},u_x\rho)+(c_u\bar{u}_{\alpha,x}-2u_{\text{res},xx}\psi,u_x\rho)\\&+(\|\omega_{\text{res}}\|_{\infty}+\|u_{\text{res}}\|_{\infty}+C|\alpha-1|)E^2\,.
\end{aligned}$$
For the first term, we can use integration by parts and Lemma  \ref{lotb} to obtain
$$CE^2(\|\psi_{\text{res},x}\|_{\infty}+\|\frac{\psi_{\text{res}}}{\sin x}\|_{\infty})\lesssim \|\psi_{\text{res},x}\|_{\infty}E^2\,.$$
For the second term, we compute $$c_u\bar{u}_{\alpha,x}-2u_{\text{res},xx}\psi=\psi_x(0)[(\alpha-1)\bar{u}_{\alpha,x}-2\sin x u_{\text{res},xx}]+2u_{\text{res},xx}(\sin x\psi_x(0)-\psi)\,.$$
Thus by Lemmas  \ref{lotb}  and \ref{residue-lemma}, we have \begin{equation}
    \label{important}
\|c_u\bar{u}_{\alpha,x}-2u_{\text{res},xx}\psi\|_{\rho}\lesssim E|\alpha-1|^{1/2} +|\alpha-1|\|\frac{\sin x\psi_x(0)-\psi}{|\sin x|x}\|_{\infty}\,.\end{equation}
Finally, for $|x|\geq \pi/2$, we have $$|\frac{\sin x\psi_x(0)-\psi}{|\sin x|x}|\lesssim |\psi_x(0)|+\|\frac{\psi}{\sin x}\|_{\infty}\lesssim E\,.$$
For $|x|< \pi/2$, we use Lemma \ref{alpha est} and $|\sin x|\geq 2/\pi |x|$ to obtain
$$|\frac{\sin x\psi_x(0)-\psi}{|\sin x|x}|\lesssim |\frac{\sin x\psi_x(0)-\psi}{x^2}|\leq |\frac{(\sin x-x)\psi_x(0)}{x^2}|+|\frac{\int_{0}^x(y-x)\omega (y)}{x^2}|\lesssim E\,,$$
where we have used the bound $\|\omega/x\|_{1}\lesssim\|\omega/x\|_{2}\lesssim\|\omega\|_{\rho}$ in the last inequality.
Thus we yield
$$((R_{1,\alpha})_x,u_x\rho)\lesssim |\alpha-1|^{1/2}E^2\,.$$
 Similarly we get $$(R_{2,\alpha},\omega\rho)\leq (-2\frac{\psi}{\sin x}\omega_{\text{res},x}\sin x,\omega\rho)+(c_u\bar{\omega}_{\alpha},\omega\rho)+(2uu_{\text{res},x},\omega\rho)+C|\alpha-1|E^2\,.$$
 For the first two terms, we can estimate them using Lemmas  \ref{lotb}  and \ref{residue-lemma}. For the third term, we use Hardy's inequality to derive
 $$\|uu_{\text{res},x}\|_{\rho}/|\alpha-1|\lesssim\|u/x/\sin x\|_{2}\lesssim \|u/x^2\|_{2}+\|u/(\pi-x)\|_{2}\lesssim\|u_x/x\|_{2}+\|u_x\|_{2}\lesssim E\,.$$
 Therefore we have 
 $$(R_{2,\alpha},\omega\rho)\lesssim|\alpha-1|E^2\,.$$
 
 For the error terms, we can just perform standard norm estimates.  We focus on the pointwise estimates for $x\geq0$ and the case $x<0$ follows by using the odd symmetry of the solution.
 $$
\begin{aligned}
    F_{2,\alpha}&=(\alpha-1)\bar{\psi}_{\alpha,x}(0)\sin^{\alpha} x+(\alpha+1)\cos x\sin^{\alpha} x-2\alpha(\psi_\text{res}+\sin x)\cos x\sin^{\alpha-1} x\\&=(\alpha-1)(\bar{\psi}_{\alpha,x}(0)-\cos x)\sin^{\alpha} x-2\alpha\frac{\psi_{\text{res}}}{\sin x}\cos x\sin^{\alpha} x\,.
\end{aligned}
$$
Therefore, we obtain $$\|F_{2,\alpha}\|_{\rho}\lesssim |\alpha-1|+\|\frac{\psi_{res}}{\sin x}\|_{\infty}\lesssim |\alpha-1|\,.$$
Similarly, we have $$
\begin{aligned}
    (F_{1,\alpha})_x&=\frac{\alpha^2-1}{2}\sin^{\frac{\alpha-1}{2}} x\cos x[\bar{\psi}_{\alpha,x}(0)-\bar{\psi}_{\alpha}\frac{\cos x}{\sin x}]+\sin^{\frac{\alpha+1}{2}}x[(\alpha+1)\bar{\psi}_{\alpha}-2\sin^{\alpha}x]\,.
\end{aligned}
$$
Further, we obtain the following estimate: 
$$\|(\alpha+1)\bar{\psi}_{\alpha}-2\sin^{\alpha}x\|_{\infty}\leq |\alpha-1|\|\bar{\psi}_{\alpha}\|_{\infty}+2\|\psi_\text{res}\|_{\infty}+2\|\sin x-\sin^{\alpha}x\|_{\infty}\lesssim|\alpha-1|\,,$$
$$\begin{aligned}\bar{\psi}_{\alpha,x}(0)-\bar{\psi}_{\alpha}\frac{\cos x}{\sin x}&=(1-\cos x)+\psi_{\text{res},x}(0)-\psi_\text{res}\frac{\cos x}{\sin x}\\&=(1-\cos x)(1+\frac{\psi_\text{res}}{\sin x})+\psi_{\text{res},x}(0)-\frac{\psi_\text{res}}{\sin x}\,.\end{aligned}$$
Combined with the estimate similar to that in \eqref{important}, we get 
$$\|[\bar{\psi}_{\alpha,x}(0)-\bar{\psi}_{\alpha}\frac{\cos x}{\sin x}]\rho^{1/2}\|_{\infty}\lesssim\|\frac{1-\cos x}{x}(1+\frac{\psi_\text{res}}{\sin x})\|_{\infty}+\|[\psi_{\text{res},x}(0)-\frac{\psi_\text{res}}{\sin x}]/x\|_{\infty}\lesssim 1\,.$$
Therefore we yield $$\|(F_{1,\alpha})_x\|_{\rho}\lesssim |\alpha-1|\,.$$
Collecting all the estimates of the residues and the error terms, we arrive at \begin{equation*}
    \frac{d}{d\tau}E(\tau)\leq -(0.16-C|\alpha-1|^{1/2}|)E+CE^2+C|\alpha-1|\,.
\end{equation*}
Similar to Subsection \ref{nonlinearest}, we can perform the bootstrap argument to conclude finite-time blowup.
 \subsection{Estimates in higher-order Sobolev norms and convergence to steady state}
 Following the ideas in Subsection \ref{h2sec}, we can perform estimates in higher-order Sobolev norms and then close the estimates to establish  convergence to a steady state. We use the same energy $K$ and only sketch the main steps here. We first have from Subsection \ref{h2sec} the estimates of $L_i$ and $N_i$ and obtain
$$\begin{aligned}
\frac{1}{2}\frac{d}{d\tau}K^2(\tau)&\leq -K^2+CEK+CEK^2+((R_{1,\alpha})_{xx},\sin^2x u_{xx}\rho)\\&+((R_{2,\alpha})_x,\sin^2x\omega_x \rho)+((F_{1,\alpha})_{xx},\sin^2xu_{xx}\rho)+((F_{2,\alpha})_x,\sin^2x\omega_x\rho)\,.
\end{aligned}$$

By Lemma \ref{residue-lemma}, $\sin x\partial x u_{\text{res}}, \sin x\partial x \omega_{\text{res}}$ shares the same pointwise estimates as $\ u_{\text{res}},  \omega_{\text{res}}$. We can obtain, similar to the estimates in $E$, the estimates $$((R_{1,\alpha})_{xx},\sin^2x u_{xx}\rho)\lesssim|\alpha-1|(E+K)K+K\|[u_{\text{res},xx}(\psi-\sin x\psi_x(0))]_x\sin x\|_{\rho}\,.$$
We further obtain the following estimate $$\begin{aligned}
    \|[u_{\text{res},xx}(\psi-\sin x\psi_x(0))]_x\sin x\|_{\rho}&\lesssim |\alpha-1|E+\|u_{\text{res},xx}(\psi_x-\cos x\psi_x(0))\sin x\|_{\rho}\\&\lesssim |\alpha-1|E+|\alpha-1|\|(\psi_x-\psi_x(0))/x\|_{\infty}\,.
\end{aligned}$$
Finally by the Cauchy-Schwarz inequality,
we have 
\begin{equation}
    \label{cs-al}|(\psi_x-\psi_x(0))/x|\leq\int_{0}^x|\omega(y)/y|dy\lesssim\|\omega/x\|_{2}\lesssim\|\omega\|_{\rho}\lesssim E\,,
\end{equation}
and therefore we conclude $$((R_{1,\alpha})_{xx},\sin^2x u_{xx}\rho)\lesssim|\alpha-1|(E+K)K\,.$$
By similar computations, we have 
$$((R_{2,\alpha})_{x},\sin^2x \omega_{x}\rho)\lesssim|\alpha-1|(E+K)K\,.$$

And for the residue terms, we use similar estimates to obtain
$$\|(F_{2,\alpha})_x\sin x\|_{\rho}\lesssim|\alpha-1|\,,$$
$$  \|(F_{1,\alpha})_{xx}\sin x\|_{\rho}\lesssim|\alpha-1|+|\alpha-1|\|(\frac{\psi_{\text{res}}}{\sin x})_x\rho^{1/2}\sin x\|_{\infty}\;.$$
We can use the triangular inequality to estimate $$|(\frac{\psi_{\text{res}}}{\sin x})_x\rho^{1/2}\sin x|\leq |[\psi_{\text{res},x}(0)-\frac{\psi_\text{res}}{\sin x}]/x|+|(\psi_{\text{res},x}-\psi_{\text{res},x}(0))/x|+|\frac{1-\cos x}{x}\frac{\psi_\text{res}}{\sin x}|\,.$$
Combined with the estimate similar to that in \eqref{important} and \eqref{cs-al}, we conclude $$  \|(F_{1,\alpha})_{xx}\sin x\|_{\rho}\lesssim|\alpha-1|\,.$$

We can finally obtain the same estimate as in Subsection \ref{h2sec} 
$$\frac{d}{d\tau}K(\tau)\leq -(1-C|a-1|)K+CE+C|a-1|+CEK\,.$$
We can again find an absolute constant $\mu_\alpha>1$ such that $$\frac{d}{d\tau}(K+\mu_\alpha E)\leq-(0.1-C|a-1|)(K+\mu_\alpha E)+C|a-1|+C(K+\mu_\alpha E)^2\,.$$
And we can use the bootstrap argument on this higher-order energy $K+\mu_\alpha E$ to conclude a priori estimate in this norm. Now we can perform weighted estimates in time using the energy $J$ as in Subsection \ref{sec ct}, where the linear estimates follow from the linear estimates of $E$, the nonlinear estimates follow from Subsection \ref{sec ct}, and the error term vanishes under time-differentiation. We can obtain the same estimates of $J$ and establish exponential convergence of $J$ to zero. Then we use the argument in \cite{chen2021finite,chen2021slightly} to establish exponential convergence to the steady state and conclude the proof of Theorem \ref{t2}.
 \section{Blowup of the Viscous Model with Weak Advection}\label{sec5}
In this section, we follow the strategy of the linear and nonlinear estimates of the weak advection model and establish blowup  for the weak advection viscous model with $a<1$.   Intuitively, the viscosity term is small in the dynamic rescaling formulation and nonlinear stability can be closed using a higher-order norm, where the viscosity term has a stability effect. Therefore we expect that the viscous weak advection model develops a finite time singularity as well. 
\subsection{Dynamic rescaling formulation}
We recall the weak advection  model with viscosity.
\begin{equation}
\label{1dwpv}
\begin{aligned}
u_{ t}+2 a\psi u_{ z} &=2 u \psi_{ z}+\nu u_{zz}\,, \\
\omega_{ t}+2 a\psi\omega_{ z} &=\left(u^{2}\right)_{z}+\nu \omega_{zz}\,, \\
- \psi_{zz} &=\omega\,.
\end{aligned}
\end{equation}
We use  the rescaled variables $$\tilde{u}(x, \tau)=C_{u}(\tau)  u( x, t(\tau))\,,\quad \tilde{\omega}(x, \tau)=C_{u}(\tau)  \omega( x, t(\tau))\,,\quad \tilde{\psi}(x, \tau)=C_{u}(\tau)  \psi(x, t(\tau))\,,$$ where $$
C_{u}(\tau)=C_{u}(0)\exp \left(\int_{0}^{\tau} c_{u}(s) d s\right)\,, \quad t(\tau)=\int_{0}^{\tau} C_{u}(s) d s\,.$$ For solutions to \eqref{1dwpv}, the rescaled variables satisfy the dynamic rescaling equation 
\begin{equation}
\label{1drfv}
\begin{aligned}
\tilde{u}_{\tau}+2a \tilde{\psi}\tilde{u}_{ x} &=2 \tilde{u} \tilde{\psi}_{ x}+c_u \tilde{u}+\nu C_u(\tau)u_{xx}\,, \\
\tilde{\omega}_{ \tau}+2a \tilde{\psi}\tilde{\omega}_{ x} &=\left(\tilde{u}^{2}\right)_{x}+c_{u}\tilde{\omega}+\nu C_u(\tau) \omega_{xx}\,, \\
-\tilde{\psi}_{xx} &=\tilde{\omega}\,.
\end{aligned}
\end{equation}
\begin{remark}
    Different from the rescaling in the inviscid case, we introduce an extra degree of freedom: the constant $C_u(0)$. We will choose it later to ensure that the viscous term has a relatively small scaling compared to the main terms.
\end{remark}
In order to establish a finite time blowup, it suffices to prove the dynamic stability of \eqref{1drfv} with scaling parameter $c_u<-\epsilon<0$ for all time; see also \cite{chen2021finite}.
As before, we will primarily work in the dynamic rescaling formulation and use the notations $\tilde{u}=\bar{u} +u$, where $\bar{u}$ is an approximation steady state and $u$ is the perturbation that we will control in time.  

We consider the following approximate steady state.
$$\bar{\omega}=\bar{u}=\bar{\psi}=\sin x\,,\quad \bar{c}_u(\tau)=2(a-1)\bar{\psi}_x(0)-\nu C_u(\tau)\bar{u}_{xxx}(0)/\bar{u}_{x}(0)=2(a-1)+\nu C_u(\tau)\,.$$
We consider odd perturbations $u$, $\omega$, $\psi$. The odd symmetry of the solution is preserved in time by equation \eqref{1drfv}.  We use the following normalization condition: $c_u=2(a-1){\psi}_x(0)-\nu C_u(\tau)u_{xxx}(0)$. This normalization ensures that $u_x(0)$ remains $0$ if the initial perturbation satisfies $u_x(0,0)=0$. In fact, if  $u_x(\tau,0)=0$, then we obtain
$$\begin{aligned}
    \frac{d}{d\tau}u_x(\tau,0)&=\frac{d}{d\tau}(u_x(\tau,0)+\bar{u}_x(\tau,0))=(2-2a)(u_x(\tau,0)+\bar{u}_x(\tau,0))(\psi_x(\tau,0)+\bar{\psi}_x(\tau,0))\\&+(c_u+\bar{c}_u)(u_x(\tau,0)+\bar{u}_x(\tau,0))+\nu C_u(\tau)(\bar{u}_{xxx}(0)+{u}_{xxx}(0))=0\,.\end{aligned}$$

This particular choice of the approximate steady state and the normalization conditions ensure that $u_x(0,\tau)=0$ for all time provided that the initial perturbation satisfies $u_x(0,0)=0$. We will perform the same weighted norm estimate in the singular weight $\rho$ and the weighted norm $E$ as in the inviscid case. 
 
       \subsection{Estimates of the viscous terms}
    Now the perturbation satisfies 
     \begin{equation}
\label{1drcv}
\begin{aligned}
{u}_{ \tau} &=L_1+(a-1)L'_1+N_1+F_1+\nu C_u(\tau)V^u\,, \\
{\omega}_{ \tau}&=L_2+(a-1)L'_2+N_2+F_2+\nu C_u(\tau)V^\omega\,, \\
-{\psi}_{xx} &={\omega}\,.
\end{aligned}
\end{equation}
Here the terms $V^u$ and $V^\omega$ correspond  to all of the terms containing the effect of the viscosity and we factor out explicitly the small factor $\nu C_u(\tau)$ for a fixed $\nu$. $$V^u=u_{xx}+\bar{u}_{xx}+(1-u_{xxx}(0))(u+\bar{u})=u_{xx}-u_{xxx}(0)\sin x +(1-u_{xxx}(0))u\,,$$$$V^\omega=\omega_{xx}-\omega_{xxx}(0)\sin x +(1-\omega_{xxx}(0))\omega\,.$$
    We invoke the  nonlinear estimates in the inviscid case and obtain for the viscous model:
\begin{equation}
\label{l2-unvis}
    \frac{1}{2}\frac{d}{d\tau}E^2(\tau)\leq -(0.16-C|a-1|)E^2+C|a-1|E+CE^3+\nu C_u(\tau)[((V^u)_x,u_x\rho)+(V^\omega,\omega\rho)]\,.
\end{equation}
    We estimate the viscous terms carefully since they involve singular weights. $$
   ((V^u)_x,u_x\rho)=(u_{xxx}-u_{xxx}(0),u_x{\rho})+u_{xxx}(0)(1-\cos x,u_x{\rho})+(1-u_{xxx}(0))\|u_x\|^2_{\rho}\,.
    $$
    Notice that $$|\rho_x|=\rho \left |\frac{-\sin x}{1-\cos x}\right | \lesssim \rho |\frac{1}{x}|\,, \quad |\rho_{xx}|=\rho \left |\frac{-\cos x}{1-\cos x}+\frac{2\sin^2 x}{(1-\cos x)^2} \right |\lesssim \rho |x^{-2}|$$ are singular near the origin and are smooth elsewhere. We can
use integration by parts twice to compute $$\begin{aligned}(u_{xxx}-u_{xxx}(0),u_x{\rho}&)=-(u_{xx}-xu_{xxx}(0),u_{xx}{\rho})-(u_{xx}-xu_{xxx}(0),\frac{x^2}{2}u_{xxx}(0){\rho_x})\\&-(u_{xx}-xu_{xxx}(0),(u_{x}-\frac{x^2}{2}u_{xxx}(0)){\rho_x})\\&\leq -\|u_{xx}\|^2_{\rho}+C[|u_{xxx}(0)|\|u_{xx}\|_{\rho}+|u_{xxx}(0)|^2+\|\frac{u_x}{x}-\frac{x}{2}u_{xxx}(0)\|_{\rho}^2]\\&\leq -\frac{1}{2}\|u_{xx}\|^2_{\rho}+C|u_{xxx}(0)|^2+C\|\frac{u_x}{x}\|_{\rho}^2\,,\end{aligned}$$
where for  the last inequality we use the weighted AM-GM inequality $ab\leq \epsilon a^2+\frac{1}{4\epsilon}b^2$ for a very small constant $\epsilon$.
Therefore we get \begin{equation}\label{visu-l2}
    ((V^u)_x,u_x\rho)\leq -\frac{1}{2}\|u_{xx}\|^2_{\rho}+C[|u_{xxx}(0)|^2+\|\frac{u_x}{x}\|_{\rho}^2 +(1+|u_{xxx}(0)|)E^2]\,.
\end{equation}
Similarly, we estimate via integration by parts$$\begin{aligned}(\omega_{xx},\omega{\rho})&=-(\omega_{x}-\omega_{x}(0),(\omega_{x}-\omega_{x}(0)){\rho})-(\omega_{x}-\omega_{x}(0),\omega_{x}(0){\rho})\\&-(\omega_{x}-\omega_{x}(0),x\omega_{x}(0){\rho_x})-(\omega_{x}-\omega_{x}(0),(\omega-x\omega_{x}(0)){\rho_x})\\&\leq-\|\omega_{x}-\omega_{x}(0)\|^2_{\rho}+C[|\omega_{x}(0)|\|\frac{\omega_{x}-\omega_{x}(0)}{x}\|_{\rho}+\|\frac{\omega}{x}-\omega_{x}(0)\|_{\rho}^2]\,.\end{aligned}$$
And we obtain
\begin{equation}\label{visw-l2}\begin{aligned}(V^\omega,\omega\rho)
   &\leq -\|\omega_{x}-\omega_{x}(0)\|^2_{\rho}+C[|\omega_{x}(0)|\|\frac{\omega_{x}-\omega_{x}(0)}{x}\|_{\rho}\\&+\|\frac{\omega}{x}-\omega_{x}(0)\|_{\rho}^2+|\omega_{xxx}(0)|^2 +(1+|\omega_{xxx}(0)|)E^2]\,.
   \end{aligned}
\end{equation}
    
    %We collect the linear estimate
    %\begin{equation}
     %   \label{l2-vis}
    %\frac{d}{dt}E(t)\leq -(0.16-C|a-1|)E+C|a-1|+CE^2+C C_u(t)(E^V+E+EE^V)\,.
  %  \end{equation}

   The essential difficulty for the viscous terms is that after integration by parts, the singular weight produces various positive terms, on top of the damping terms $-\|u_{xx}\|_{\rho}$; see \eqref{visu-l2}, \eqref{visw-l2}. Fortunately, the positive terms contribute only to higher-order terms near the origin.
   
   Consider the interval $I=[-\pi/2,\pi/2]$. $\rho$ and $|1/x|$ are upper bounded by a positive constant outside of the interval and we have $$\|\frac{u_x}{x}\|_{\rho}^2\lesssim \|{u_x}\|_{\rho}^2+\|\frac{u_x}{x^2}\|_{L^2(I)}^2\lesssim E^2+\|u_{xxx}\|_{L^\infty(I)}^2\,,$$
   $$\|\frac{\omega}{x}-\omega_{x}(0)\|_{\rho}^2\lesssim \|\omega-x\omega_{x}(0)\|_{\rho}^2+\|\frac{\omega}{x^2}-\omega_{x}(0)/x\|_{L^2(I)}^2\lesssim E^2+|\omega_{x}(0)|^2+\|\omega_{xx}\|_{L^\infty(I)}^2\,,$$
   $$\|\frac{\omega_{x}-\omega_{x}(0)}{x}\|_{\rho}\lesssim\|{\omega_{x}-\omega_{x}(0)}\|_{\rho}+\|\frac{\omega_{x}-\omega_{x}(0)}{x^2}\|_{L^2(I)}\lesssim\|{\omega_{x}-\omega_{x}(0)}\|_{\rho}+\|\omega_{xxx}\|_{L^\infty(I)}\,.$$
   Plugging these estimates into the \eqref{visu-l2}, \eqref{visw-l2} and using again the weighted AM-GM inequality, we can yield 
   \begin{equation}\label{visuw-l2}(V^\omega,\omega\rho)+((V^u)_x,u\rho)
   \leq -\frac{1}{2}(\|\omega_{x}-\omega_{x}(0)\|^2_{\rho}+\|u_{xx}\|^2_{\rho})+C[E_V^2 +(1+E_V)E^2]\,, 
\end{equation}
where $$E_V=\|\omega_{xxx}\|_{L^\infty(I)}+\|u_{xxx}\|_{L^\infty(I)}+|\omega_{x}(0)|\,.$$

   \subsection{Estimates in a higher-order norm}
    We will use a weighted higher-order norm to close the estimates. To have a good estimate in this higher-order norm, it needs to satisfy three criteria. First, we need to extract damping in the leading order linear term. Secondly, we need to bound the terms like $\omega_{xxx}(0)$ using interpolation between the lower and the higher-order norms via the Gagliardo-Nirenberg inequality; therefore it needs to be at least as strong as a regular higher-order norm near the origin. Thirdly, we need damping for the diffusion terms to close the estimates. This motivates us to choose a combination of the $k-$th order weighted norms for $k\geq1$:
    $$E_{k}^2(\tau)=(u^{(k+1)},u^{(k+1)}\rho_k)+(\omega^{(k)},\omega^{(k)}\rho_k)\,, \quad \rho_k=(1+\cos x)^k\;,$$
    where we use the notation that $f^{(k)}=\partial_x^k f$. We denote $E_0=E$ and $\rho_0=\rho$.
    \begin{remark}
    \label{remark rm}
        This weighted norm immediately satisfies criterion 2 and we will verify in the linear estimates that it satisfies criterion 1. Finally, a clever combination of the weighted norms can produce damping for the viscous terms and we make the damping terms in the estimates of the $(k-1)$-th order norms greater than the positive terms in the estimates of the $k$-th order norms. We will elaborate on those points and establish the nonlinear estimates.
    \end{remark}

    Now we can estimate $\frac{d}{d\tau}E_{k}(\tau)$ for $k>0$ as follows
    $$\begin{aligned}
    &\frac{1}{2}\frac{d}{d\tau}E_{k}^2(\tau)=(L_2^{(k)}+(a-1)(L_2')^{(k)}+N_2^{(k)}+F_2^{(k)}+\nu C_u(t)(V^\omega)^{(k)},\omega^{(k)}\rho_k)\\&+(L_1^{(k+1)}+(a-1)(L_1')^{(k+1)}+N_1^{(k+1)}+F_1^{(k+1)}+\nu C_u(t)(V^u)^{(k+1)},u^{(k+1)}\rho_k)\,,
    \end{aligned}$$ 
    where the parts $L_i,L_i',F_i,N_i$ are defined exactly the same as in the inviscid case. 

    We first look at the viscous terms.
     We have for example
    $$((V^u)^{(k+1)},u^{(k+1)}(1+\cos x)^k)\leq (u^{(k+3)},u^{(k+1)}(1+\cos x)^k)+CE_VE_k+(1+E_V)E_k^2\,.$$
    We use integration by parts twice to obtain
    $$\begin{aligned}
        &(u^{(k+3)},u^{(k+1)}(1+\cos x)^k)=-(u^{(k+2)},u^{(k+2)}(1+\cos x)^k-u^{(k+1)}k(1+\cos x)^{k-1}\sin x)\\&=-(u^{(k+2)},u^{(k+2)}(1+\cos x)^k)+\frac{1}{2}(u^{(k+1)},u^{(k+1)}k(1+\cos x)^{k-1}[(k-1)-k\cos x])\\&\leq -(u^{(k+2)},u^{(k+2)}(1+\cos x)^k)+C(k)(u^{(k+1)},u^{(k+1)}(1+\cos x)^{k-1})\,.
    \end{aligned}$$
    We can also get a similar bound for $V^\omega$. Therefore combined with the leading order estimate \eqref{visuw-l2} and using the idea in Remark \ref{remark rm}, we conclude that for small enough constants $0<\mu<\mu_0({k_0})<1$, we have the following viscous estimate \begin{equation}\label{viscous-final}
        \sum_{k=0}^{k_0}\mu^k[((V^u)^{(k+1)},u^{(k+1)}\rho_k)+((V^\omega)^{(k)},\omega^{(k)}\rho_k)]\leq C(k)[E_V^2+(1+E_V)\sum_{k=0}^{k_0}\mu^kE_k^2]\,. 
    \end{equation}
    Here $\mu_0(k_0)$ is a generic constant depending on $k_0$. We can choose $k_0$ large enough later so that $E_V$ can be bounded using the interpolation inequalities.
    
    Now we look at the linear terms and extract damping. We denote the terms as lower order terms (l.o.t. for short) if their $\rho_k$-weighted $L^2$-norms are bounded by $\sum_{i=0}^{k-1}E_i$. For the terms of intermediate order, since  $\rho_k\leq C(k)\rho_i$ for $i<k$, combined with the classical elliptic estimate,  we can show that $u^j$, $\psi^i$ for $0\leq j<k+1$ and $0\leq i<k+2$ are l.o.t.
    Using the l.o.t. notation, we keep track only of the higher-order terms $$(L_1^{(k+1)},u^{(k+1)}\rho_k)=(-2\sin x u^{(k+2)}-2k\cos x u^{(k+1)}+2\sin x \psi^{(k+2)}+\text{l.o.t.},u^{(k+1)}\rho_k)\,,$$
    $$(L_2^{(k)},\omega^{(k)}\rho_k)=(-2\sin x \omega^{(k+1)}-2k\cos x \omega^{(k)}+2\sin x u^{(k+1)}+\text{l.o.t.},\omega^{(k)}\rho_k)\,.$$
    Again we have a crucial cancellation of the cross terms and for the leading order terms we use integration by parts to obtain for example $$\begin{aligned}(-2\sin x u^{(k+2)}-2k\cos x u^{(k+1)},u^{(k+1)}\rho_k)&=(u^{(k+1)},u^{(k+1)}(-k-(k-1)\cos x)\rho_k)\\&\leq -(u^{(k+1)},u^{(k+1)}\rho_k)\,.\end{aligned}$$
    Therefore we derive the following estimate $$(L_1^{(k+1)},u^{(k+1)}\rho_k)+(L_2^{(k)},\omega^{(k)}\rho_k)\leq -E_k^2+C(k)\sum_{i=0}^{k-1}E_i E_k\leq -\frac{1}{2}E_k^2+C(k)\sum_{i=0}^{k-1}E_i^2\,.$$
    
    Similarly we have $$(L_1'^{(k+1)},u^{(k+1)}\rho_k)+(L_2'^{(k)},\omega^{(k)}\rho_k)\leq 2E_k^2+C(k)\sum_{i=0}^{k-1}E_i^2\,.$$
    
    We have the trivial bound for the error term 
    $$(F_1^{(k+1)},u^{(k+1)}\rho_k)+(F_2^{(k)},\omega^{(k)}\rho_k)\leq C(k) (a-1)E_k\,.$$
    
    The nonlinear terms are more subtle. We will show that \begin{equation}\label{nonlinea-f}
        (N_1^{(k+1)},u^{(k+1)}\rho_k)\leq  C(k)\sum_{i=0}^{k}E_i^2E_k\,,
    \end{equation}
    and we can have the same bound for $\omega$.
    In fact, for a canonical term in $N_1^{(k+1)}$ and $N_2^{(k)}$, it is of the form $\psi^{(i)}u^{(k+2-i)}$ or $\psi^{(i)}\psi^{(k+3-i)}$ or $u^{(i)}u^{(k+1-i)}.$ For the terms $\psi u^{(k+2)}$ and $\psi \omega^{(k+1)}$, we can use integration by parts and Lemma \ref{lotb} to show that  $$(\psi u^{(k+2)},u^{(k+1)}(1+\cos x)^k)\leq C(k)(\|\psi_x\|_{\infty}+\|\frac{\psi}{\sin x}\|_{\infty})E_k^2\leq C(k)E E_k^2\,.$$
    The terms associated with $\psi_x u^{(k+1)}$, $\psi_x \omega^{(k)}$, $u u^{(k+1)}$, and $u \omega^{(k)}$  have the same bound trivially. We can then focus on controlling the weighted norms of $\omega^{(i)}u^{(k-i)}$, $\omega^{(i)}\omega^{(k-1-i)}$, $u^{(i+1)}u^{(k-i)}$ for indices $0<i<k$ to establish the bound \eqref{nonlinea-f}. For example, we get
    $$\|\omega^{(i)}u^{(k-i)}(1+\cos x)^{k/2}\|_{2}\leq \|\omega^{(i)}(1+\cos x)^{(i+1)/2}\|_{\infty}E_{k-1-i}\,.$$
    Finally, by the fundamental theorem of calculus, we can bound the $L^{\infty}$-norm by  $$\begin{aligned}
        & C(K)[\|\omega^{(i+1)}(1+\cos x)^{(i+1)/2}\|_{1}+\|\omega^{(i)}(1+\cos x)^{(i-1)/2}\sin x\|_{1}]\\&\leq C(K)[E_{i+1}+\|\omega^{(i)}(1+\cos x)^{(i)/2}\|_{1}]\leq C(k)\sum_{i=0}^{k}E_i\,.
    \end{aligned}
        $$
        Therefore we conclude that \eqref{nonlinea-f} holds.
\subsection{Collection of norms and finite time blowup}
We collect the bounds \eqref{viscous-final} \eqref{nonlinea-f} for viscous and nonlinear terms, along with the linear bounds and the leading order estimate \eqref{l2-unvis}. For any fixed ${k_0}$, there exists a small enough constant $0<\mu_1({k_0})<\mu_0({k_0})$, such that the following estimate holds
$$\frac{d}{d\tau}I_{k_0}^2\leq -(0.1-C|a-1|)I_{k_0}^2+C|a-1|I_{k_0}+CI_{k_0}^3+C\nu C_u(\tau)[E_V^2+(1+E_V)I_{k_0}^2]\,,$$
where the energy is defined as $$I_{k_0}^2=\sum_{k=0}^{k_0}\mu_1^k({k_0})E_k^2\,.$$
Here the constants depend on $k_0$ and $\mu$ but once we first prescribe $k_0$ then $\mu=\mu_1({k_0})$, they become just constants. We will later make our $C_u(\tau)$ and $|a-1|$ small to close the argument.

    Finally, by the Gagliardo-Nirenberg inequality, for $k=1,3$, we have $$\|\omega^{(k)}\|_{L^{\infty}(I)}\lesssim \|\omega^{(4)}\|^{\theta}_{L^{2}(I)}\|\omega\|^{1-\theta}_{L^{2}(I)}\,, \quad \theta=\frac{k+1/2}{4}\,.$$
    This is the classical Gagliardo-Nirenberg inequality applied to a bounded domain, and we can just use the extension technique to prove it; see for example \cite{li2021note}. We get similar bounds involving $u$ and  conclude that $E_V\lesssim I_{k_0}$, for any fixed $k_0\geq 4$. For example, we just take $k_0=4$ and obtain
    $$\frac{d}{d\tau}I_{4}\leq -(0.1-C|a-1|)I_{4}+C|a-1|+CI_{4}^2+C\nu C_u(\tau)(1+I_{4})I_{4}\,.$$

Now we choose $C_u(0)=|a-1|^2$ for $|a-1|<\delta$ with a small enough $\delta>0$. It is easy to check that the bootstrap argument for $I_{4}\leq C|a-1|$ and $C_u(\tau)\leq C_u(0)\exp((a-1)t)\leq C_u(0)$ will hold for all time provided that it holds initially. We again use the estimate for the normalization constants $$c_u+\bar{c}_u=2(a-1)+\nu C_u(t)(1-u_{xxx}(0))+2(a-1)\psi_x(0)<(a-1)<0\,.$$
    Thus we can obtain a blowup in finite time in the physical variables. 
\section{Appendix}
\label{sec app}
\begin{lemma}
\label{cap}
    Assume $\sum_{k\geq 1} a_k^2+c_k^2<\infty$, then we have the following inequality for \eqref{dample}
    $$\begin{aligned}
 F(a,c)&\coloneqq \sum_{k\geq1}\{ a_k^2(0.84+\frac{1}{k^2}-\frac{1}{(k-1)^2})+c_k^2(0.84+\frac{1}{k(k+1)})+ 2a_k a_{k+1}\frac{1}{(k+1)^2}\\&+2 a_k \sum_{j>k+1}a_j(\frac{1}{j^2}-\frac{1}{(j-1)^2})+2a_kc_k\frac{1+2k-k^2}{2k^2(k+1)}+2a_{k+1}c_k\frac{k^2-k-1}{2k^2(k+1)^2}\\&-2a_{k+2}c_k\frac{k+2}{2(k+1)^2}+\sum_{j>k}2a_k c_j\frac{1}{j(j+1)}\}\geq0\,.\end{aligned}$$
\end{lemma}
\begin{proof}
    Denote the summation of terms in  $F(a,c)$ that only involve $a_i$, $c_j$ for $i,j\leq N$ as $F_N(a,c)$. Here $N=200$. This quadratic form $F_N(a,c)$ can be expressed as $a^{(N),T}F^{(N)}c^{(N)}$, where $a^{(N)}, c^{(N)}$ are two vectors with entries $a_i$,$c_j$ respectively and $F^{(N)}$ is a symmetric matrix. We numerically verify using interval arithmetic in Matlab that the smallest eigenvalue of $F^{(N)}$ is greater than $0.01$; see remarks after the proof for details. Therefore we have 
    $$F_N(a,c)\geq0.01\sum_{k=1 }^N (a_k^2+c_k^2)\,.$$

For the remainder $F(a,c)-F_N(a,c)$, we estimate it term by term via the trivial bound $2ab\geq-(a^2+b^2)$ and obtain
    $$\begin{aligned}&F(a,c)-F_N(a,c)\geq\sum_{k>N}[ a_k^2(0.84+\frac{1}{k^2}-\frac{1}{(k-1)^2})+c_k^2(0.84+\frac{1}{k(k+1)})] \\&+\sum_{k=1}^{ N} a_k^2(-\frac{2}{N^2}-\frac{1}{N+1}) +\sum_{k=N-1}^{ N} c_k^2(-\frac{N^2-N-1}{2N^2(N+1)^2}-\frac{N+1}{2N^2}) \\&+\sum_{k> N} a_k^2(-\frac{3}{N^2}-N(\frac{1}{(N)^2}-\frac{1}{(N+1)^2})-\frac{N^2-N-1}{2N^2(N+1)^2}-\frac{N+1}{2N^2}-\frac{N^2-2N-1}{2N^2(N+1)}\\&-\frac{1}{N+2})+ \sum_{k> N} c_k^2(-\frac{N^2-N-1}{2N^2(N+1)^2}-\frac{N+1}{2N^2}-\frac{N^2-2N-1}{2N^2(N+1)}-\frac{1}{N+2})\,. \end{aligned}$$
For $N=200$, we estimate all of the coefficients by a lower bound and obtain
    $$F(a,c)-F_N(a,c)\geq -\frac{2}{N}\sum_{k=1}^{ N} (a_k^2+c_k^2)+(0.84-\frac{3}{N})\sum_{k> N} (a_k^2+c_k^2)\geq-\frac{2}{N}\sum_{k=1}^{ N} (a_k^2+c_k^2)\,.$$
    Therefore we conclude $F(a,c)\geq 0$.
\end{proof}
\begin{remark}
    We now explain how to verify that the smallest eigenvalue of the symmetric matrix $F^{(200)}$ is greater than $0.01$. We proceed in three steps.
    \begin{enumerate}
        \item We first use Matlab to perform an (approximate) SVD decomposition of $$F^{(200)}-0.011I\approx VDV'\,.$$
        Here $D$ is the diagonal matrix consisting of (approximate) eigenvalues of $F^{(200)}-0.011I$, and $V$ is the unitary matrix consisting of (approximate) eigenvectors of $F^{(200)}-0.011I$.
        \item We use interval arithmetic to verify that the maximal absolute value of entries of $F^{(200)}-0.011I- VDV'$ is at most $10^{-10}$. Therefore the spectral norm of $F^{(200)}-0.011I- VDV'$, which is bounded by its $1$-norm, is rigorously bounded from above by $200 \times 10^{-10}$.
        \item Since $D$ has positive entries, we know that $VDV'$ is positive definite. We  conclude that  $$F^{(200)}-0.01I=VDV'+0.001I+F^{(200)}-0.011I- VDV'$$ is positive definite.
    \end{enumerate}
\end{remark}

\chapter{High Precision PINNs in Unbounded Domains: Application to singularity formation in PDEs}
\label{append:pinns}
We investigate the high-precision training of Physics-Informed Neural Networks (PINNs) in unbounded domains, 
with a special focus on applications to singularity formation in PDEs. 
We propose a modularized approach and study the choices of neural network ansatz, sampling strategy, and optimization algorithm.
When combined with rigorous computer-assisted proofs and PDE analysis,
the numerical solutions identified by PINNs, provided they are of high precision, 
can serve as a powerful tool for studying singularities in PDEs.
For 1D Burgers equation, our framework can lead to a solution with very high precision,
 and for the 2D Boussinesq equation, which is directly related 
 to the singularity formation in 3D Euler and Navier-Stokes equations, 
 we obtain a solution whose loss is $4$ digits smaller than that obtained in \cite{wang2023asymptotic} with fewer training steps. 
  We also discuss potential directions for pushing towards machine precision for higher-dimensional problems.
  \section{Introduction}
singularity formation is one of the key challenges in the study of partial differential equations (PDEs).
Unlike well-posed equations, where one can apply classical existence and uniqueness theorems,
singularities often occur in certain solutions of nonlinear PDEs, where we only have guarantees 
of existence for a short time, but the solution may blow up in finite time.
The study of singularities often involves a case-by-case approach and is related to 
some of the most intriguing mathematics and physical properties, such as the onset of turbulence in the Navier-Stokes equations.
The singularity of the Navier-Stokes equations is one of the seven Millennium Prize Problems \cite{fefferman2006existence}:
widely regarded as the most fundamental and challenging problem in analysis, and is still open.
One of the key difficulties in the study of singularities is the lack of 
understanding of the singularity pattern and its mechanism. The computation of the 
singularity itself, or infinity, is intractable numerically.
A general roadmap is thus to first propose a plausible singularity ansatz that renders
the computation feasible,
then find candidates of such blowup by numerical simulations, and
finally verify the stability of such an ansatz by PDE analysis.

We are often interested in a special structure of singularity: self-similar singularity.
Self-similarity relates to the invariance of the solution under scaling transformations 
and reduces singularity to the existence of a self-similar profile. To be precise, for the quantity of interest $u(x,t)$,
we put the ansatz $u(x,t) = (T-t)^{-\alpha} U(x(T-t)^{-\beta})$, where $U$ is the profile function
independent of time, $T$ is the blowup time, and $\alpha>0, \beta$ are the scaling exponents to be determined.
Now we can reduce the computation of an infinite $u$ to the computation of a finite, smooth profile $U$, 
along with scaling exponents $\alpha, \beta$ to be inferred. Physics-Informed Neural Networks (PINNs) \cite{raissi2019physics}
serve as a powerful tool to find such profiles, with the scaling parameters jointly inferred as inverse problems. It was first introduced in the context of identifying singularity profiles in \cite{wang2023asymptotic} and has seen success 
even in problems that are unstable for traditional numerical methods. While PINNs offer a powerful tool to search
for candidate profiles, solutions identified by PINNs are often of limited accuracy and far from applicable to
rigorous PDE analysis, to the best of the authors' knowledge.

In this work, we aim to systematically study the high-precision training of PINNs,
with a special focus on applications  to solve profile equations governing singularity formations
in PDEs, in fluid dynamics and beyond. We will take a modularized perspective without diving 
into sophisticated tricks and investigate the following aspects:
a good neural network ansatz representing the profile function, a good sampling strategy to tackle
the infinite domain with a special focus on imposing boundary conditions,
and a good optimization algorithm to train the neural network. 
We apply our findings to the 1D Burgers equation and the 2D Boussinesq equation, 
obtaining a precision amenable to rigorous PDE analysis for the 1D Burgers equation 
and $4$ digits better than \cite{wang2023asymptotic} with fewer training steps for the 2D Boussinesq equation.
%We also discuss potential directions of pushing towards machine precision for higher-dimensional problems.
The 2D Boussinesq equation shares many similarities with the 3D axisymmetric Euler equation for the ideal
fluid without viscosity; see the pioneering works of \cite{chen2021finite,chen2022stable,chen2025stable} for the connection between the two equations,
where the authors used the connection to establish singularity formation for the 3D axisymmetric Euler with boundary.

\section{Related Works}
\subsection{PINNs}
Neural networks have witnessed success in solving PDEs and surrogate modeling in math and science.
PINNs in particular have been widely used due to their flexibility and applicability to a wide range of problems \cite{cai2021physics, raissi2019physics, karniadakis2021physics}.
The key idea of PINNs is to enforce the PDE constraints at a set of collocation points and to minimize the residual
of the PDEs as a loss function. By posing the solution of PDEs as an optimization problem, PINNs are especially suited
to solve inverse problems \cite{raissi2020hidden, yuan2022pinn, lu2021physics, yu2022gradient}, where the solution of the PDE and the underlying parameters can be jointly inferred.
In \cite{wang2023asymptotic}, the authors used PINNs to study the blowup of the 1D Burgers equation, 
the 1D family of generalized Constantin-Lax Majda equations, and the 2D Boussinesq equation.

Another line of work, operator learning \cite{kovachki2023neural}, focuses on learning the solution operator instead of learning 
a single instance of solution, where Fourier Neural Operators (FNOs) \cite{li2020fourier, li2023fourier, li2023geometry, li2024scale} and DeepONets \cite{lu2021learning} are two families of representative works in this direction.
Once a solution operator is learned, it can be evaluated in a resolution-free manner at any point in the domain.
Data are often augmented to enhance the solution accuracy, while the loss function  can also incorporate
the PDE constraints, termed Physics Informed Neural Operator (PINO) \cite{li2021physics}. In \cite{maust2022fourier}, the authors used PINO with Fourier continuation 
to study the blowup of the 1D Burgers equation.
\subsection{Self-similar singularity and computer assisted proofs}
Self-similar singularity of the ansatz $u(x,t) = (T-t)^{-\alpha} U(x(T-t)^{-\beta})$
is generic in the study of singularity formation in PDEs, where 
one uses the scaling invariance of the PDE and
 can reduce the computation of an infinite $u$ to the computation of a finite, smooth (approximate) profile $U$.
Such structures exist even in the simple Riccati ODE $u_t=u^2$ with an exact solution $u=(T-t)^{-1}$.

The approximate profile can be identified via explicit construction or numerical computation. 
Working in the rescaled, self-similar variables and performing 
stability analysis around the profile $U$ provides a powerful tool to establish the singularity formation,
for nonlinear Schrodinger equations \cite{mclaughlin1986focusing, merle2005blow}, incompressible fluids \cite{elgindi2019finite,chen2019finite2,chen2021finite,chen2022stable,hou2024blowup}, compressible fluids \cite{MRRSam22a, MRRSam22b}, and beyond. 
Until recently, most of the works relied on an explicit profile and spectral information 
of the associated linearized operator to establish linear and nonlinear stability. In
\cite{chen2022stable,chen2025stable}, the authors used computer-assisted proofs with a sophisticated numerical profile obtained by solving the dynamic rescaling equations in time to obtain an approximate steady state. By analyzing the stability of the approximate profile, they established the singularity formation for the 2D Boussinesq equation
and the 3D axisymmetric Euler equation with boundary.
And in \cite{hou20242,chen2024stability}, the authors provided a framework using only local information for stability analysis,
bypassing spectral information and allowing for numerical profiles with computer-assisted proofs,
for problems beyond self-similarity.

\subsection{Towards high precision training}
Various methods have been proposed in the literature to improve the accuracy of PINNs. One line of work focuses on a 
better representation of the solution. In \cite{michaud2023the, wang2024multi}, the boosting technique was proposed, where a sum of a sequence of neural networks
with decreasing magnitude was used to learn the solution; at each stage, a new neural network is trained to learn the residual.
To overcome the spectral bias \cite{rahaman2019spectral} of multilayer perceptrons (MLPs), or the favor of learning low-frequency modes \cite{xu2019training,xu2019frequency} in the solution, 
one can use Fourier feature encoding \cite{sitzmann2020implicit, tancik2020fourier,ng2024spectrum}, or different activation functions \cite{jagtap2020adaptive, jagtap2020locally,hong2022activation, zhang2023shallow,kanbias}. In particular, Kolmogorov-Arnold Networks (KANs) \cite{liu2024kan,kan1} that leverage nonlinear learnable activation functions and the Kolmogorov-Arnold representation theorem were proposed and further investigated in the PINN setting \cite{wang2024kolmogorov,shukla2024comprehensive,toscano2025pinns}.
Another line of work improves the optimization landscape during the training of PINNs.
Various optimizers, which we will detail in Subsection \ref{subsec:optimizer}, have been proposed to improve the convergence rate.
Adaptive design of points sampling \cite{anagnostopoulos2024residual,wu2023comprehensive,rigas2024adaptive} and adaptive weighting of different terms \cite{wang2021understanding,xiang2022self,mcclenny2023self} in the loss function
were also proposed to improve the accuracy of PINNs.

We will only focus on applying hard constraints and choosing a good optimizer in this work and leave 
the exploration of more sophisticated tricks for future work.
\section{Methodology} 
We outline our methodology of high-precision training for PINNs on the whole space in this section.
We work under the general formulation of the profile equation $$L(U,\lambda)=0,$$ where 
$U(y)$ is the profile function, $\lambda$ is a set of scaling parameters to be determined,
and $L$ is the nonlinear differential operator. For our problems of interest, $U$=0 will be a trivial 
solution satisfying the equation.
\subsection{Infinite domain}\label{subsec:inf}
The key challenges we are facing here are sampling and learning on an infinite domain.
For a given budget of a finite number of sampling points, we need to sample the domain in a way that
the resulting solution is accurate and generalizes well throughout the domain. In the meantime,
We want the neural network to be able to represent the profile function and the initialization of parameters
to favor learning of such representations. To this end, we adopt an exponential "mesh" in our sampling strategy:
Consider an auxiliary  variable $z$ such that $y = \sinh(z)=\frac{e^z-e^{-z}}{2}$, 
and sample $z$ uniformly in a finite region. Here we choose the $\sinh$ transformation as in \cite{wang2023asymptotic}
to respect the parity of the functions, detailed in the subsequent subsection. Such a transformation
maps roughly $z\in[-30,30]$ to $y\in[-5\times 10^{12},5\times 10^{12}]$.

Boundary conditions are another important aspect when learning on the whole space.
For our problems of interest, $U$ by itself will not have sufficient decay at infinity,
and one approach adopted in \cite{wang2023asymptotic} is to impose Neumann boundary conditions at infinity, or numerically 
on the boundary of the domain of the $z$ variables.
To rule out the trivial solution $U=0$, we need to enforce a nondegeneracy condition,
often posed at the origin. We will discuss the enforcement of these conditions in the following subsection.
We refer to this formulation as boundary conditions using \textbf{weak asymptotics}.

On the other hand, we can enforce stronger information on the boundary. 
If we know the exact asymptotic behavior of the solution at infinity as $g$, for example a power law,
we can introduce a smooth cutoff function $\chi$ with $\chi(0)=0,\chi(\infty)=1$
 and the ansatz $U=\tilde{U}+\chi g$. We can then enforce Dirchlet boundary conditions at infinity for $\tilde{U}$
 represented by the neural network. We refer to this formulation as boundary conditions using \textbf{exact asymptotics}. We will demonstrate for the 1D example that PINNs using exact asymptotics will outperform those using weak asymptotics by a large margin.

A priori, the exact asymptotics information is not available, and one can first train a neural network $U_w$ with 
boundary conditions using weak information, and distill the information of asymptotics $g$ from $U_w$.
We refer to this formulation as boundary conditions using \textbf{hybrid asymptotics}. For example,
for the 2D Boussinesq equation, borrowing ideas from \cite{chen2022stable},
one can use function fitting and symbolic regression to extract asymptotics $g$ from $U_w$, filtering out the noisy residues, such that $g$ is a symbolic function approximating $U_w$ at infinity. We will leave this approach to future work.

\subsection{Hard constraint}\label{subsec:hard}
Hard constraints are important concepts in the parametrization of the solution space for PINNs.
When enforced properly, they will guarantee physical properties of the solution \cite{lu2021physics, richter2022neural,mohan2023embedding,duruisseaux2024towards}, and
can impose the solution to be in the correct manifold. While most of the previous works focus on hard constraints of boundary conditions,
we emphasize the enforcement of hard constraints in the following senses: the parity of the learned function
and the nondegeneracy conditions. Empirically we observe a better convergence rate and a more stable solution
when enforcing hard constraints.

\paragraph{Parity.} For a function $f(y_i,\hat{y}_i)$ even/odd in the variable $y_i$, we train a neural network 
with the following ansatz $f=(f_{nn}(y_i,\hat{y}_i)\pm f_{nn}(-y_i,\hat{y}_i))/2$. 

\paragraph{Nondegeneracy conditions.} As discussed in the previous subsection, we need to enforce 
nondegeneracy conditions to rule out the trivial solution $U=0$ when using weak asymptotics.
For example, for the 1D Burgers equation, we know that $U$ is odd and necessarily $U'(0)=-1$; 
we can enforce $U'''(0)=6$. We will enforce a hard constraint
via Taylor expansion at the origin as $U=-z+z^3+z^4U_1$, for an odd function $U_1$. 
Similarly for the 2D Boussinesq equation, we enforce $\partial_{1}\Omega(0,0)=-1$ and $\Omega$
is odd in $z_1$ via a Taylor expansion as $\Omega=-z_1+z_1z_2\Omega_1+z_1^2\Omega_2$, 
where $\Omega_1,\Omega_2$ are even and odd functions in $z_1$ respectively.

\subsection{Optimizer: Self-Scaled BFGS methods}\label{subsec:optimizer}
A common practice of training PINNs is to use the Adam optimizer.
As a stochastic first-order method, Adam is known to be robust and efficient in training deep neural networks
and can empirically escape local minima. To further improve convergence to the minimizer, one 
can apply second-order methods with a higher convergence rate, like L-BFGS, after training with Adam for 
a few epochs. While this seems to be a gold standard in the training of PINNs \cite{rathore2024challenges}, various optimizers have been investigated,
 including variants of 
second-order quasi-Newton methods \cite{rathore2024challenges,wang2025gradient}, and optimizers using natural gradients \cite{muller2023achieving,jnini2024gauss,chen2024teng}.
We highlight and use the self-scaled BFGS methods proposed in \cite{al1998numerical,al2014broyden} and introduced to the PINNs context in \cite{urban2025unveiling,kiyani2025optimizer}.
BFGS methods use an approximation of the inverse of the Hessian matrix to precondition the gradient for the update direction. To be precise, consider the parameters $\Theta_k$ and learning rate $\alpha_k$ at step $k$, with loss function $\mathcal{J}(\Theta)$, then the update rule for $\Theta$ is $$\Theta_{k+1}=\Theta_k-\alpha_k H_k\nabla\mathcal{J}(\Theta_k).$$ Different choices of updating the approximate inverse Hessian $H_k$ lead to different optimizers, and L-BFGS in particular is a memory-efficient way for the updates by storing only vectors instead of the whole matrix.
The self-scaled BFGS methods use a scaling compared to the standard BFGS update of the inverse Hessian. More precisely, for the auxiliary variables $$\begin{aligned}
    s_k&=\Theta_{k+1}-\Theta_k,\quad y_k=\mathcal{J}(\Theta_{k+1})-\mathcal{J}(\Theta_k),\\{v}_k&=\sqrt{{y}_k \cdot H_k {y}_k}\left[\frac{{s}_k}{{y}_k \cdot {s}_k}-\frac{H_k {y}_k}{{y}_k \cdot H_k {y}_k}\right],\end{aligned}$$
    we have  for the scalers $\tau_k$ and $\phi_k$: $$H_{k+1}=\frac{1}{\tau_k}\left[H_k-\frac{H_k {y}_k \otimes H_k {y}_k}{{y}_k \cdot H_k {y}_k}+\phi_k {v}_k \otimes {v}_k\right]+\frac{{s}_k \otimes {s}_k}{{y}_k \cdot {s}_k},$$
   where the original BFGS corresponds to the choices $\tau_k=\phi_k=1$.
While this is only a simple modification of the original BFGS, the authors in \cite{urban2025unveiling} demonstrated a much improved convergence rate
across a variety of benchmarks, including the  Helmholtz equation, the nonlinear Poisson equation, the nonlinear Schrödinger equation, the Korteweg-De Vries equation, the viscous Burgers equation, the Allen-Cahn equation, 3D Navier-Stokes: Beltrami flow, and the lid-driven cavity.
We use the self-scaled Broyden methods proposed in \cite{urban2025unveiling}; see equations (13)-(23) therein for details on the choices of  $\tau_k$ and $\phi_k$.

\paragraph{On the role of minibatch training or random resampling.}
    One of the common practices when training PINNs is to use random resampling of the collocation points. This can enhance the performance of SGD-based methods like Adam empirically. However, full-batch second-order methods like BFGS with supposedly higher-order accuracy do not adapt well to random resampling since they rely on past trajectories for Hessian updates. 
 One empirical observation, as proposed in \cite{wang2024multi}, is that when one uses an optimizer with fixed resolution like BFGS, it will be able to generalize in the regions where sampling points are sufficient.
 However, in the undersampled regions, the learned solution generalizes poorly. In an abstract form, there exists a critical batchsize $N_c$, such that when $N>N_c$, fixed sampling will be preferred, while for  $N<N_c$, fixed sampling will have very bad generalization. %Random resampling almost does not deteriorate even when $N$ becomes smaller. 
 $N_c$ would depend on both the equation and the scale of the neural network.
     Empirically, we observe that roughly 10k points are sufficient for generalization with fixed training points. We resample every 1000 epochs to further reduce overfitting. See details on the choices of the batch size in the experiments section.
    \section{Experiments}
    In this section, we describe our numerical experiments on the blowup profiles for 1D Burgers equation and 2D Boussinesq Equation. The codes are available at \url{https://github.com/RoyWangyx/High-precision-PINNs-unbounded-domains-/tree/main}. When training both equations, we denote the PDE by $L(U(y))=0$ and the boundary condition by $B(U)=0$. We use auxiliary variables $z=\sinh^{-1}y$ as in Subsection \ref{subsec:inf} and consider the following combination of interior, boundary, and smoothness losses as in \cite{wang2023asymptotic}
    \begin{equation}\label{eqn_loss}
 \begin{aligned}loss&=
        0.1(L_i+L_s)+L_b\\&= 0.1(\frac{1}{N_i}\sum_{j=1}^{N_i} [\hat{L}(U_{nn}(z_j))]^2+\frac{1}{N_s}\sum_{j=1}^{N_s} |\nabla_{z_j} \hat{L}(U_{nn}(z_j))|^2)+\frac{1}{N_b}\sum_{j=1}^{N_b} [\hat{B}(U_{nn})]^2,
    \end{aligned}       
    \end{equation}
    where $U_{nn}(z)$ is supposed to approximate $\hat{U}(z)=U(y)$ in the $z$-variables, and $\hat{L}$, $\hat{B}$ denotes the PDE and the boundary condition transformed in the $z$-variables; see \cite{wang2023asymptotic} for a concrete formula for the 2D Boussinesq equation.
\subsection{Burgers equation}
For the 1D Burgers equation \begin{equation}
    u_t+uu_x=0,
\end{equation}
consider the self-similar ansatz that respects the scaling symmetry 
\begin{equation}
    u(x,t)=(1-t)^\lambda U(y),\quad y={x}{(1-t)^{-1-\lambda}}.
\end{equation}
The profile equation for $U$ used for the PDE loss in \eqref{eqn_loss} is
\begin{equation}\label{burgers_eqn}
    -\lambda U+((1+\lambda)y+U)U_y=0.
\end{equation}
We impose an odd symmetry on $U$, and the profile equation has implicit solutions  \begin{equation}
    y+ U+CU^{1+1/\lambda}=0.
\end{equation}
for any constant $C$, as in the setting of \cite{wang2023asymptotic}, where we know that the most stable solutions correspond to $\lambda =0.5$ and there are nonsmooth solutions at e.g. $\lambda =0.4$.

In this example, we assume that we first train the neural network on a bounded domain and infer the correct $\lambda$ already, for example via the method in \cite{wang2023asymptotic}. Now we focus on fixing $\lambda$ and learn $U$ on the unbounded domain. Using an MLP with activation function $\tanh$, $4$ layers and $20$ neurons per layer and a hard constraint on parity, we use the optimizer SSBroyden1 as in \cite{urban2025unveiling} with $20000$ epochs and resampling every $1000$ epochs. $z$ is sampled uniformly on $[0,30]$ with a batchsize $10000$ for both the interior and smoothness losses, corresponding to a domain $[0,5\times 10^{12}]$ in the $y$ variables.

For the formulation using weak asymptotics as in Subsection \ref{subsec:inf}, we use the Neumann boundary condition $U_y=0$ and enforce hard constraint of nondegeneracy conditions as in Subsection \ref{subsec:hard}. For the formulation using exact asymptotics as in Subsection \ref{subsec:inf}, we use Dirichlet boundary condition $\tilde{U}=0$ and the cutoff function $\chi=(\frac{y}{1+y})^{15}$, since the far field is captured by the exact asymptotics $g=-y^{\frac{\lambda}{1+\lambda}}$.

We present the following results of $\lambda=0.4, 0.5$ using weak and exact asymptotics: see Figure \ref{fig:burgers_final} for the equation residue of the solution at the final stage and Figure \ref{fig:burgers_loss} for the evolution of the losses. We are able to achieve high accuracy over a large domain, but using exact asymptotics is preferred for both the smooth and nonsmooth case of $\lambda$.
\begin{figure}[htbp]
    \centering
    \begin{minipage}[b]{0.48\textwidth}
        \centering
        \includegraphics[width=\textwidth]{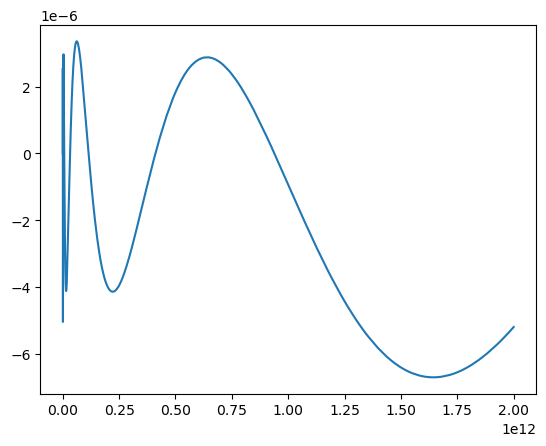}
    
    \end{minipage}
    \hfill
    \begin{minipage}[b]{0.48\textwidth}
        \centering
        \includegraphics[width=\textwidth]{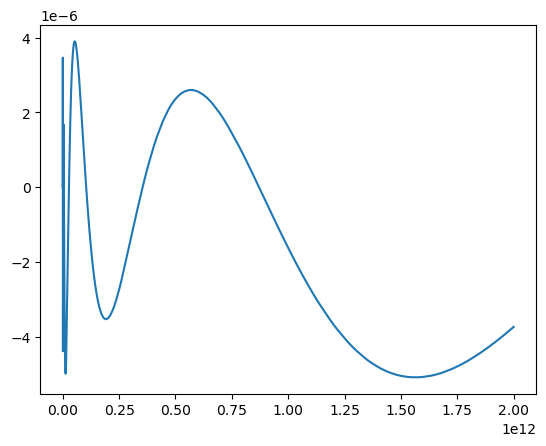}
    \end{minipage}

    \vspace{1em}

    \begin{minipage}[b]{0.48\textwidth}
        \centering
        \includegraphics[width=\textwidth]{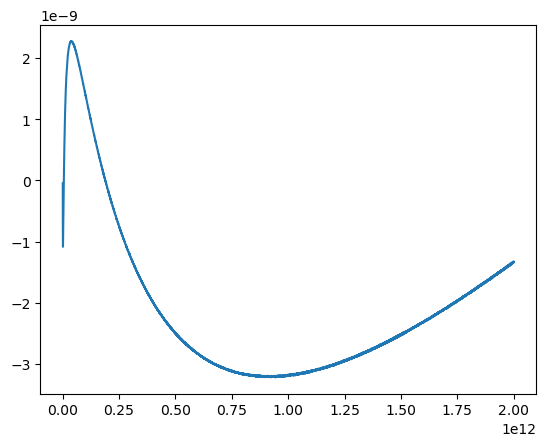}
    \end{minipage}
    \hfill
    \begin{minipage}[b]{0.48\textwidth}
        \centering
        \includegraphics[width=\textwidth]{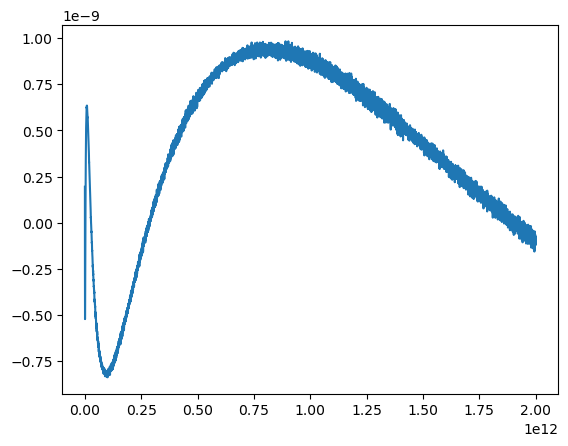}
    
    \end{minipage}
    
    \caption{Final residue in a large domain for 1D Burgers. Upper: weak asymptotics; down: exact asymptotics; left: $\lambda=0.4$; right: $\lambda=0.5$.}
    \label{fig:burgers_final}
\end{figure}

\begin{figure}[htbp]
    \centering
    \begin{minipage}[b]{0.48\textwidth}
        \centering
        \includegraphics[width=\textwidth]{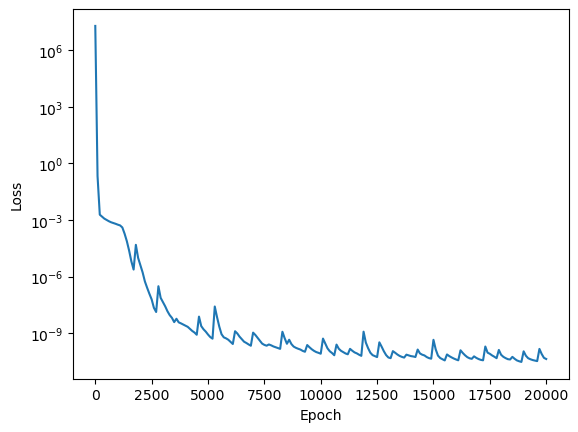}
    
    \end{minipage}
    \hfill
    \begin{minipage}[b]{0.48\textwidth}
        \centering
        \includegraphics[width=\textwidth]{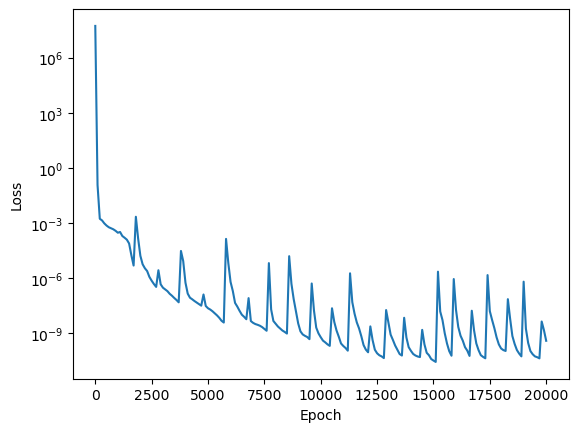}
    \end{minipage}

    \vspace{1em}

    \begin{minipage}[b]{0.48\textwidth}
        \centering
        \includegraphics[width=\textwidth]{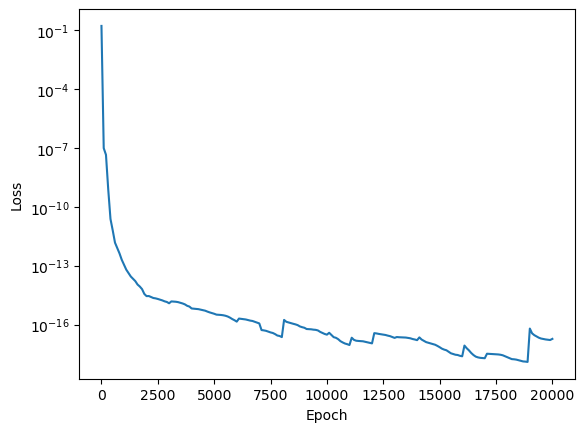}
    \end{minipage}
    \hfill
    \begin{minipage}[b]{0.48\textwidth}
        \centering
        \includegraphics[width=\textwidth]{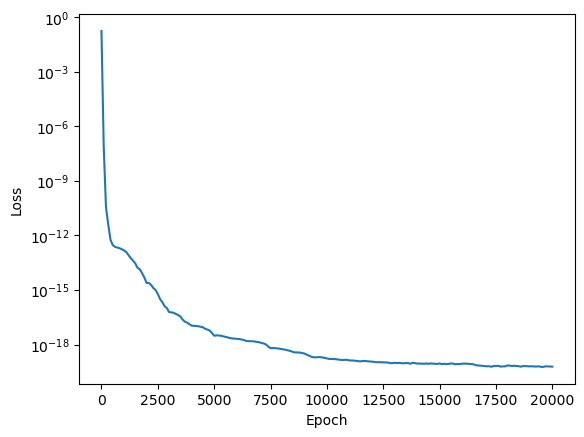}
    
    \end{minipage}
    
    \caption{Trajectory of losses for 1D Burgers. Upper: weak asymptotics; down: exact asymptotics; left: $\lambda=0.4$; right: $\lambda=0.5$.}
    \label{fig:burgers_loss}
\end{figure}
\subsection{Boussinesq equation}
For the 2D Boussinesq equation on the half plane, in vorticity form with the self-similar ansatz, we get the following profile equations for $( \Omega,U_1,U_2,\Phi,\Psi)$ as in \cite{wang2023asymptotic}:\begin{align*}
    \Omega+((1+\lambda)(y_1,y_2)^T+(U_1,U_2)^T)\cdot\nabla\Omega&=\Phi,\\
    (2+\partial_{y_1}U_1)\Phi+((1+\lambda)(y_1,y_2)^T+(U_1,U_2)^T)\cdot\nabla\Phi&=-\partial_{y_1}U_2\Psi,\\
    (2+\partial_{y_2}U_2)\Psi+((1+\lambda)(y_1,y_2)^T+(U_1,U_2)^T)\cdot\nabla\Psi&=-\partial_{y_2}U_1\Phi,\\
    \partial_{y_1}U_1+\partial_{y_2}U_2=0, \quad \Omega = \partial_{y_1}U_2-\partial_{y_2}U_1, \quad \partial_{y_1}\Psi &= \partial_{y_2}\Phi,
\end{align*}
where $( \Omega,U_1,\Phi)$ are odd and $(U_2,\Psi)$ are even in $y_1$ and we are in the half plane $y_2\geq0$.

For the boundary conditions, we impose a non-penetration boundary condition $U_2(y_1,0)=0$ along with decaying weak asymptotics at the far field, with Dirichlet boundary conditions $\Phi=\Psi=0$ and Neumann boundary conditions for the velocity field $\nabla(U_1,U_2)^T=0$. For the nondegeneracy condition, we impose $\partial_{y_1}\Omega(0,0)=-1$ and use Taylor expansion to enforce a hard constraint as in Subsection \ref{subsec:hard}. We find that enforcing a hard constraint is much more effective to avoid converging to a trivial solution than enforcing soft constraints.

For each function, we use a $7$-layer MLP with width $30$, hard constraints on parity, and activation function $\mathrm{SiLU}=\frac{x}{1 + e^{-x}}$ to better model the growth at the far field. For sampling, we sample $1000$ points on each boundary of the square $(z_1,z_2)\in[0,30]^2$, and $5000$ points each for the interior and smoothness losses, where we sample  $(z_1,z_2)$ with equal probability uniformly on $[0,30]^2$ and $[0,5]^2$ for the interior loss and with equal probability uniformly on $[0,3]^2$ and $[0,0.5]^2$ for the smoothness loss, ensuring smoothness near the origin. Again, we are computing effectively in a large domain $[0,5\times 10^{12}]^2$ in the $y$ variables.

For optimization, we use Adam for $10000$ epochs with resampling, followed by the optimizer SSBroyden 1 as in \cite{urban2025unveiling} with $40000$ epochs and resampling every 1000 epochs. The learning rate of Adam is set to be $0.001$ for the functions and $0.1$ with $\beta=(0.9,0.9)$ for $\lambda$, with a decay of $0.9$ after $5000$ epochs.

We present the final profiles and the residue of the PDEs near the origin respectively in Figure \ref{fig:bous_prof} and \ref{fig:bous_res}, and the evolution of losses in Figure \ref{fig:traj}.
Compared to \cite{wang2023asymptotic}, we achieve a training loss of $4$ digits smaller and equation residues of $2$ digits smaller. We remark that due to computational constraints and the cost of a full batch optimizer involving the approximation of the Hessian matrix, this is the largest neural network affordable. We use 10 days of CPU time on a MacPro 2019 with 2.5GHz 28-core Intel Xeon W processor.

\begin{figure}[htbp]
    \centering
    \begin{minipage}[b]{0.19\textwidth}
        \centering
        \includegraphics[width=\textwidth]{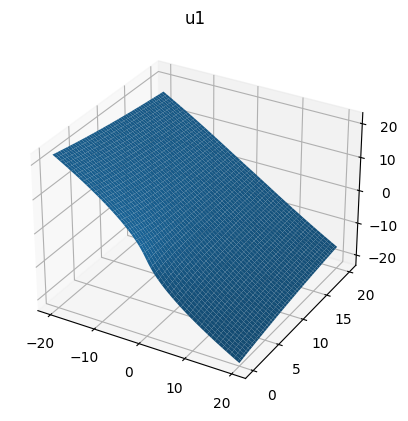}
    
    \end{minipage}
    \hspace{0.1mm}
    \begin{minipage}[b]{0.19\textwidth}
        \centering
        \includegraphics[width=\textwidth]{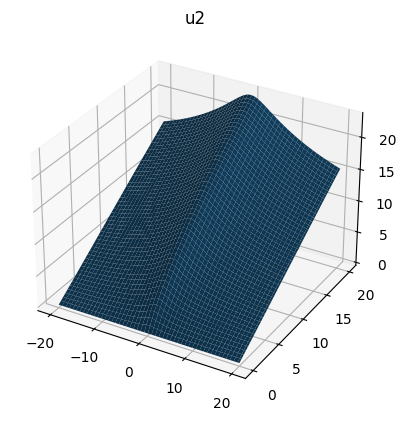}
    \end{minipage}
\hspace{0.1mm}
 \begin{minipage}[b]{0.19\textwidth}
        \centering
        \includegraphics[width=\textwidth]{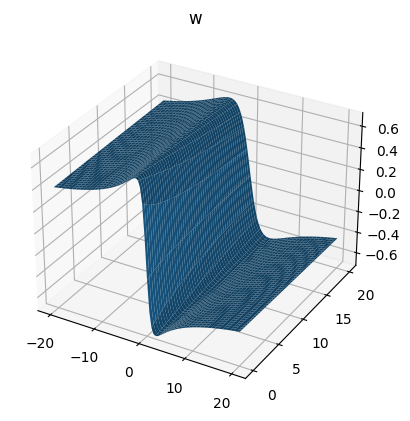}
    
    \end{minipage}
    \hspace{0.1mm}
    \begin{minipage}[b]{0.19\textwidth}
        \centering
        \includegraphics[width=\textwidth]{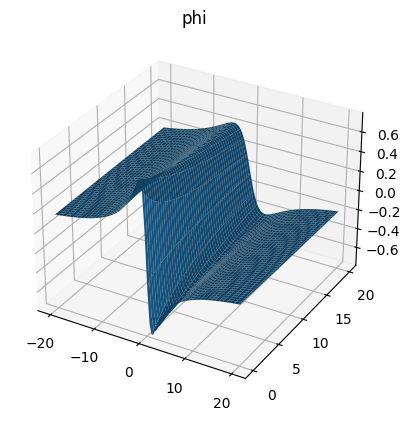}
    
    \end{minipage}
    \hspace{0.1mm}
    \begin{minipage}[b]{0.19\textwidth}
        \centering
        \includegraphics[width=\textwidth]{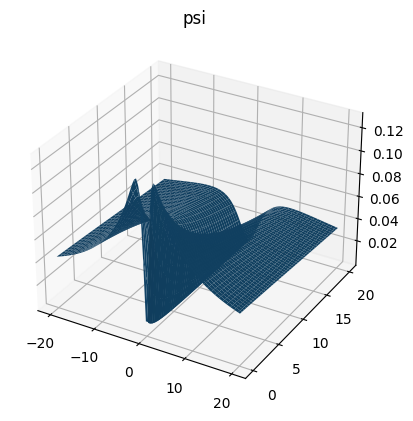}
    \end{minipage}
    
    \caption{Final profiles for 2D Boussinesq}
    \label{fig:bous_prof}
\end{figure}
 
\begin{figure}[htbp]
    \centering
    \begin{minipage}[b]{0.32\textwidth}
        \centering
        \includegraphics[width=\textwidth]{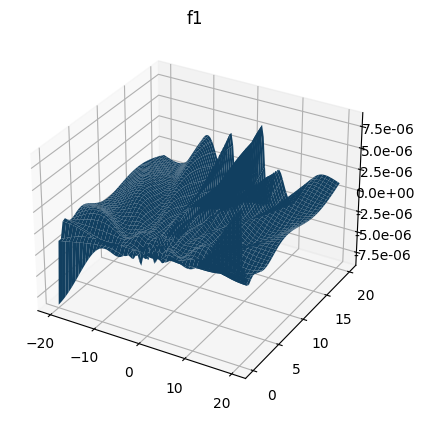}
    
    \end{minipage}
    \hspace{0.1mm}
    \begin{minipage}[b]{0.32\textwidth}
        \centering
        \includegraphics[width=\textwidth]{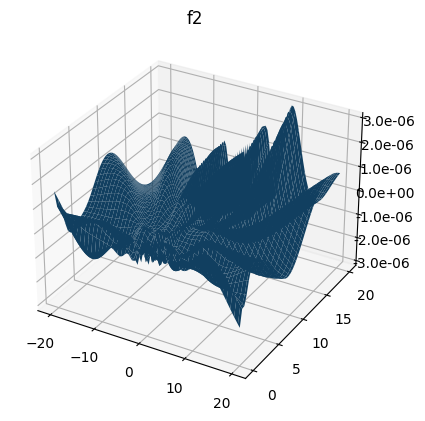}
    \end{minipage}
\hspace{0.1mm}
 \begin{minipage}[b]{0.32\textwidth}
        \centering
        \includegraphics[width=\textwidth]{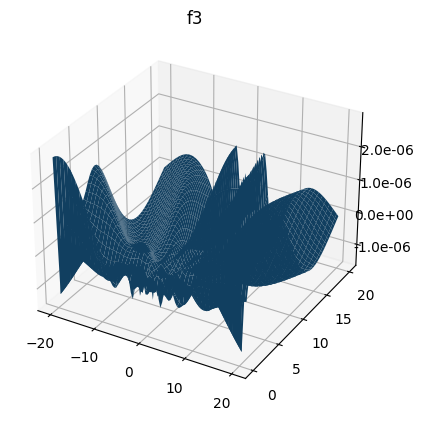}
    
    \end{minipage}
    \vspace{1em}
    \begin{minipage}[b]{0.32\textwidth}
        \centering
        \includegraphics[width=\textwidth]{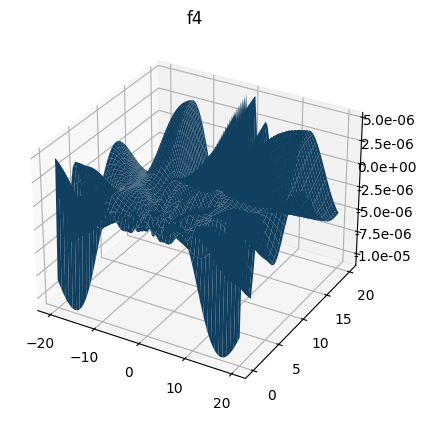}
    
    \end{minipage}
    \hspace{0.1mm}
    \begin{minipage}[b]{0.32\textwidth}
        \centering
        \includegraphics[width=\textwidth]{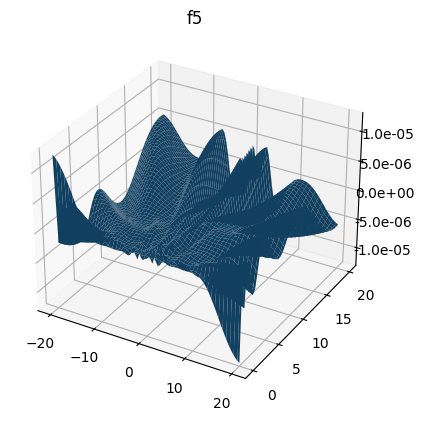}
    
    \end{minipage}
    \hspace{0.1mm}
    \begin{minipage}[b]{0.32\textwidth}
        \centering
        \includegraphics[width=\textwidth]{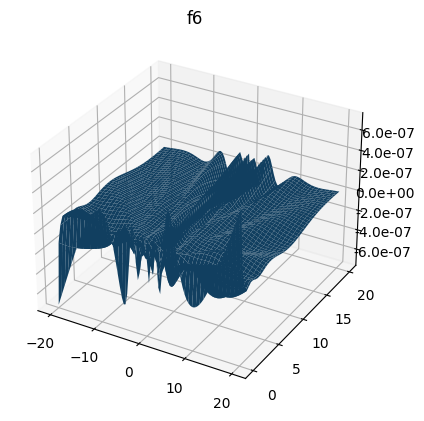}
    \end{minipage}
    
    \caption{Final equation residues for 2D Boussinesq}
    \label{fig:bous_res}
\end{figure}

\begin{figure}
    \centering
\includegraphics[width=0.5\linewidth]{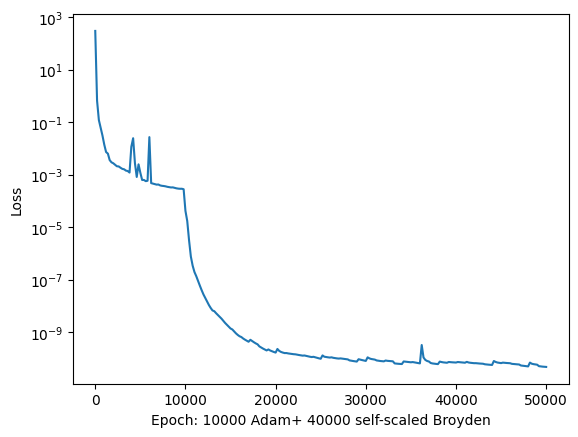}
    \caption{Trajectory of losses for 2D Boussinesq:  
$10000$ Adam iterations followed by $40000$ self-scaled Broyden iterations.}
    \label{fig:traj}
\end{figure}
    \section{Conclusions and Future Work}
    We demonstrate the importance of enforcing appropriate asymptotics, enforcing hard constraints, and adopting a better optimizer for solving PDEs using neural networks on an infinite domain. We achieve better accuracy for problems crucial to the study of singularity formations. As a future direction, we believe a better enforcement of far-field asymptotics, formulated as hybrid asymptotics, might have the potential of driving PDE residues to machine precision, potentially amenable to rigorous computer-assisted proofs using the profiles identified by the neural networks. Another direction is to use PINO, the idea of operator learning on a range of scaling parameters, to learn a collection of profiles with different scalings.

\section{On Weak Convection Model to 3D Euler and Numerical Stability Analysis}

A natural next step of the present program is to study the weak-convection regime, namely the dynamic-rescaling system in which the transport is weakened to encourage blowup; as in \cite{liu2017spatial}. \begin{align}
\label{eq:new}
\partial_t u_{1}
+ {\varepsilon} \ u^r\,\partial_r u_{1}
+ {\varepsilon} \ u^z\,\partial_z u_{1}
&= 2u_{1}\,\partial_z \psi_{1}  \\
\label{eq:transformed_2}
\partial_t \omega_{1}
+ {\varepsilon} \ u^r\,\partial_r \omega_{1}
+ {\varepsilon} \ u^z\,\partial_z \omega_{1}
&= 2u_{1}\,\partial_z(u_{1})  \\
\label{eq:transformed_3}
-\Bigl[\partial_r^2 + \frac{3}{r}\partial_r + \partial_z^2\Bigr]\psi_{1}
&= \omega_{1}.
\end{align} In this setting, the profile equations retain the main hyperbolic scaling structure while the nonlinear convection terms enter with a small prefactor $\varepsilon$, which makes the model both analytically tractable and still rich enough to capture singular behavior. Our computations and formal asymptotics indicate that this weak-convection model admits a nontrivial singular profile: after rescaling, one finds a coherent stationary state $(U,\Omega,\Psi)$ together with scaling parameters $(\lambda,C,c_u,c_\omega)$ for which the original solution concentrates and the relevant amplitude grows according to the self-similar law. In this sense, the model provides clear evidence of singularity formation, while also preserving enough of the geometric structure of the full equation to serve as an informative intermediate problem between purely leading-order toy models and the fully coupled dynamics. We are able to push to a higher $\varepsilon=0.3$ compared to the reported in \cite{liu2017spatial}, using neural networks.

An important direction for future work, in collaboration with Prof. Anima's group, is to complement the existence and numerical construction of such singular profiles with a stability theory. The main difficulty is that the correct coercive functional is not known a priori: because the singular profile is strongly anisotropic and exhibits nontrivial growth/decay across spatial regions, standard polynomial or isotropic Sobolev weights are unlikely to capture the true dissipative mechanism of the linearized operator. To address this, we propose to parametrize a family of singular weights by a neural network and to learn them through a min--max formulation, maximizing the coercivity while minimizing over all admissible functions. 
Such a neural min--max procedure is attractive because it allows the weight to adapt to the hidden geometry of the singularity, rather than imposing an ansatz by hand. If successful, it would provide a data-informed route toward identifying the correct singular weights, proving linear stability in a sharp weighted space, and ultimately opening the door to a nonlinear stability argument for the weak-convection singularity. We also study the linearized spectrum numerically to infer empirically the stability and instability of the profiles.

\chapter{Second Order Ensemble Langevin Method}
\label{append:ehmc}
We propose a sampling method based on an ensemble approximation of second order Langevin dynamics. {The log target density is appended with a quadratic term in an auxiliary momentum variable and 
damped-driven Hamiltonian dynamics, is introduced; the resulting stochastic
    differential equation is invariant to the Gibbs measure, with marginal on
    the position coordinates given by the target. A preconditioner based on 
covariance under the law of {position coordinates under}
the dynamics does not change this invariance property, and is introduced 
to accelerate convergence to the Gibbs measure.}
        The resulting mean-field dynamics may be approximated by an ensemble 
method; this results in a gradient-free and affine-invariant stochastic 
dynamical system {with desirable provably 
uniform convergence properties across
the class of all Gaussian targets.} Numerical results demonstrate the
potential of the method as basis for a numerical sampler in Bayesian inverse 
problems, {beyond the Gaussian setting.}

\section{Introduction}
    \subsection{Set-up}
    Consider sampling the density $$\pi(q)=\frac{1}{Z_q}\exp(-\Phi(q))\,,$$ where $\Phi: \mathbb{R}^{N} \to \mathbb{R}$ is termed the \emph{potential} function and $Z_q$  the \emph{normalization constant.}
    A broad family of problems can be cast into this formulation; the Bayesian approach to inverse problems provides a particular focus for our work \cite{kaipio2006statistical}. 
        The point of departure for the algorithms considered in this chapter is the following mean-field
    Langevin equation:
\begin{equation}
\label{eqol}
\frac{\mathrm{d} q}{\mathrm{d} t}=-\mathcal{C}(\rho) D \Phi(q)+\sqrt{2 \mathcal{C}(\rho)} \frac{\mathrm{d} W}{\mathrm{d} t}\,,
\end{equation}
where $D$ {denotes} the gradient operator, $W$ is an $N$-dimensional standard Brownian motion, $\rho$ is the density associated to the law of $q$, and $\mathcal{C}(\rho)$ is the covariance under this density. This
constitutes a mean-field generalization \cite{garbuno2020interacting} of the standard Langevin
equation \cite{pavliotis2014stochastic}. Applying a particle approximation to the mean-field
model results in an interacting particle system, and coupled Langevin dynamics 
\cite{leimkuhler2018ensemble,garbuno2020interacting,nusken2019note,garbuno2020affine}. 
The benefit of preconditioning using the covariance is that it leads to
mixing rates independent of the problem, provably for quadratic $\Phi$ and
empirically beyond this setting \cite{garbuno2020interacting}, 
{because of the affine invariance \cite{goodman2010ensemble} of the
resulting algorithms \cite{garbuno2020affine}.}

    In order to {further} accelerate mixing and achieve sampling efficiency, 
    we introduce an auxiliary variable  $p\in \mathbb{R}^{N}$ and consider the Hamiltonian \begin{equation}
    \label{ham}
        \mathcal{H}(z)=\frac{1}{2}\langle p,\mathcal{M}^{-1}p\rangle+\Phi(q)\,,
    \end{equation} where the new state variable $z:=(q^\top, p^\top)^\top \in\mathbb{R}^{2N}$. Define a measure {via its density} on $\mathbb{R}^{2N}$ by
     
\begin{equation}
\label{eq:Hneed}
\Pi(z)=\frac{1}{Z_{q,p}}\exp(-\mathcal{H}(z))\,,
\end{equation}
where  $Z_{q,p}$ is the normalization constant. The marginal distribution of $\Pi$ in the $q$ variable gives the desired distribution $\pi$, i.e. $\int \Pi(z) \mathrm{d}p=\pi(q)$.
    We now aim at sampling the joint distribution. To this end, consider the following underdamped Langevin dynamics in $\mathbb{R}^{2N}$:
\begin{equation}
\label{eqsol}
\frac{\mathrm{d} z}{\mathrm{d} t}=\mathcal{J} D {\mathcal{H}}(z)-\mathcal{K} D {\mathcal{H}}(z)+\sqrt{2 \mathcal{K}} \frac{\mathrm{d} W_0}{\mathrm{d} t}\,,
\end{equation}
with the choices 
\begin{equation}
\label{eqsol2}
\mathcal{J}=\begin{pmatrix}
0 & \mathcal{C} \\
-\mathcal{C} & 0
\end{pmatrix}\,,\quad
\mathcal{K}=\begin{pmatrix}
\mathcal{K}_{1} & 0 \\
0 & \mathcal{K}_{2}
\end{pmatrix}\,.
\end{equation}
Here $W_0$ is a  standard Brownian motion in $\mathbb{R}^{2N}$ with components $W', W \in \mathbb{R}^{N}$. Then we have the following, proved in Subsection
\ref{proofprop0}:

\vspace{0.1in}

\begin{proposition}
\label{prop.prop0}
{Assume that $\mathcal{K}_1$, $\mathcal{K}_2$ are symmetric and 
non-negative definite, and that $\mathcal{C}$ is symmetric positive definite. 
Assume further that $\mathcal{C}$, $\mathcal{K}$ and $\mathcal{M}$ depend
on the law of $z$ under the dynamics defined by \eqref{eqsol} and \eqref{eqsol2},
but are are independent of $z$: all derivatives with respect to $z$
are zero. Then the Gibbs measure $\Pi(z)$ is invariant under the dynamics defined
by \eqref{eqsol}, \eqref{eqsol2}.}
\end{proposition}

\vspace{0.1in}

{In practice, to simulate from such a mean-field model, it will
be necessary to consider a particle approximation of the form
\begin{equation}
\label{eqsolp}
\frac{\mathrm{d} z^{(i)}}{\mathrm{d} t}={J(Z) D_{z} H(z^{(i)};Z)-K(Z) D_{z} H(z^{(i)};Z)}+\sqrt{2 K(Z)} \frac{\mathrm{d} W_0^{(i)}}{\mathrm{d} t}\,,
\end{equation}
for the set of $I$ particles $Z=\{ z^{(i)} \}_{i=1}^I,$ and where {$M(Z), K(Z), J(Z)$ are
appropriate empirical approximations of $\mathcal{M}(\rho), \mathcal{K}(\rho), \mathcal{J}(\rho)$ based on replacing $\rho$ by $\rho^I$ where}
$$\rho^I = {\frac{1}{I}}\sum_{i=1}^{I} \delta_{z^{(i)}}\,,$$}
and the Hamiltonian is given by
\begin{equation}
\label{eq:surelyneed}
H(z;Z)=\frac{1}{2}\langle p,M(Z)^{-1}p\rangle+\Phi(q)\,.
\end{equation} 
Thus
$$D_z H(z;Z)=\bigl(D\Phi(q)^\top, (M(Z)^{-1}p)^\top\bigr)^\top.$$
%We use the notation $H_Z$ to emphasize the dependence on the ensemble $Z$, 
%but when the context is clear we also use the abbreviated notation $H$. 
{Note that $H$ is the appropriate finite particle approximation of
$\mathcal{H},$ given the particle approximation $M$ of $\mathcal{M}.$ The 
dependence of $\mathcal{H}$ on the law of $z$ has been replaced by
dependence on the collection of particles $Z$.\footnote{The reader is asked
to note that collection of particles $Z$ is different from normalization
constant $Z_{q,p}$ appearing in \eqref{eq:Hneed}.}}

\vspace{0.1in}

\begin{remark}
\label{rem:12}
{Unlike \eqref{eqsol}, the equation \eqref{eqsolp} is no longer 
a damped-driven Hamiltonian system; this is because of the dependence
of the Hamiltonian on the particle positions $Z$, through the mass
matrix. Furthermore, its marginal on any 
coordinate $q^{(i)}$ does not necessarily
preserve the desired target measure under the dynamics. However we
expect it to do so \emph{approximately} when $I$ is large. This
justifies the use of algorithms based on \eqref{eqsolp}.}
\end{remark}

\vspace{0.1in}

{In this chapter we will concentrate on a specific choice of mean-field operators within
the above general construction, which we now describe.} Let $\mathcal{C}_q(\rho)$ denote the $q$-marginal in the covariance 
under the law of \eqref{eqsol}. We make the choices
{$\mathcal{K}_{1}=0$, $\mathcal{C}=\mathcal{M}=\mathcal{C}_q(\rho)$, and $\mathcal{K}_2=\gamma\mathcal{C}_q(\rho)$, for a scalar damping parameter $\gamma>0$.} Then the underdamped Langevin dynamics yields
\begin{equation}
\label{eqmom}
\begin{aligned}
\frac{\mathrm{d} q}{\mathrm{d} t}&= p\,, \\
\frac{\mathrm{d} p}{\mathrm{d} t}&=-\mathcal{C}_q(\rho)D \Phi(q)- {\gamma p+\sqrt{2\gamma \mathcal{C}_q(\rho)} \frac{\mathrm{d} W}{\mathrm{d} t}}\,.
\end{aligned}
\end{equation}

{To implement a particle approximation of the}
mean-field dynamics \eqref{eqmom} {we introduce  particles in the form} $z^{(i)}(t)=\bigl( (q^{(i)}(t))^\top, (p^{(i)}(t))^\top\bigr)$  and
we use the {ensemble covariance and mean approximations}
\begin{subequations}
\label{eq:ecov}
\begin{align}
    \mathcal{C}_q(\rho)\approx C_q(Z)&\coloneqq\frac{1}{I} \sum_{i=1}^{I}\left(q^{(i)}-\bar{q}\right) \otimes\left(q^{(i)}-\bar{q}\right)\,,\\
    \bar{q}&:=\frac{1}{I} \sum_{i=1}^{I}q^{(i)}\,.
\end{align}
\end{subequations}
 In order to 
{obtain} affine invariance, we take the generalized square root of the ensemble covariance $C_q(Z)$, {similarly} to \cite{garbuno2020affine}.  We introduce the $N\times I$ matrix $$Q \coloneqq\left(q^{(1)}-\bar{q},q^{(2)}-\bar{q},\cdots,q^{(I)}-\bar{q}\right)\,,$$ which allows us to  
define the empirical covariance and generalized (nonsymmetric) square root  via
$${C}_q(Z)=\frac{1}{{I}} QQ^T\,,\quad
\sqrt{{C}_q(Z)}\coloneqq\frac{1}{\sqrt{I}} Q\,.$$ 
Now with $I$ independent standard Brownian motions $\{ W^{(i)} \}_{i=1}^I\in \mathbb{R}^{I}$, a natural
particle approximation of \eqref{eqmom} is, for $i=1, \dots, I$,
\begin{equation}
\label{ensol}
\begin{aligned}
\frac{\mathrm{d} q^{(i)}}{\mathrm{d} t}&= p^{(i)}\,, \\
\frac{\mathrm{d} p^{(i)}}{\mathrm{d} t}&=-C_q(Z)D \Phi(q^{(i)})- {\gamma p^{(i)}+\sqrt{2 \gamma C_q(Z)} \frac{\mathrm{d} W^{(i)}}{\mathrm{d} t}}\,.
\end{aligned}
\end{equation}

{In subsequent sections we will 
employ an ensemble approximation of $C_q(Z)D \Phi(\cdot)$, as in \cite{garbuno2020interacting}, thereby
avoiding the need to compute adjoints of the forward model; we note also that in the linear case this approximation is exact. We will show that the resulting interacting
particle system has the potential to provide accurate derivative-free inference
for {certain classes of} inverse problems.}

{In the remainder of this section we provide a literature review, we
highlight our contributions, and we outline the structure of the chapter.}

%\begin{theorem}
%For the second order Langevin equation \eqref{SOL}, assume that $J^T=-J$, $K^T=K$, $K\geq 0$, and that $(\nabla\cdot K)(z)=0$, then the fokker planck
%\end{theorem}

\subsection{Literature review}
\label{ssec:lit}
The overdamped Langevin equation is the canonical SDE that is invariant with respect to {a given target density}. Many sampling algorithms are built upon this idea, and in particular, it is shown to govern a large class of Monte Carlo Markov Chain (MCMC) methods; see \cite{roberts2001optimal,ottobre2011asymptotic,ma2015complete}.
To enhance mixing and accelerate convergence, a second-order method, Hybrid Monte Carlo (HMC, also referred to as Hamiltonian Monte Carlo) \cite{duane1987hybrid,neal1994improved} has been proposed, leading
to underdamped Langevin dynamics. There have been many attempts to justify the empirically
observed fast convergence speed of second-order methods in comparison with first-order methods \cite{betancourt2017conceptual}. Recently a quantitative convergence rate is established in \cite{cao2019explicit}, showing that the underdamped Langevin dynamics converges faster than the overdamped Langevin dynamics when the log of the target $\pi$ has a small Poincar\'e constant; see also \cite{eberle2019couplings}. 

{The idea of introducing preconditioners within the context of interacting particle systems used for sampling is developed in  \cite{leimkuhler2018ensemble}.} Preconditioning via ensemble covariance {is} shown to boost convergence and numerical performance \cite{garbuno2020interacting}. Other choices of preconditioners in sampling will lead to different forms of SDEs and associated Fokker-Planck equations with different structures, which can result in different sampling methods effective in different scenarios; see for example \cite{lindsey2021ensemble}.
Affine invariance is introduced in \cite{goodman2010ensemble} where it is argued that
this property leads to desirable convergence properties for interacting
particle systems used for sampling: methods that satisfy affine invariance are invariant under an affine change of coordinates, and are thus uniformly effective for problems that can be rendered  well-conditioned under an affine 
transformation; {in particular for the sampling of the class of
all Gaussian measures.} An affine invariant version of the mean-field 
underdamped Langevin dynamics of \cite{garbuno2020interacting} 
is proposed in \cite{nusken2019note,garbuno2020affine}.

     {Kalman methods have shown wide success in state estimation problems since their
     introduction by Evensen; see \cite{evensen2009data} for an overview of the field,
     and the papers \cite{reich2011dynamical,iglesias2013ensemble} for discussion of
     their use in inverse problems. Using the ensemble covariance as a preconditioner  leads to affine invariant \cite{garbuno2020affine} and gradient-free 
     \cite{garbuno2020interacting} approximations of Langevin dynamics; this is desirable in practical computations in view of the intractability of derivatives in many large-scale models arising in science and engineering \cite{cleary2021calibrate,haber2018never,kovachki2019ensemble,huang2021unscented}. 
     See \cite{bergemann2010mollified,schillings2017analysis} for analysis of these methods and, in the context of continuous data-assimilation, see \cite{del2017stability,bergemann2012ensemble,taghvaei2018kalman}. {There are other
     derivative-free methods that can be derived from the mean-field perspective,
     and in particular consensus-based methods show promise for optimization
     \cite{carrillo2018analytical} and have recently been extended to sampling in
     \cite{carrillo2022consensus}.}}

     Recent work has established the convergence of ensemble preconditioning methods to mean field limits; see for example \cite{ding2021ensemble}. For other works on rigorous derivation of mean field limits of interacting particle systems, see \cite{sznitman1991topics,carrillo2010particle,jabin2017mean}. For underpinning theory of Hamiltonian-based sampling, see
     \cite{bou2017randomized,bou2018geometric,livingstone2019geometric,betancourt2017geometric}.

\subsection{Our contributions}
\label{ssec:con}
The following contributions are made in this chapter:
\begin{enumerate}
    \item We introduce an underdamped second order mean field Langevin dynamics, with a covariance-based preconditioner. 
    \item {In the case of {Bayesian inverse problems defined by}
    a linear forward map, we show that 
    that {this mean field model} preserves Gaussian distributions under time evolution and, if initialized at
    a Gaussian, converges to the desired target at {a} rate independent of the linear map.}
    \item {We introduce finite particle approximations of the mean field model, resulting in an affine invariant method.}
    \item For {Bayesian inverse problems}, we introduce a gradient-free approximation of the algorithm, based on ensemble Kalman methodology.
    \item {In the context of {Bayesian inverse problems} we provide numerical examples to demonstrate that the algorithm resulting
    from the previous considerations has desirable sampling properties.}
\end{enumerate}

In Section \ref{2}, we introduce the inverse problems context that motivates us.
In Section \ref{3} we discuss the equilibrium distribution of the mean field
model.
Section \ref{EKA} introduces the ensemble Kalman approximation of the
finite particle system; {and in that section we also
demonstrate affine invariance of the resulting method.} Section \ref{2l} presents analysis of the finite particle system in the case of linear inverse problems, where the ensemble Kalman approximation is exact; we demonstrate
that the relaxation time to equilibrium is independent of the specific linear inverse
problem considered, a consequence of affine invariance. In Section \ref{4} we provide numerical results which demonstrate the efficiency and potential value of our method, 
and in Section \ref{5} we draw conclusions. Proofs of the propositions are given in the appendix, Section \ref{6}.

\section{Inverse Problem}\label{2}

Consider the Bayesian inverse problem of finding $q$ from an observation $y$ determined by the forward model 
$$y=\mathcal{G}(q)+\eta\,.$$
{Here $\mathcal{G}:\mathbb{R}^N \to \mathbb{R}^J$ is a {(in general)} nonlinear forward map. We assume a prior zero-mean Gaussian $\pi_{0}=\mathsf{N}(0, \Gamma_0)$ on unknown $q$ and assume that
{the random variable} $\eta \sim \mathsf{N}(0, \Gamma)$, representing measurement error, is independent of the prior on $q$. 
We also assume that $\Gamma, \Gamma_0$ are positive definite.} Then by Bayes rule, the posterior density that we aim to sample is given by\footnote{{In what follows $\|\cdot\|_{C}=\|C^{-\frac12}\cdot\|$, with analogous notation for the
inducing inner-product, for any positive definite covariance $C$ and for $\|\cdot\|$ the Euclidean norm.}}
{$$\pi(q) \propto \exp \Bigl(-\frac{1}{2}\|y-\mathcal{G}(q)\|_{\Gamma}^{2}\Bigr) \pi_{0}(q) \propto {\exp}(-\Phi(q))\,,$$
where potential function $\Phi(q)$ has the following form:}
\begin{equation}
\label{invf}
\Phi(q)=\frac{1}{2}\|y-\mathcal{G}(q)\|_{\Gamma}^{2}+\frac{1}{2}\|q\|_{\Gamma_0}^{2}\,.
\end{equation}
 
{In the linear case when $\mathcal{G}(q)=Aq$, $\Phi(q)$ is quadratic
 and the gradient $D\Phi(q)$ can be written as a linear function:
\begin{subequations}
\label{eq:quad}
\begin{align}
\Phi(q)&=\frac{1}{2}\|y-Aq\|_{\Gamma}^{2}+\frac{1}{2}\|q\|_{\Gamma_0}^{2}\,,\\
D\Phi(q)&=B^{-1}q-c\,,\\
B&=(A^T\Gamma^{-1} A+\Gamma_0^{-1})^{-1}\,,\quad
c=A^T\Gamma^{-1} y\,.
\end{align}
\end{subequations}
{In this linear setting,} the posterior distribution $\pi(q)$ is the Gaussian $\mathsf{N}(Bc,B)$.}

\section{{Equilibrium Distributions for the Mean Field Fokker-Planck Equation}}
\label{3}
{The mean-field underdamped Langevin equation \eqref{eqsol} has an associated 
nonlinear and nonlocal Fokker-Planck equation giving the evolution of the law of particles
$z(t)$, denoted $\rho(z,t)$. The equation for this law is
{(see proof in subsection \ref{proofprop0}.)}
\begin{equation}
\label{fp}
\partial_t \rho=\nabla \cdot \bigl((\mathcal{K}-J)(\rho\nabla {\mathcal{H}}+\nabla\rho)\bigr)\,.
\end{equation}
{By Proposition \ref{prop.prop0} this Fokker-Planck equation has $\Pi(z)$ as its equilibrium; this follows as for standard linear Fokker-Planck equations \cite{pavliotis2014stochastic} since the dependence of the sample paths on $\rho$ involves only the mean and covariance; see the proof in subsection \ref{proofprop0}
and see also \cite{duncan2017using, graham1977covariant,jiang2004mathematical,risken1996fokker,pavliotis2014stochastic} for discussions on how to derive Fokker-Planck equations with a prescribed stationary distribution. 
In the specific case \eqref{eqmom}, recall that $\mathcal{J}$ and $\mathcal{K}$ are 
given by}}
\begin{equation} 
\mathcal{J}=\begin{pmatrix}
0 & \mathcal{C}_q(\rho) \\
-\mathcal{C}_q(\rho) & 0
\end{pmatrix}\,,\quad
\mathcal{K}=\begin{pmatrix}
0 & 0 \\
0 & {\gamma}\mathcal{C}_q(\rho)
\end{pmatrix}\,.
\end{equation}
{These choices satisfy the assumption of Proposition \ref{prop.prop0}.
We approximate \eqref{eqmom} by the interacting particle system \eqref{ensol}.
In this context, we note Remark \ref{rem:12} to motivate computational methods
based on integrating \eqref{ensol}.}

%\begin{remark}
%    The approximation of \eqref{eqmom} by the interacting particle
%system \eqref{ensol} does not result in an equilibrium distribution
%with marginals on the position coordinates given by independent products of $\pi$ exactly. We believe it to be close to independent products of $\pi$  in the regime where we have a reasonable number of particles. 
%\end{remark}

\section{Ensemble Kalman Approximation}\label{EKA}
  
\subsection{Derivatives via differences}
\label{dvd}

{We now make the ensemble Kalman approximation to approximate the gradient term by differences, as in \cite{garbuno2020interacting}:
\[D \mathcal{G}\left(q^{(i)}\right)\left(q^{(k)}-\bar{q}\right) \approx\left(\mathcal{G}\left(q^{(k)}\right)-\bar{\mathcal{G}}\right)\,,\]
where $\bar{\mathcal{G}}:=\frac{1}{I} \sum_{k=1}^{I} \mathcal{G}\left(q^{(k)}\right)$. 
Invoking this approximation within \eqref{ensol}, using the specific form
\eqref{invf} of $\Phi$,
yields the following system of interacting particles in $\mathbb{R}^N$, for $i=1, \dots, I:$}
\begin{equation}
    \label{encgf}
\begin{aligned}
\dot{q}^{(i)}&= p^{(i)}\,,  \\
\dot{p}^{(i)}&=-C_q(Z)\Gamma_0^{-1} {q}^{(i)}
-\frac{1}{I} \sum_{k=1}^{I}\langle  \mathcal{G}(q^{(k)})-\bar{\mathcal{G}}, \mathcal{G}(q^{(i)})-y\rangle_{\Gamma} q^{(k)}- {\gamma}p^{(i)}+\sqrt{2{\gamma} C_q(Z)} \dot{W}^{(i)}\,.
\end{aligned}
\end{equation}  
{We will use this system as the basis of all our numerical experiments.}

\subsection{Affine invariance}
\label{2a}

In this subsection, we show the affine invariance property \cite{goodman2010ensemble,garbuno2020affine,leimkuhler2018ensemble}  for the Fokker-Planck equations in the mean-field regime \eqref{fp}, for the particle equation in the mean-field regime \eqref{eqmom}, for the ensemble approximation \eqref{ensol}, and the gradient-free approximation  \eqref{encgf}. For simplicity of presentation, we only state the results in the case of ensemble approximation, and the mean-field case is a straightforward analogy upon dropping all of the particle superscripts.

\vspace{0.1in}

\begin{definition}[Affine invariance for particle formulation]
\label{af}
We say a particle formulation is affine invariant, if under all affine 
transformations of the form 
\begin{equation}
\label{eq:aft}
q^{(i)}=A v^{(i)}+b\,,\quad p^{(i)}=A u^{(i)}\,,
\end{equation}
the equations on the transformed particle systems are given by the same equations with $q^{(i)}$, $p^{(i)}$ replaced by $v^{(i)}$, $u^{(i)}$ respectively, and with potential $\Phi$ replaced by $\tilde{\Phi}$ via \[\tilde{\Phi}(v^{(i)})=\Phi(q^{(i)})=\Phi(Av^{(i)}+b)\,.\]
Here $A$ is any invertible matrix and $b$ is a vector.
\end{definition}

\vspace{0.1in}

\begin{definition}[Affine invariance for Fokker-Planck equation]
\label{af1}
We say a Fokker-Planck equation is affine invariant, if under all affine transformations of the form \[q^{(i)}=A v^{(i)}+b\,,\quad p^{(i)}=A u^{(i)}\,,\]
the equations on the pushforward PDF $\tilde{\rho}^I$ are given by the same equation on ${\rho}^I$ with $q^{(i)}$, $p^{(i)}$ replaced by $v^{(i)}$, $u^{(i)}$ respectively, and with Hamiltonian $H$ replaced by $\tilde{H}$ via
\[ \tilde{H}(v^{(i)}, u^{(i)})=H(q^{(i)}, p^{(i)})=H(Av^{(i)}+b,Au^{(i)})\,.\]
Here $A$ is any invertible matrix and $b$ is a vector.
\end{definition}

\vspace{0.1in}

The key dynamical systems introduced in this chapter are affine invariant:

\vspace{0.1in}

\begin{proposition}
\label{aft}
The particle formulations \eqref{eqmom}, \eqref{ensol} and \eqref{encgf} are affine invariant. The Fokker-Planck equation \eqref{fp} is also affine invariant.
\end{proposition}

\vspace{0.1in}

{We defer the proof to Subsection \ref{proofaf}. The significance of affine invariance
is that it implies that the rate of convergence is preserved under affine transformations.
The proposed methodology is thus uniformly effective for problems that become well-conditioned under an affine transformation.
Proposition \ref{afli}, which follows in the next section, illustrates this property
in the setting of linear forward map $\mathcal{G}(\cdot).$}

\vspace{0.1in}

\begin{remark}
The affine invariance of the methodology introduced
in \cite{leimkuhler2018ensemble} involves a definition different
from that in Definition \ref{af}. In particular \eqref{eq:aft}
is replaced by 
\begin{equation}
\label{eq:aft2}
q^{(i)}=A v^{(i)}+b\,,\quad p^{(i)}=u^{(i)}\,,
\end{equation} 
\end{remark}
\section{Mean Field Model for Linear Inverse Problems}
\label{2l}

{We consider the mean field SDE \eqref{eqmom}
in the linear inverse problem setting of Section \ref{2} with $\mathcal{G}(q)=Aq;$ thus
\eqref{eq:quad} holds.  We note that $B$ in \eqref{eq:quad}
is both well-defined and symmetric positive definite since
$\Gamma_0, \Gamma$ are assumed to be symmetric positive definite.
{The two Propositions \ref{afli}, \ref{prop3}
demonstrate problem-independent rates of convergence, across the
set of all linear Gaussian inverse problems; this is a consequence} 
of affine invariance which in turn is a consequence
of our choice of preconditioned mean field system.

In the setting of the linear inverse problem, the mean field model \eqref{eqmom} reduces to
\begin{equation}
\label{linearp}
\begin{aligned}
\dot{q}&= p\,, \\
\dot{p}&=-\mathcal{C}_q(\rho)(B^{-1}q-c) - {\gamma}{p}+\sqrt{2 {\gamma}\mathcal{C}_q(\rho)}\dot{W}\,.
\end{aligned}
\end{equation}
We prove the following result about this system in Subsection \ref{proof3.3}:}

\vspace{0.1in}

\begin{proposition} 
\label{afli}
Write the mean $m(\rho)$ and the covariance $\mathcal{C}(\rho)$ of the law $\rho(z)$  of particles in equation \eqref{linearp} in the block form
$$m(\rho)=\begin{pmatrix}
m_q(\rho) \\
m_p(\rho)
\end{pmatrix}\,, \quad \mathcal{C}(\rho)=\begin{pmatrix}
\mathcal{C}_q(\rho) & \mathcal{C}_{q,p}(\rho) \\
\mathcal{C}_{q,p}^T(\rho) & \mathcal{C}_p(\rho)
\end{pmatrix}\,.$$ The evolution of the mean and covariance is prescribed by the following system of ODEs:
\begin{equation}
\label{vare}
\begin{aligned}
\dot{m_q}&= m_p\,, \\
\dot{m_p}&=-\mathcal{C}_q(B^{-1}m_q-c) - {\gamma}m_p\,,\\
\dot{\mathcal{C}_q}&=\mathcal{C}_{q,p}+\mathcal{C}_{q,p}^T\,,\\
\dot{\mathcal{C}_{p}}&=-\mathcal{C}_{q} B^{-1}\mathcal{C}_{q,p}-(\mathcal{C}_{q} B^{-1}\mathcal{C}_{q,p})^T- 2{\gamma}\mathcal{C}_{p}+2{\gamma}\mathcal{C}_{q}\,,\\
\dot{\mathcal{C}}_{q,p}&=- {\gamma}\mathcal{C}_{q,p} -\mathcal{C}_{q} B^{-1}\mathcal{C}_{q} +\mathcal{C}_{p}\,.
\end{aligned}
\end{equation}
{The unique steady solution with positive definite covariance is the Gibbs
measure $m_q=Bc, m_p=0, C_{q,p}=0, C_q=C_p=B;$ the marginal on $q$ gives the
solution of the linear Gaussian Bayesian inverse problem.
All other steady state solutions have degenerate covariance, are unstable, and take the
form $m_q=B(c+m), m_p=0, C_{q,p}=0, C_q=C_p=B^{1/2}XB^{1/2}$ for a projection matrix $X$ 
 and $m$ in the nullspace of $C_q$.
 The equilibrium point with positive definite covariance is hyperbolic
and linearly stable and hence its basin of attraction is an open set, containing
the equilibrium point itself, in the set of all mean vectors and 
positive definite
covariances. Furthermore, the mean and covariance converge to this equilibrium,
from all points in its basin of attraction, with a speed independent of $B$ and $c$.}  
\end{proposition} 
\vspace{0.1in}
\begin{remark}
\label{r:5.2}
{Proposition \ref{afli} demonstrates that the convergence speed to the
non-degenerate equilibrium point is independent of the specific linear
inverse problem to be solved; the rate does, however, depend on $\gamma$.
Analysis of the linear stability problem shows that $\gamma\approx1.83$ gives the best local convergence rate; see Remark \ref{optg}. In the case where mean field 
preconditioning is not used, the optimal choice of $\gamma$ for underdamped
Langevin dynamics depends on the linear inverse problem being solved, 
and can be identified explicitly in the scalar setting \cite{pavliotis2014stochastic}; for analysis in the non-Gaussian setting see \cite{chak2021optimal}. Motivated by analogies with the work in
\cite{pavliotis2014stochastic,chak2021optimal} we expect the optimal choice
of $\gamma$ to be problem dependent in the nonlinear case, but motivated
by our analysis in the linear setting we expect to see a good choice
which is not too small or too large and can be identified by straightforward
optimization.}
\end{remark}
\vspace{0.1in}

Now we show that for Gaussian initial data, the solution remains Gaussian and is thus determined by the evolution of the mean and covariance via the ODE \eqref{vare}.
We prove this by investigating the mean field Fokker-Planck equation in the linear case. 
The evolution in time of the law $\rho(z)$ of equation \eqref{linearp} is governed
by the equation
\begin{equation}
\label{linearfp}
\partial_t \rho=\nabla\cdot\left(
\begin{pmatrix}
-p \\
\mathcal{C}_q(\rho)(B^{-1}q-c) + {\gamma}p
\end{pmatrix}\rho+\begin{pmatrix}
0 & -\mathcal{C}_q(\rho) \\
\mathcal{C}_q(\rho) & {\gamma}\mathcal{C}_q(\rho)
\end{pmatrix}\nabla\rho\right)\,.
\end{equation}

{We prove the following result in Subsection \ref{proof3}; by
virtue of Proposition \ref{afli} it establishes
that the eigenvalues that determine the
local stability of the posterior Gaussian density are
the same across all linear Gaussian inverse problems:} 

\vspace{0.1in}

\begin{proposition}
\label{prop3}
{Let $m(t)$, $C(t)$ solve the ODE \eqref{vare} with initial conditions $m_0$ and $C_0$.
Assume that $\rho_0$ is a Gaussian distribution, so that
\[\rho_{0}(z):=\frac{1}{(2 \pi)^{N}}\left(\operatorname{det} {C}_{0}\right)^{-1 / 2} \exp \left(-\frac{1}{2}\left\|z-m_{0}\right\|_{{C}_{0}}^{2}\right)\] 
with mean $m_0$ and covariance $C_0$. Then the Gaussian profile \[\rho(t,z):=\frac{1}{(2 \pi)^{N}}\left(\operatorname{det} {C}({t})\right)^{-1 / 2} \exp \left(-\frac{1}{2}\left\|z-m({t})\right\|_{{C}({t})}^{2}\right)\] solves the Fokker-Planck equation \eqref{linearfp} with initial condition $\rho(0,z)=\rho_{0}(z)$.}
\end{proposition}

\vspace{0.1in}

\section{Numerical Results}
\label{4}

{We introduce, and study, a numerical method for sampling the Bayesian inverse problem of Section \ref{2}; the method is based on numerical time  stepping of the interacting particle system \eqref{encgf}.}
{In this section, we demonstrate that the proposed sampler, which we refer to as
\emph{ensemble Kalman hybrid Monte Carlo} (EKHMC) can effectively approximate posterior distributions for two widely studied inverse problem test cases. We compare EKHMC with its first-order version EKS \cite{garbuno2020interacting} and a gold standard MCMC \cite{brooks2011handbook}. EKHMC inherits two major advantages of EKS: (1) exact gradients are not required (i.e., derivative-free); (2) the ensemble can faithfully approximate the spread of the posterior distribution, rather than collapse to a single point as happens with the
basic EKI algorithm \cite{schillings2017analysis}.
Furthermore, we show empirically that EKHMC can obtain samples of  similar quality to EKS, and has faster convergence than EKS. We detail our numerical time-stepping
scheme in the first subsection,
before studying two examples (one low dimensional, one a PDE inverse problem for a field) in
the subsequent subsections.}

\subsection{Time integration schemes}

We employ a splitting method to integrate the stochastic {dynamical} system
given by equation~(\ref{encgf}): {the first capturing the finite
particle approximation of the Hamiltonian evolution,} 
and the second capturing an OU process in momentum space.  
{The Hamiltonian evolution follows the equation}:
\begin{equation}\label{exp:Ham}
\begin{aligned}
\dot{q}^{(i)}&= p^{(i)}\,,  \\
\dot{p}^{(i)}&= F_{H}:= -C_q(Z)\Gamma_0^{-1} {q}^{(i)}
-\frac{1}{I} \sum_{k=1}^{I}\langle  \mathcal{G}(q^{(k)})-\bar{\mathcal{G}}, \mathcal{G}(q^{(i)})-y\rangle_{\Gamma} q^{(k)}\,.
\end{aligned}
\end{equation}
The OU process follows the equation:
\begin{equation}\label{exp:OU}
\begin{aligned}
\dot{q}^{(i)}&= 0\,,  \\
\dot{p}^{(i)}&=-\gamma p^{(i)} +\sqrt{2\gamma C_q(Z)} \dot{W}^{(i)}\,.
\end{aligned}
\end{equation}  
{We implement a symplectic Euler  integrator \cite{sanz2018numerical} for 
the part of the particle system arising from approximation of
the Hamiltonian contribution
to the damped-driven mean-field equations. That is, we take a half 
step $\epsilon/2$ of momentum updates, then a full step $\epsilon$ of position updates, and finally a half step $\epsilon/2$ of momentum updates.
With the ensemble approximation the system is only approximately Hamiltonian; 
the splitting used in symplectic Euler is still well-defined, however. And we
also expect it to perform well in the large particle limit because
the mean-field limit is itself Hamiltonian.} {Let $Z_j$ be the collection of
all position and momentum particles at time $j$: $\{q_{j}^{(i)},p_{j}^{(i)}\}_{i=1}^I.$ Starting from time $j$ this
symplectic Euler integration gives map $Z_{j} \mapsto \hat{Z}_{j}.$
We set $q_{j+1}^{(i)}$ to be the $i^{\text{th}}$ position coordinates of $\hat{Z}_{j}$
and then update the momentum coordinates using the OU process {which provides
the damped-driven component of the mean-field limiting process.} 
The damping coefficient $\gamma>0$ is treated as a hyperparameter of EKHMC. 
{Vector $Z$ is set at the value given 
by output of the preceding symplectic Euler integrator, denoted by $\hat{Z}_{j}$.}
The update of the $i^{\text{th}}$ momentum coordinate, {given by 
solving the OU process exactly in law,} is then
\begin{equation}
\begin{aligned}
    & \tilde{p}^{(i)}_j={\text{exp}}(-\gamma\epsilon)\hat{p}^{(i)}_j\,, \\
    & p^{(i)}_{j+1} = \tilde{p}^{(i)}_j+\eta,\quad \eta\sim \mathsf{N}
    (0,(1-{\text{exp}}(-2\gamma\epsilon))C_q(\hat{Z}_{j}))\, ,
\end{aligned}
\end{equation}
{where $\hat{p}^{(i)}_j$ are the momentum coordinates from $\hat{Z}_{j}$.}
Within the OU process the damping coefficient $\gamma>0$ is treated as a hyperparameter of EKHMC.

{It is posible to consider use of a Metropolis-Hastings (Metropolization)
step to correct the dynamics; however because the underlying continuous
time system is not (for finite number of particles) invariant
with respect to the target, doing so would be very complicated
and so we do not pursue this. Aside from invariance,
Metropolization also imparts stability of the integrator and we may address
this in different ways. Indeed, similarly} 
to~\cite{garbuno2020interacting}, and as there 
with the goal of improving stability
and convergence speed towards posterior distribution, we implement an adaptive step size, i.e., the true step size is rescaled by the magnitude of the ``force field" $F_H$ (defined in equation~(\ref{exp:Ham})):
\begin{equation}
    \tilde{\epsilon}=\frac{\epsilon}{a|F_H|+1}\,.
\end{equation}

\subsection{Low dimensional parameter space}

We follow the example presented in Section 4.3 of the paper~\cite{garbuno2020interacting}. We start by defining the forward map which is given by the one-dimensional elliptic boundary value problem
\begin{subequations}
\begin{align}
    &-\frac{d}{dx}\Bigl({\text{exp}}(u_1)\frac{d}{dx}p(x)\Bigr)=1\,,\quad x\in(0,1)\,,\\
&\quad\quad\quad p(0)=0\,, \quad p(1)=u_2\,.
\end{align}
\end{subequations}
The solution is given explicitly by
\begin{equation}
    p(x) = u_2x+{\text{exp}}(-u_1)\Bigl(-\frac{x^2}{2}+\frac{x}{2}\Bigr)\,,
\end{equation}
The forward model operator $\mathcal{G}$ is then defined by
\begin{equation}
    \mathcal{G}(u) = 
    \begin{pmatrix}
    p(x_1)\\
    p(x_2)
    \end{pmatrix}\,.
\end{equation}
Here $u=\bigl(u_1,u_2\bigr)^T$ is a constant vector that we want to find and we assume that we are given noise measurements $y$ of $p(\cdot)$ at locations $x_1=0.25$ and $x_2=0.75$. Parameters are chosen according to~\cite{garbuno2020interacting}, but we summarize them here for completeness:

\begin{itemize}
    \item noise $\eta\sim\mathsf{N}(0,\Gamma)$, $\Gamma=0.1^2I_2$;
    \item prior $\pi_0=\mathsf{N}(0,\Gamma_0)$, $\Gamma_0=\sigma^2I_2$, $\sigma=10;$
    \item measurement $y=(27.5,79.7)^T;$
    \item number of particles $I=10^3;$
    \item initialization: $(u_1,u_2)\sim \mathsf{N}(-3.5,0.1^2)\times \mathsf{U}(70,110).$
\end{itemize}
Here $\mathsf{U}$ is the uniform distribution.

\begin{figure}
    \centering
    \includegraphics[width=1.0\linewidth]{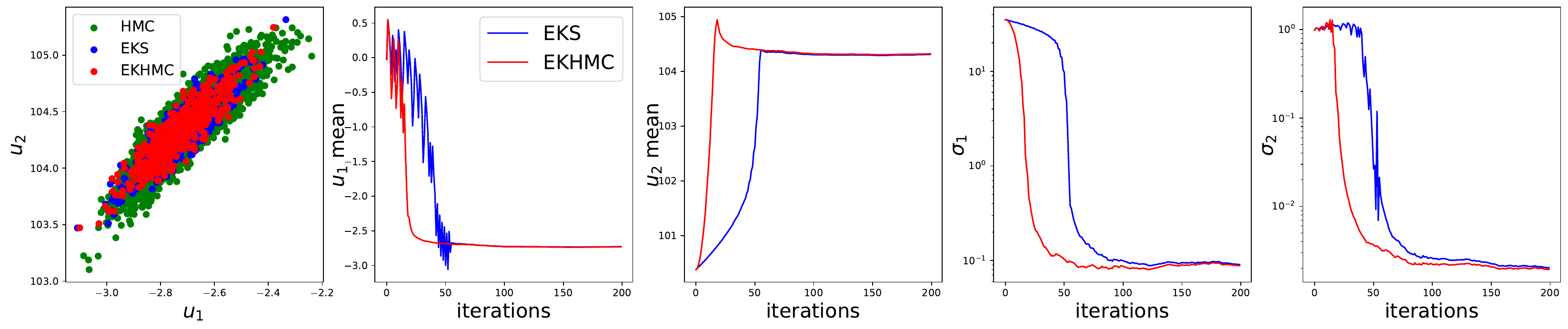}
    \caption{{The low dimensional parameter space example. From left to right: samples; mean $u_1$; mean $u_2$; the first singular value $\sigma_1$; the second singular value $\sigma_2$.}}
    \label{fig:2d}
\end{figure}

\begin{figure}
    \centering
    \includegraphics[width=1\linewidth]{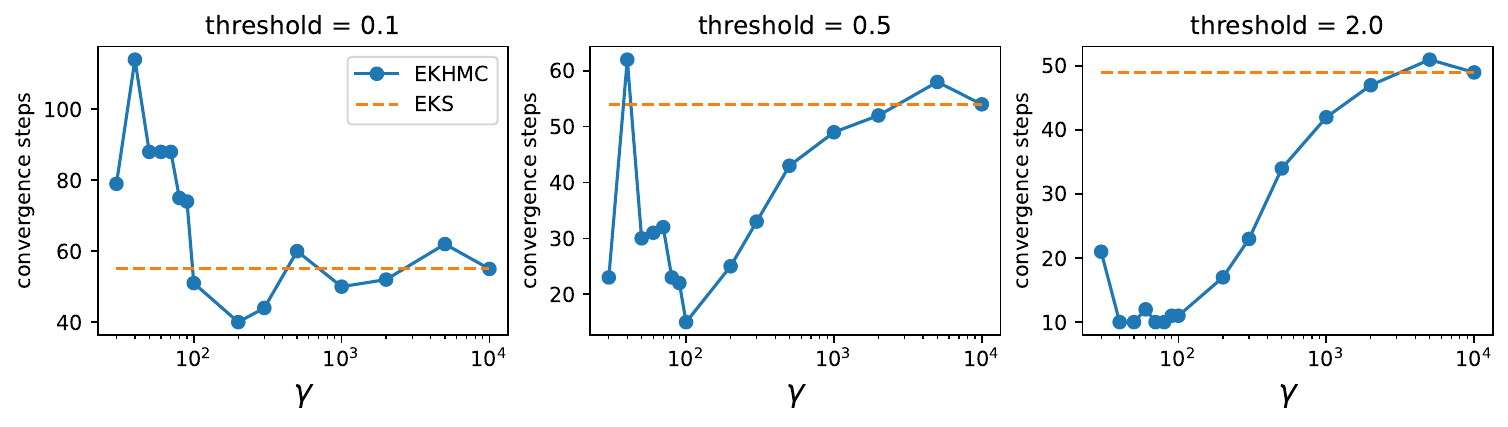}
    \caption{{Convergence time (of $u_2$ mean) as a function of damping coefficient $\gamma$. Left, Middle, Right: thresholds are 0.1, 0.5, 2.0, respectively. The takeaway from these plots is: large damping converges faster eventually (small threshold), while small damping converges faster initially (large threshold).}}
    \label{fig:gamma}
\end{figure}

 We choose $a=0.01$, $\epsilon=0.2$ in both EKS and EKHMC. {We find 
choosing $\gamma=1$ causes overshooting {(the trajectory exhibits 
oscillatory rather than monotonic convergence)}, a phenomenon which
can be ameliorated by increasing to {$\gamma=100$}; indeed this latter
value appears, empirically, to be close to optimal in terms of convergence speed. We conjecture this difference from the linear case, where the
optimal value is $\gamma=1.83$ (see Remark \ref{r:5.2}),
is due to the non-Gaussianity of the desired target: the particles have accumulated considerable momentum when entering the linear convergence regime, and so extra damping is required to counteract this and avoid overshooting. Despite the
desirable problem independence of the optimal $\gamma$ in the linear case,
we expect case-by-case optimization to be needed for nonlinear
problems.} {We also empirically study how $\gamma$ affects the convergence speed for this 2D problem: we sweep $\gamma\in [30,10^4]$, compute the number of steps needed to be close to the convergent point by a threshold, shown in FIG.~\ref{fig:gamma}. When the threshold is small, large $\gamma$ converges faster; when the threshold is large, small $\gamma$ converges faster. The takeaway is that: large $\gamma$ converges faster eventually (small threshold), while small $\gamma$ converges faster initially (large threshold).}

We evolve the ensemble of particles for 200 iterations, and record their positions in the last iteration as the approximation of the posterior distribution, {shown in Figure \ref{fig:2d}.} The EKHMC samples are quite similar to EKS samples, both of which are reasonable approximations of the samples obtained from the gold standard HMC \cite{duane1987hybrid} simulations -- but both approximations of the gold standard miss the full spread of the
true distribution, because of the ensemble Kalman approximation they invoke.
We also compute four {ensemble} quantities to characterize the evolution of EKHMC and EKS: $u_1$ mean, $u_2$ mean, two eigenvalues $\sigma_1$ and $\sigma_2$ of the covariance matrix $C_q(Z)$. EKHMC has faster convergence towards equilibrium than EKS, benefiting from the second-order dynamics.

\subsection{Darcy flow}

\begin{figure}
    \centering
    \includegraphics[width=1.0\linewidth]{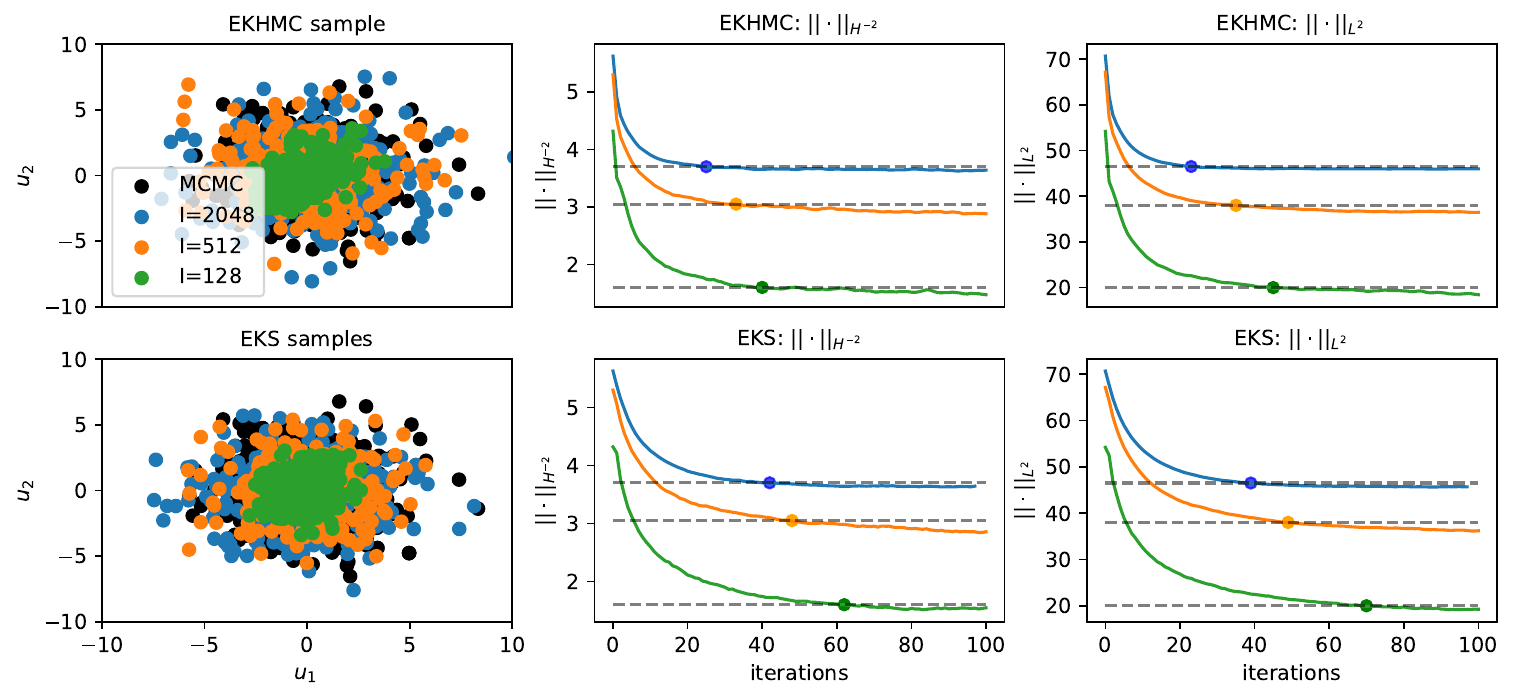}
    \caption{{Darcy flow. Left: samples obtained from EKHMC (top) and EKS (bottom), compared with MCMC. Middle: Evolution of $||u||_{H^{-2}}$ for EKHMC and EKS for different $I=128,512,2048$. Right: The same as the middle, but $||u||_{L^2}$ instead of $||u||_{H^{-2}}$. EKHMC converges faster than EKS.}}
    \label{fig:darcy}
\end{figure}

This example follows the problem setting detailed in~\cite{garbuno2020interacting}. We summarize the essential problem specification here for completeness.{The forward
problem of porous medium flow, defined by permeability $a(\cdot)$ and source term $f(\cdot)$, is to find the pressure field $p(\cdot)$
where $p(\cdot)$ is solution of the following elliptic PDE:}
\begin{equation}
\begin{aligned}
    -\nabla\cdot \Bigl(a(x)\nabla p(x)\Bigr)&=f(x)\,, \quad x\in D=[0,1]^2\,, \\
    p(x)&=0\,, \quad x\in \partial D\,.
\end{aligned}
\end{equation}
We assume that the permeability $a(x)=a(x;u)$ depends on some unknown parameters $u\in\mathbb{R}^d$. 
{The resulting  inverse problem is, given (noisy) pointwise measurements of $p(x)$
on a grid, to infer $a(x;u)$ or $u$.}
We model $a(x;u)$ as a log-Gaussian field. The Gaussian underlying 
this log-Gaussian model has mean zero and has
precision operator defined as $\mathcal{C}^{-1}=(-\triangle+\tau^2\mathcal{I})^\alpha$; here
$\triangle$ is equipped with Neumann boundary conditions on the spatial-mean zero functions. We set $\tau=3$ and $\alpha=2$ in the experiments. Such parametrization yields as Karhunen-Loeve (KL) expansion

\begin{equation}
    {\text{log}} a(x;u)=\sum_{l\in K} u_l\sqrt{\lambda_l}\varphi_l(x)\,,
\end{equation}
where $\varphi_l(x)={\cos}(\pi\langle l,x\rangle)$, $\lambda_l=(\pi^2|\ell|^2+\tau^2)^{-\alpha}$, $K\equiv \mathbb{Z}^2$. {Draws from this
random field are H\"older with any exponent less than one.}
In practice, we truncate $K$ to have dimension $d=2^8$. We generate a truth random field by sampling $u\sim \mathsf{N}(0,I_d)$. We create data $y$ with $\eta\sim \mathsf{N}(0,0.1^2 I_K)$. We choose the prior covariance $\Gamma_0=10^2 I_d$. 

To characterize the evolution of EKHMC and compare it with EKS, we compute two metrics:
\begin{equation}
    d_{H^{-2}}(\cdot)=\sqrt{\frac{1}{I}\sum_{j=1}^I  ||u^{(j)}(t)-\cdot||_{H^{-2}}^2}\,,\quad d_{L^2}(\cdot)=\sqrt{\frac{1}{I}\sum_{j=1}^I  ||u^{(j)}(t)-\cdot||^{2}_{L^2}}\,.
\end{equation}
For these metrics, we use norms 
\begin{equation}
    ||u||_{H^{-2}}=\sqrt{\sum_{l\in K_d}|u_l|^2\lambda_l}\,, \quad ||u||_{L^2}=\sqrt{\sum_{l\in K_d} |u_l|^2}\,,
\end{equation}
where the first is defined in the negative Sobolev space $H^{-2}$, whilst the second in the $L^2$ space.

We set $a=0.01$ and $\epsilon=1.0$ for both EKHMC and EKS. We set $\gamma=1$ in EKHMC; {unlike the previous example, this choice does not lead to over-shooting.}
We simulate the particles for $100$ iterations, and use their positions in the last iteration as the approximation to the posterior distribution. We compare the samples obtained from MCMC (we use the pCN variant on RWMH~\cite{cotter2013mcmc}), EKS ($I=128,512,2048$) and EKHMC ($I=128,512,2048$) in Figure~\ref{fig:darcy}. The evolution of $||u||_{H^{-2}}$ and $||u||_{L^2}$ are plotted to show that convergence is achieved, and that EKHMC converges faster than EKS.

\section{Conclusions}\label{5}

{Gradient-free methods for inverse problems are increasingly important in large-scale
multi-physics science and engineering problems; see \cite{huang2022efficient} and
the references therein. In this chapter
we have provided an initial study of gradient-free methods which leverage the potential
power of Hamiltonian based sampling. Analysis of the resulting method is hard, and
our initial theoretical results leave many avenues open; in particular convergence to
equilibrium {is not established for the underlying nonlinear, nonlocal
Fokker-Planck equation arising from the mean field model,
and optimization of convergence rates over $\gamma$ is not understood,
except in the setting of linear
Gaussian inverse problems.}
The Fokker-Planck equation associated with standard underdamped Langevin dynamics has 
been studied in the context of hypocoercity -- see \cite{bakry1985diffusions,villani2006hypocoercivity} -- and generalization of these methods will be of potential interest. Preconditioned HMC is also proven to converge in \cite{bou2021two}; generalization to ensemble approximations would be of interest. On the computational side there are other approaches to avoiding gradient computations, yet leveraging Hamiltonian structure, which need to be evaluated and compared with what is proposed here \cite{lan2016emulation}.}

\section{Appendix}\label{6}
We present the proofs of {various results from the chapter here.}

\subsection{Proof of Proposition \ref{prop.prop0}}
\label{proofprop0}
\begin{proof}
{The determination of the most general form of diffusion process which
is invariant with respect to a given measure has a long history; see
\cite{duncan2017using}[Theorem 1] for a statement 
and the historical context. Note, however, that this theory
is not developed in the mean-field setting and concerns only the
linear Fokker-Planck equation. However, the dependence of the matrices 
$\mathcal{K}, \mathcal{J}$ and $\mathcal{M}$ on $\rho$ readily allows use
of the approach taken in the linear case. We find it expedient to use
the exposition of this topic in \cite{ma2015complete}. The same
derivation as used to obtain Eq. (5) from \cite{ma2015complete} 
shows\footnote{We use the notational convention concerning divergence
of matrix fields that is standard in continuum mechanics \cite{gonzalez2008first};
this differs from the notational convention adopted in \cite{ma2015complete}.}
that the density $\rho$ associated with the dynamics \eqref{eqsol}, \eqref{eqsol2} satisfies Eq. \eqref{fp}; this follows from the (resp. skew-)
symmetry properties of (resp. $\mathcal{J}$) $\mathcal{K}$ and the fact
that $\mathcal{J}$, $\mathcal{K}$ and  $\mathcal{M}$ are assumed independent
of $z$, despite their dependence on $\rho.$ It is also manifestly the case
that the Gibbs measure is invariant for \eqref{fp} since the
mass term $\mathcal{M}$ appearing in $\mathcal{H}$ is independent of $z$ so that
$$\rho\nabla \mathcal{H}+\nabla\rho=0$$
when $\rho$ is the Gibbs measure
and \eqref{fp} shows that then $\partial_t \rho=0.$}
\end{proof}

\subsection{Proof of Proposition \ref{aft}}
\label{proofaf}
\begin{proof}
In fact, consider the affine transformation 
\[q^{(i)}=A v^{(i)}+b\,,\quad p^{(i)}=A u^{(i)}\,,\] along with  $$\tilde{\Phi}(v^{(i)})=\Phi(q^{(i)})=\Phi(Av^{(i)}+b)\,,\quad\tilde{H}(v^{(i)}, u^{(i)})=H(q^{(i)}, p^{(i)})=H(Av^{(i)}+b,Au^{(i)})\,.$$ Then we have that the gradient term scales like $\nabla_{v^{(i)}}\tilde{\Phi}(v^{(i)})$= $A^T\nabla_{q^{(i)}}\Phi(q^{(i)})$ and the covariance preconditioner scales like $C_v={A^{-1}} C_q {A^{-1}}^T$. The generalized square root scales like $\sqrt{C_v}={A^{-1}} \sqrt{C_q}$. Therefore we can check that affine invariance holds for the particle systems \eqref{ensol} and \eqref{eqmom}. Affine invariance for the gradient free approximation  \eqref{encgf} is more easily seen to be true.  Finally for the Fokker-Planck equation \eqref{fp}, we can check similarly via the scaling of the terms. See a similar argument in the proof of Lemma 4.7 in \cite{garbuno2020affine}.
\end{proof}
\subsection{Proof of Proposition \ref{afli}}
\label{proof3.3}
\begin{proof}
{We can take expectations in \eqref{linearp} to obtain the first two ODEs
in \eqref{vare}. Now define $$\hat{z}=z-m(\rho)\triangleq\begin{pmatrix}
\hat{q} \\
\hat{p}
\end{pmatrix}\,,$$  to obtain the following evolution for $\hat{z}$:
\[\begin{aligned}
\dot{\hat{q}}&= \hat{p}\,, \\
\dot{\hat{p}}&=-\mathcal{C}_q(\rho)B^{-1}\hat{q} - {\gamma}\hat{p}+\sqrt{2 {\gamma}\mathcal{C}_q(\rho)}\dot{W}\,.
\end{aligned}\]
Note that $\mathcal{C}_q=\mathop{\mathbb{E}}[\hat{q}\otimes \hat{q}]$, $\mathcal{C}_p=\mathop{\mathbb{E}}[\hat{p}\otimes \hat{p}]$, and $\mathcal{C}_{q,p}=\mathop{\mathbb{E}}[\hat{q}\otimes \hat{p}]$. We can 
use Ito's formula and the above evolution to derive the last three ODEs in \eqref{vare}. 
The form of the steady state solutions is immediate from setting the right hand side of
\eqref{vare} to zero.}

{Now, we will establish the independence of the essential dynamics of the
mean and covariance on the specific choice of $B,c$. To this end, denote $x_1 = B^{1/2}(B^{-1}m_q-c)$, $x_2 = B^{-1/2}m_p$, $X = B^{-1/2}\mathcal{C}_qB^{-1/2}$, $Y = B^{-1/2}\mathcal{C}_pB^{-1/2}$, $Z = B^{-1/2}\mathcal{C}_{q,p}B^{-1/2}$. Applying this
change of variables we obtain, from \eqref{vare}, the system} 
\begin{equation}
\label{vares}
\begin{aligned}
\dot{x_1}&= x_2\,, \\
\dot{x_2}&=-X x_1 - {\gamma}x_2\,,\\
\dot{X}&=Z+Z^T\,,\\
\dot{Y}&=-X Z-Z^T X- 2{\gamma}Y+2{\gamma}X\,,\\
\dot{Z}&=- {\gamma}Z -X^2 +Y\,.
\end{aligned}
\end{equation}
We notice that the steady state solutions take the form $(x_1,x_2,X,Y,Z)=(w,0,X,X,0)$ for a projection matrix $X^2=X$ and  $w$ such that $Xw=0$.
{The  steady state} with $X=I_N$ and $w=0$ corresponds to the desired posterior mean.
The transformed equation is indeed independent of $B$ and $c$. Therefore the speed of convergence  $(x_1,x_2,X,Y,Z) \to (0,0,I_N,I_N,0)$, within its basin of
attraction, is indeed independent of $B$ and $c.$ The same is true in the original
variables.

{To complete the proof of stability it suffices} to establish the local exponential 
convergence to the
steady solution $(x_1,x_2,X,Y,Z)=(0,0,I_N,I_N,0)$. As a first step, we show that the following system converges to $(X,Y,Z)=(I_N,I_N,0)$:
\begin{equation}
\label{meanco}
\begin{aligned}
\dot{X}&=Z+Z^T\,,\\
\dot{Y}&=-X Z-Z^T X- 2{\gamma}Y+2{\gamma}X\,,\\
\dot{Z}&=-{\gamma}Z -X^2 +Y\,.
\end{aligned}
\end{equation}
{Linearization around $(X,Y,Z)=(I_N,I_N,0)$ {in variables of entries of $(X,Y,Z+Z^T)$} gives the matrix $$\begin{pmatrix}
0 & 0 & 2I_{N^2} \\
2\gamma I_{N^2} & -2\gamma I_{N^2} & -2I_{N^2} \\
-2I_{N^2} & I_{N^2} & -\gamma I_{N^2} 
\end{pmatrix}$$
{whose eigenvalues satisfy $x^3+3\gamma x^2+(2\gamma ^2+6)x+4\gamma=0$ which all have
strictly negative real part for $\gamma>0$, since we can show the unique real eigenvalue is in the interval $(-\gamma,0)$.} Therefore we conclude the local exponential convergence of covariances $X,Y,Z$ in a neighbourhood of $(X,Y,Z)=(I_N,I_N,0)$.
The exponential convergence of means $(x_1,x_2) \to (0,0)$ then follows 
linearizing the first two linear ODEs in the system \eqref{vares} at $X=I_N.$}

{Similarly, for any other steady state $(X,Y,Z)=(X,X,0)$ with a projection matrix $X\neq I_N$, we will show that there is an unstable direction in \eqref{meanco}. In fact, since $X$ is symmetric with only eigenvalues $1$ and $0$, there exists a nonzero vector $v$ such that $Xv=0$. Linearization around $(X,Y,Z)=(X,X,0)$ in the direction $(avv^T,bvv^T,cvv^T)$ gives the following $3\times3$ matrix for scalars $a,b,c$:
$$\begin{pmatrix}
0 & 0 & 2 \\
2\gamma  & -2\gamma  & 0 \\
0 & 1 & -\gamma
\end{pmatrix}$$
{whose eigenvalues satisfy $x^3+3\gamma x^2+2\gamma ^2x-4\gamma=0$. {For
all $\gamma>0$ it may be shown that there exists a real eigenvalue in the 
interval $(0,1)$, {since $f(0)f(1)<0$}; thus there is an unstable direction determined by $(a,b,c)$,} and therefore in the original formulation.}} 
\end{proof}
\begin{remark}
\label{optg}
{We can further investigate the spectral gap around $(X,Y,Z)=(I_N,I_N,0)$, which is the absolute value in the real root of equation $x^3+3\gamma x^2+(2\gamma ^2+6)x+4\gamma=0$.
To be precise, we can show that for $x_0=-\sqrt{12-\sqrt{128}}$ and $\gamma_0=-(4+3x_0^2)/(4x_0)\approx 1.83$, the spectral gap is maximized as $-x_0$ and we expect fastest convergence for our method in the linear setting.}

{To establish this claim we proceed as follows. By the intermediate value theorem, in order to show there always exists a root in $[x_0,0)$ and the spectral gap is at most $-x_0$, we only need to show that $$x_0^3+3\gamma x_0^2+(2\gamma ^2+6)x_0+4\gamma\leq0\,.$$  The claim is true because $x_0<0$ and $2x_0(x_0^3+6x_0)=(4+3x_0^2)^2/4$. By the basic inequality $a^2+b^2\geq 2ab$, we have  $$x_0^3+3\gamma x_0^2+(2\gamma ^2+6)x_0+4\gamma=2x_0\gamma^2+(4+3x_0^2)\gamma+x_0^3+6x_0\leq0\,.$$
The maximal spectral gap is attained if and only if the equality in $a^2+b^2\geq 2ab$ holds, i.e. when $\gamma=\gamma_0$.}

We also plot the spectral gap as a function of $\gamma$ to help visualize the dependence of the rate of convergence on damping for the linear problem; see Figure \ref{fig:damp}. {The clear message from this figure is that there is a natural
optimization problem for $\gamma$ in this linear Gaussian setting; this is used
to motivate searches for optimal parameters in the non-Gaussian setting.}

\begin{figure}
    \centering
    \includegraphics[width=0.6\linewidth]{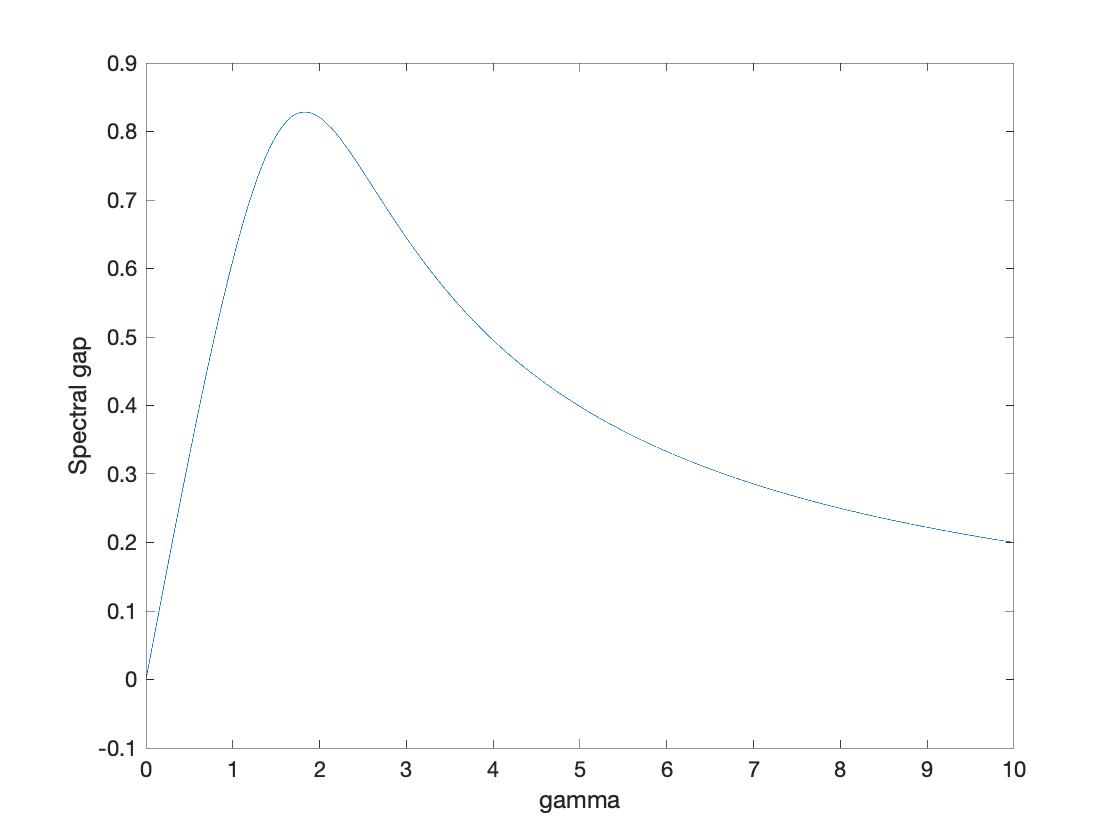}
    \caption{Spectral gap as a function of $\gamma$}
    \label{fig:damp}
\end{figure}
\end{remark}

\subsection{Proof of Proposition \ref{prop3}}
\label{proof3}
\begin{proof}
Note that since $m(\rho)$, $\mathcal{C}(\rho)$ are independent of the particle $z$, we have
\[\nabla \rho=-{\mathcal{C}}(\rho)^{-1}(z-m(\rho)) \rho\,,\]
\[D^2 \rho=\left(-{\mathcal{C}}(\rho)^{-1}+\left({\mathcal{C}}(\rho)^{-1}(z-m(\rho))\right) \otimes\left({\mathcal{C}}(\rho)^{-1}(z-m(\rho))\right)\right) \rho\,.\]
{Using this we can substitute the Gaussian profile into equation \eqref{linearfp}. We compute the first term on the right hand to obtain}
\[\left(-{\mathcal{C}}(\rho)^{-1}(z-m(\rho)) \cdot \begin{pmatrix}
-p \\
\mathcal{C}_q(\rho)(B^{-1}q-c) + {\gamma}p
\end{pmatrix}+{\gamma}N\right) \rho \triangleq I_1\rho\,.\]
The second term on the right hand side can be written as:
\[\begin{pmatrix}
0 & -\mathcal{C}_q(\rho) \\
\mathcal{C}_q(\rho) & {\gamma}\mathcal{C}_q(\rho)
\end{pmatrix}:\left(-{\mathcal{C}}(\rho)^{-1}+\left({\mathcal{C}}(\rho)^{-1}(z-m(\rho))\right) \otimes\left({\mathcal{C}}(\rho)^{-1}(z-m(\rho))\right)\right) \rho\triangleq (I_2+I_3)\rho \,.\]
For the left hand side of \eqref{linearfp}, note that \[\frac{\mathrm{d}}{\mathrm{d} t} \operatorname{det} {\mathcal{C}}(\rho)=\operatorname{Tr}\left[\operatorname{det} {\mathcal{C}}(\rho) {\mathcal{C}}(\rho)^{-1} \frac{\mathrm{d}}{\mathrm{d} t} {\mathcal{C}}(\rho)\right], \,\, \frac{\mathrm{d}}{\mathrm{d} t} {\mathcal{C}}(\rho)^{-1}=-{\mathcal{C}}(\rho)^{-1}\left(\frac{\mathrm{d}}{\mathrm{d} t} {\mathcal{C}}(\rho)\right) {\mathcal{C}}(\rho)^{-1}\,.\]
{Using the evolution of the mean \eqref{vare} we obtain}
\[\begin{aligned}
\frac{\mathrm{d}}{\mathrm{d} t}\|z-m(\rho)\|_{{\mathcal{C}}(\rho)}^{2}=& 2\left\langle\frac{\mathrm{d}}{\mathrm{d} t}(z-m(\rho)), {\mathcal{C}}(\rho)^{-1}(z-m(\rho))\right\rangle \\
&+\left\langle(z-m(\rho)), \frac{\mathrm{d}}{\mathrm{d} t}\left({\mathcal{C}}(\rho)^{-1}\right)(z-m(\rho))\right\rangle \\
=&2\left\langle \begin{pmatrix}
-m_p(\rho) \\
\mathcal{C}_q(\rho)(B^{-1}m_q(\rho)-c) + {\gamma}m_p(\rho)
\end{pmatrix}, \mathcal{C}(\rho)^{-1}(z-m(\rho))\right\rangle\\&-\left\langle{\mathcal{C}}(\rho)^{-1}(z-m(\rho)), \frac{\mathrm{d}}{\mathrm{d} t}\left({\mathcal{C}}(\rho)\right){\mathcal{C}}(\rho)^{-1}(z-m(\rho))\right\rangle\triangleq 2(I_4-I_5)\,.
\end{aligned}\]
Therefore we can compute the left hand side as:\[\begin{aligned}
\partial_{t} \rho &=\left[-\frac{1}{2}(\operatorname{det} {\mathcal{C}}(\rho))^{-1}\left(\frac{\mathrm{d}}{\mathrm{d} t} \operatorname{det} {\mathcal{C}}(\rho)\right)-\frac{1}{2} \frac{\mathrm{d}}{\mathrm{d} t}\left\|z-m(\rho)\right\|_{{\mathcal{C}}(\rho)}^{2}\right] \rho \\
&=\left[-\frac{1}{2}\operatorname{Tr}\left[{\mathcal{C}}(\rho)^{-1} \frac{\mathrm{d}}{\mathrm{d} t} {\mathcal{C}}(\rho)\right]-I_4+I_5\right] \rho\triangleq (I_6-I_4+I_5)\rho\,.
\end{aligned}\]

In order to show that the Gaussian profile is indeed a solution to \eqref{linearfp}, we only need to show that $I_1+I_2+I_3=-I_4+I_5+I_6$. Note that \[\begin{aligned}I_1+I_4&={\gamma}N+{\mathcal{C}}(\rho)^{-1}(z-m(\rho)) \cdot \begin{pmatrix}
p-m_p(\rho) \\
-\mathcal{C}_q(\rho)B^{-1}(q-m_q(\rho)) - {\gamma}(p-m_p(\rho))
\end{pmatrix}\\&={\gamma}N+{\mathcal{C}}(\rho)^{-1}(z-m(\rho)) \cdot\begin{pmatrix}
0 &I_N \\
-\mathcal{C}_q(\rho)B^{-1} &- {\gamma}I_N
\end{pmatrix} (z-m(\rho))\,.\end{aligned}\]
Defining $\hat{z}=z-m(\rho)$, we collect the terms in the to-be-proven identity $I_1+I_2+I_3=-I_4+I_5+I_6$ as:
\[\begin{aligned}
&\hat{z}^T\left[{\mathcal{C}}(\rho)^{-1}\begin{pmatrix}
0 &I_N \\
-\mathcal{C}_q(\rho)B^{-1} &- {\gamma}I_N
\end{pmatrix}+{\mathcal{C}}(\rho)^{-1} \left(\begin{pmatrix}
0 & -\mathcal{C}_q(\rho) \\
\mathcal{C}_q(\rho) & {\gamma}\mathcal{C}_q(\rho)
\end{pmatrix}-\frac{1}{2}\frac{\mathrm{d}}{\mathrm{d} t}\left({\mathcal{C}}(\rho)\right)\right){\mathcal{C}}(\rho)^{-1}\right]\hat{z}\\+&{\gamma}N-\operatorname{Tr}\left[\begin{pmatrix}
0 & -\mathcal{C}_q(\rho) \\
\mathcal{C}_q(\rho) & {\gamma}\mathcal{C}_q(\rho)
\end{pmatrix}{\mathcal{C}}(\rho)^{-1}\right]+\frac{1}{2}\operatorname{Tr}\left[\frac{\mathrm{d}}{\mathrm{d} t} {\mathcal{C}}(\rho){\mathcal{C}}(\rho)^{-1} \right]\triangleq \hat{z}^T M_1\hat{z}+M_2=0\,.
\end{aligned}\]
We only need to show that $M_1+M_1^T=M_2=0$.

In fact, we compute ${\mathcal{C}}(\rho)(M_1+M_1^T){\mathcal{C}}(\rho)$ to be:
\[\begin{pmatrix}
0 &I_N \\
-\mathcal{C}_q(\rho)B^{-1} &- {\gamma}I_N
\end{pmatrix}{\mathcal{C}}(\rho)+{\mathcal{C}}(\rho)\begin{pmatrix}
0 &-(B^{-1})^T\mathcal{C}_q(\rho) \\
I_N&- {\gamma}I_N
\end{pmatrix}+\begin{pmatrix}
0 & 0 \\
0 & 2{\gamma}\mathcal{C}_q(\rho)
\end{pmatrix}-\frac{\mathrm{d}}{\mathrm{d} t}\left({\mathcal{C}}(\rho)\right)\,,\]
which equals zero upon blockwise computation via \eqref{vare}. Thus $M_1+M_1^T=0$.
Moreover, note that $M_2=-\operatorname{Tr}\left[\mathcal{C}(\rho)M_1\right]$: $M_1+M_1^T=0$ implies $\mathcal{C}(\rho)(M_1+M_1^T)=0$ and taking the trace on both sides {leads to the
conclusion} that $M_2=0$.
Therefore the Gaussian profile indeed solves the Fokker-Planck equation \eqref{linearfp}.
\end{proof}

\end{document}